\documentclass[10pt]{amsart}
\usepackage{graphicx,amssymb,amsfonts,amsmath,amsthm,newlfont}
\usepackage{epsfig}
\usepackage{axodraw4j}
\usepackage{color,url}
\usepackage[all,2cell]{xy} \UseAllTwocells \SilentMatrices

\usepackage{axodraw4j}
\usepackage{pstricks}

\newtheorem{thm}{Theorem}%[section]
\newtheorem{cor}[thm]{Corollary}
\newtheorem{lem}[thm]{Lemma}
\newtheorem{prop}[thm]{Proposition}

\theoremstyle{definition}
\newtheorem{defn}[thm]{Definition}

\newtheorem{notat}[thm]{Notation}
\newtheorem{conj}[thm]{Conjecture}

\newtheorem{ex}[thm]{Examples}
\newtheorem{example}[thm]{Example}

\theoremstyle{remark}
\newtheorem{rem}[thm]{Remark}
\newcommand{\Z}{\mathbb Z}
\newcommand{\C}{\mathbb C}
\newcommand{\Pro}{\mathbb P}

\newcommand{\R}{\mathbb R}
\newcommand{\N}{\mathbb N}
\newcommand{\Q}{\mathbb Q}

\newcommand{\Reg}{\mathrm{Reg}}

\newcommand{\q}{/\!\!/}
\newcommand{\adj}{\mathrm{adj} \,}

\newcommand{\ovj}{{}_{\bullet}}
\newcommand{\tvj}{{}^{\bullet}_{\bullet}}

\newcommand{\To}{\longrightarrow}

\newcommand{\Mod}{\mathfrak{M}}

\newcommand{\Sym}{\mathfrak{S}}

\newcommand{\Or}{\mathcal{O}}

\newcommand{\A}{\mathbb{A}}

\newcommand{\F}{\mathcal{F}}

\newcommand{\gr}{\mathrm{gr}}

\newcommand{\minor}{\preccurlyeq}

\newcommand{\Pe}{\mathfrak{P}}
\newcommand{\Pel}{\mathfrak{P}^{\log}}

\newcommand{\G}{\mathbb{G}}
\newcommand{\Spec}{\mathrm{Spec\,\,}}
\newcommand{\MZV}{\mathcal{Z}}

\newcommand{\E}{\mathcal{E}}

\newcommand{\BP}{\mathcal{P}}
\newcommand{\Mot}{\mathcal{M}}

\begin{document}
\title{On the periods of some Feynman integrals}
\author{Francis Brown}
%\date{14 July 2008}
 \maketitle

\vspace{-0.3in}
\begin{abstract} We study the related questions: $(i)$ when Feynman amplitudes in massless $\phi^4$ theory evaluate to multiple zeta values,
and $(ii)$ when their underlying motives are mixed Tate.
More generally, by considering configurations of singular hypersurfaces  which  fiber linearly over each other, we deduce sufficient geometric and combinatorial criteria on Feynman graphs for both $(i)$ and $(ii)$  to hold. These criteria hold for some infinite classes of graphs which essentially contain all cases previously known to physicists. 
 Calabi-Yau varieties appear at the  point where these criteria fail.
\end{abstract}

\vspace{0.1in}

\subsection{Background}
Let $G$ be a connected graph. To each edge $e\in E_G$ associate a variable $\alpha_e$, known as a Schwinger parameter,  and consider the graph polynomial
$$\Psi_G = \sum_{T\subset G} \prod_{e \notin E(T)} \alpha_e\ ,$$
where the sum is over all spanning trees  $T$ of $G$. It is homogeneous of degree $h_1(G)$, the loop  number of $G$. 
When  $G$ satisfies $|E_G|=2h_1(G)$ and  is primitive for the  Connes-Kreimer coproduct, one can show that the Feynman integral:
\begin{equation}\label{introIG}  I_G= \int_0^{\infty}\ldots \int_0^{\infty} {\prod_{e\in E_G}  d\alpha_e \over \Psi_G^2} \,\delta\big(\sum_{e\in E_G} \alpha_e-1\big) \ ,
\end{equation} 
converges  and defines a real number.   This quantity is  renormalization-scheme independent, and 
it is  a general fact that Feynman integrals in tensor quantum field theories can be reduced to scalar integrals at the cost of modifying
 only the numerator, and  not the denominator, of these integrals.
Thus the  residues $(\ref{introIG})$, and their variants with  numerators,  capture much of the number-theoretic content of  any massless, single-scale quantum field theory  in four  dimensions.

 Broadhurst and Kreimer \cite{BK}, and  recently Schnetz \cite{SchnetzCensus} have computed the residues $I_G$ by a  variety of impressive numerical and analytic methods for all such graphs in $\phi^4$ theory up to six loops, and for some graphs up to nine loops. They found in all identifiable cases that $I_G$ is a rational linear combination of multiple zeta values 
$$\zeta(n_1,\ldots, n_r) = \sum_{1\leq k_1<\ldots <k_r} { 1 \over k_1^{n_1}\ldots k_r^{n_r}} \quad n_i \in \N, n_r \geq 2\ , $$
which are  periods of the mixed Tate motives  of the unipotent fundamental group of $\Pro^1\backslash \{0,1,\infty\}$. 
For a long time,  a widely held view was  consequently  that all such residues $(\ref{introIG})$ should  evaluate to multiple zeta values, and our original goal was to prove this. However, the
methods of this paper have recently led to counter-examples which make this very unlikely, even for planar graphs \cite{BrSch}.   Therefore the  question we seek to address  is to determine  for which graphs $G$ is $I_G$ a combination of multiple zeta values. This is the analytic side of the problem.

The algebro-geometric  approach to this problem  was initiated by Bloch, Esnault and Kreimer in the foundational paper \cite{B-E-K}, where they  
interpreted  $I_G$  as the period of a mixed Hodge structure, as follows. Let $X_G \subset \Pro^{E_G-1}$ denote the singular graph hypersurface defined by the 
zero locus of $\Psi_G$, and let $\Delta \subset \Pro^{E_G-1}$ denote the union of  the coordinate hyperplanes. Since $X_G$ and $\Delta$ do not cross normally, they  constructed a  blow-up $P \rightarrow \Pro^{E_G-1}$, and defined the `graph motive' to be:
\begin{equation}\label{BEKmotive}
m_G=H^{ E_G-1} ( P \backslash Y, B \backslash (B \cap Y))\ ,\end{equation}
where $B$ denotes the total transform of $\Delta$ and $Y$  the strict transform of $X_G$. %Note that $m_G$ only depends on the denominator of   $(\ref{introIG})$
They proved that the residue $I_G$ is a period of the  mixed Hodge structure underlying $m_G$. % It is important that the motive only depends on the denominator of . 
The algebraic version of the question raised earlier is to determine for which 
 graphs  $G$ is  $m_G$   (or some piece of $m_G$ which  carries the period $I_G$) of mixed Tate type. Note that this would  not quite suffice to conclude that $I_G$ evaluates to multiple zeta values. In  \cite{B-E-K}, \cite{Doryn}, some progress  was made in computing a certain graded piece of $(\ref{BEKmotive})$ for some special families of graphs (namely, the wheels and zig-zags).
 %
%They also studied in detail  the wheels with spokes graphs, and managed to  compute  a graded piece of  in this case.  This has been extended and generalized in the  recent thesis of D. %Doryn \cite{Doryn}. 

A  more accessible problem is instead to consider the point-counting function 
$$R_G : q \mapsto  |X_G(\mathbb{F}_q)|$$
which to any graph $G$ associates the number of points of $X_G$ over finite fields $\mathbb{F}_q$, as $q$ ranges over all prime powers.
Motivated by the philosophy of mixed Tate motives, Kontsevich informally conjectured in 1997  that $R_G$ should be polynomial in $q$ for all graphs. This was studied by Stanley \cite{Sta}, and Stembridge \cite{Stem} proved  by computer that it  holds for all  graphs up to 12 edges.  Belkale and Brosnan subsequently showed in \cite{BB} that the conjecture is in general false, and moreover that the functions $R_G$ are of general type. This result implied that the cohomology of $X_G$ can be very complicated, but
did not rule out the possibility  that the $m_G$ are always mixed Tate, still less that the $I_G$ evaluate to multiple zetas (especially if the graphs $G$ are constrained to lie in $\phi^4$).
 Indeed, the first `potentially-Tate'  counter-examples to Konstevich's conjecture were only discovered 
very recently (\cite{SchnetzFq}, \cite{Doryn}).

In this paper,   we show for certain infinite families of graphs that  a variant of the graph motives $(\ref{BEKmotive})$ are mixed Tate, and that
their periods $I_G$ evaluate to multiple zeta values. Our methods also have consequences for the point-counting problem, which are addressed in a separate paper \cite{BrSch}.
 The idea is that for some graphs, the complement of the  graph hypersurfaces can be
  related to  moduli spaces  of curves of genus $0$  via  a sequence of linear fibrations. The existence of such linear fibrations results from the vanishing of certain graph invariants.
When these invariants do not vanish, we can extract  Calabi-Yau varieties from the graph hypersurfaces which yield explicit (modular) counter-examples  \cite{BrSch} to  Kontsevich's conjecture. 

\subsection{Overview}% Since they are not equivalent, we were forced to address both the algebraic and analytic aspects of the problem separately.
On the analytic side,  the main idea is to  choose an order on the $N=|E_G|$ edges of $G$, and consider the partial Feynman integrals
\begin{equation} \label{intropartialI} I^i_G(\alpha_{i+1},\ldots, \alpha_N)= \int_0^{\infty}\ldots  \int_0^{\infty} {1\over \Psi_G^2} d\alpha_{1}\ldots d\alpha_i\ .
\end{equation}
These are multivalued functions of the Schwinger parameters $\alpha_{i+1},\ldots, \alpha_N$  with singularities along a certain discriminant locus we denote $L_i$. 
We call these the Landau varieties of $G$, by analogy with  the  case of Feynman integrals which depend on masses and external momenta.
For certain graphs,  the varieties $L_i$ can be computed using stratified Morse theory, and   the  monodromy of the functions $I^i_G$ controlled. When the monodromy is unipotent, the $I^i_G$ are  periods of unipotent fundamental groups. The functions 
$(\ref{intropartialI})$ can therefore be expressed in terms of multiple polylogarithms, and this explains the appearance of 
 multiple zeta values.

The algebraic interpretation is to consider the graph hypersurface $X_G$    in  $(\Pro^1)^N$, with coordinates $\alpha_1,\ldots, \alpha_N$,  along with the  hypercube $B=\bigcup_{i=1}^N \{\alpha_i=0,\infty\}$. Together they define a stratification on $(\Pro^1)^N$. Consider the projection:
$$
\xymatrix{ (\Pro^1)^N   \ar[d]^{\pi_i} & (\alpha_1,\ldots,\alpha_N) \ar[d] \\
(\Pro^1)^{N-i} &  (\alpha_{i+1},\ldots,\alpha_N)}
$$
and let $L_i\subset (\Pro^1)^{N-i}$ denote the %codimension 1 part of the
(reduced)  discriminant locus. Thus  $L_i$ is the smallest   variety such that $\pi_i:(\Pro^1)^N\backslash \pi_{i}^{-1}(L_i) \rightarrow (\Pro^1)^{N-i} \backslash L_i$ is a
locally trivial map of stratified varieties, and therefore has topologically constant fibers. When all  components of $L_i$ are linear in one of the Schwinger parameters,
one can show that the fundamental group of the base acts unipotently on the relative cohomology of the fibers. By  induction, this implies that a variant of $(\ref{BEKmotive})$
is mixed Tate.

Most of this paper is therefore devoted to computing the Landau varieties  associated to an edge-ordered graph, and relating them  to its combinatorics.

% argument is obstructed whenever  the $L_i$ contain non-linear components, and these can interpreted combinatorially. Moduli spaces.

\subsection{Plan of the paper}

In $\S1$ we recall some basic  concepts from graph theory, and define the notion of vertex width (\S\ref{sectvertexwidth}). A connected graph $G$ has \emph{vertex width} at most $n$ if there exists a
filtration of subgraphs $G_i\subset G$,  where $G_i$ has $i$ edges: 
$$\emptyset  \subset G_1 \subset \ldots  \subset G_{N-1}  \subset G_N= G$$
such that  the number of vertices in the intersection $G_i \cap (G\backslash G_i)$ is at most $n$, for all $i$.  We then show that `physical' graphs (i.e., primitive divergent graphs in $\phi^4$ theory) are in fact completely general from the point of view of their graph minors.

Section $2$ is a detailed study of certain polynomials related to graphs.  From the matrix-tree theorem, the graph polynomial can be expressed as a determinant
$$\Psi_G = \det M_G$$
where $M_G$ is obtained from  the incidence matrix of $G$. We define a related set of polynomials (`Dodgson polynomials'), indexed by  any sets of edges $I,J,K$ of $G$, where $|I|=|J|$, 
and denoted by  $\Psi^{I,J}_{G,K}$. These are the determinants of matrix minors of $M_G$ obtained by deleting rows and columns. The key to understanding the periods of Feynman integrals rests in the algebraic relations between the $\Psi^{I,J}_{G,K}$, which are of two types. The first type are completely general identities (due to Dodgson, Jacobi, Pl\"ucker,\ldots) relating the minors of matrices, and the second type depend on  the combinatorics of $G$, namely the existence of cycles and corollas. 

In $\S3$, we prove some  properties of Feynman integrals under simple operations of graphs using the parametric representation. The main new  result is that there is a natural definition
of the set of periods of any (not necessarily convergent) graph, and that this set is minor monotone.   In particular, it follows from well-known results in graph minor theory that the set of graphs whose periods lie in a fixed ring (e.g., the ring of multiple zeta values), is determined by a finite set of forbidden minors, i.e., a finite number of `critical counterexamples'.  
This gives some insight into the structural properties of graphs whose periods are  multiple zeta values.

In $\S4$ we study the geometry of the graph hypersurface $X_G$, viewed as a subvariety of $(\Pro^1)^N$, and its intersections with the coordinate hypercube. The reason for this choice of ambient space is because the Schwinger parametrization is naturally adapted to a cubical representation, and  this point of view preserves  the symmetry between contracting and deleting edges in graphs.  We explain how to blow up linear subspaces contained in $(\Pro^1)^N$ to obtain a divisor which is normal crossing in the neighbourhood of the hypercube $[0,\infty]^N$,  in much the same way as \cite{B-E-K}, and show that $I_G$  (in the case when $G$ is primitively divergent) is a period of the corresponding mixed Hodge structure, or `motive', which we denote by $\Mot_G$.

In $\S5$ we recall some definitions from stratified Morse theory, and define the Landau varieties $L(G, \pi_i)$  of a graph $G$, relative to a projection $\pi_i:(\Pro^1)^N\rightarrow (\Pro^1)^{N-i}$.
We show, by a standard argument,  that the partial  Feynman integrals are multivalued functions with singularities along the $L(G,\pi_i)$. Unfortunately, discriminant varieties are in general  difficult to compute, and the standard methods  are  ill-adapted to the case of Feynman graphs, which are very degenerate.   In particular,  the  divisors $L(G,\pi_i)$ have large numbers of components. 

To this end, in $\S6$ we develop an inductive method to compute an upper bound for the Landau varieties $L(G,\pi_i)$ by computing iterated resultants of polynomials, and removing spurious components. The method  applies to any family of hypersurfaces $S\subset (\Pro^1)^N$.  Such a family is defined to be 
\emph{linearly reducible}   if, for all $i$, all components of the Landau varieties $L(S,\pi_i)\subset (\Pro^1)^{N-i}$ are of degree at most 1 in one coordinate. This notion depends on a choice of ordering on the coordinates.

In $\S7$ we apply this method to the graph polynomials $\Psi_G$. Using all the identities obtained in $\S2$ we show  that  generically,  the Landau varieties  $L(G,\pi_i)$ always contain  the zero locus of the Dodgson polynomials $\Psi^{I,J}_{G,K}=0$ where $I\cup J\cup K = \{1,\ldots, i\}$.  We define a graph $G$ to be  of \emph{matrix type} if the converse is true: i.e.,
\begin{equation}\label{introLub}
L(G,\pi_i) \subseteq \{\Psi^{I,J}_{G,K}=0:   I\cup J\cup K =\{1,\ldots, i\}\}\ \hbox{ for all } i \ . \end{equation}
This is the most accessible   family of graphs one can define, and in particular,  they are linearly reducible since each $\Psi^{I,J}_{G,K}$ is of degree at most one in all its variables.
Unfortunately, not all graphs are of matrix type. For $i\leq 4$,   condition $(\ref{introLub})$ actually holds for all graphs,   but  in the  general case, $L(G,\pi_5)$ contains a component given by the zero locus of a new polynomial
\begin{equation}\label{intro5inv}
{}^5\Psi_G(i,j,k,l,m)\ , \end{equation}
which we call the \emph{5-invariant} of a set of five edges $i,j,k,l,m$ in a graph. For general graphs, this is an irreducible polynomial which is quadratic in each Schwinger parameter and gives the first obstruction for a graph to  be of matrix type. But under certain combinatorial conditions,  it factorizes into a product of Dodgson polynomials and $(\ref{introLub})$ still holds for $i=5$. In particular, by studying the effect of triangles and 3-valent vertices in a graph on the degeneration of the 5-invariant, we show that it (and all higher possible obstructions) vanish when $G$ has small vertex-width. This provides an infinite supply of non-trivial graphs of matrix type.
 
 \begin{thm} If $G$ has vertex-width $\leq 3$, then $G$ is of matrix type.  
 \end{thm}

In $\S8$ we return to the more general situation of an arbitrary configuration of hypersurfaces $S$, and consider what happens when it is linearly reducible. 
We have a sequence of Landau varieties $L(S,\pi_i) \subset (\Pro^1)^{N-i}$, and  maps  $\pi_{i+1}: (\Pro^1)^{N-i} \rightarrow  (\Pro^1)^{N-i-1}$ which project out each successive coordinate.  The linearity assumption implies that each projection can be 
completed to a commutative diagram mapping to the universal curve of genus $0$ and $m_i$ marked points, for some $m_i$:
\begin{equation} \label{introMonsquare}
\xymatrix{ (\Pro^1)^{N-i}\backslash L(S,\pi_i)  \ar[r]^{\qquad \rho} \ar[d]^{\pi_{i+1}} & \overline{\Mod}_{0,m_i+1}  \ar[d]^{f} \\
 (\Pro^1)^{N-i-1} \backslash L(S,\pi_{i+1}) \ar[r]^{\qquad \overline{\rho}} & \overline{\Mod}_{0,m_i}}
\end{equation}

Thus we relate any linearly reducible configuration to  moduli spaces, and do this explicitly for graphs of matrix type in \S\ref{sectMonGmatrixtype}.
Putting these diagrams together for different $i$ yields connecting maps between certain (fiber products) of moduli spaces. The induced maps on the fundamental groups, plus the integers $m_i$, completely encode all   the data about the periods.
  One application we have in mind  is   for the algorithmic computation of the periods of Feynman integrals: rather than study the geometry and function theory of each individual graph, one needs only implement the function theory on the moduli spaces $\Mod_{0,n}$ once and for all. 

In $\S9$ we explain how to compute any period integral whose singularities are given by a set of linearly reducible hypersurfaces. For this, refer to the diagram $(\ref{introMonsquare})$ above.
The partial Feynman integrals naturally live on the spaces on the left-hand side, but using the maps $\rho$, we show that they can be pulled back to the moduli spaces $\Mod_{0,m_i+1}$ and  computed by working entirely in the de Rham fundamental group of the $\Mod_{0,n}$'s,  which we studied in \cite{BrENS}. In particular,  we proved in that the periods one obtains in this way  are multiple zeta values. Thus we deduce:

\begin{thm}   Let $G$ be positive and  of matrix type. Then $I_G$ is a rational linear combination of multiple zeta values of weight $\leq N-3$, where $N=|E_G|$. 
\end{thm}
The positivity condition means that the coefficients of the polynomials $\Psi^{I,J}_{G,K}$ which occur in $L(G,\pi_i)$ should be positive, but for a general graph of matrix type one should obtain multiple polylogarithms evaluated at roots of unity $\pm 1$. We show (theorem \ref{thmvw3positive}) that if $G$ has vertex width $\leq 3$, then it is positive of matrix type.
The method gives an algorithm for the computation of any terms in the $\varepsilon$-expansion of such graphs to all orders, and with arbitrary dressings \cite{BrCMP}.

In $\S10$ we define an iterative procedure (the \emph{denominator reduction}) to compute the denominators in the partial Feynman integrals. Thus, for a linearly reducible graph, the 
partial Feynman integral at the $n^{th}$ stage of integration is of the form
$$ \hbox{ Multiple polylogarithms  in the $\Psi^{I,J}_{G,K}$ } \over P_n\ ,$$
where $P_n$ is a polynomial  which can be computed very easily. The polynomials $P_n$ give the deepest contributions to the Landau varieties, and the first place that non-Tate phenomena appear.  We explain how the denominator reduction gives a mechanism for  proving  if the residue of a graph $I_G$ has transcendental weight drop, and these methods have subsequently been used   in \cite{WD} to give a complete combinatorial explanation of all known weight-drop phenomena  in $\phi^4$ theory.

In $\S11$,  we use the iterated fibrations of $\S7,8$ to prove that: 

\begin{thm} If $G$ is linearly reducible then the motive of $G$ is mixed Tate.
\end{thm}
In particular, graphs of vertex width 3 are mixed Tate.\footnote{We have recently given a constructive  proof  in \cite{BrSch} that graph hypersurfaces of vertex $\leq 3$ have polynomial point counts over finite fields  using the methods of $\S\S2,7$.}
The idea is that a composition of linear fibrations has global unipotent monodromy, and so one can compute the motive inductively using the Leray spectral sequence. However, in the present paper we do not discuss the rather essential  question of ramification, and only touch on the issue of framings in \S\ref{sectDRrevisited}.

Finally, in $\S12$, we gather examples and  counter-examples of critical graphs at various loop orders which are related to the discussions of $\S\S1,2,7,10$.
  Using  the denominator reduction,  we  search for examples which lie just outside the scope of the present method, and find that the first serious obstructions occur at 8 loops. Since writing the first versions of this paper, we have shown in \cite{BrSch}   that one of these 8 loop examples  contains a singular K3 surface which is not Tate. 
  Thus  the first non-Tate examples occur as soon as our criteria for linear reducibility fail. Furthermore, there exists a non-Tate example with vertex width 4,
  which shows that   the   $vw(G)\leq 3$ condition cannot be improved.
  In $\S\ref{sectCY}$, we explain how the denominator reduction associates  a Calabi-Yau variety to any ordered graph,  and   list   the  modular forms
  which conjecturally correspond to them, 
  for all graphs up to 8 loops. They are all expressible in terms of the  Dedekind eta function. 
   We conclude the paper with some remarks on classification, and a list of open problems. 
\\

The paper  \cite{BrENS} may serve  as an introduction to this one, and in particular contains a worked example in the case of the wheel with 3 spokes.

\subsection{Outlook} Here follow some general remarks on possible future developments.

In this paper,  we only considered massless, subdivergence-free graphs, but we expect the same methods to work for  
 graphs with subdivergences and trivial dependence on a single external momentum. In the case where there are several kinematic variables,
 we expect the present method to prove that the Feynman integrals are polylogarithmic functions of the external kinematics, but only at much lower loop orders (examples of elliptic integrals are known to occur).
  
Secondly,  although $\phi^4$ theory is considered to be `unphysical', it is the universal scalar quantum field theory from the point of view of its periods. Thus the methods of this paper give  upper bounds for the types of numbers which can occur in any such scalar  massless quantum field theory. In practice, the periods often tend to cancel when one sums over all graphs in the perturbative expansion at a given loop order, but the new phenomena occurring at 8 loops in $\phi^4$  may suggest  that the transcendentality is merely postponed to higher loop orders. We also believe that one might gain an understanding of  the cancellation phenomenon by studying the interaction between the motivic Galois group acting on the periods of $\phi^4$, and the possible symmetries of a given  quantum field  theory. 
   
   Finally, we  believe that the motivic approach  to the Feynman amplitudes may also have consequences for the possible convergence of the perturbative series of primitive graphs.    In particular, it follows from \cite{B-E-K} that there is a well-defined Hodge and weight filtration on the perturbative series of primitively-divergent graphs, and we expect from $\S12$,   \cite{BrSch} and  \cite{WD},   that these filtrations are  highly  non-trivial. It therefore makes sense to sum this particular perturbative series  according to its Hodge and weight filtrations, which might improve its convergence.
    \\

%\begin{rem}
%Note that one might have expected the argument to go the other way round: first prove that $M_G$ is mixed Tate, and then deduce information about the periods of $M_G$ using the full %machinery of the motivic theory. However, this is at present unavailable since it is not  currently known whether the periods of (unramified) mixed Tate motives are multiple zeta values (let %alone for general mixed Tate motives).. The methods of this paper sidesteps these issues by essentially showing that the Feynman motives lie in the subcategory of motives coming from the %moduli space $\Mod_{0,n}$, for which the periods are known. Integrality, 
%\end{rem}

Very many thanks  to Spencer Bloch and Dirk Kreimer   for numerous  discussions and many helpful suggestions and  comments. Thanks also to David Broadhurst, Jean-fran\c{c}ois Burnol,  
H\'el\`ene Esnault, Chris Peters, Oliver Schnetz, and  Karen Yeats for their interest and remarks.
Most of this paper was written up  during stays at the IHES in 2008 and 2009 and at the ESI, Vienna, in 2009.

\section{Preliminaries on  Graphs}

\subsection{Terminology} Throughout, a \emph{graph} $G$ will denote an undirected multigraph with no external edges. Such a graph $G$ can be represented
by a pair $(V_G,E_G)$ where $V_G$ is a finite set (the \emph{vertices} of $G$), and $E_G$ is a finite set of unordered pairs $\{v_i,v_j\}$ of elements of $V_G$ (the \emph{edges} of $G$), which may occur with multiplicity. % Automorphisms.

A \emph{subgraph} $\gamma$ of $G$, denoted $\gamma \subseteq G$, is defined by specifying   a subset of edges $E'\subseteq E_G$, and setting $\gamma=(V',E')$, where $V'\subset V_G$ is the set of vertices which occur in $E'$. A \emph{tadpole} (following the terminology in physics) is an edge of the form $\{v_i,v_i\}$. The \emph{degree} or \emph{valency}  of a vertex $v$ is the number of edges in $E$ which are incident to $v$. A graph $G$ is said to be in $\phi^4$ if every vertex has degree at most 4.

We  write $e_G=|E_G|$, $v_G=|V_G|$, $h_G=h_1(G)$, and $h_0(G)$  for the number of edges, vertices, independent cycles (the \emph{loop number}), and connected components of a graph $G$. They are related by Euler's formula: $h_G-h_0(G) = e_G -v_G$. 
%Following the physics terminology, we call $h_G$ the \emph{loop number} of $G$.

\begin{defn} A graph $G$ is said to be  \emph{primitive divergent} if     $e_G=2h_G$, and for all strict subgraphs $\gamma \subsetneq G$, $ e_\gamma >2 h_\gamma$.
\end{defn}
%It follows from Euler's formula that if every vertex of  $G$ has degree greater than 2, then it has exactly 4 vertices of degree 3, and all the rest are of degree 4.

Given  two disjoint sets of edges $C,D \subset E_G$,  we write
$\gamma = G \backslash D\q C$ for the graph obtained by deleting the edges in $D$, and contracting
the edges in $C$  ({\it i.e.}, removing each edge in $C$ and identifying its endpoints). 
Since the operations of deleting and contracting disjoint edges commute, $\gamma$ is well-defined.
We need the convention that the contraction of any $C$ such that $h_C>0$ is the empty graph.
 Any graph $\gamma$
obtained from $G$ by contracting and deleting edges is called a \emph{minor} of $G$ and will be denoted $\gamma \minor G$.

The \emph{complete graph} $K_n$ for $n\geq 2$, is the graph with  $V_{K_n}=\{v_1,\ldots, v_n\}$ and $E_{K_n}=\{\{v_i,v_j\}$ for all $1\leq i<j\leq n\}$. The \emph{complete bipartite graph} $K_{p,q}$ is the graph with vertex set $\{v_1,\ldots, v_p, w_1, \ldots, w_q\}$ and edges $\{v_i, w_j\}$ for $1\leq i,j \leq q$.

\subsection{Standard operations on graphs} \label{sectOperationsongraphs}
The theory of electrical circuits suggests the five basic operations  pictured below:
\begin{figure}[h!]
  \begin{center}
%    \leavevmode
    \epsfxsize=6.0cm \epsfbox{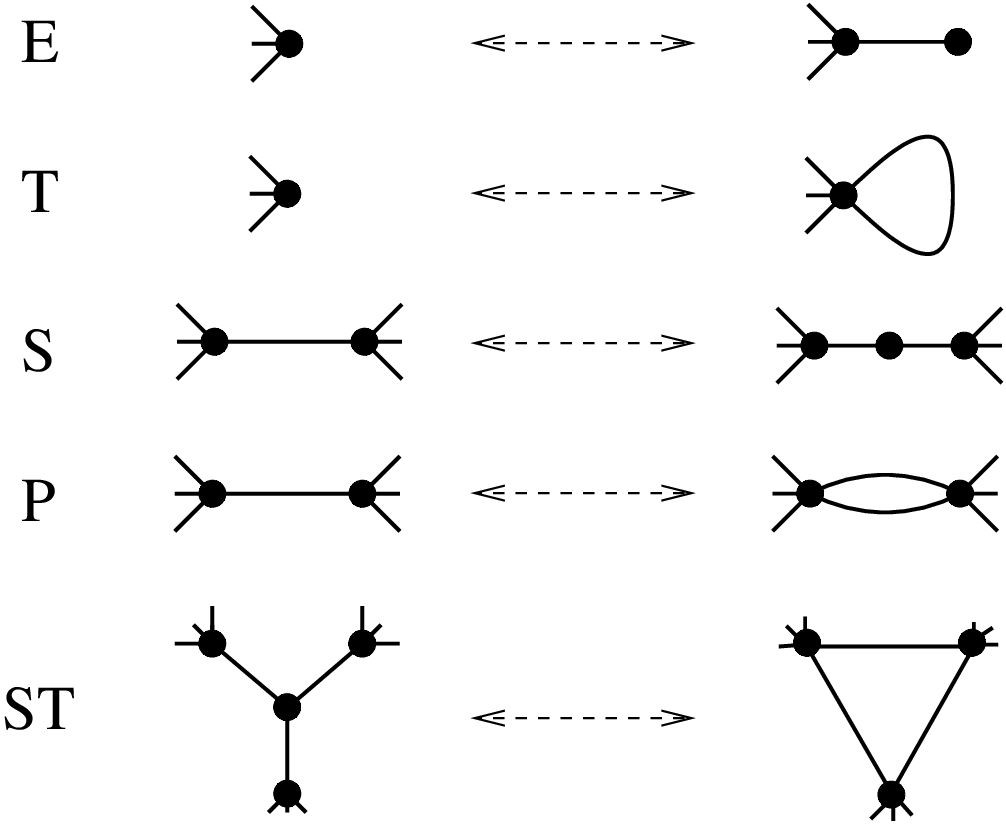}
  \label{6ops}
  \end{center}
%\caption{The five basic operations on graphs  are: Core, Loop, Series, Parallel, and Star-Triangle (or $\Delta-Y$) operations.}
\end{figure}

\noindent
%We denote these by:
They are: External leg, Tadpole, Series, Parallel, and Star-Triangle %(or $\Delta-Y$) 
operations.
A graph is said to be \emph{1PI} (1-particle irreducible) if  %the deletion of any edge causes the number of loops to drop, {\it i.e.},
 $h_{G\backslash e}< h_G$
for all edges $e\in E_G$.
% Every graph $G$ has a maximal 1PI subgraph $G'\subseteq G$ obtained by applying operation $C$ repeatedly.

 %In the contrary case, it follows from the Euler formula that removing $e$ must disconnect $G$ (it is  \emph{1-particle irreducible} or 1PI).

\begin{defn} We define the \emph{simplification} of a graph $G$ to be  the smallest minor $G'\minor G$ 
which is obtained from $G$ by applying operations $E,T,S,P$ above. We say that a graph $G$ is \emph{simple} if it 
is equal to its simplification.
\end{defn}

%\subsection{Joining graphs}
Now let $G_1$ and $G_2$ denote two connected graphs, and let $v_i\in V_{G_i}$ for $i=1,2.$ The \emph{one vertex join} $G_1\ovj G_2$ of $G_1$ and $G_2$ 
%which we denote $G_1\ovj G_2$,  
is  the graph obtained by gluing $G_1$ and $G_2$
together by identifying $v_1$ and $v_2$.
%Likewise, let $G_1$ and $G_2$ denote two connected graphs, and 
Now let  
$v_i\neq w_i$ be two vertices in $G_i$ for $i=1,2$ which are connected by a single edge $e_i$. A \emph{two vertex join} of $G_1$ and $G_2$, which we denote by $G_1\tvj G_2$, is obtained by identifying $v_1$ with $v_2$, and $w_1$ with $w_2$, and deleting the edges $e_1$, $e_2$ (see below). We have $h_{G_1 \tvj G_2}= h_{G_1}+h_{G_2}-1$.
\begin{center}
\fcolorbox{white}{white}{
  \begin{picture}(246,67) (330,-210)
    \SetWidth{1.5}
    \SetColor{Black}
    \Arc(364,-176)(31,270,630)
    \Arc(443,-176)(31,270,630)
    \Arc(443,-176)(31,270,630)
    \Arc(544,-176)(31,270,630)
 %left fig
    \Line(333,-176)(365,-176)
    \Line(382,-150)(365,-176)
    \Line(365,-176)(380,-203)
    %middle fig
    \Line(475,-176)(444,-176)
    \Line(426,-151)(444,-176)
    \Line(444,-176)(427,-203)
    %rightfig
    \Line(511,-176)(529,-176) %h
    \Line(559,-176)(574,-176) %h
    \Line(544,-144)(529,-176)
    \Line(559,-176)(544,-208)
    \Line(529,-176)(544,-208)
    \Line(559,-176)(544,-145)
    \SetWidth{1.0}
    \Vertex(365,-176){2}
    \Vertex(333,-176){2}
    \Vertex(444,-176){2}
    \Vertex(380,-202.5){2}
    %middle
    \Vertex(382,-151){2}
    \Vertex(427,-202.5){2}
    \Vertex(426,-150.5){2}
    \Vertex(474,-176){2}
    %rightfig
\Vertex(513,-176){2}
    \Vertex(529,-176){2}
    \Vertex(544,-145){2}
    \Vertex(559,-176){2}
    \Vertex(544,-207){2}
    \Vertex(575,-176){2}
    \Text(386,-150)[lb]{\Large{\Black{$v_1$}}}
    \Text(386,-209)[lb]{\Large{\Black{$w_1$}}}
    \Text(412,-150)[lb]{\Large{\Black{$v_2$}}}
    \Text(412,-209)[lb]{\Large{\Black{$w_2$}}}
    \Text(384,-180)[lb]{\Large{\Black{$e_1$}}}
    \Text(418,-180)[lb]{\Large{\Black{$e_2$}}}
    \Text(324,-205)[lb]{\Large{\Black{$G_1$}}}
    \Text(475,-205)[lb]{\Large{\Black{$G_2$}}}
    \Text(490,-180)[lb]{\Large{\Black{$=$}}}
    \Text(403,-182)[lb]{\Large{\Black{$\tvj$}}}
  \end{picture}
}
\end{center}

The \emph{connectivity} $\kappa(G)$ of a graph $G$ is the 
minimal number of vertices required to disconnect $G$.  We say  $G$ is \emph{n-vertex reducible}, or $n$VR, if  $\kappa(G)\leq n$.
The one (resp. two) vertex join of two graphs is 1VR (resp. 2VR).  

\subsection{Forbidden  minors and local minors}
We recall some well-known concepts concerning forbidden minors which will be important for the sequel.
\begin{defn} A set of graphs $S$  is \emph{minor closed} if for all $\Gamma \in S$,
%$$\hbox{for all } \Gamma \in S \ ,  \gamma \minor \Gamma \implies  \gamma \in S \ .$$
$\gamma \minor \Gamma \implies  \gamma \in S.$
\end{defn}

The set of planar graphs or the set of trees are examples of minor closed sets.
%Examples of minor closed sets include the set of  planar graphs, or the set of trees.
One way to define a set of minor-closed graphs is to specify a finite set of forbidden minors $\F=\{\gamma_1,\ldots, \gamma_N\}$.
Then if one defines $\F^c$ to be the set of all graphs which do not contain any of the forbidden minors $\gamma_i$, then $\F^c$ is clearly minor-closed.
The converse is a celebrated theorem due to Robertson and Seymour.

\begin{thm} \label{RobSey} Let $S$ be a minor closed set of graphs. Then there exists a finite set of forbidden minors $\F_S$ such that $S=\F_S^c$.
% is the set of all graphs $\Gamma$
%which satisfy $\gamma\notminor \Gamma$ for all $\gamma\in \F_S$.
\end{thm}

Let $S$ be a  minor-closed set of graphs.  A \emph{critical minor} for $S$ is a graph $G\notin S$ such that every minor of $G$ is in $S$. A minor closed property for graphs is in principle completely determined by its set of critical minors.
For example,  a theorem due to Wagner states that the  forbidden minors for the set of planar graphs is 
$\{ K_{3,3},K_5\}$.  

In this paper, we will require a related notion of local minors.

\begin{defn} \label{defnlocalminor} 
An \emph{ordered graph} $(G,O)$ is a graph $G$ with a total ordering  $O$ on its set of edges $E_G$.  For each $1\leq i\leq e_G$, we obtain  
a partition $G=L_i\cup R_i$ into disjoint subgraphs, where $L_i$ is  defined by the first $i$ edges and $R_i$ by the remaining $e_G-i$ edges. Let
$$V_i(G,O) = V_{L_i}\cap V_{R_i} \ , \quad \hbox{ for } 1\leq i\leq e_G$$
be the set of vertices which are common to $L_i$ and $R_i$.  We define a \emph{local k-minor} of $(G,O)$ to be  any minor of $L_i$ which has exactly $k$ edges. It has a distinguished subset of  vertices which  come from $V_i(G,O)$.
\end{defn}

Thus for any ordered graph $(G,O)$,  we obtain a finite list of local $k$-minors  which reflects the local structure of the graph after contracting and deleting certain edges
according to  $O$.
 Our main results hold for  graphs with forbidden local 5-minors.

\subsection{The vertex width of a graph } \label{sectvertexwidth} One simple way to control the local minors  which can occur in a graph is with the notion of vertex width.
\begin{defn} \label{defnvw} Let $(G,O)$ be  an ordered graph,  and let  $V_i(G,O)$ be %the set of vertices meeting $L_i$ and $R_i$ 
 as in definition \ref{defnlocalminor}.  We define the \emph{vertex width of} $(G,O)$ to be:
$$vw(G,O) = \max_{1\leq i\leq n-1}\, |V_i(G,O)|  \ .$$
Finally, we define the \emph{vertex width} of the graph $G$ to be the smallest vertex width of all  possible orderings $O$ of the set of edges of $G$:
$$vw(G) = \min_{O} \, vw(G,O)  \ .$$
\end{defn}

The connectivity  of $G$ is bounded by its vertex width: $\kappa(G) \leq vw(G).$
It is clear that if $\gamma\minor G$ is a graph minor of $G$, then
$vw(\gamma)\leq vw(G)$.
The set of graphs satisfying $vw(G)\leq n$ is therefore minor closed, for each $n\geq 1$. 
%It is easy to verify that $vw(G) = vw(G')$ where $G'$ is a simplification of $G$. 
%It is also clear that   $vw(G)=1$ if and only if $G$ is a tree.

%It follows that $vw(G) = vw(G')$ if $G'\subseteq G$ is the maximal core subgraph of $G$. 

\begin{example}
For all $m,n\geq 3$, it is easy to check that  $vw(K_n)=n-1$ and $vw(K_{m,n}) = \min\{m,n\}+1$.  The  square lattice $B_n$ with $n^2$ boxes
has vertex width $n+1$.
 In particular, 
the vertex width is unbounded on the set of all planar graphs.
\end{example}

Some simple families of graphs of vertex width $\leq 3$ are given by the wheel with $n$ spokes and  what are known as  zigzag graphs in the  physics literature.
A simple  way to generate  infinite families of such graphs is by subdividing triangles. % see    such as the left-right graphs of $\S??$.

\subsection{Minors of $\phi^4$ graphs} 
Restricting oneself to primitive-divergent graphs in $\phi^4$ theory does not change the minors which can occur, but only delays them to higher loop orders.
%does not affect the periods one obtains.

First observe that if $G$ is simple, connected, and in $\phi^4$, then it only has vertices of degree three and four. It then follows from Euler's formula that the relation $2 h_G= e_G$
is equivalent to $G$ having  exactly four vertices of valency 3.

\begin{lem} Every  complete graph $K_n$   occurs as a minor of a primitive-divergent graph in $\phi^4$ theory (perhaps at higher loop order).\end{lem}

\begin{proof} By embedding $K_n\hookrightarrow K_{n+1}$, we can assume that $n$ is large and odd. Consider the graph with vertices
 arranged in  a square lattice with $n$ elements  $x^i_1,\ldots, x^i_n$, arranged in $n-1$ rows $1\leq i \leq n-1$. Take vertical edges $\{x^i_j, x^{i+1}_j\}$ for all $i,j$, and twisted horizontal edges $\{x^i_j,x^i_{i+j}\}$,
 where the lower indices are taken modulo $n$. Add edges $\{x^1_2, x^{n-1}_2\}$, \ldots, $\{x^1_{n-1}, x^{n-1}_n \}$. In the resulting graph $G_n$, every vertex is 4-valent except for the four corners $x^1_1, x^1_n, x^n_1, x^n_n$, which have valency 3. By contracting the vertical edges in each column, one sees that $G_n$ contains the complete graph $K_n$ as a minor. To see that it is primitive-divergent, it suffices to consider connected strict  subgraphs  $\gamma \subset G_n$ which only have vertices of degree $2,3,4$. By Euler's formula, the equation $2h_{\gamma} \geq  e_{\gamma}$ is equivalent to 
 $2v_2 + v_3 \leq 4$, where $v_2, v_3$ denotes the number of 2 and 3 valent vertices in $\gamma$. Since every remaining vertex  of $\gamma$ is 4-valent,
 $\gamma$ must also contain each of its $G_n$-neighbours. Thus by considering the sets of entire  rows and columns of $G_n$ contained in $\gamma$, one easily verifies that no such divergent subgraphs can exist. Thus  $G_n$ is primitive divergent. 
 
  A different, and more loop-efficient  way to embed $K_n$ as a minor of a primitive-divergent graph in  $\phi^4$ theory, is simply by subdividing its vertices.
\end{proof}

It follows that every graph  occurs as a minor of a primitive-divergent graph in $\phi^4$,  since
 one can subdivide multiple edges using operation $S$, until no more remain. The graph one obtains can then be embedded as a subgraph of the complete graph on its set of vertices, and thence into $\phi^4$ using the previous lemma.
Conversely, every primitive-divergent graph in $\phi^4$ with $h$ loops is a minor of $K_{h+1}$.
\begin{lem} The square lattice $B_n$ occurs as a minor of  a planar primitive-divergent graph in $\phi^4$  with the same number of vertices (i.e., $n^2-1$ loops).
\end{lem}
\begin{proof}
The graph $B_n$ has $n^2$ vertices, each of which has  degree 4 except the ones lying along an outer edge (degree 3) and the four corners, which have degree 2.
According to whether $n$ is odd or even, connected the outer vertices as shown in the two cases below. The resulting graph has exactly 4 vertices of degree 3 (highlighted), and is primitive divergent, by a similar argument to the one given in the previous lemma.
\end{proof}
Since every planar graph is a minor of  a square lattice,
it follows that every planar graph occurs as a minor of a planar primitive-divergent graph in $\phi^4$ (perhaps at higher loop order), 
and conversely, every planar primitive-divergent graph $G$ is a minor of $B_N$, for some sufficiently large $N$, by the same argument.

\begin{figure}[h!]
  \begin{center}
%    \leavevmode
    \epsfxsize=8.0cm \epsfbox{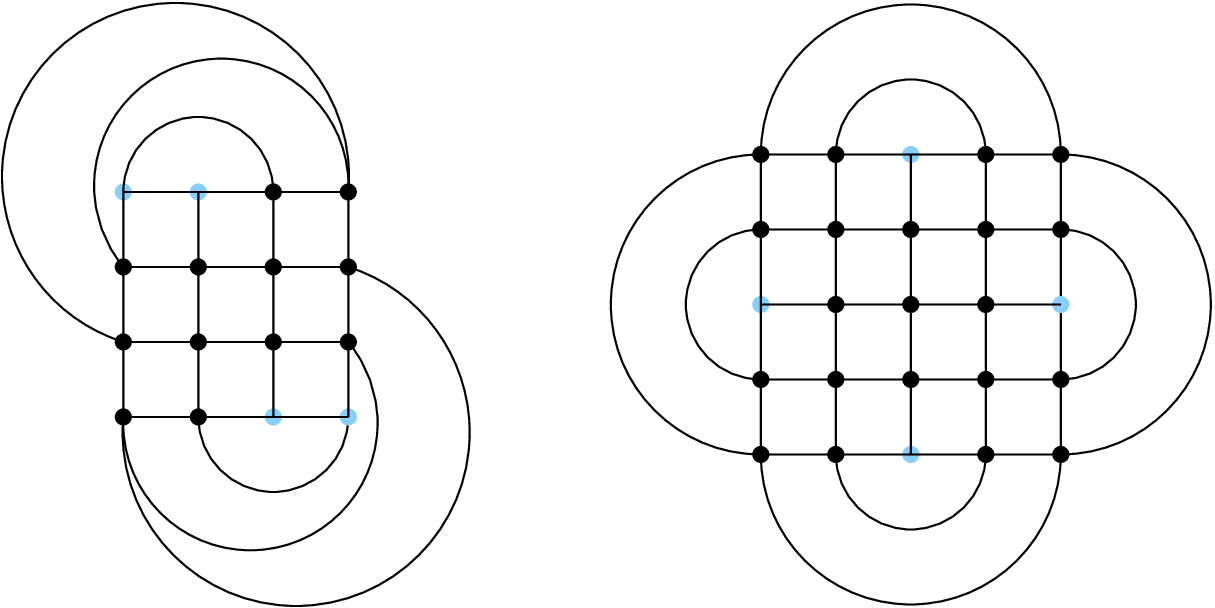}
  \end{center}
\end{figure}

In conclusion, any minor-closed property of graphs  holds for all graphs in $\phi^4$ (planar $\phi^4$)
if and only if it holds for the basic family of graphs $K_n$ (resp. $B_n$), and if and only if it holds for all graphs (resp. all planar graphs).
Thus, from the point of view of minors, there is no such thing as an `unphysical' graph.
% it  suffices to search for counter-examples amongst the two basic families of graphs $K_n$ and $B_n$.
\vspace{0.1in}

\section{Graph Polynomials and Determinantal identities}

\subsection{Reminders on graph polynomials}
%We recall %the definition and 
%some basic properties of graph polynomials. % which we will  refer to frequently.
%Throughout, let $G$ be a connected graph, %and  choose an order on the set of edges of $G$.
To each edge $e$ of a connected graph $G$, we  associate a
variable  $\alpha_e$,  known as the Schwinger coordinate of $e$.

\begin{defn} \label{defgraphpoly}Let $G$ be a connected graph. The \emph{graph polynomial} of $G$ is
\begin{equation} \label{eqngraphpolydef}
\Psi_G = \sum_{T \subseteq G} \prod_{e \notin T} \alpha_e\in \Z[\alpha_e, e\in E_G]\ ,
\end{equation}
where the sum is over all spanning trees $T$ of $G$, {\it i.e.}, all connected subgraphs $T\subseteq G$ such that $T$ contains every
vertex of $G$ and does not contain a loop. If  $G$ is not connected there are no spanning trees, and so we set
$\Psi_G = 0.$
\end{defn}
In the case where $G$ is a tree,   $\Psi_G=1$. By convention, $T$ can be empty, and thus if $G$
is a  tadpole consisting of a single edge $e$, we have  $\Psi_G=\alpha_e$.

\begin{lem} For any edge $e$ of $G$, there is the contraction-deletion formula
\begin{equation}\label{contractdelete}
\Psi_G = \Psi_{G\backslash e } \alpha_e + \Psi_{G\q e}\ .
\end{equation}
\end{lem}
\begin{proof}
%In  $(\ref{eqngraphpolydef})$, 
A spanning tree $T$ of $G$ either contains the edge $e$ or does not contain it. In the first case,
$T\backslash  e $ defines a spanning tree of $G\q e$; in the second, $T$ defines a spanning tree of $G\backslash  e$. The details are left as an
exercise. %, since the result is well-known (see ...)
  Note  that the validity of  $(\ref{contractdelete})$ requires  the contraction of tadpoles to be zero.
\end{proof}

The previous lemma is useful for doing inductions, and gives an alternative definition of $\Psi_G$. For example, if $G$ is connected, it follows from  Euler's formula that 
%\begin{equation}\label{degPsiG}
$\deg \Psi_G = h_G.$ %\ .\end{equation}
The following corollary follows  immediately from the definitions.
\begin{cor}\label{corVanishingsubgraphs} (Vanishing condition) Let $C,D$ be two disjoint   sets of edges of $G$. Then
$\Psi_{G\backslash C\q D}= 0$ if and only if either $G\backslash C$ is disconnected, or $D$ contains a loop. In particular,
$\Psi_{G\backslash C\q D}=0$ implies that either $\Psi_{G\backslash C}=0$ or $\Psi_{G\q D}=0$.

%\begin{eqnarray}
%&\Psi_{G\backslash \{e_1,\ldots,e_k\}}= 0& \hbox{if and only if } \,G\backslash \{e_1,\ldots,e_k\} \hbox{ is disconnected}\ ,\nonumber \\
%&\Psi_{G\q \{e_1,\ldots,e_k\}}= 0& \hbox{if and only if } \{e_1,\ldots, e_k\}  \hbox{ contains a loop} \ .\nonumber
%\end{eqnarray}
\end{cor}
It follows from the next lemma that  the  graph polynomial $\Psi_G$ of a connected graph is a product $\prod_{i=1}^n \Psi_{G_i}$ of 
graph polynomials of 1PI subgraphs.
\begin{lem} If $G_1$ and $G_2$ are connected graphs, then 
$\Psi_{G_1\ovj G_2}=\Psi_{G_1}\Psi_{G_2}.$
\end{lem}
\begin{proof} The map which takes $T_1\subset G_1$ and $T_2\subset G_2$ to $T_1\cup T_2 \subset G$, is a bijection between the set of spanning trees
of $G$ and pairs of spanning trees of $G_1$ and $G_2$.
\end{proof}

Conversely, one can also show that if $G$ is a 1PI graph,  then $\Psi_G$ is  reducible if and only if it is the one-vertex join of two subgraphs $G_1$ and $G_2$.

%\begin{equation} \label{psicoreequalspsi}
%\Psi_{G'}= \Psi_G \ .
%\%end{equation}
% The previous lemma implies the graph polynomial
%$\Psi_G$ is a product of the graph polynomials of its core subgraphs.

%
\begin{lem} Let $G_1$ and $G_2$ be as in \S\ref{sectOperationsongraphs}, then 
\begin{equation} \label{twovertexjoin} \Psi_{G_1\tvj G_2} = \Psi_{G_1\backslash e_1} \Psi_{G_2\q e_2}+  \Psi_{G_1\q e_1} \Psi_{G_2\backslash e_2} \ .
\end{equation}
\end{lem}
\begin{proof} A subgraph $T\subset G_1\tvj G_2$ defines a pair of subgraphs $T_1\subset G_1\backslash e_2$ and $T_2\subset G_2\backslash e_2$. One checks that  $T$ is a spanning tree of $G_1 \tvj G_2$ if and only if either $T_1$ and $T_2 \cup \{e_2\} $ are  spanning trees of $G_1$ and $G_2$ respectively, or $T_1\cup\{e_1\}$ and  $T_2$ are. 
\end{proof}

The following lemma is well-known and follows from $(\ref{contractdelete})$.
\begin{lem} \label{lemsimplif} Let $G$ be a connected graph, and let $G_S$ denote the graph obtained from $G$ by subdividing an edge $e$ into two new edges $e_1,e_2$ in series,
and let $G_P$ denote the graph obtained by replacing $e$ with two new edges $e_1,e_2$ in parallel. Then
$$\Psi_{G_S}= \Psi_G(\alpha_{e_1}+\alpha_{e_2}) \quad \hbox{and}\quad  \Psi_{G_P} = (\alpha_{e_1}+\alpha_{e_2})\Psi_G\big({\alpha_{e_1}\alpha_{e_2} \over
\alpha_{e_1}+\alpha_{e_2}} \big)\ .$$
\end{lem}

The relationship between graph polynomials and the star-triangle operation will be examined later.
Finally, suppose that a graph $G$  has   a planar embedding, and let  $G^\vee$ be  its planar dual. It is well-known  that
\begin{equation} \label{planardual}
\Psi_{G^\vee} ( \alpha_e) = \Psi_{G} ( \alpha_e^{-1}) \prod_{e\in E_G} \alpha_e\ .
\end{equation}

%\newpage

\subsection{The graph matrix} Graph polynomials can be expressed as determinants of a Laplacian matrix
in several  ways. The following presentation is the %simplest and the
 most convenient for our purposes. Let
us fix a connected graph $G$, and  choose an orientation of its edges. For each edge $e$ and each vertex $v$ of $G$ define:
$$\varepsilon_{e,v} = \left\{
                           \begin{array}{ll}
                             1, & \hbox{if } \quad s(e)=v\ , \\
                             -1, & \hbox{if } \quad t(e) = v\ ,\\
                             0, & \hbox{otherwise}\ ,
                           \end{array}
                         \right.
 $$
where $s(e)$ denotes the source of the oriented edge $e$, and $t(e)$ its target.
Let $\E_G$ be the $e_G\times (v_G-1)$ matrix obtained by deleting one of the  columns of $(\varepsilon)_{e,v}$.
%
%$$
%D(\alpha_1,\ldots, \alpha_n)=\left(
%  \begin{array}{ccc}
%    \alpha_1 &  &      \\
%     &    \ddots &    \\
%     &  &  \alpha_n   \\
%  \end{array}
%\right)
%$$
%

\begin{defn} Let $n=e_G+ v_G-1$, and consider the $(n\times n)$  matrix:
\begin{equation}\label{MGdefn}
M_G =
\left(
  \begin{array}{ccc|cc}
    \alpha_1 &  &  &  &  \\
     & \ddots &  & \E_G  &\\
     &  & \alpha_{e_G} &  &  \\
   \hline
  &  &  &   & \\
  & -{}^T\!\E_G &  &  0 & \\
  \end{array}
\right)
\end{equation}
%$$M_G=(-1)^{v(G)+1}
%\left(
%  \begin{array}{cc}
%    D&  E\\
%    E^T & 0 \\
%  \end{array}
%\right)\ .
%$$
It is not well-defined, because it depends on the choice of the deleted column in  $\E_G$, the orientation, and the chosen order of the edges and vertices. Throughout this paragraph,
$M_G$ will refer to any such choice of matrix. % bearing in mind that its minors well be well-defined.
\end{defn}

The following lemma relating $\E_G$ to  trees  goes back to Kirchhoff.
\begin{lem} \label{lemmatrixtree} Let $I$ denote a subset of edges of $G$ such that $|I|=h_G=e_G-v_G+1$, and let $\E_G(I)$ denote the square $(v_G-1)\times (v_G-1)$ matrix
obtained by deleting the rows of $\E_G$ corresponding to elements of $I$. Then
$$\det(\E_G(I)) = \left\{
              \begin{array}{ll}
                \pm 1, & \hbox{ if  } I \hbox{ is  a  spanning tree of } G\ , \\
                0 , & \hbox{otherwise.}
              \end{array}
            \right.$$
\end{lem}

%The following result   was shown to me by V. Rivasseau.
\begin{prop} \label{propcontractdelete} If $G$ is connected, then  $\Psi_G = \det M_G$. If $G$ is not connected, then $\Psi_G = \det M_G =0$. Deleting an edge $e$ corresponds
to taking the determinant of the minor of $M_G$ obtained by deleting the row and column corresponding to $e$, and contracting  $e$ corresponds
to setting the variable $\alpha_e$ to $0$, i.e.:
$$\Psi_{G\backslash  e} = \det M_G (e,e)\ ,\quad \hbox{ and } \quad \Psi_{G\q e} = \det M_G\big|_{\alpha_e=0}\ .$$
\end{prop}
\begin{proof} It is clear from the shape of the matrix $M_G$ that
$$\det(M_G) =  \sum_{I\subset G} \prod_{i\notin I} \alpha_i \det \left(
                                                                   \begin{array}{c|c}
                                                                     0 & \E_G(I) \\ \hline
                                                                     -{}^T\!\E_G(I) & 0 \\
                                                                   \end{array}
                                                                 \right) = \sum_{I\subset G ,  \\ |I|=h_G} \prod_{i\notin I} \alpha_i  \det(\E_G(I))^2\ .
$$
In the second expression, the sum is over all subsets of edges $I$ of $G$. In the case when $|I|<h_G$ or $|I|>h_G$, the 
columns of the  matrix are not independent and so its  determinant vanishes. This leaves  the case $|I|=h_G$, when $\E_G(I)$ is a square matrix.
The previous lemma implies that $\det(\E_G(I))^2 =1$ if $I$ is a spanning  tree and zero otherwise, which gives back formula $(\ref{defgraphpoly})$, and
therefore $\Psi_G=\det(M_G)$.  It follows
from the contraction-deletion formula $(\ref{contractdelete})$ that $\Psi_{G\backslash e}$ is the coefficient of $\alpha_e$ in $\det(M_G)$,
which is precisely given by the minor $\det(M_G(e,e))$. Likewise, $(\ref{contractdelete})$ implies that $\Psi_{G\q e}$ is obtained by setting $\alpha_e$ to zero.
\end{proof}

\subsection{Dodgson Polynomials} Proposition \ref{propcontractdelete} motivates the following definition.
\begin{defn} Let $I,J,K$ denote sets of edges of $G$ such that  $|I|=|J|$, and  %and  $I\cup J \subset S$.
let $M_G(I,J)$ denote the matrix obtained from $M_G$ by deleting the rows $I$ and columns $J$ \cite{Sta,B-E-K}. Then we define the \emph{Dodgson  polynomial} to be:
 \begin{equation} \label{PsiIJdef}
\Psi_{G,K}^{I,J} =    \det M_G(I,J)\big|_{\alpha_e=0, e\in K}\ .
\end{equation}
Changing the choice of matrix $M_G$ may change $\Psi_{G,K}^{I,J}$ by a sign.
When the graph $G$ is clear from the context, we will drop the $G$ from the notation.
\end{defn}

%Since $S$  contains $I\cup J$, and 
It is easy to verify that $\deg(\Psi^{I,J}_{G,K} ) = h_G -  |I| = h_G- |J|$  if   $\Psi^{I,J}_{G,K}\neq 0$, and it is immediate   from the definition  that $\Psi^{I,J}_{G,K} = \Psi^{I,J}_{G,K\cup A}$ for any $A\subset I\cup J$, and $\Psi_{G,K}^{I,J} = \Psi_{G,K}^{J,I}$.
%Since $\det(M_G(I,J))$ does not depend on the variables $\alpha_i$ for $i\in I\cup J$, we have $\Psi^{I,J}_K=\Psi^{I,J}_{I\cup J\cup K}.$
It follows from the previous proposition that for all $A,B\subset E_G$, 
\begin{equation}\label{DodgsonContractDelete}
\Psi^{I,J}_{G\backslash A\q B, K} =  \Psi^{I\cup A, J\cup A}_{G, K\cup B} \ .
\end{equation}
Therefore by passing to a minor we will often assume that $I\cap J=\emptyset$, and $K=\emptyset$. 

Also, when $I=J$, we will sometimes  write $\Psi^{I}_K$ instead of $\Psi^{I,I}_K$.

%$\Psi^I_K$ is the graph polynomial of $G\backslash I \q K$.

%\begin{defn} We define the \emph{depth} of the pair  $(I,J)$, and by extension, the polynomial $\Psi^{I,J}_S$, to be the quantity
% $|I|-|I\cap J|$\,\,\, ($=|J|-|I\cap J|)$.
%\begin{equation}\label{kdef}
%\kappa(I,J) = |I|-|I\cap J|\qquad (=|J|-|I\cap J|)\ .\end{equation}
%\end{defn}
%It is easy to verify that the following formula  holds whenever $\Psi^{I,J}_S\neq 0$:
%\begin{equation} \label{degform}
%\deg(\Psi^{I,J}_S ) = h_G -  |I|\qquad (= h_G- |J|)\ .
%\end{equation}

\begin{prop}\label{PropGenPsitreeformula} Let $I,J, K$ be as above. Then
\begin{equation}\label{GenPsitreeformula}
\Psi^{I,J}_{G,K} = \sum_{T\subseteq G} \pm \prod_{e\notin T} \alpha_e\ .\end{equation}
where the sum is over all  subgraphs  $T\subseteq G$ which are simultaneously spanning trees for both
$G \backslash  I \q (J\backslash (I\cap J) \cup K) $ and $G  \backslash J \q (I\backslash (I\cap J) \cup K)$. In particular, every monomial which occurs in $\Psi^{I,J}_{G,K}$ also occurs in both $\Psi^{I,I}_{G,J\cup K}$
and $\Psi^{J,J}_{G,I \cup K}$.
\end{prop}
\begin{proof}
By passing to the minor $G\backslash (I\cap J) \q K$, we can assume that $I\cap J=\emptyset$, and $K=\emptyset$.
As before, it follows from the shape of  the matrix $M_G(I,J)$ that
\begin{eqnarray}
\det(M_G(I,J)) &= & \sum_{U\subset G\backslash (I \cup J)} \prod_{u\notin U} \alpha_u \det \left(
                                                                   \begin{array}{c|c}
                                                                     0 & \E_G(U\cup I) \\ \hline
                                                                     -{}^T\!\E_G(U\cup J) & 0 \\
                                                                   \end{array}
                                                                 \right)     \nonumber \\
&=& \sum_{U\subset G \backslash (I \cup J)} \prod_{u\notin U } \alpha_u  \det(\E_G(U\cup I)) \det(\E_G(U \cup J)) \ .\nonumber
\end{eqnarray}
For both $\det(\E_G(U\cup I))$ and $\det(\E_G(U\cup J))$ to be non-zero, it follows from lemma $\ref{lemmatrixtree}$ that both  $U\cup I$  and $U\cup J$ must be  spanning trees
 in $G$. The tree $U\cup I$ does not involve any edges from $J$ (since, by assumption $I\cap J=\emptyset$), and so
it follows that $U$ is a spanning tree in $G\backslash I \q J$ and likewise
$G\backslash J \q I$, by symmetry. Conversely, for any such $U$, lemma $\ref{lemmatrixtree}$ implies that
$\det(\E_G(U\cup I))$ and $\det(\E_G(U\cup J))$, and hence their product, are equal to $\pm 1$.
\end{proof}

\begin{rem} \label{signsinDodgsons} A   formula for the signs in $(\ref{GenPsitreeformula})$  is given  in terms of  spanning forests in  \cite{WD}. We will only need the following fact: 
if  $i,j$ are adjacent edges in $G$ meeting at a vertex $v$, then all the coefficients of $\Psi_G^{i,j}$ have the same sign. 
This easily follows from the previous proof on choosing  the removed vertex in $M_G$ to be $v$.
\end{rem}

A spanning tree $T$ in $G\backslash I \q J$ lifts to a subgraph  $T\cup I\cup J\subset G$ which has $|I|=|J|$ loops. One can therefore view  Dodgson polynomials
as  sums over subgraphs of $G$ containing cycles which satisfy  certain properties.
The previous proposition implies  the following  vanishing condition for the Dodgson polynomials.
\begin{cor} \label{corvanishingcond}
Let $I,J,K$ be as above. Then $\Psi^{I,J}_{G,K} = 0$ if and only if there are no subgraphs $T\subset G$ such that $T\cup I$ is a spanning tree in $G\backslash (I\cap J)\q K$
and $T\cup J$ is a spanning tree in $G\backslash (I\cap J) \q K$.
\end{cor}

\begin{ex} Consider the wheel with four spokes, and let $I=\{1,2\}$, $J=\{3,4\}$ with the numbering of the edges shown below.
\begin{center}
\fcolorbox{white}{white}{
  \begin{picture}(76,65) (313,-220)
    \SetWidth{1.5}
    \SetColor{Black}
    \Arc(352,-187)(32,270,630)
    \SetWidth{1.5}
    \Line(352,-155)(352,-219)
    \Line(320,-187)(384,-187)
    \Vertex(352,-155){2}
    \Vertex(320,-187){2}
    \Vertex(352,-187){2}
    \Vertex(384,-187){2}
    \Vertex(352,-219){2}
   \Text(320,-160)[lb]{{\Black{$1$}}}
   \Text(380,-160)[lb]{{\Black{$2$}}}
   \Text(380,-220)[lb]{{\Black{$3$}}}
   \Text(320,-220)[lb]{{\Black{$4$}}}  
   \Text(335,-185)[lb]{{\Black{$5$}}}
   \Text(365,-198)[lb]{{\Black{$7$}}}
   \Text(358,-170)[lb]{{\Black{$6$}}}
   \Text(343,-210)[lb]{{\Black{$8$}}}
  \end{picture}
}
\end{center}It is clear that the unique common spanning tree of $G\backslash I \q J$ and $G\backslash J \q I$ is $\{6,8\}$, and therefore
$\Psi^{12,34}= \pm \alpha_5\alpha_7$. The subgraph $\{1,2,3,4,6,8\}$ therefore forms a double cycle.
Likewise, $\Psi^{14,23} =\pm \alpha_6\alpha_8$, and $\Psi^{13,24} =\pm( \alpha_6\alpha_8-\alpha_5\alpha_7)$.
% so we check in this case that
%$\Psi^{12,34}- \Psi^{13,24} + \Psi^{14,23}=0$, which we will prove in general below.

\end{ex}

\subsection{Pl\"ucker identities}
From now on, when considering identities between Dodgson polynomials, we will fix a representative matrix $M_G$ once and for all. 
This also fixes a representative matrix for all minors of $G$, so it makes sense to  write, e.g.,
\begin{equation}\label{DodgsonCD}
\Psi^{I,J}_{G,K}= \Psi^{I\cup e, J\cup e}_{G,K} \alpha_e + \Psi^{I,J}_{G,K\cup e}=  \Psi^{I , J }_{G\backslash e,K} \alpha_e + \Psi^{I,J}_{G\q e,K} \ .
\end{equation}

\begin{lem}
Let $M$ be a $N\times N$ symmetric matrix, and let $i_1,\ldots, i_{2n}$ be distinct indices between $1$ and $N$. Then
\begin{equation} \label{Plucker}
\sum_{k=n}^{2n} (-1)^k \det M(\{i_1,\ldots, i_{n-1}, i_k\},\{i_{n},\ldots, \widehat{i}_k,\ldots, i_{2n}\}) = 0 \ .
\end{equation}
\end{lem}
\begin{proof} By doing row  expansions with respect to  rows not in $\{i_1,\ldots, i_{2n}\}$, it suffices by linearity  to consider the case $N=2n$. Then, after deleting rows $i_1,\ldots, i_n$ from $M$, one sees that $(\ref{Plucker})$
reduces to the  classical Pl\"ucker identity.
\end{proof} 
The  determinant of the matrix $M_G$  is $\pm1$ times the determinant of a  symmetric matrix. Applying the previous lemma gives linear identities %between graph polynomials 
of the form:
\begin{equation} \label{PsiPlucker}
\sum_{k=n}^{2n} (-1)^k \Psi_G^{\{i_1,\ldots, i_{n-1}, i_k\},\{i_{n},\ldots, \widehat{i}_k,\ldots, i_{2n}\}} = 0 \ .
\end{equation}
We will mainly require the special case $n=2$  where $i,j,k,l$ are distinct indices:
\begin{eqnarray}
\Psi_G^{ij,kl} -\Psi_G^{ik,jl}+\Psi_G^{il,jk} & = & 0 \ . 
%\Psi_G^{ijk,lmn} -\Psi_G^{ijl,kmn}+\Psi_G^{ijm,kln} -\Psi_G^{ijn,klm}& = & 0 \  , \nonumber
\end{eqnarray}
%where $i,j,k,l$ are distinct indices.

\subsection{Determinantal identities}
%The following result is  well-known.

\begin{lem} (Jacobi's determinant formula)
Let $M= (a_{ij})$ be an  invertible $n\times n$ matrix, and let $\adj  M =(A_{ij})$ denote the adjoint of $M$, i.e., the transpose of the  matrix of cofactors of $M$.
Then for any integer $1\leq k \leq n$,
\begin{equation}\label{Jacobiformula}
\det (A_{ij})_{k<i,j\leq n} = \det(M)^{n-k-1} \det (a_{ij})_{1\leq i,j\leq k}\ .
\end{equation}
\end{lem}

\begin{proof}
If $I_n$ denotes the identity matrix of size $n$, we have:
$$(\adj M)\, M = I_n (\det M) \ .$$
In particular, $\det(\adj M) = \det(M)^{n-1}$. If we replace $M$  in this equation with the matrix obtained from $M$ by replacing
the last $k$ columns with the corresponding columns of the identity matrix $I_n$, we obtain the equation:
$$
\left(
  \begin{array}{ccc}
   & & \\
& A_{ij} &\\
& & \\
  \end{array}
\right)
\left(
  \begin{array}{c|c}
     &     0  \\
    \underset{1\leq j\leq k}{a_{ij}}    &  \underset{-\!\!-\!\!-\!\!-\!\!-\!\!-}{}  \\
      &    I_{n-k}    \\
  \end{array}
\right)  =  \left(
              \begin{array}{c|c}
                \det(M) I_k  & * \\ \hline
                0 & \underset{k<i,j\leq n}{A_{ij}} \\
              \end{array}
            \right)\ .
$$
Taking the determinant of this equation gives
$$(\det M)^{n-1} \det (a_{ij})_{1\leq i,j\leq k} = (\det M)^k \det (A_{ij})_{k<i,j\leq n}\ .$$
Finally, dividing by $(\det M)^k$ yields $(\ref{Jacobiformula})$.
\end{proof}
In the special case $k=n-2$, we can  rearrange the indices to give the following quadratic identity, for any $1\leq p<q\leq n$ and $1\leq r<s\leq n$:
\begin{equation}\label{quadraticCondensation}
A_{p,r} A_{q,s} - A_{p,s} A_{q,r} = \det(M) \det M(pq,rs)\ , \end{equation}
where $M(pq,rs)$ denotes the matrix $M$ with rows $p$ and $q$, and columns $r$ and $s$ removed.
 This identity  is usually attributed to C. L.  Dodgson. % (see.. for an interesting discussion on the history of this identity).

Now let $G$ be any graph, and let $A,B$ denote two subsets of the set of edges of $G$ with $|A|=|B|$.
Applying the previous identity to the matrix obtained from  $M_G$ by removing  rows $A$ and columns $B$, gives the identity:
\begin{equation}\label{Psiquadratic}
\Psi_S^{Ap,Br} \Psi_S^{Aq,Bs} - \Psi_S^{Ap,Bs} \Psi_S^{Aq,Br} = \Psi_S^{A,B}  \Psi_S^{Apq,Brs}\ , \end{equation}
for any  $p,q,r,s\notin A\cup B$, and  where $S=A\cup B \cup \{p,q,r,s\}$.
\begin{lem} \label{lemTechVanish} Let $A=\{a_1,\ldots, a_n\}$ and $B=\{b_1,\ldots, b_n\}$ be sets of edges of $G$ with $A\cap B=\emptyset$.
Then for any fixed $1\leq i\leq n$, 
\begin{equation}
\Psi^{A\backslash\{a_i\},B\backslash\{b_j\}} = 0\quad \hbox{ for all } \quad j=1,\ldots, n
\quad  \Longrightarrow \quad \Psi^{A,B}=0\  . \nonumber
\end{equation}
\end{lem}
\begin{proof}  Let $C=A\backslash \{a_i\}$, $D=B\backslash \{b_j\}$. 
It follows from  $(\ref{Psiquadratic})$  that  $\Psi^{C a_i, Db_j} \Psi^{Cb_j,Da_i}= \Psi^{Ca_ib_j,Da_ib_j}\Psi^{C,D}- \Psi^{Ca_i,Da_i}\Psi^{Cb_j,Db_j}$.
If $\Psi^{C,D}=0$ the contraction-deletion formula $(\ref{DodgsonCD})$  implies that all terms on the right-hand side are zero, and so
 $$\Psi^{A,B} \, \Psi^{A\cup\{b_j\}\backslash\{a_i\} ,B\cup\{a_i\}\backslash\{b_j\}} = 0. $$
The Pl\"ucker identity gives
$ \Psi^{A,B} +\sum_{j=1}^n \pm  \Psi^{A\cup\{b_j\}\backslash\{a_i\} ,B\cup\{a_i\}\backslash\{b_j\}} = 0,$
which together imply that $\Psi^{A,B}=0$.
\end{proof}
The following two quadratic  identities will be crucial for the sequel, and will later be reformulated more simply in terms of resultants in  \S\ref{sectMatrixRed}
$(\ref{resid1})$, $(\ref{resid2})$. 
\begin{lem} Let $I,J$ be subsets of $E_G$ satisfying $|I|=|J|$, and let $a,b,x \notin I\cup J$. If  we set $S= I\cup J \cup\{a,b,x\}$, then
we have the following identity:
\begin{eqnarray} \label{FirstDodgsonId}
\Psi_S^{I,J} \Psi_S^{Iax,Jbx} - \Psi_S^{Ix,Jx}\Psi_S^{Ia,Jb} = \Psi_S^{Ix,Jb}\Psi_S^{Ia,Jx}\ .
\end{eqnarray}
Now let $I,J$ be two subsets of $E_G$ satisfying $|J|=|I|+1$, and let $a,b, x \notin I\cup J$. If we set $S=I\cup J\cup \{a,b,x\}$, then we have  the following identity:
\begin{eqnarray} \label{SecondDodgsonId}
\Psi_S^{Ia,J} \Psi_S^{Ibx,Jx} - \Psi_S^{Iax,Jx}\Psi_S^{Ib,J} = \Psi_S^{Ix,J}\Psi_S^{Iab,Jx}\ .
\end{eqnarray}
\end{lem}
\begin{proof} The first identity $(\ref{FirstDodgsonId})$ follows immediately from $(\ref{Psiquadratic})$ on setting $A=I, B=J$, and $(p,q,r,s) = (a,x,b,x)$ and rearranging terms.
For the second identity $(\ref{SecondDodgsonId})$, let $A=I$, and set $J=B\cup \{y\}$, where $y\notin A$. Set $(p,q,r,s) = (a,b,x,y)$ in $(\ref{Psiquadratic})$, and writing
$S=A\cup B\cup \{a,b,x,y\}$  gives
$$\Psi_S^{Aa,Bx}\Psi_S^{Ab,By} - \Psi_S^{Aa,By}\Psi_S^{Ab,Bx} = \Psi_S^{A,B} \Psi_S^{Aab,Bxy}\ . $$
Multiplying this equation through by $\Psi_S^{Ax,By}$ gives
\begin{equation}\label{pfeqn1}
\Psi_S^{Ax,By}\Psi_S^{Aa,Bx}\Psi_S^{Ab,By} - \Psi_S^{Aa,By}\Psi_S^{Ab,Bx}\Psi_S^{Ax,By} = \Psi_S^{A,B}\Psi_S^{Ax,By} \Psi_S^{Aab,Bxy}\ . \end{equation}
Two further applications of  $(\ref{FirstDodgsonId})$ gives
$$ \Psi_S^{A,B}\Psi_S^{Aax,Byx} - \Psi_S^{Ax,Bx} \Psi_S^{Aa,By}= \Psi_S^{Aa,Bx}\Psi_S^{Ax,By} \ , $$
$$ \Psi_S^{A,B}\Psi_S^{Abx,Byx} - \Psi_S^{Ax,Bx} \Psi_S^{Ab,By}= \Psi_S^{Ab,Bx}\Psi_S^{Ax,By} \ . $$
Substituting these into $(\ref{pfeqn1})$ and cancelling terms, gives
$$\Psi_S^{A,B}\big(\Psi_S^{Ab,By}\Psi_S^{Aax,Byx} - \Psi_S^{Aa,By}\Psi_S^{Abx,Byx}\Big) = \Psi_S^{A,B}\Psi_S^{Ax,By} \Psi_S^{Aab,Bxy}\ .$$
If $\Psi_S^{A,B}$ is non-zero,  we are done.  The opposite case requires that $\Psi_S^{A,B}=0$ 
for all $B$. Since we can assume that $I\cap J = \emptyset$, we can choose the element $y$ as we wish.
  But then lemma \ref{lemTechVanish} implies that the identity $(\ref{SecondDodgsonId})$ reduces identically to zero.
\end{proof}

\subsection{Graph-specific identities}
The Jacobi identity $(\ref{Jacobiformula})$  implies some linear relations between Dodgson polynomials in the case
when its right-hand side vanishes.

\begin{lem} \label{lemlinid} Let $|A|=|B|$, and let  $I=\{i_1,\ldots, i_k\}$, where $I \cap A=I \cap  B=\emptyset$.  
%$A'=A\cup\{i_1,\ldots, i_k\}$ and $B'=B\cup \{i_1,\ldots, i_k\}$, and set
%$S=A\cup B$, and $S'=A'\cup B'$. 
 If  %either
 %$\Psi^{A,B}_{I}$ or 
  $\Psi^{A\cup I, B\cup I}$ vanishes, 
then for all $1\leq r\leq k$ we have an identity
\begin{equation}\label{linearrelation}
\Psi^{Ai_r, Bi_r} = \sum_{j\neq r} \pm \Psi^{A i_r, B {i_j}}  \qquad \big(= \sum_{j\neq r} \pm \Psi^{A i_j, B {i_r}}\big)\ ,
\end{equation}
which should be viewed as a polynomial identity in the variables $\alpha_{i_1},\ldots, \alpha_{i_k}$.
%If $A=B$, then all the coefficients are positive.
\end{lem}
\begin{proof}  %First we treat the case where $\Psi^{A\cup I, B\cup I}=0$. 
Assume that $I$ is minimal, {\it i.e.}, $\Psi^{A\cup J, B \cup J}\neq 0$ for all $J\subsetneq I$.
Applying the Jacobi identity $(\ref{Jacobiformula})$ to the matrix $M_G$ with rows $A$ and columns $B$ removed gives:
%After appropriately ordering the rows and columns, we deduce that:
$$ \det\left(
  \begin{array}{ccc}
    \Psi^{Ai_1,Bi_1} & \cdots & \Psi^{Ai_1,Bi_k} \\
    \vdots&  \ddots &  \vdots\\
    \Psi^{Ai_k,Bi_1} & \cdots & \Psi^{Ai_k,Bi_k} \\
  \end{array}
\right)= \big(\Psi^{A,B}\big)^{k-1} \Psi^{A \cup I,B \cup I }=0\ .
$$
 The matrix on the left therefore has rank $<k$, and so there is a non-trivial linear relation between its columns.
In particular, there exist $\lambda_1,\ldots, \lambda_k$ such that:
\begin{equation}\label{pf1} \lambda_1 \Psi^{Ai_1, Bi_1} +\lambda_2 \Psi^{A i_1, B {i_2}}  +\ldots  +\lambda_k \Psi^{Ai_1,Bi_k} =0 \ ,
\end{equation}
where  $\Psi^{Ai_1,Bi_j}$ should be viewed as  a polynomial in the variables $\alpha_{i_2},\ldots, \widehat{\alpha_{i_j}}, \ldots,  \alpha_{i_k}$ and the $\lambda_j$ do not depend on
these variables. Choosing a  monomial which occurs in $(\ref{pf1})$, say,  $\alpha_{i_3}...\alpha_{i_k}$,  and taking its coefficient  yields
\begin{equation} \label{pf2}
\lambda_1 \Psi^{Ai_1i_3\ldots i_k, Bi_1i_3\ldots i_k} + \lambda_2 \Psi^{A i_1i_3\ldots i_k, B {i_2}i_3\ldots i_k}=0\ .\end{equation}
By assumption, $\Psi^{Ai_1i_3\ldots i_k, Bi_1i_3\ldots i_k}$, and hence $\Psi^{A i_1i_3\ldots i_k, B {i_2}i_3\ldots i_k}$ is non-zero.  Using the interpretation
of Dodgson polynomials as sums of trees, we know that every monomial in $\Psi^{A i_1i_3\ldots i_k, B {i_2}i_3\ldots i_k}$ also occurs
in $\Psi^{Ai_1i_3\ldots i_k, Bi_1i_3\ldots i_k}$. Taking the coefficient of any monomial in $(\ref{pf2})$ therefore implies that
$\lambda_1\pm \lambda_2 =0$. %(In the case $A=B$, we necessarily have $\lambda_1+\lambda_2=0$). 
Similarly, we deduce that $\lambda_1=\pm \lambda_j$ for all $2\leq j\leq k$, and therefore
$(\ref{pf1})$ implies $(\ref{linearrelation})$.
%$$ \Psi^{Ai_1, Bi_1} = \pm  \Psi^{A i_1, B {i_2}}  \pm \ldots  \pm \Psi^{Ai_1,Bi_k}_S  \ .$$
%as required.  In the case where $A=B$, $ \Psi^{Ai_1i_3\ldots i_k, Bi_1i_3\ldots i_k} $
%is a graph polynomial, and therefore has positive coefficients and all the other terms in $(\ref{pf1})$ are Dodgson polynomials of depth 1 which contain at least one positive% %monomial. Therefore in this case, $\lambda_1=-\lambda_j$ for all $j$ and completes the proof when $\Psi^{A\cup I,B\cup I}=0$.
%This settles the case where $\Psi^{A\cup I, B\cup I}_{S}$ vanishes, i.e., when taking the coefficient of $\alpha_{i_1}\ldots \alpha_{i_k}$ in
%$\Psi_{S}^{A,B}$ is zero.
%The  case when $\Psi^{A, B}_{I}=0$ %is precisely when setting $\alpha_{i_1}=0,\ldots,\alpha_{i_k}=0$ in $\Psi^{A,B}_S$ causes it to vanish.
 %can be deduced from the previous case by inverting all variables $\alpha_i$ ({\it i.e.}, by repeating the above proof for  the dual matroid of $G$).
\end{proof}

%Equation $(\ref{linearrelation})$ gives % to a large number of
%many linear relations between $\Psi^{I,J}_{G,K}$. 
A similar result holds  when $\Psi^{A,B}_I=0$, %as can be seen 
by inverting the variables $\alpha_i$. 
We will not require the general case for the sequel, only the following special cases. 
%Rather than write down the most general case, it is more
% we will instead consider the  examples which are important for the sequel. %instructive to look at a particular example in detail, which will be important in the sequel.
\begin{example} \label{ExStar}  
%The previous lemma is  useful for understanding the local structure of  graph polynomials. To illustrate this,  
%Consider the simple case where
Suppose $G$ contains a three-valent vertex, or star:
\begin{figure}[h!]
%\vspace{-0.2in}
  \begin{center}
%    \leavevmode
    \epsfxsize=3.0cm \epsfbox{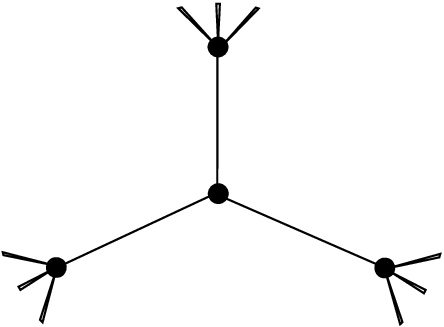}
  \label{2gen}
 \put(-30,20){ $3$}
 \put(-67,10){ $1$}
 \put(-45,40){ $2$}
  \end{center}
\end{figure}
%\vspace{-0.2in}

\noindent
Deleting edges $1,2,$ and $3$  disconnects the central vertex, so  $\Psi^{123}=0$, by  corollary \ref{corVanishingsubgraphs}, and we can therefore  apply the previous lemma.  
Another obvious identity is that  $\Psi^{12}_3 = \Psi^{23}_1 = \Psi^{13}_2$,
since deleting any two of the three edges and contracting the third always gives rise to the same minor.
Let us define:
\begin{equation}
 f_{123}= \Psi_{123} \ , \ f_1=\Psi^{2,3}_1 \ , \ f_2= \Psi^{1,3}_2 \ , \ f_3=\Psi^{1,2}_3\ , \ f_0=\Psi^{12}_3 = \Psi^{23}_1 = \Psi^{13}_2\ .
\end{equation} 
The Jacobi identity implies that
$$\det \left(
         \begin{array}{ccc}
           \Psi^{1} & \Psi^{1,2} & \Psi^{1,3} \\
           \Psi^{2,1} & \Psi^{2} & \Psi^{2,3} \\
           \Psi^{3,1} & \Psi^{3,2} & \Psi^{3} \\
         \end{array}
       \right) =0 \ .$$
Applying the  previous lemma to the first row gives the equation  $\Psi^1 = \Psi^{1,2} + \Psi^{1,3}.$  The coefficients are all positive, since we know by remark \ref{signsinDodgsons} that $\Psi^{i,j}$ for adjacent edges $i,j$ always has positive coefficients.
This implies  the identity
$$ \Psi^{13}_2 \alpha_3+ \Psi^{12}_3 \alpha_2  + \Psi^1_{23} = \Psi^{13,23}\alpha_3 +\Psi^{1,2}_3 + \Psi^{12,23}\alpha_2+ \Psi^{1,3}_2\ . $$
 By symmetry, we deduce that  
for all $\{i,j,k\}=\{1,2,3\} $:
%\begin{eqnarray}
%f_0&=& \Psi^{ij}_k=\Psi^{ij,jk} \quad \ , \\
%\Psi^{i}_{jk}& =&   f_j + f_k    \quad (=\Psi^{i,k}_j+\Psi^{i,j}_k)\ .\nonumber 
%\end{eqnarray}
$$f_0 =  \Psi^{ij}_k=\Psi^{ij,jk} \qquad \hbox{ and } \qquad \Psi^{i}_{jk} =   f_j + f_k    \quad (=\Psi^{i,k}_j+\Psi^{i,j}_k)\ .$$ 
%We therefore have
%$$\Psi^{12}_3=\Psi^{12,23} \ , \ \Psi^{13}_2= \Psi^{13,23} \ , \ \Psi^1_{23}=\Psi^{1,2}_3+ \Psi^{1,3}_2\ .$$
%Similar considerations for the second column gives, in particular, $\Psi^{23}_1 = \Psi^{13,23}$, which implies that $\Psi^{23}_1=\Psi^{13}_2$
%when combined with the previous equations. Continuing in this way, or by symmetry,  we deduce
%that
It follows that  the graph polynomial $\Psi=\Psi_G$ can be written:
$$\Psi = f_{123}+  (f_2+f_3) \alpha_1+  (f_1+f_3)\alpha_2 +(f_1+f_2)\alpha_3 + f_0( \alpha_1\alpha_2+\alpha_1\alpha_3+\alpha_2\alpha_3)\ ,$$
The Dodgson identity $(\ref{FirstDodgsonId})$ implies that
$\Psi^{23}_1 \Psi_{123} - \Psi^{2}_{13}\Psi^3_{12} = \Psi^{2,3}_1 \Psi^{3,2}_1,$
which, in the new notation, is $ f_0 f_{123}- (f_1+f_3)(f_1+f_2) = f_1^2$. We therefore have:
\begin{equation} 
f_0 f_{123} = f_1f_2 +f_1f_3+f_2f_3\ .
\end{equation}
\end{example}

\begin{example} \label{ExTriangle}  Now suppose that $G$ contains an internal triangle:

%\vspace{-0.15in}
\begin{figure}[h!]
  \begin{center}
%    \leavevmode
    \epsfxsize=3.0cm \epsfbox{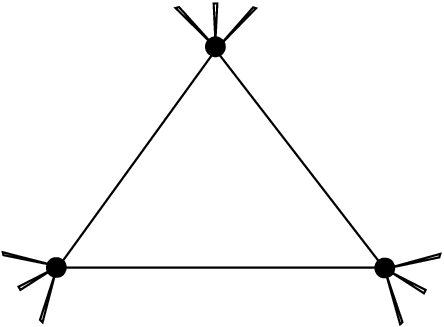}
  \label{2gen}
 \put(-30,35){ $3$}
 \put(-75,35){ $1$}
 \put(-50,15){ $2$}
  \end{center}
\end{figure}
%\vspace{-0.15in}

\noindent In this case it is clear, since $1,2,3$ forms a loop, that $\Psi_{123}=0$. Define:
\begin{equation}
 f^{123}= \Psi^{123},  f^1=\Psi^{12,13}  , f^2= \Psi^{12,23}  ,  f^3=\Psi^{13,23} , \ f^0=\Psi_{12}^3 = \Psi_{23}^1 = \Psi_{13}^2\ .
\end{equation} 
By a similar argument to the case of a 3-valent vertex, we have:
\begin{eqnarray}
f^0&=& \Psi^{k}_{ij}=\Psi^{i,k}_j \quad \ , \\
\Psi_{i}^{jk}& =&   f^j + f^k    \ , \quad \hbox{ for all } \quad \{i,j,k\}=\{1,2,3\} \  .\nonumber 
\end{eqnarray}
%We therefore have
%$$\Psi^{12}_3=\Psi^{12,23} \ , \ \Psi^{13}_2= \Psi^{13,23} \ , \ \Psi^1_{23}=\Psi^{1,2}_3+ \Psi^{1,3}_2\ .$$
%Similar considerations for the second column gives, in particular, $\Psi^{23}_1 = \Psi^{13,23}$, which implies that $\Psi^{23}_1=\Psi^{13}_2$
%when combined with the previous equations. Continuing in this way, or by symmetry,  we deduce
%that
We deduce that the graph polynomial $\Psi$ can be written:
$$\Psi = f^{123}\alpha_1\alpha_2\alpha_3+  (f^1+f^2) (\alpha_1\alpha_2)+  (f^1+f^3)(\alpha_1\alpha_3) +(f^2+f^3)(\alpha_2\alpha_3) + f^0( \alpha_1+\alpha_2+\alpha_3)\ ,$$
The Dodgson identity  implies that
\begin{equation} 
f^0 f^{123} = f^1f^2 +f^1f^3+f^2f^3\ .
\end{equation}

\end{example}

 \begin{cor}\label{universal3joinformula}
Let $G$ be a connected graph which contains a 3-valent vertex (resp. triangle). Then
with the previous notations,
$f_0 \Psi_G$ (resp. $(f^0)^2 \Psi_G$) is the graph polynomial of the following graph on the right (resp. left),
where a labelled edge means replacing its Schwinger coordinate with that labelling.

%\begin{center}
\fcolorbox{white}{white}{
  \begin{picture}(109,48) (388,-287)
    \SetWidth{1.5}
    \SetColor{Black}
    \Line(436,-231)(436,-184)
    \SetWidth{1.5}
    \Vertex(436,-184){2}
    \Vertex(436,-231){2}
    \SetWidth{1.5}
    \Line(436,-231)(397,-256)
    \Line(436,-231)(475,-258)
    \SetWidth{1.5}
    \Vertex(436,-184){1}
    \Vertex(396,-256){2}
    \Vertex(474,-258){2}
    \SetWidth{1.5}
    \Arc(436,-231)(47,270,630)
     \Text(380,-280)[lb]{\Large{\Black{$K_4$}}}
    \Text(485,-221)[lb]{\Large{\Black{$f^0\alpha_2$}}}
    \Text(404,-242)[lb]{\Large{\Black{$f_2$}}}
    \Text(457,-242)[lb]{\Large{\Black{$f_1$}}}
    \Text(441,-211)[lb]{\Large{\Black{$f_3$}}}
    \Text(431,-273)[lb]{\Large{\Black{$f^0\alpha_3$}}}
    \Text(399,-221)[lb]{\Large{\Black{$f^0\alpha_1$}}}
  \end{picture}
  \begin{picture}(100,113) (293,-203)
    \SetWidth{1.5}
    \SetColor{Black}
    \Vertex(418,-199){2}
    \Vertex(418,-97){2}
    \Vertex(367,-148){2}
    \Vertex(418,-148){2}
    \Vertex(469,-148){2}
    \SetWidth{1.5}
    \Arc(418,-148)(51,270,630)
    \Line(418,-99)(418,-199)
     \Text(336,-200)[lb]{\Large{\Black{$K_{2,3}$}}}
    \Text(380,-174)[lb]{\Large{\Black{$f_0\alpha_1$}}}
    \Text(424,-174)[lb]{\Large{\Black{$f_0\alpha_2$}}}
    \Text(471,-174)[lb]{\Large{\Black{$f_0\alpha_3$}}}
    \Text(380,-126)[lb]{\Large{\Black{$f_1$}}}
    \Text(424,-126)[lb]{\Large{\Black{$f_2$}}}
    \Text(471,-126)[lb]{\Large{\Black{$f_3$}}}
  \end{picture}}
\end{cor}

\begin{example} (Star-Triangle duality) Consider a pair of graphs $G_{Y}$ and $G_{\triangle}$ related 
by a Star-Triangle transformation. 
With the notations of the two previous examples, we have $f^0=f_{123}$, $f_0=f^{123}$,  and $f^i=f_i$ for $i=1,2,3$.
It follows that the general star-triangle duality follows from the special case of the two graphs $K_{2,3}$ and $K_4$ above. 
%We now write down the general
%relationship between the Dodgson polynomials of $G_{Y}$ and $G_{\triangle}$.  
One can show that for every set of subsets $I,J,K\subset E_{G_Y}$ such that 
their union $I\cup J\cup K$ contains the edges $1,2,3$, we have
$$\Psi^{I,J}_{G_Y,K} = \Psi^{I',J'}_{G_{\triangle},K'}\ ,$$
for some  $I',J',K'\subset E_{G_{\triangle}}$  obtained from $I,J,K$, and vice-versa. In other words, there is an (explicit) bijection between the Dodgson polynomials
$\Psi^{I,J}_K$ of $G_Y$ and $G_{\triangle}$ provided that $I\cup J\cup K$ contains the edges of the 3-valent vertex (resp. triangle).

%Write $I=I_0\cup A$, $J=J_0\cup B$, $K=K_0\cup C$, where $A,B,C\cap \{1,2,3\}=\emptyset$ and $C\cap (A\cup B) = \emptyset$.
%\begin{itemize}
% \item If $I_0=\{1,2,3\}$  then $\Psi^{I,J}_K (G_{Y}) =0$.
%\item If $I_0=\{1,2\}$   then $I'= \{1,2,3\}\cup A$.
%\item If $I_0= \{1\}$ then $I' = \{2,3\} \cup A$.
%\item If $I_0=\emptyset$ and $K_0=\{1\}$ then $I'=\{2,3\}\cup A$.
%\item If $I_0=\emptyset$ and $K_0=\{1,2\}$ then $I'=\{3\}\cup A$.
%\item If $K_0=\emptyset$ then $K'=C$.
%\item If $K_0=\{1\}$ then $K'=C$.
%\item If $K_0= \{1,2\}$ then $K'=\{3\} \cup C$.
%\item If $K_0=\{1,2,3\}$ then $K'=\{2,3\}\cup C$, $I'=\{1\}\cup A$, $J'=\{1\} \cup B$.
%\item   Symmetrize the above with respect to $\{1,2,3\}$ and $I\leftrightarrow J$.
 %\end{itemize}
% We leave the reverse map $(I',J',K')\mapsto (I,J,K)$  to  the reader. %In any case, we have the following:
 
% \begin{cor} There is a bijection between the set of Dodgson polynomials $\Psi^{I,J}_K$ of $G_{Y}$ and $G_{\triangle}$ such that $I\cup J\cup K$ contains the three edges of the %three-valent vertex (resp. triangle).
 %\end{cor}

\end{example}

These examples illustrate the large numbers of identities between the $\Psi^{I,J}_{G,K}$   which arise  from constraining the local structure of $G$. 
It is these identities which  determine the nature of the periods $I_G$.
%Since it is these identities which determine the nature of the periods, a systematic study of such identities would be worthwhile.

\newpage

\section{Analytic properties of Feynman integrals }
%The set of minors of a Feynman graph plays a crucial role in determining its periods. 
We prove that the periods of Feynman graphs are minor monotone.
As a result, any graph invariant  which relates to periods should also satisfy this property.
We then give parametric proofs of some  identities between periods of Feynman integrals which are well-known in  momentum or coordinate space.

\subsection{Parametric Feynman integrals for primitive divergent graphs}
%For the first definition of the periods of  a graph, we require a condition of convergence.
%\begin{defn}
  Let $G$ be a primitive divergent graph (and so $e_G=2h_G$ is even), and set
\begin{equation}\label{omegaGdef} \Omega_{G} = \sum_{i=1}^{e_G} (-1)^i \alpha_i\, d\alpha_1\ldots \widehat{d\alpha}_i \ldots d\alpha_{e_G}\ .
\end{equation}
The \emph{residue} of  $G$ is defined to be the projective integral \cite{B-E-K}:
\begin{equation} I_G = \int_{ \Delta} {\Omega_G \over \Psi^2_G} \ ,
\end{equation}
where $\Delta=\{(\alpha_1:\ldots: \alpha_{e_G}) \in \Pro^{e_G-1}(\R)| \alpha_i \geq 0\}$ is the real coordinate simplex.  This  integral is (absolutely) convergent, and defines a positive real number.
In this paper we only consider Feynman  integrals whose momentum dependence factors out  of the integral (c.f.  `bpd' graphs in  \cite{BrCMP}). 
In this case one sometimes considers propagators raised to powers $1+n\,\varepsilon$ where $n\in \Z$ (the dressed case). In the parametric
setting, this corresponds to choosing  for each edge $e$ of $G$, an integer $n_e$ such that the following homogeneity condition is satisfied:
\begin{equation} \label{homogeneitycond}
\sum_{i=1}^{e_G}n_i  =- h_G\ .
\end{equation}
The general form of the parametric Feynman integral in dimension $D=4-2\,\varepsilon$ is:
\begin{equation}\label{IGconvdef} I_G(n_1,\ldots, n_{e_G}, \varepsilon) = \int_{\Delta} \prod_{i=1}^{e_G} \alpha_i^{n_i\varepsilon} {\Omega_G \over \Psi_G^{2-\varepsilon}}\in \R[[\varepsilon]]\ ,
\end{equation} viewed as a Taylor expansion in $\varepsilon\geq 0$.  We are interested in its   coefficients of $\varepsilon^k$, for all possible
values of decorations $n_e$. Note that these coefficients are given by integrals involving  powers of   $\log(\alpha_e)$ and $\log(\Psi_G)$ in the numerator \cite{BrCMP}.

It is convenient to rewrite the integral $(\ref{IGconvdef})$ as an affine integral:
\begin{equation} \label{affineint}
I_G(n_1,\ldots, n_{e_G}, \varepsilon) = \int_{\Delta_H} {\prod_{i=1}^{e_G} \alpha_i^{n_i\varepsilon} \over \Psi_G^{2-\varepsilon}}\, \delta(H)\prod_{i=1}^{e_G} d\alpha_i \ ,
\end{equation}
where $H$ is any hyperplane in $\A^{e_G}$ not passing through the origin, and $\Delta_H$ is the  subset of points of $H$ which  project  onto $\Delta$. In the physics literature it is common to take $H$ to be the
hyperplane $\sum_e \alpha_e =1$, but in this paper, we shall always take $H$ to be $\alpha_{e_0}=1$ for some particular choice of edge $e_0$.

Since the set of all primitive divergent graphs do not form a minor closed set, we need to define the periods for an arbitrary graph.

\subsection{Periods of arbitrary  graphs} \label{sectPeriodsallgraphs}
%Let $G$ be  any connected  graph with $e_G$ edges. %and define $\Omega_G$ of homogeneous degree $e_G$ by $(\ref{omegaGdef})$. 
 For any connected graph $G$, we define    two  $\Q$-vector spaces of periods, in the following naive sense.

\begin{defn} \label{defnrealperiods}
 We define the \emph{real periods} of $G$, denoted $\Pe(G)$, to be the
$\Q$-vector space spanned by the real numbers
\begin{equation}\label{periodint} \int_{\Delta_H} {P(\alpha_1,\ldots, \alpha_{e_G}) \over \Psi_G^n} \, \delta(H)\prod_{i=1}^{e_G} d\alpha_i  \ ,\end{equation}
where $n \in \N$ and  $P\in \Q[\alpha_e, e\in E_G]$  such that the integral converges absolutely.
 Likewise, we define the \emph{real logarithmic periods} of $G$, denoted $\Pel(G)$, to be the $\Q$-vector space spanned by
the coefficients in the Taylor expansion at $\varepsilon=0$ of all absolutely convergent integrals
of the form: %$\int_{H} P \, d\alpha_1\ldots d\alpha_n$,
\begin{equation}\label{logperiodint} \int_{\Delta_H} {P(\alpha_1,\ldots, \alpha_{e_G}) \over \Psi_G^n} \Big({\prod_{i=1}^{e_G} \alpha_i^{n_i\varepsilon}
 \over \Psi_G^{m\varepsilon}} \Big)\, \, \delta(H)\prod_{i=1}^{e_G} d\alpha_i \ ,\end{equation}
where $m, n \in \N, n_i\in \Z$,  and $P$ is as above. %and $\sum_i n_i = mh_G$.
\end{defn}

 It is clear that $\Pe(G)$ is finite-dimensional, but $\Pel(G)$ will be infinite-dimensional in general.
 %\footnote{Assuming standard, but unproven, conjectures on the linear independence of periods.}.
When $G$ is primitive divergent,  the residue $I_G$ is an element of $\Pe(G)$, and the coefficients
of the Taylor expansion of $I_G(n_1,\ldots, n_{e_G},\varepsilon)$ are in $\Pel(G)$.

\begin{prop}\label{propminor} The real periods (resp.  real logarithmic periods) of Feynman graphs are minor monotone. More precisely, for every minor $\gamma \minor G$,
$$\Pe(\gamma) \subseteq \Pe(G)\quad\hbox{ and }\quad \Pel(\gamma) \subseteq \Pel(G) .$$
\end{prop}

\begin{proof} It suffices to prove the formulae in the two cases $\gamma = G\backslash  e$, and $\gamma = G \q e$, and the general
case will follow by induction. % since we can always write $\gamma = G\backslash D\q C$.
First observe that for all  $s>0 $,
\begin{equation} \label{waslem}
\int_{0}^\infty {s\, a \over (ax+ b)^{s+1}} dx  = {1 \over b^s }  \quad \hbox{ and } \quad
\int_{0}^\infty { s\, b \,x^{s-1} \over (ax+ b)^{s+1}} dx = {1 \over a^s}  \ ,
\end{equation}
where the second equation follows from the first by changing variables $x\mapsto x^{-1}$.
Using the fact $(\ref{contractdelete})$ that $\Psi_G = \Psi_{G \backslash e} \alpha_e +\Psi_{G\q e}$, these formulae imply that:
\begin{equation}\label{nologinsert}
\int_{0}^\infty {s\Psi_{G\backslash e} \over \Psi_G^{s+1}}\, d\alpha_e  = {1 \over \Psi_{G\q e}^s }  \quad \hbox{ and } \quad
\int_{0}^\infty { s \Psi_{G\q e} \,\alpha_e^{s-1} \over \Psi_G^{s+1}} d\alpha_e = {1 \over \Psi_{G\backslash e}^s}  \ .
\end{equation}
Consider a Feynman integral of the form $(\ref{logperiodint})$ for the graph $\gamma =G\backslash e$ or $\gamma = G\q e$, and write it
 as an affine integral over the hyperplane $H=\{\alpha_{1}=1\}$, for simplicity:
$$I_\gamma =\int_{[0,\infty]^{e_{\gamma}-1}} {P(1,\alpha_2,\ldots, \alpha_{e_\gamma}) \over \Psi_{\gamma}^n\big|_{\alpha_1=1}} \Big({\prod_{i=2}^{e_{\gamma}} \alpha_i^{n_i\varepsilon}
 \over \Psi_{\gamma}^{m\varepsilon}\big|_{\alpha_1=1}} \Big) \,  \prod_{i=2}^{e_\gamma} d\alpha_i\ .$$
By substituting the appropriate formula from   $(\ref{nologinsert})$ with $s=n+m\varepsilon$,
we deduce in both cases  $\gamma =G\backslash e$ and $\gamma = G\q e$,    that $I_\gamma$ can be formally rewritten
\begin{equation} \label{Igammanew}
I_\gamma =\int_{0}^{\infty} \int_{[0,\infty]^{e_{\gamma}-1}}  (n+m\varepsilon) {Q(1,\alpha_2,\ldots, \alpha_{e_G}) \over \Psi_{G}^{n+1}\big|_{\alpha_1=1}} \Big({\prod_{i=2}^{e_{G}} \alpha_i^{n_i\varepsilon}
 \over \Psi_{G}^{m\varepsilon}\big|_{\alpha_1=1}} \Big) \,  \prod_{i=2}^{e_G} d\alpha_i\ .
\end{equation}
where $n_e=m$. %By homogenizing with respect to the coordinate $\alpha_1$ and passing back to a projective integral, 
We deduce that  the coefficients of  $\varepsilon^k$ in  $I_\gamma$  can be written as linear combinations of  coefficients in the Taylor expansion of
$(\ref{logperiodint})$ for the bigger graph $G$.
To justify the interchange of integrals, and to show that  $(\ref{Igammanew})$ converges absolutely, observe that any graph polynomial $\Psi_G$, $\Psi_{G\backslash e}$
or $\Psi_{G\q e}$ is a sum of monomials with positive coefficients, and is therefore positive on the domain of integration $(0,\infty)^{e_\gamma-1}$. 
%
%We thus have
%for all $\alpha_2,\ldots, \alpha_{e_\gamma}>0 $:
%$$
%\int_{0}^\infty \Big|{s\Psi_{G\backslash e} \over \Psi_G^{s+1}}\Big|\, d\alpha_e  = {1 \over \big|\Psi_{G\q e}^s \big|}  \quad \hbox{ and } \quad
%\int_{0}^\infty \Big|{ s \Psi_{G\q e} \,\alpha_e^{s-1} \over \Psi_G^{s+1}}\Big| d\alpha_e = {1 \over \big|\Psi_{G\backslash e}^s\big|}  \ .
%$$
The conclusion follows from   a standard application of Fatou's lemma for the Lebesgue integral by taking a
compact exhaustion of the domain of integration $[0,\infty]^{e_G-1}$.
\end{proof}

\begin{cor} \label{corperiodforbiddenminors} The set of graphs with periods contained in a fixed subring of $\R$ is minor closed, and thus determined by a finite set of forbidden minors (theorem \ref{RobSey}).
\end{cor} 
One  aim of this paper is to establish properties of the set of graphs whose periods are multiple zetas or, more generally, values of polylogarithms.

%@@@@@@@@@@@@@@@@@@@@@@@@@@@@@@@@@@@@@@@@@@@@@@@@@@@@@
%@@@@@@@@@@@@@@@@@@@@@@@@@@@@@@@@@@@@@@@@@@@@@@@@@@@@@
%In \S\ref{sectblowups} we will define the set of periods of a mixed Hodge structure associated to a connected graph $G$,  which will also have the property of being minor %monotone.
%@@@@@@@@@@@@@@@@@@@@@@@@@@@@@@@@@@@@@@@@@@@@@@@@@@@@@
%@@@@@@@@@@@@@@@@@@@@@@@@@@@@@@@@@@@@@@@@@@@@@@@@@@@@@

\subsection{Simplification} As is well-known, the periods of a graph $G$ and its simplification $G'$ are related in a rather trivial way via  Euler's  beta function \cite{Sm}.
\begin{lem} For all $0<\rho,\sigma$ and $u,v \neq 0$:
 \begin{equation} \label{EulerBeta}
\int_{0}^{\infty} {x^{\rho-1} \over (u\,x+v)^\sigma} dx = {1 \over u^{\rho}v^{\sigma-\rho}}{ \Gamma(\rho)\Gamma(\sigma-\rho) \over \Gamma(\sigma)}\ .
\end{equation}
\end{lem}
The graph $G$ can be deduced from $G'$ by applying the operations $S$ and $P$ successively. %Since the relation between the periods is essentially well-known, we only
Consider the case of the operation S. Therefore let $G_S$ be a graph obtained from $G$ by subdividing  edge $N$ into two  edges $N$, $N+1$.
We have 
$$\Psi_{G_S}  = \Psi_G(\alpha_1,\ldots,\alpha_{N-1}, \alpha_{N}+\alpha_{N+1}) . $$ 
Now let $r_1,r_2\in \R$ and let  $f$ denote any function  such that the following integrals converge absolutely. Then
$$\int_{0}^\infty \int_0^{\infty} x_1^{r_1} x_2^{r_2} f(x_1+x_2) dx_1 dx_2 = {\Gamma(r_1+1)\Gamma(r_2+1)\over \Gamma(r_1+r_2+2)} \int_0^{\infty} x^{r_1+r_2+1} f(x) dx$$
which is proved by a  change of variables $y=x_1,x=x_1+x_2$ and the  definition of Euler's beta function.
We conclude that
$$\int { \prod_{i=1}^{N+1} \alpha_i^{n_i\varepsilon}  d\alpha_i \over \Psi_{G_S}^{d-2\varepsilon} } = {\Gamma(\rho_1 ) \Gamma(\rho_2) \over \Gamma(\rho_1+\rho_2)}  \int   {\alpha_N^{\rho_1+\rho_2-1} d\alpha_N\prod_{i=1}^{N-1} \alpha_i^{n_i\varepsilon} d\alpha_i 
\over \Psi_G^{d-2\varepsilon}}$$
where $\rho_1=n_N \varepsilon+1$, $\rho_2= n_{N+1} \varepsilon+1$. 
Therefore the decorated Feynman integral  for $G_S$ is a product of gamma factors with 
a period integral for $G$. The  case of a parallel reduction P is similar.
%Thus a graph and its simplification are equivalent from the point of view of the periods
Since the Taylor expansion of Euler's beta function involves products of zeta values, and the graphs we are interested in typically evaluate to multiple zeta values, 
we henceforth  consider a graph to be equivalent to its simplification from the point of view of its periods
(compare also lemma \ref{lemsimp}).  
% The integers $n_e$ can be interpreted as considering the (divergent) Feyman integral of the graph $G$ with $n$ bubbles
%inserted along the edge $e$  \cite{IZ,Sm}.
\subsection{The two-vertex join of $G_1$ and $G_2$}

Let $G_1$, $G_2$ denote two connected primitive-divergent graphs with edges $e_0,\ldots, e_r$ and $e_0',\ldots, e'_s$, respectively, and let 
 $G=G_1\tvj G_2$ be their two vertex join along the endpoints of $e_0, e_0'$. It follows that $G$ is also primitive-divergent, with edges $E_G=\{e_1,\ldots, e_r, e_1',\ldots, e'_s\}$. Let $n_1,\ldots, n_r$ be integers corresponding to the edges $e_i$, and let $m_1,\ldots, m_s$ be integers corresponding to the edges $e_i'$ which satisfy the homegeneity condition $(\ref{homogeneitycond})$:
\begin{equation}\label{tvjhomcond} \sum_{i=1}^r n_i +\sum_{j=1}^s m_j  = - h_{G} \ . \end{equation}
 Define integers $n_0$, $m_0$ by the corresponding  conditions for $G_1$, $G_2$:
$$n_0 =   -(h_{G_1}  +\sum_{i=1}^r n_i) \quad \hbox{ and }  \quad m_0=  -(h_{G_1}  +\sum_{j=1}^s m_j) \ .$$
Let us write $\underline{n}$ (respectively  $\underline{m}$) for $(n_1,\ldots, n_r)$ (resp. $(m_1,\ldots, m_s)$).
\begin{prop} With these notations, we have
\begin{equation}
I_{G}(\underline{n},\underline{m}, \varepsilon) ={\Gamma( -n_0\varepsilon)   \Gamma(-m_0\varepsilon )\over \Gamma(2-\varepsilon )} I_{G_1}(n_0, \underline{n},\varepsilon) \,I_{G_2}(m_0, \underline{m},\varepsilon)\ .
\end{equation}
\end{prop}
\begin{proof}  Let us denote the Schwinger parameters corresponding  to $e_0,\ldots, e_r$ (resp. $e_0',\ldots, e_s'$) by $\alpha_0,\ldots, \alpha_r$ (resp. $ \beta_0,\ldots, \beta_s$).
The decorated Feynman integral of $G_1$  along the hyperplane $\alpha_0=1$, is given by:
$$I_{G_1} (n_0, \underline{n},\varepsilon)  = \int_{\alpha_0=1} {\prod_{i=0}^{r}  \alpha_i^{n_i\varepsilon} \over \Psi_{G_1}^{2-\varepsilon}} \prod_{i=1}^r d\alpha_i\ .$$
By multiplying each $\alpha_i$, for $1 \leq i\leq r$, by the same parameter $\lambda>0$, we obtain:
$$I_{G_1} (n_0, \underline{n},\varepsilon)  ={\lambda^{r+\sum_{i=1}^r n_i \varepsilon} \over \lambda^{(2-\varepsilon)h_{G_1}-1} } \int_{[0,\infty]^r} {\prod_{i=1}^{r}  \alpha_i^{n_i\varepsilon} \over (\Psi_{G_1\backslash  e_0 }\lambda  +\Psi_{G_1\q e_0})^{2-\varepsilon}} \prod_{i=1}^r d\alpha_i$$
Recall from $(\ref{twovertexjoin})$ that $\Psi_{G}=\Psi_{G_1\backslash e_0}\Psi_{G_2\q e'_0} + \Psi_{G_1\q e_0}\Psi_{G_2\backslash e'_0}$. Therefore,
set $\lambda = \Psi_{G_2\q e'_0} / \Psi_{G_2\backslash e'_0}$, which is of homogeneous degree 1.  Since $G_1$ is primitive divergent, we have $r_1+1= 2\,h_{G_1}$.
We deduce that:
$$I_{G_1} (n_0, \underline{n},\varepsilon) \, \left({\Psi^{n_0\varepsilon}_{G_2\backslash e'_0 } \over \Psi^{2+(n_0-1)\varepsilon}_{G_2\q e'_0}}\right) = \int_{[0,\infty]^r} {\prod_{i=1}^{r}  \alpha_i^{n_i\varepsilon} \over \Psi_{G}^{2-\varepsilon}} \prod_{i=1}^r d\alpha_i\ .$$
%Since the left-hand side clearly only depends on the integers $n_i$ and $\varepsilon$, we have:
%Writing $\underline{n}$ for $n_1,\ldots, n_r$,
Now  we can multiply through by $\prod_{i=1}^s {\beta^{m_i\varepsilon}_i}d\beta_i$ and integrate over the intersection $\Delta$ of a suitable  hyperplane 
with $[0,\infty]^s$. This  gives%\footnote{We allow ourselves the irritating convention of writing the integrand outside the integral}:
$$I_{G_1} (n_0, \underline{n},\varepsilon)\int_{ \Delta}\prod_{i=1}^s {\beta^{m_i\varepsilon}_i} d\beta_i 
 \, \left({\Psi^{n_0\varepsilon}_{G_2\backslash e'_0} \over \Psi^{2+(n_0-1)\varepsilon}_{G_2\q e'_0 }}\right)
  =
\int_{ [0,\infty]^{r}\times \Delta} {\prod_{i=1}^{r}  \alpha_i^{n_i\varepsilon}d\alpha_i \prod_{i=1}^s \beta_i^{m_i\varepsilon}d\beta_i\over \Psi_{G}^{2-\varepsilon}} \ .$$
%$$I_{G_1} (0, \underline{n},\varepsilon)\int_{ [0,\infty]^s}\prod_{i=1}^s {\beta^{m_i\varepsilon}_i} { \Psi^{\bullet}_{G'_1 \backslash \{e_0\} }\over  \Psi^{\bullet}_{G'_1\q \{e_0\}}  }  \prod_{i=1}^s d\beta_i=
%\int_{[0,\infty]^{r+s}} {\prod_{i=1}^{r}  \alpha_i^{n_i\varepsilon} \prod_{i=1}^s \beta_i^{m_i\varepsilon}\over \Psi_{G}^{2-\varepsilon}} \prod_{i=1}^r d\alpha_i \prod_{i=1}^s d\beta_i\ .$$
By $(\ref{IGconvdef})$, this is just:
$$I_{G_1} (n_0, \underline{n},\varepsilon)\int_{ [0,\infty]^s}\prod_{i=1}^s {\beta^{m_i\varepsilon}_i} d\beta_i 
\, \left({\Psi^{n_0\varepsilon}_{G_2\backslash e'_0} \over \Psi^{2+(n_0-1)\varepsilon}_{G_2\q e'_0}}\right)
 = I_G(\underline{n},\underline{m},\varepsilon)$$
Now apply $(\ref{EulerBeta})$ with $\rho=-n_0\varepsilon$ and $\sigma=2-\varepsilon$, and writing $\beta_0$ for $x$. By the contraction-deletion formula for $\Psi_{G_2}$, we obtain
$$I_{G_1} (n_0, \underline{n},\varepsilon)\int_{ \Delta\times [0,\infty]} \beta_0^{-n_0\varepsilon-1} d\beta_0 {\prod_{i=1}^s {\beta^{m_i\varepsilon}_i} d\beta_i \over \Psi_{G_2}^{2-\varepsilon }}  = I_G(\underline{n},\underline{m},\varepsilon) {\Gamma(2-\varepsilon) \over \Gamma(-m_0\varepsilon) \Gamma(-n_0\varepsilon)}$$
We have by $(\ref{tvjhomcond})$ and the definitions of $m_0,n_0$ that:
$$m_0+n_0 = h_G-h_{G_1}-h_{G_2}  = -1 \ .$$
This implies that:
$$I_{G_1} (n_0, \underline{n},\varepsilon)\int_{ \Delta\times [0,\infty]}{\prod_{i=0}^s {\beta^{m_i\varepsilon}_i} d\beta_i \over \Psi_{G_2}^{2-\varepsilon }}d\beta_0  = I_G(\underline{n},\underline{m},\varepsilon) {\Gamma(2-\varepsilon) \over \Gamma(-m_0\varepsilon) \Gamma(-n_0\varepsilon)}\ ,$$
which is precisely the statement of the proposition.
\end{proof}
Setting $\varepsilon=0$, and all $n_i, m_i$ to zero,  we retrieve the well-known result:
\begin{equation} \label{ProductFormula}
I_{G} = I_{G_1} \,I_{G_2}\ ,
\end{equation}
for the leading term of the Feynman integral of  a two-vertex join. Note that  this induces a drop in the expected transcendental weight of $I_G$ \cite{WD}.

%and the transcendental weight of $I(G)$ is ..

\subsection{The star-triangle relations} The  star-triangle relations have been studied extensively by physicists, but mainly from the point of view
of momentum or coordinate space. We give a short parametric proof here.

\begin{lem} Let $G_{\triangle}, G_{Y}$ be a pair of graphs related by the Star-Triangle operation. Let $\alpha_1,\alpha_2,\alpha_3$ denote the 
Schwinger parameters corresponding to the edges of the star or triangle.
We have the following identities between periods of $G_{\Delta}$ and $G_{Y}$:
$$  \int{ (f^{123})^{\kappa}{\prod_{i=1}^n \alpha_i^{\lambda_i} d\alpha_i}\over \Psi_{\triangle}^{\mu}}\delta(H) = 
 \int   { (f^0)^{\kappa} {\prod_{i=1}^n \alpha_i^{\lambda'_i} d\alpha_i}\over \Psi_{Y}^{\mu}} \delta(H') $$
where  $\lambda_i,\lambda'_i, \mu$ are parameters  such that the integrals converge which also satisfy $\lambda_i+\lambda'_i=\mu-2$  for $i=1,2,3$,
and $\lambda_j=\lambda_j'$ otherwise. Also, 
 $f^{123}= \Psi(G_\triangle\backslash\{1,2,3\})$ and $f^0 = \Psi(G_Y\q \{1,2,3\})$,  and  $\kappa= \sum_{i=1}^3 \lambda_i+3-2\mu$.
 
\end{lem}
\begin{proof}
Let $\alpha_1,\alpha_2,\alpha_3$ denote the Schwinger coordinates of $G_{\triangle}$ corresponding to the edges of the triangle, and 
let $\beta_1,\beta_2,\beta_3$ denote the Schwinger coordinates of $G_Y$ corresponding to the dual 3-valent vertex. We know from examples \ref{ExStar} and \ref{ExTriangle}
that
$$\Psi_{G_\triangle}= f^{123} \alpha_1\alpha_2\alpha_3 + (f^1+f^2)(\alpha_1\alpha_2) + (f^1+f^3)(\alpha_1\alpha_3) + (f^2+f^3)(\alpha_2\alpha_3) + f^0 (\alpha_1+\alpha_2+\alpha_3).$$
$$\Psi_{G_Y}= f_{123} + (f_1+f_2)\beta_3 + (f_1+f_3)\beta_2 + (f_2+f_3)\beta_1 + f_0 (\beta_1\beta_2+\beta_1\beta_3+\beta_2\beta_3). $$
where $f^{123}=f_0$, $f^0=f_{123}$ and $f_i=f^i$ for $i=1,2,3$. One checks that
$$\int_{[0,\infty]^3} {\prod_{i=1}^3\alpha_i^{\lambda_i} d\alpha_i \over \Psi_{G_\triangle}(\alpha_1,\alpha_2,\alpha_3)^{\mu} }  = \Big( {f^0\over f^{123}}\Big)^{\sum_{i=1}^3 \lambda_i+3-2\mu}\int_{[0,\infty]^3} {\prod_{i=1}^3\beta_i^{\mu-2-\lambda_i} d\beta_i \over \Psi_{G_Y}(\beta_1,\beta_2,\beta_3)^{\mu} } $$ 
by making the change of variables $\alpha_i=f^0 (f^{123} \beta_i)^{-1}$.
\end{proof}
If $2 \mu =D-2\varepsilon$, and say $D=3$ dimensions, then the  numerator terms $(f^0)^{\kappa}$ and $(f^{123})^{\kappa}$ can be made to disappear from the formula to give  
an  identity between the residues of graphs  (the uniqueness relations).
In 4 spacetime dimensions  this fails, but
 one obtains  relations between certain subsets of the set of periods of $G_{\Delta}$ and $G_{Y}$,  which are sometimes known as the almost-uniqueness relations.
However, the relationship between the star-triangle relations and the residues (leading terms) of Feynman graphs is not at all clear.
 % containing a triangle with edges $e_1,e_2,e_3$, and let $G_{Y}$ denote the graph obtained by replacing the triangle with a star.  
 %and $f_0f_{123}=f_1f_2+f_1f_3+f_2f_3$.

\subsection{The complete graph and universal Feynman integral}
It follows from proposition $\ref{propminor}$   that the Feynman integral for the complete graph is the universal Feynman integral for graphs with a given number of vertices.
\begin{cor} Let $G$ be any simple Feynman graph with $v_G$ vertices. Then since $G$ is a subgraph of the complete graph $K_{v_G}$, we have
$$\Pe(G)\subseteq \Pe(K_{v_G}) \hbox{ and } \Pel(G)\subseteq \Pel(K_{v_G})\ .$$
\end{cor}
It follows from the proof of proposition $\ref{propminor}$ that for any primitive divergent graph $\gamma$ with at most $v$ vertices, there exists a polynomial $N_\gamma$,
 such that
$$I_\gamma = \int_{\Delta} {\Omega_{\gamma}  \over \Psi_\gamma^2} =   \int_{\Delta} { N_{\gamma} \over \Psi_{K_v}^{r-e_\gamma}}\Omega_{K_v}\ ,$$
where $r=\binom{v}{2}$ is the number of edges of $K_v$. By repeatedly applying $(\ref{nologinsert})$, the polynomial $N_\gamma$ is a product of graph polynomials and powers of Schwinger parameters $\alpha_e$,
but is not unique since it depends on the choice of embedding of $\gamma$ into $K_v$ and the order in which the edges are removed from $K_v$ to obtain $\gamma$.
One can make it unique using the natural action of the symmetric group $\Sym_v$ on $v$ letters which acts on $K_v$ by permuting its vertices.

%\begin{defn}  For any graph $\gamma$ with $v$ vertices,  let  $\widetilde{N}_\gamma=\Psi_{K_v}^{e_\gamma}\big( \sum_{\sigma \in \Sym_v} \sigma(N_\gamma)\big)$ % %where $\sigma$ acts
%on the Schwinger coordinates of the complete graph $K_v$. 
%\end{defn}

\begin{rem} If $X_{K_v} = \{\Psi_{K_v}=0\}$ denotes the graph hypersurface of $K_v$, and writing
 $r=\binom{v}{2}-1$ for the number of edges of $K_v$, observe that
$$\Pro^r \backslash X_{K_v}  \cong GL(v-1)/ O(v-1)$$
can be identified with the symmetric space of  symmetric $(v-1)\times (v-1)$  non-singular square matrices \cite{Sta}.
\end{rem}

For physical applications, one is required to sum the Feynman amplitudes over all graphs in a theory at a given loop order. 
The previous corollary  should allow one  in principle to place all such (unrenormalized) integrands over a common denominator, and reduce the sum of the contributions
of all graphs  to a single integral.

%In particular, the same idea allows one to compute primitive linear combinations of products of graphs in $\phi^4$ (linear combinations
%of graphs in $\phi^4$ which are primitive for the Connes-Kreimer coproduct) by finding a graph in which all the elements are minors and reducing to a common denominator.
%Very few examples of Feynman amplitudes for such primitive elements are known, since previous methods do not work in this case.

%In conclusion, there is a huge number of identities between generalized periods of graphs which change their underlying combinatorics. An  
%important question is to determine  the  irreducible quantities (transcendents) for these relations, since these will  be the building blocks of perturbation theory.

%\begin{rem} One reason we took $\ref{defnrealperiods}$ as our definition of real periods was because it substantially simplifies the proof of proposition $\ref{propminor}$.
%The same result should hold  for more sophisticated definition of the periods. One natural way to do this is to define the
%periods of a primitive divergent graph to be the periods of its corresponding mixed Hodge structure  as defined in \cite{B-E-K}. Generalizing this
%construction to all graphs, the equivalent of proposition $\ref{propminor}$ will then correspond to the inclusion of a face.
%\end{rem}

%\newpage

\section{The graph hypersurface and blow-ups}\label{sectblowups}
Following \cite{B-E-K}, we blow up coordinate linear subspaces in $(\Pro^1)^N$ which are  contained in the graph hypersurface so that 
the Feynman integral defines a period of the corresponding mixed Hodge structure.

\subsection{The graph hypersurface}\label{sectgraphhypersurface}
Let $G$ be a graph, with edges numbered $1,\ldots, N$, and let $\alpha_i$ denote the corresponding Schwinger parameters, viewed
as affine  coordinates on each copy of $\Pro^1$ in  $(\Pro^1)^{N}$. Equivalently, let
\begin{equation} \widetilde{\psi}_G = \sum_{T\subseteq G} \prod_{e\notin T} a_e \prod_{e \in T} b_e\ ,
\end{equation}
where the sum is over all spanning trees $T$ of $G$, and $(\alpha_e:1)=(a_e:b_e)$.  The variables $a_e$ and $b_e$ 
are interchanged on passing to the planar dual graph (matroid).
%Later in the paper,  we will always work in the affine coordinates $\alpha_i$ to lighten the notation.

Let  $B,X_G\subset (\Pro^1){N}$ be the (resp. coordinate, graph) hypersurfaces defined by:
$$B: \prod_{i=1}^{N} a_ib_i=0\quad  \hbox{ and } \quad  X_G: \widetilde{\psi}_G = 0 \ .$$
Let $D=[0,\infty]^{N}$ denote the real hypercube in $(\Pro^1(\R))^{N}$ with positive coordinates $\alpha_i$. Its boundary is contained in $B$.
%The variety $X_G$ is known as the graph hypersurface  and is highly singular in general.
%but is usually considered in projective space $\Pro^N$. 
Since the coefficients
of $\widetilde{\psi}_G$ are positive, $X_G$ does not meet the interior of  ${D}$. To see how it meets the boundary of $D$, let 
%define, for any  two disjoint sets $S,T\subseteq \{1,\ldots, N\}$,
$$L_{S,T}: \prod_{s\in S} a_s \prod_{t\in T} b_t=0 \qquad\hbox{and} \qquad F_{S,T} =  D\cap L_{S,T}\ .$$
where $S,T\subseteq \{1,\ldots, N\}$ are disjoint. It is clear that  $F_{S,T}= F_{S,\emptyset} \cap F_{\emptyset, T}$. 

\begin{lem} \label{lemfaces} (cf \cite{B-E-K}, lemma 7.1) Let $S,T$ be as above. 
The following are equivalent:

 \qquad i). $F_{S,T} \cap X_G \neq \emptyset$.

 \qquad ii). $F_{S,T} \subset X_G$.

\qquad iii). $F_{S,\emptyset} \subset X_G$  or $F_{\emptyset, T}\subset X_G$.

\noindent 
The first case  $F_{S,\emptyset} \subset X_G $  occurs if and only if the subgraph of $G$ defined by $S$ contains a loop ($h_1(S)>0$), 
and the second case  $F_{\emptyset,T} \subset X_G $  occurs if and only if removing the set of edges $T$ from $G$  causes it to disconnect $(h_0(G\backslash T)>0)$.
\end{lem}
\begin{proof} It follows from %the contraction-deletion formula 
$(\ref{contractdelete})$ that the restriction of  $\widetilde{\psi}_G$ to the face $F_{S,T}$
is   $\widetilde{\psi}_{G\backslash T \q S}$.  Since it has positive coefficients, the zero locus $X_G$ meets $F_{S,T}$ if and only if it is identically zero along $F_{S,T}$. 
The result then follows from corollary \ref{corVanishingsubgraphs}.
\end{proof}
\begin{defn}
\label{defnPro1G} Let $G$ be a  graph. For any set of edges $S\subset E_G$, write $\Pro^1_S$ for  $(\Pro^1)^{|S|}$ whose affine coordinates are the Schwinger parameters of $S$, and let 
\begin{equation} \label{defforgetmap}
\pi_S:  \Pro^1_G \To \Pro^1_S\ ,\end{equation}
denote the map $\pi_S:( \alpha_e)_{e\in E_G} \mapsto (\alpha_e)_{e\in E_S}$. If we identify $\Pro^1_S$ with $\mathrm{Hom}(S,\Pro^1)$, then $\pi_S$  is the natural map
induced by  the inclusion $S\subset E_G $. 
\end{defn}
\subsection{Combinatorics of blow ups} As in  \cite{B-E-K}, we can blow up linear spaces $L_{S,T}$ in such a way that the strict transform of the  graph hypersurface is moved away from  the inverse image of the domain of integration $D$. 
In general, suppose we are given a set of faces
$\mathcal{F}=\{ F_{S_i,T_i}\},$
 of codimension $|S_i|+|T_i|\geq 2$. Let $\widetilde{\mathcal{F}}$  be the set of all intersections  of faces in $\mathcal{F}$,
and  blow-up  $\Pro^1_G$ along  $L_{S,T}$ for each $F_{S,T} \in\widetilde{\mathcal{F}}$, in order of increasing dimension. %  set of  coordinate hyperplanes 
% corresponding to each intersection of faces. % faces in $\widetilde{\mathcal{F}}$,
More precisely, for each $0\leq k\leq N-2,$ let $\mathcal{F}^{(k)}$ denote the set of intersections of faces $F_{S,T}=\cap_i F_{S_i,T_i}$ of dimension $k$. % where $F_i$ are elements of  $\mathcal{F}$.
%Each element $F_{S,T}$ in $\mathcal{F}^{(k)}$ corresponds to a coordinate hyperplane
%$L_{S,T}\subset (\Pro^1)^N$.
As is standard practice, we blow up the  coordinate hyperplanes $L_{S,T}$ corresponding to elements $F_{S,T}$ in $\mathcal{F}^{(0)}$,  %followed by the set of hyperplanes in $\mathcal{F}^{(1)}$, and
%finally the set of hyperplanes in $\mathcal{F}^{(N-2)}$, 
$\mathcal{F}^{(1)}$,\ldots, $\mathcal{F}^{(N-2)}$ 
in turn. %order
% of increasing dimension. 
Denote the resulting space by $\pi_{\widetilde{\mathcal{F}}}:\BP_{\widetilde{\mathcal{F}}}\rightarrow  \Pro^1_G.$ It  does not depend on the choice of order of the blow-ups, and only depends on $\widetilde{\mathcal{F}}.$
%
%\begin{defn}
%We  denote the corresponding blow-up  by
%$\pi_{\widetilde{\mathcal{F}}}:\Pro_{\widetilde{\mathcal{F}}}\rightarrow  (\Pro^1)^N,$
% and 
 Let $D_{\widetilde{\mathcal{F}}}$  be the (real analytic) closure  of the
inverse image of   $D=(0,\infty)^N$. It is a compact manifold with corners, whose   boundary stratification is a polytope whose
 poset of faces  is determined by the following lemma.
 %we denote by $P_{\widetilde{\mathcal{F}}},$ and is determined by the following lemma.
%\end{defn}

\begin{lem}\label{lemfaces}
For every face $F_{S,T} \in \widetilde{\mathcal{F}}\cup \{ F_{(i,\emptyset)},  F_{(\emptyset, i)}, {1\leq i\leq N}\}$, let $P_{S,T}$ denote  the  strict transform of 
 $F_{S,T}$.
The resulting  map from $ \widetilde{\mathcal{F}}\cup \{ F_{(i,\emptyset)},  F_{(\emptyset, i)}, {1\leq i\leq N}\}$ to the set of facets of $D_{\widetilde{\mathcal{F}}}$  is a bijection.
 Two such facets $P_{S,T}$ and $P_{S',T'}$ meet
if and only if one of the following  holds:

i).   %one facet $F_{S,T}$ is contained in the other $F_{S',T'}$. This is equivalent to the condition
$S\subset S'$ and $T\subset T'$,

ii). $S'\subset S$ and $T'\subset T$,

iii).  $(S\cup S') \cap (T\cup T')= \emptyset\ ,$ and $F_{S\cup S', T\cup T'}$ is not an element of $ \widetilde{\mathcal{F}}$.

\end{lem}
\begin{proof} It is easily verified that truncating the faces of a hypercube in increasing order of dimension leads to the above poset structure.
\end{proof}

\begin{defn}\label{defnblownhypercube}
A set of faces $\mathcal{F}$ is \emph{polarized} if, for every $F_{S,T}\in \mathcal{F}$, either  $S=\emptyset$ or
$T=\emptyset$.
In this case we  write 
$\mathcal{F}= \mathcal{F}_0 \cup \mathcal{F}_\infty,$
where $\mathcal{F}_0=\{F_{S,\emptyset}: F_{S,\emptyset}\in \mathcal{F} \}$ and $\mathcal{F}_\infty=\{F_{\emptyset,T}: F_{\emptyset,T}\in \mathcal{F} \}.$ %\end{defn}
\end{defn}

Now let $G$ be a connected graph, and let $X_G, B$ be as in  $\S\ref{sectgraphhypersurface}$. 
The set of  faces $F_{S,T}$ which need blowing up are given by lemma $\ref{lemfaces}$. 
%For every set of edges $R$ in $G$, denote the corresponding
%subgraph by $R$ also. 
Let
\begin{eqnarray} \label{Fuvdef}
\quad \mathcal{F}_{0} &= &  \{ \, F_{S,\emptyset} : S\subset E_G \hbox{ minimal s.t. } |S|\geq 2, \hbox{ and } S \hbox{ contains a loop} \} \\
\quad \mathcal{F}_{\infty} &= &   \{\, F_{\emptyset,T} : T\subset E_G \hbox{ minimal s.t.  }|T|\geq 2, \hbox{ and }  G\backslash T  \hbox{  is disconnected} \}  \nonumber 
\end{eqnarray}
and consider the polarized set $\mathcal{F}=\mathcal{F}_0 \cup \mathcal{F}_\infty$.

\begin{example} \label{expresunset}
Consider the sunset diagram (the graph with vertex set $V=\{1,2\}$ and edges $E=\{\{1,2\},\{1,2\},\{1,2\}\}$, numbered 1,2,3).  Then $(\ref{Fuvdef})$ gives 
$\mathcal{F}_0= \{12,23,13\}$, and $\mathcal{F}_\infty= \{123\}.$
Then
$\widetilde{\mathcal{F}}=\{F_{12,\emptyset},F_{23,\emptyset},F_{13,\emptyset},F_{123,\emptyset},F_{\emptyset,123}\}$, so
 one must first blow up the two points $\alpha_1=\alpha_2=\alpha_3=0$, and $\alpha_1=\alpha_2=\alpha_3=\infty$, followed by the
three lines $\alpha_1=\alpha_2=0$, $\alpha_1=\alpha_3=0$, and $\alpha_2=\alpha_3=0$. 
\end{example}

\begin{defn}  \label{defngeneralblowup} Write $\BP_G$ for $\Pro_{\widetilde{\mathcal{F}}}$, where $\mathcal{F}$ is the polarized set given by $(\ref{Fuvdef})$.
Let $B'\subset \BP_G$ denote the total transform of $B$, and let $X'_G\subset \BP_G$ denote the strict transform of the graph hypersurface $X_G$. The irreducible
components of $B'$ are given in lemma $\ref{lemfaces}$. 
Let $D_G$ denote the closure (in the real analytic topology) of the strict transform of the hypercube $(0,\infty)^N$. It is a compact manifold with corners
whose boundary is contained in $B'$. Its poset of faces is determined by lemma $\ref{lemfaces}$. 
\end{defn}
The following lemma is well-known (implicit in  \cite{B-E-K}, \S3)  and is essentially  the Schwinger-parametric interpretation of the renormalization Hopf algebra.

\begin{lem} Let $G$ be a connected graph, and let $\gamma\subset G$  be a subgraph, with connected components $\gamma=\gamma_1 \cup \ldots \cup \gamma_k$. Then
\begin{equation}\label{graphquotR}
\Psi_G =   \Psi_{G/ \gamma}  \prod_{i=1}^k \Psi_{\gamma_i}  + R, \end{equation}
 where $R$ is a polynomial of total degree $\geq h_{\gamma}+1$ in the parameters $\{\alpha_e, e\in E(\gamma)\}$. 
Here,   $\Psi_{G/ \gamma}$ is the graph polynomial of the  graph-theoretic quotient $G/ \gamma$  where each component  $\gamma_i$ is shrunk to a point (note the shrinking of  loops is not zero here).
\end{lem}

%Now suppose that $G$ is primitive divergent, and let
For any connected graph $G$, define:
$$\omega_G = {d\alpha_1 \wedge \ldots \wedge  d\alpha_N \over \Psi_G^2} \in \Omega^N(\Pro^1_G \backslash X_G)\ .$$

\begin{prop} \label{propblowup} \cite{B-E-K}  Let $\pi_G: \BP_G \rightarrow \Pro^1_G$ denote the blow-up defined above. Then
%$B'\subset \Pro_G$ is a smooth normal crossings divisor, and
 $B'  \subset \BP_G$ is a normal crossings divisor,  and  $D_G$ does not intersect $X'_G$. If $G$ is primitive divergent, then 
 $\pi_G^*(\omega_G)$ has no poles along components of $B'$.
\end{prop}
\begin{proof}  The fact that $B'$ is a normal crossings divisor follows from general properties of blowing up linear subpaces in projective space.
To show that $D_G$ does not intersect $X'_G$ (no further blow-ups are required), is entirely analogous to the proof of proposition 7.3 in 
  \cite{B-E-K}, so we omit the details. The idea is that the restriction of $\Psi_G$ to each face of $B'$ can be expressed in terms of graph polynomials
 of sub and quotient graphs of $G$ by the contraction-deletion relations, or by $(\ref{graphquotR})$ for the exceptional divisors. The statement follows by an induction on the number of edges.

To see how the primitive divergence comes in, the key remarks are:% the statement about $\pi_G^*(\omega_G)$ uses the primitive divergence hypothesis.
\footnote{Given that $D_G$ does not intersect $X'_G$,  the fact that $\pi_G^*(\omega_G)$ has no poles along $B'$ is equivalent to the convergence of the  residue $I_G$ for $G$ primitive divergent.} 
%To see where this comes in, the key remarks are:

(i) The order of the pole of  $\omega_G$ along a facet of the form $F_{S,\emptyset}$ is $ 2 h_1(S)$.

 This follows from $(\ref{graphquotR})$ on  setting
$\gamma =S$.  Writing  $(\ref{graphquotR})$ in homogeneous coordinates $a_e,b_e$, where $(\alpha_e:1) = (a_e:b_e)$, shows that the order of vanishing of $\widetilde{\psi}_G$ along a facet
$F_{\emptyset, T}$, where $T$ is the set of edges of $G$ not in $\gamma$, 
 is $|G/ \gamma|-h_1(G/\gamma )$.  By Euler's formula, this is $V(G/ \gamma)-1$, which is $0$ if and only if $G\backslash T$ is connected. For any subset of edges $T\subset G$, write $\gamma_T$ for the subgraph of $G$ defined by the  complement of the set of edges of $T$ in $G$. We have shown:
 
(i')  The order of the pole of  $\omega_G$ along a facet of the form $F_{\emptyset,T}$ is $ 2(V(G/ \gamma_T)-1)$.

\noindent Now, the order of the poles of $\pi_G^*(\omega_G)$ along components of the exceptional divisors of $B'$ are computed using the fact that
 blowing up a linear subspace of codimension $p$ decreases the order of the pole by $p-1$. This  follows from a direct calculation.

That $\pi_G^*(\omega_G)$ has no poles along $B'$, involves checking, in the case (i), that 
$(|S|-1)-2h_1(S)\geq 0$, which holds precisely because $G$ is primitive divergent. The corresponding inequality in the case (i') is: 
%$|T|-1-2(V(G\q \gamma_T)-1)\geq 0$.   We have
\begin{eqnarray}
(|T|-1)-2(V(G/ \gamma_T)-1) &=& |G/\gamma_T| - 1 - 2(|G/ \gamma_T| - h_1(G/ \gamma_T))\nonumber \\
&= &|G|- |\gamma_T| - 2( |G|-|\gamma_T|- h_1(G) + h_1(\gamma_T))-1 \nonumber \\
& =& \big(2h_1(G)-|G|\big) + \big(|\gamma_T| - 2h_1(\gamma_T)-1\big)  \geq 0 \nonumber
\end{eqnarray}
which is again non-negative since $G$ is primitive divergent.  \end{proof} 
%The rest of the proof is just as in \cite{B-E-K}. \end{proof}

Recall that the domain of integration of a Feynman integral is  $\Delta_H=[0,\infty]^N\cap H$, where $H$ is a hyperplane  in $\Pro^1_G$ such that $0\notin H$. 
%which does not pass through the origin.
If $H$ is the hyperplane $\sum_{i=1}^N \alpha_i=1$, then $H$ intersects the hypercube $[0,\infty]^{N}$ in faces of the form $F_{S,\emptyset}$ only. After blowing up, the subspace $\pi_G^{-1}(H)\cap \BP_G$ therefore has the  identical geometry as the blow-ups constructed in \cite{B-E-K}, \S7.
In this paper,  we  will always fix an edge, say  $e=N$, and 
choose the hyperplane $H\subset \Pro^1_G$ defined by  $\alpha_e=1$. 
We can therefore define the mixed Hodge structure of any connected graph  $G$:  %primitive divergent graph $G$:
\begin{equation} \label{GraphMotive}
\Mot_G = H^{N-1}(\BP_G \backslash X_G', B'\backslash (B'\cap X_G'))\ ,
\end{equation}
where, by abuse of notation, $\BP_G, X'_G, B'$ actually denotes their intersection with the hyperplane $\alpha_N=1$.
In a standard way, the domain of integration $\pi_G^{-1}(\Delta_H)$ defines a relative Betti homology class
in $\gr^W_0 H_{N-1}(\BP_G \backslash X_G', B'\backslash B'\cap X_G')$ (see \cite{B-E-K},  proposition 7.5).
When $G$ is primitive divergent,  proposition \ref{propblowup}  implies that   the differential form  
$$ \omega_G \,\delta(H)=\prod_{i=1}^{N-1} {d\alpha_i \over \Psi_G^{2}\big|_{\alpha_N=1}}$$
%$$ {\prod_{i=1}^N d\alpha_i \over \Psi_G^2} \delta(H)
%$$
defines a relative de Rham cohomology  class in $H^{N-1}(\BP_G \backslash X_G', B'\backslash B'\cap X_G')$. 

\begin{cor} If $G$ is primitively divergent, then 
the residue 
$$I_G=\int_{\Delta_H} \omega_G \delta(H)$$ 
converges, and defines  a period of $\Mot_G$.
\end{cor}

\begin{rem}  Since the boundary components of $B'$  corresponding to facets of the form $F_{i,\emptyset}$ (resp. $F_{\emptyset, i}$) intersect $X_G$ 
along $X_{G\backslash i}$ (resp. $X_{G\q i}$),  the inclusion of a facet defines a morphism  %-mixed Hodge structure 
$\Mot_{\gamma} \rightarrow \Mot_G$  where $\gamma= G\backslash i$ or  $G\q i$.
By induction, for any minor $\gamma \minor G$ (which contains the chosen edge $N$)  there is a  morphism % of mixed Hodge structures
$\Mot_{\gamma} \rightarrow \Mot_G$ of mixed Hodge structures. 
%and  it follows that the set of periods of the mixed Hodge structures $M_G$  %(and the $M_G$ themselves) %(considered for arbitrary, and not just primitive divergent, graphs)
%are also minor monotone 
%(up to  the minor detail that  the edge $N$ should lie in $\gamma$, and that $\Mot_{\gamma}$ and $\Mot_G$ a priori depend on this choice of edge).
Proposition \ref{propminor}  computes a representative for  the image of the class of the  Feynman differential form $\omega_{\gamma}$ under this map. 
\end{rem}

\newpage
\section{Stratification  and singularities of integrals} \label{SectStrat}
We recall some results from stratified Morse theory \cite{Morse}, % define the  Landau variety of a proper map, 
and use them  to compute  the   singularities of integrals associated to graph hypersurfaces. 
\subsection{The Landau varieties  of a graph}
Let $P,T$ be  smooth connected complex analytic manifolds, and  consider a smooth proper   map
$$\pi: P \To T\ .$$
Let $S$ be a  closed analytic subset of  $P$. It gives rise to a stratification of $P$, i.e.,  
  there is   a sequence of closed analytic submanifolds
$$S^0=P\supset S^1=S\supset S^2 \supset  \ldots \supset  S^k$$
such that the complements $S^i\backslash S^{i+1}$  are smooth. Furthermore, the irreducible components $A_k$ (the open strata)  of $S^{i} \backslash S^{i+1}$ have the property that
the boundary $\partial A_k = \overline{A}_k\backslash A_k$ is a union of strata of  lower dimension, and 
satisfy Whitney's conditions A and B. % The discriminant locus can be defined as follows.
The \emph{critical set}  $cA_i$ of a stratum $A_i$ is defined to be  the (analytic)  set of points of $A_i$ where $\pi$ fails to be submersive: 
$$ cA_i=\{ x\in A_i: \, \mathrm{rank }\,\, T_x \pi < \dim T\}\ .$$
 In particular, all strata of dimension less than $\dim T$ are critical. 
\begin{defn}  The \emph{Landau variety}    $L(S,\pi)$ is  the codimension $1$ part of 
$\pi(\cup_i cA_i),$ where the union is taken over all strata of $S$ (cf \cite{Pham}).
\end{defn} 
 
 It follows from Thom's isotopy theorem that $L(S,\pi)$ has the following property:   for all open  $U\subset T\backslash L(S,\pi)$, the  restriction of
$\pi : \pi^{-1}(U) \rightarrow U $
is a locally trivial stratified map, i.e.,    each point $u\in U$ has an open neighbourhood $V$ 
such that $\pi^{-1}(V)$  is homeomorphic (as a stratified space)  to $\pi^{-1}(u)\times V$ (\cite{Morse}, \S1.7). % A more general (and more precise) theorem is stated in \S1.7 of ..
In other words, $L(S,\pi)$ is the smallest subvariety of $T$ outside of which $\pi$ has topologically constant (but not necessarily smooth) fibers.

%The Landau variety $L(S,\pi)$ is then defined  by $\pi(\cup_i cA_i),$ where the union is taken over all strata of $S$.
%The result then follows from Thom's isotopy theorem [ref]: the restriction of $\pi$ to any open set $U\subset T\backslash L(S,\pi)$  is  a smooth proper map which
%is submersive on each stratum, and  is therefore  locally trivial. 
%\footnote{Note that the Landau variety $L(S,\pi)$ is \emph{not} the locus where the fibers of $\pi$ are singular.  It may happen (see example \ref{3redexampleDODGSON} below) %that all the fibers are singular, but that the singular locus
%defines a stratum which projects submersively onto the base $T$, and is therefore not critical.}

\begin{defn} Let $G$ be a connected graph, and let 
 $B,X_G$ denote the coordinate and graph hypersurfaces in $\Pro^1_G$.
For any  subset $K\subset E_G$, consider the map $(\ref{defforgetmap})$
$$\pi_K: \Pro^1_G \To \Pro^1_K\ .$$ 
%$$\pi_K\circ \pi_G:\BP_G \To \Pro^1_G \To \Pro^1_K\ .$$ 
 We define the \emph{Landau variety} of $G$ relative to $K$ to be   $L(X_G\cup B, \pi_K)$. %\Pro^1_K$.
  \end{defn}

Since  $L(X_G\cup B,\pi_K)$ is a hypersurface in  $\Pro^1_K$, we will often represent it
 as the zero locus of a set of 
polynomials $\{f_1,\ldots, f_n\}$, where $f_i \in  \Q[(\alpha_{e})_{e\in K}]$.

\subsection{Singularities of integrals}
%We sketch the well-known result that the singularities of a function obtained by integrating along the fibers of $\pi$ are contained in the Landau variety of $\pi$.
%Leray and Pham,  extended to the relative case.
%Throughout this section, we  shall work over the complex numbers  and  we shall therefore write $P$ instead of $P(\C)$, and so on.

Let $\pi:P \rightarrow T$ be  as above, % proper smooth morphism as above, and
let $S=X \cup B\subset P$,
%\footnote{We do not need this assumption, only that the integrals
%we consider should be absolutely convergent on their domains of integration}, 
%where no intersection of components of $B$ is contained in $X$, 
and let $L(S,\pi)$ be its  Landau variety.
%Let  $L\subset T$ be a subvariety  such that
%$\pi : P\backslash \pi^{-1}(L) \rightarrow  T \backslash L $ is a locally trivial  fibration of stratified varieties.
For each complex point $t\in T\backslash L$, let $P_t$ denote the fiber of $\pi$ over $t$, and let $\ell =\dim P_t$ denote its dimension.
Set $X_t= X\cap P_t$ and $B_t = B \cap P_t$. 
% and let $\omega$ be a closed multivalued differential form of degree $\ell$ on $X\backslash A$.
%Let $\Delta \subset X\backslash A$ be a smooth compact submanifold of $X\backslash A$ of real dimension $\ell$, whose boundary is contained in $B$. We are interested in absolutely
%convergent integrals of the form
%$$I = \int_\Delta \omega\ , $$
%which converges. Can have logarithmic singularities. Single-valued on $\Delta$.
%NOTE: need $\omega$ single-valued on $\Delta$ to be well-defined.
%
%
For all $t$ in a neighbourhood of a  fixed complex point $t_0 \in T\backslash L$, suppose that
we are given a continuous family of real compact submanifolds   $\Delta_t\subset P_t\backslash X_t$ of dimension $\ell$ such that $\partial \Delta_t \subset B_t$, and $\Delta_t\cap X_t=\emptyset$.
Suppose that we are given  a family   $\omega_t\in \Omega^\ell (P_t \backslash X_t)$  of  closed differential
$\ell$-forms which depend analytically on $t$ for all $t\in T\backslash L$.   Let %We are interested in  the following integral:
\begin{equation} \label{itDodef} I(t) = \int_{\Delta_t} \omega_t\ ,
\end{equation}
defined in a neighbourhood of $t_0$.
It is absolutely convergent since  $\omega_t$ has singularities contained in $X_t$, which is disjoint from the compact set $\Delta_t$.
%Assume furthermore that $\omega_t$ extends to a global family of forms in $\Omega^\ell (X_t \backslash A_t)$ for all $t\in T\backslash L$.

\begin{thm}\label{thmsingofint}  $I(t)$ extends to a multivalued function on  $T\backslash L$. %, where $L^1$ denotes  the codimension 1 part of $L$. %$T\backslash L$.
\end{thm}

\begin{proof} (see \cite{Pham}, Chapter X).
%Let $j:P\backslash X \backslash B \rightarrow P\backslash X$ be the inclusion map, and considerthe sheaf $\mathcal{F} = (R^\ell \pi_* j_! \Q)^{\vee}$ on $T$ whose stalks at $t\in T\backslash L$ are isomorphic to:
By repeating the construction of (\cite{Pham},  Appendix A) in the relative case, one can construct a homology sheaf $\mathcal{F}$ whose stalks at $t \in T \backslash L$ are isomorphic to:
%
%
%
%
%Since $\pi$ is smooth,  one can  consider the sheaf
%$\mathcal{F} =R^{\ell} \pi_* \Omega_{X\backslash A,B\backslash (A\cap B)}$ on $T$ whose 
%define for all $p \in \N$ a sheaf of homology groups on $T$ to be the sheaf associated to the presheaf:
%$$U \mapsto H_{p+r }(X, X_{T\backslash U})\ ,  $$
%where $r=2k$ is the real dimension of $T$. Its stalks are (\cite{Pham}, Appendix A):
%$$H_p(X_t)\cong H_{p+r}(X,X_{T-t}) \cong \lim_{\underset{U\supset t}{\rightarrow}} H_{p+r}(X,X_{T-U})\ . $$
%Repeating this construction for  $X_t\backslash A_t$ and $B_t \backslash (B_t\cap A_t)$, and taking the mapping cone of the inclusion
%of $B_t \backslash (B_t\cap A_t)$ into $X_t\backslash A_t$, we obtain a sheaf of relative homology groups $\mathcal{F}$ on $T$ whose
%stalks at $t$ are isomorphic to
$$\mathcal{F}_t \cong H_\ell( P_t\backslash X_t, B_t \backslash (B_t\cap X_t))\ .$$
%Equivalently, one can define  $\mathcal{F}$ to be the sheaf associated to the presheaf:
%$$U\mapsto  H_{\ell+r} \big(X\backslash A, (B\cup X_{T\backslash U}) \backslash (A \cap (B\cup X_{T\backslash U})\big)\ ,$$
%for all open sets $U\subset T$. %An excision argument shows that the germs of $\mathcal{F}$ at the point $t$ are isomorphic to
By assumption, the relative homology class $[\Delta_t]$ defines a continuous section of $\mathcal{F}$ in a neighbourhood of $t_0$.
Since $\pi : P\backslash \pi^{-1}(L) \rightarrow  T \backslash L $ is a locally trivial fibration, the sheaf $\mathcal{F}$ is locally constant,
and therefore the section $[\Delta_t]$ extends to a global multivalued section of $\mathcal{F}$.
%Let $T'$ denote a covering space
%of $T\backslash L$ on which $[\Delta_t]$ becomes single-valued,  and fix a point $t'\in T'$ which lies above $t'_0 \in T$.
Now fix a  point $s\in T\backslash L$,  and  let $U$ be a simply-connected
open neighbourhood of $s$. For any $t\in U$, let $[\Delta'_t]$ denote a fixed local  branch of this multivalued section.
 Since  $\omega_t$  extends to  an analytic family $\omega_t \in \Omega^\ell(P_t \backslash X_t)$ for all $t\in T\backslash L$,
%the previous argument shows that for all $t\in U$,
 we can define a (multivalued) continuation of $I(t)$ by the formula
$$ I(t) = \int_{[\Delta'_t]} \omega_t\quad \hbox{ for all }  t \in U\ .$$
To finish the proof, it suffices to   show that the function $I(t)$  defined in this way is analytic in a neighbourhood of $s$.
Note that by Hartogs' theorem, a function is analytic at $t\in T$ if and only if it is analytic at $t$ along every smooth curve in $T$,
 so we can assume that $T$ is of dimension $1$.
  We can replace $[\Delta'_t]$ with a homology class which does not depend on $t$  by considering the relative Leray coboundary map:
$$\delta_t: H_\ell\big(P_t\backslash X_t , B_t\backslash (B_t\cap X_t)\big) \To H_{\ell+1}\big(P\backslash (X \cup  P_t), B\backslash (B\cap (X \cup  P_t))\big)\ .$$
 The group on the right is locally constant 
 %(one way to see this is to
%take the boundary of the locally constant sheaf $\mathcal{F}$ under the exact sequence for the triple
%$(X\backslash A, (B\cup X_{T-W})\backslash ((B\cup X_{T-W})\cap A), B\backslash (B\cap A))$, take the limit $W=t$, and apply an excision argument),
 and is therefore
 constant on the simply-connected set $U$. Therefore
$\delta_t[\Delta'_t]$ is  constant for all $t \in U$, and let  $h$ denote a  representative of this relative homology class. We can assume that
$h$ is a real smooth submanifold of $X$ such that $\partial h \subset B$, and $h\cap A=\emptyset$ ($h$ is a tubular neighbourhood of a
representative of $[\Delta'_t]$). By the residue formula we can write the local branch of $I(t)$ as:
$$I(t) =  \int_{[\Delta'_t]} \omega_t=  {1 \over 2i\pi} \int_{h} {\omega_t \wedge du \over t-u}\ ,$$
for all $t\in U$.
%we thus replace the domain of integration with one which is independent of $t$. Restrict to simply-connected neighbourhood.
By doing a  Taylor expansion under the integral, we see at once that $I(t)$ is  holomorphic in a neighbourhood of $s$ in $T\backslash L$. \end{proof}
\begin{rem} \label{remmultivalued}
The proof  extends to the case where $\omega_t$ is a multivalued function on $P\backslash (X\cup B)$, since we can pass to a universal cover  of $P\backslash (X\cup B)$
and the projection to $T$  will still be a locally trivial stratified map on an open subset of $T\backslash L$. One must add to the initial assumptions, however,  that $\omega_t$ is single-valued along 
$\Delta_t$ for some $t$ in a neighbourhood of $t_0$, otherwise  $(\ref{itDodef})$ may not  be well-defined.
\end{rem}
\subsection{Application to partial Feynman integrals}
%We apply the above to the Feynman integrals  considered in
%by partially integrating with respect to a subset of variables.
%Therefore, let $G$ be a graph with $e_G$ edges,   let  $X=(\Pro^1)^{e_G}$, $A_G$, and $B$ be as above,
 % where for each edge $e$ of $G$, the corresponding
%copy of $\Pro^1$ has  coordinate $\alpha_e$.  Let $A_G$ be the graph hypersurface and  $B$ be the coordinate
 %hypercube as defined  in $\S\ref{sectgraphhypersurface}$, 

Let $G$ be a connected graph, $K$ a subset of edges of $G$, and   consider the coefficient of some power of $\varepsilon$ in an absolutely convergent Feynman integral 
 %as in $\S\ref{sectPeriodsallgraphs}$. 
of the form $(\ref{logperiodint})$.
Integrating only  the edges in $E_G \backslash E_K$ gives a partial integral:
%We can assume that the hyperplane $H$ is given by $\alpha_e=1$, for $e \in  E(K)$, and consider the partial integral:   
  \begin{equation}\label{IdefPF}
 I^{K}((\alpha_e)_{e\in E_K})= \int_0^{\infty}\!\!\!\ldots \int_0^{\infty} {P(  (\alpha_{e}), (\log(\alpha_{e})), \log(\Psi_G)) \over \Psi_G^n}  \prod_{e\in E_{G\backslash K}} d\alpha_e\ ,
\end{equation}
where $P$ is a polynomial with coefficients in $\Q$, and $n\in \N$.  
 %and $0\notin H\subset \Pro^1_G$ is a hyperplane. % as in $\S \ref{sectPeriodsallgraphs}$.
%Choose a  subset $K$ of edges of $G$,  and let $T=\Pro^1_{G\backslash K}$. 
%denote the product of $\Pro^1$'s
%corresponding to  the edges of $G\backslash K$.
%We can assume that $\dim T\geq 2$ and that the hyperplane $H$ is entirely contained in $T$.%
%Now consider the partial Feynman integral 
%$$I(\{\alpha_e\}_{e\in G\backslash K}) =\int_0^{\infty}\ldots \int_0^{\infty}{P( \ldots, \alpha_{e}, \log(\alpha_{e}),\ldots, \log(\Psi_G)) \over \Psi_G^n}\,\delta(H) \prod_{e \in K} d\alpha_e  %$$ % obtained by integrating out the variables $\alpha_e$, for $e\in K$. 
It is a multivalued
 function on some  open subset of  $\Pro^1_K$.  By Fubini's theorem, the full Feynman integral $(\ref{logperiodint})$ is the integral of $I^K$ over a hyperplane $H$ chosen to lie in $\Pro^1_K$.
 
%$$I =\int_{\Delta_H} I\big(\{\alpha_e\}_{e\in G\backslash K}\big)   \, \delta(H) \prod_{e\in G\backslash K}d\alpha_e\ .$$
\begin{thm} \label{thmsingofFeynInt} The  integral $(\ref{IdefPF}) $ is a multivalued  function on $\Pro^1_K\backslash L(G,K)$, and has
 no singularities on the interior of the  hypercube $[0,\infty]^{|K|}\subset \Pro^1_K$. % upon which it
%has a canonical  single-valued  branch. % which takes values in the real numbers.
\end{thm}
\begin{proof}
Consider the blow-up of definition \ref{defngeneralblowup}.  Since the exceptional divisors lie over hyperplanes $\alpha_i=0,\infty$ which are critical, 
the map $\pi=\pi_K\circ \pi_G: P_G \To \Pro^1_K$, where $P_G$ is stratified by $X'_G\cup B'$, is still a  locally trivial fibration outside $L(G,K)$ (see \S \ref{sectrelativecohom}).
By proposition \ref{propblowup}, the domain of integration $D_G$ and $X'_G$ do not meet, so the argument of 
 theorem $\ref{thmsingofint}$ still applies in this slightly modified situation. By  assumption, $(\ref{IdefPF})$ is absolutely convergent, so the pull-back of the integrand  has no polar singularities along $\pi^{-1}(\Delta_t)$.  By remark \ref{remmultivalued}, we can then %apply theorem $\ref{thmsingofint}$ after 
pass to a  universal cover, 
since the multivalued functions $\log(\alpha_e)$ and $\log(\Psi_G)$ are ramified only along $B'$ and $X'_G$. We conclude that $I(\{\alpha_e\})$ extends to a multivalued holomorphic function on $\Pro^1_K\backslash L(G,K)$, which proves the first part of the theorem.  
In this case, the germ of the chain of integration $\Delta_t=[0,\infty]^{|G\backslash K|}$ is constant for all $t\in [0,\infty]^{|K|} \subset \Pro^1_K$.
 %To prove the second part,  
 %observe that
Since $\Psi_G$  is a  sum of monomials with positive coefficients, 
we can %directly
 differentiate under the integral %sign 
 to   deduce that
   $(\ref{IdefPF})$ is analytic on the interior
of %the hypercube
 $[0,\infty]^{|K|}$. 
%Since $\Delta_t$ is constant on $[0,\infty]^{|K|}$, that and
%deduce that the branch $I(\{\alpha_e\})$ is analytic.
 \end{proof}

\subsection{Landau varieties of linear hypersurfaces} Let $k\geq 1$, and 
consider  the situation where $P=(\Pro^1)^k\times T$,  and $T=(\Pro^1)^{2^k}$ with coordinates we denote by $\psi^{I}_J$, where
$I\cap J=\emptyset$ and $I\cup J=\{1,\ldots, k\}$. Let $x_1,\ldots, x_k$ be the coordinates in the fibers $(\Pro^1)^k$
 of the projection $\pi_k:P \rightarrow T$, and consider a general
linear form
$$\psi=\sum_{I\subset \{1,\ldots, k\}} \psi^I_{J} \prod_{i\in I} x_i \, $$ where $J$ is the complement of $I$ in $\{1,\ldots, k\}$. We omit $I$ or $J$ from the notation  if 
they are the empty set. As before, 
let  $X$ denote the zero locus of $\psi$ in $P$, and let $B$ denote the hypercube $ \cup_{i=1}^k \{x_i=0,\infty\}$.
 Note that $\pi_k$ fails to be submersive
at a smooth point $p$ of the  hypersurface $X$ if and only if:
\begin{equation}
{\partial \psi \over \partial x_i}(p) =0 \ , \quad \hbox{ for all } 1\leq i \leq k\ .
\end{equation}
At a singular point $p$, these partial derivatives will also  vanish, but it can happen  that a smooth  stratum in the singular locus of $X$ may
project submersively onto $T$.
In the following examples, we write down the Landau varieties $L(X\cup B,\pi_k)$ in the cases $k=1,2,3$, and consider a Feynman-type integral of the form:
$$I^k = \int_{[0,\infty]^k} {dx_1\ldots dx_k \over \psi^2} \ .$$
 Strictly speaking, we should first stratify $X\cup B$, and then compute the critical strata, but a posteriori it will be clear that $\pi_k$ is  locally trivial on $T\backslash L(X\cup B,\pi_k)$.

%desingularize it in the case when it isn't, but this problem will take care of itself.

\begin{example}\label{1redexample}  
%First consider the trivial case of a one-dimensional fibration:
%$$\pi: \Pro^1 \times T \To T\ ,$$
% Borrowing the notation from $\S $, we  write the equation of our linear hypersurface as
If $k=1$, we have $\psi=\psi_1 +\psi^1 x_1$, and  $B$ is   $x_1=0, \infty$. The critical strata $X\cap B$
are  given by  $\{\psi=0\} \cap \{x_1=0\}$ and $\{\psi=0\}\cap \{x_1=\infty\}$, which project down to $\{\psi^1=0\}$ and $\{\psi_1=0\}$ respectively.
Clearly,  $\pi_1|_X$ is submersive everywhere outside this locus.
The Landau variety is therefore $\psi^1 \psi_1=0$, and indeed $I^1=   \int_0^\infty {dx \over \psi^2} = {1 \over \psi_1 \psi^1}\ $has singularities at both $\psi^1=0$ and $\psi_1=0$.
\end{example}

\begin{example} \label{2redexample} When $k=2$,  we have 
$\psi=\psi_{12} +\psi_2^1 x_1+\psi^2_1 x_2 + \psi^{12} x_1x_2$, and $B$ is the square $x_1,x_2=0,\infty$. The four strata $\{\psi=0\} \cap \{x_1=0,\infty\} \cap \{x_2=0,\infty\}$
are of codimension $3$ and are therefore critical. They project down to $\psi_{12}=0, \psi^1_2=0, \psi^2_1=0, \psi^{12}=0$. The intersection of $X$ with a
single face of $B$ defines a stratum which is  linear, and therefore projects submersively onto $T$, by the previous example. The stratum $X$ fails to be submersive when
$\psi=   {\partial \psi \over \partial x_1 } =  {\partial  \psi \over \partial x_2 } = 0,$
which can only occur when $\psi_{12}\psi^{12} - \psi^2_1 \psi^1_2 = 0$.  The Landau variety is therefore
$$L_2= \{\psi^{12}=0\}\cup \{\psi^1_2=0\} \cup \{\psi^2_1=0\}\cup \{\psi_{12}=0\}\cup \{\psi_{12}\psi^{12} - \psi^2_1 \psi^1_2=0\}\ .$$
%From the one-dimensional computation above we have:
%$$I_2=\int_0^\infty \int_0^\infty {dx_1 dx_2\over \psi^2} = \int_0^\infty {dx_2 \over (\psi_{12}+\psi^1_2 x_2)(\psi^2_1 + \psi^{12} x_2 )}\ .$$
%
%
%
Four components of $L_2$ correspond to the locus where $X$ passes through a corner of the square $B$, and one when $X$ degenerates into a product of two lines (dashed):
\begin{center}
\fcolorbox{white}{white}{
  \begin{picture}(170,119) (234,-128)
%  \begin{picture}(170,124) (234,-123)
    \SetWidth{1.0}
    \SetColor{Black}
    \Line(262,-22)(262,-122)
    \Line(246,-106)(375,-106)
    \Line(246,-39)(375,-39)
    \Line(352,-22)(352,-122)
    \SetWidth{1.5}
    \Bezier(246,-118)(314,-118)(364,-89)(375,-22)%JaxoID:FBez
    \Text(379,-56)[lb]{\Large{\Black{$X:\psi=0$}}}
    \Text(385,-110)[lb]{\Large{\Black{$B$}}}
    \Text(200,-110)[lb]{\Large{\Black{$x_2=0$}}}
    \Text(334,-20)[lb]{\Large{\Black{$x_1=\infty$}}}
    \Text(200,-42)[lb]{\Large{\Black{$x_2=\infty$}}}
    \Text(244,-20)[lb]{\Large{\Black{$x_1=0$}}}
    \SetWidth{1.0}
    \Line[dash,dashsize=5](296,-22)(296,-122)
    \Line[dash,dashsize=5](246,-62)(375,-62)
  \end{picture}
}
\end{center}
%
%\begin{center}
%\fcolorbox{white}{white}{
%\begin{picture}(223,164) (441,-257)
 %  \begin{picture}(273,64) (291,-57)
%    \SetWidth{1.0}
 %   \SetColor{Black}
  %  \Line(481,-117)(481,-256)
%    \Line(459,-233)(637,-233)
  %  \Line(459,-140)(637,-140)
 %   \Line(606,-117)(606,-256)
%   \SetWidth{1.0}
%    \Bezier(459,-249)(552,-249)(622,-210)(637,-117)%JaxoID:FBez
%    \Text(550,-210)[lb]{\Large{\Black{$\psi=0$}}}
 %   \Text(418,-235)[lb]{\Large{\Black{$x_2=0$}}}
 %   \Text(590,-114)[lb]{\Large{\Black{$x_1=\infty$}}}
%   \Text(418,-143)[lb]{\Large{\Black{$x_2=\infty$}}}
  %  \Text(465,-114)[lb]{\Large{\Black{$x_1=0$}}}
  %  \SetWidth{0.3}
 %   \Line[dash,dashsize=5](528,-117)(528,-256)
 %  \Line[dash,dashsize=5](459,-172)(637,-172)
% \end{picture}
%}
% \end{center}
Decomposing $I^1$ into partial fractions and integrating gives:%  the following expression for $I^2$:
\begin{equation}\label{exI2} 
  I^2  %= \int_0^\infty \!\!\!\!{dx_2 \over (\psi_{12}+\psi^1_2 x_2)(\psi^2_1 + \psi^{12} x_2 )}
  =  { \log(\psi^{1}_2)\! +\!\log(\psi^2_1 ) \!-\! \log(\psi^{12})\!- \! \log(\psi_{12})  \over \psi^{1}_2\psi^2_1-\psi^{12}\psi_{12}}\ ,
\end{equation}
which is  ramified along $L_2$, as expected, and has unipotent monodromy.
\end{example}
%\begin{figure}[h]
%\caption{The fibers  of $\pi$  change  over  the Landau variety $L_2$. Over a generic point of $T$, $A:\psi=0$ meets the  square in four distinct points.
%This changes when $A$ passes through a corner of the square, which is exactly when one of $\psi^{12},\psi^1_2, \psi^2_1, \psi_{12}$ vanishes. Finally, $A$
%degenerates into a product of two lines (indicated by dashed lines) precisely when $\psi_{12}\psi^{12} - \psi^2_1 \psi^1_2$ vanishes. }
%These lines are parallel to
%the coordinate axes because $\psi$ is linear.}
% \end{figure}

\begin{example} \label{3redexample} The situation  quickly gets out of hand. If $k=3$, then
%$$\psi=\psi_{123} +\psi^1_{23} x_1+ \psi^2_{13} x_2 + \psi^3_{12} x_3 + \psi^{12}_3 x_1x_2+ \psi^{13}_2 x_1x_3+ \psi^{23}_1 x_2x_3+\psi^{123}x_1x_2x_3\ ,$$
 $B$ is a cube   $x_1,x_2,x_3=0,\infty$. The intersection of $X:\psi=0$ with the corners of the cube are critical strata whose projections give
the contributions
\begin{equation} \label{exlandau1}
\{\psi_{123} ,\psi^1_{23} , \psi^2_{13} , \psi^3_{12}, \psi^{12}_3 , \psi^{13}_2,  \psi^{23}_1,\psi^{123}\}
\end{equation}
to the Landau variety. Next, the intersection of $\psi=0$ with any one-dimensional edge of the cube is always submersive, by the first example. The intersection
of $\psi=0$ with any of the six faces of the cube puts us in the situation of the previous example, and gives  rise to six quadratic terms
\begin{equation} \label{exlandau2}
\psi^{ij}_k\psi_{ijk}-\psi^i_{jk}\psi^j_{ik}\quad \hbox{ and }\quad \psi^{ijk}\psi^k_{ij}-\psi^{ik}_{j}\psi^{jk}_{i} \quad \hbox{ where } \quad \{i,j,k\}=\{1,2,3\}\ .
\end{equation}
%$$\psi^{12}_3\psi_{123}-\psi^1_{23}\psi^2_{13}\ , \psi^{123}\psi^3_{12}-\psi^{13}_{2}\psi^{23}_{1} ,  $$
Finally, the hypersurface itself is not submersive when $\psi=0$ and $\partial \psi/\partial x_i =0$ for all $1\leq i\leq 3$. These four equations
admit a solution when the discriminant:
\begin{eqnarray} \label{exlandau3}
F& =&  (\psi^{123}\psi_{123})^2+  4 \,\psi^1_{23}\psi^2_{13}\psi^3_{12}\psi^{123}+ 4\, \psi_{123}\psi^{23}_1\psi^{13}_2\psi^{12}_3 \\
& &
\sum_{\{i,j,k\}=\{1,2,3\}} (\psi^i_{jk}\psi^{jk}_i)^2   -2 \, \psi^i_{jk} \psi^j_{ik}\psi^{ik}_j\psi^{jk}_i -2 \,\psi^{ijk}\psi_{ijk} \psi^i_{jk}\psi^{jk}_i  \nonumber
\end{eqnarray}
vanishes. Thus the Landau variety $L_3$ for a general hypersurface $\psi$
 is the union of the zero loci  of the polynomials $(\ref{exlandau1})$, $(\ref{exlandau2})$, and $(\ref{exlandau3})$. 
%As before, we can consider
%the following integral:
%\begin{equation}\label{I3general}
%I_3= \int_0^\infty {dx_1 dx_2 dx_3\over \psi^2} = \int_0^\infty { \log(\psi^{1}_2) +\log(\psi^2_1 ) - \log(\psi^{12}) -\log(\psi_{12})  \over \psi^{1}_2\psi^2_1-\psi^{12}\psi_{12}} dx_3\ .
%\end{equation}
%where $\psi^{1}_2$ stands for $\psi^{13}_2x_3+\psi^1_{23}$ and so on. The function $F$ is the discriminant of the denominator $\psi^{1}_2\psi^2_1-\psi^{12}\psi_{12}$, which is %quadratic in $x_3$. 
One
can verify that  the integral  $I^3$
%$I_3=  \int_0^\infty { \log(\psi^{1}_2) +\log(\psi^2_1 ) - \log(\psi^{12}) -\log(\psi_{12})  \over \psi^{1}_2\psi^2_1-\psi^{12}\psi_{12}} dx_3\ $
 can be written as an ugly expression involving the dilogarithm function, and has unipotent monodromy
on a degree 2 covering of the complement of the Landau variety $L_3$ (it has
 quadratic ramification around $F=0$).
\end{example}

\begin{example} \label{3redexampleDODGSON} 
What saves us in the  graph polynomial case are the Dodgson identities. Let us reconsider the previous example when $\psi=\Psi$ is a  (generic) graph polynomial.
By identity $(\ref{FirstDodgsonId})$, the quadratic terms in $(\ref{exlandau2})$ factorize as follows:
\begin{equation} \label{newlandau2}
\Psi^{ij}_k\Psi_{ijk}-\Psi^i_{jk}\Psi^j_{ik}= (\Psi^{i,j}_k)^2\ , \hbox{ and } \Psi^{ijk}\Psi^k_{ij}-\Psi^{ik}_{j}\Psi^{jk}_{i}= (\Psi^{ik,jk})^2 \ .
\end{equation}
The polynomial $F$ in $(\ref{exlandau3})$ is now %the discriminant of a perfect square, and is therefore
 identically zero.
What happens is  that there is a unique solution to the equations   $\Psi=0$ and ${\partial \Psi\over \partial x_i} =0$ for $i=1,2,3$  given by:
$$(x_1,x_2,x_3) = \big(-{\Psi^{2,3}_1 \over \Psi^{12,13}},-{\Psi^{1,3}_2 \over \Psi^{12,23}},-{\Psi^{1,2}_3 \over \Psi^{13,23}} \big) \ .$$
It follows that  every fiber of $\pi$ has a singular point. These singular points  define a stratum which is everywhere submersive over the subvariety of  $T$ implied by $(\ref{newlandau2})$, and meets the  boundary of the cube $x_1,x_2,x_3=0,\infty$ over loci of the form  $\Psi^{i,j}_k=0$ and $\Psi^{ik,jk}=0$,
which we have already taken into account in the situation $(\ref{newlandau2})$.

\begin{cor} The Landau variety obtained by projecting a  graph hypersurface $\Psi$   is the zero locus of  graph polynomials corresponding to the eight corners of a cube:
\begin{equation}
\{\Psi_{123}, \Psi^1_{23}, \Psi^2_{13}, \Psi^3_{12}, \Psi^{12}_3, \Psi^{13}_2, \Psi^{23}_1, \Psi^{123}\}\ ,
\end{equation}
and  Dodgson polynomials corresponding to the six faces of a cube:
\begin{equation}
\{\Psi^{1,2}_3, \Psi^{13,23}, \Psi^{1,3}_2, \Psi^{13,23}, \Psi^{2,3}_1, \Psi^{12,13}\}\ .
\end{equation}
\end{cor}
\noindent
Likewise, the integral  $(\ref{exI2})$ also simplifies  since its  denominator  $\Psi^{1}_2\Psi^2_1-\Psi^{12}\Psi_{12}$
is now  a perfect square $(\Psi^{1,2})^2$, and can be easily integrated one more time.
%$$I_3 =\int_0^\infty {dx_1 dx_2 dx_3\over \psi^2} = \int_0^\infty { \log(\psi^{1}_2) +\log(\psi^2_1 ) - \log(\psi^{12}) -\log(\psi_{12})  \over (\psi^{1,2})^2} dx_3\ .$$
\begin{cor} Using both  Dodgson identities $(\ref{FirstDodgsonId})$ and $(\ref{SecondDodgsonId})$, we calculate:
\begin{equation}
I^3 ={\Psi^{123} \log \Psi^{123} \over    \Psi^{12,13}\Psi^{12,23}\Psi^{13,23} }-{\Psi_{123} \log \Psi_{123} \over    \Psi^{2,3}_1\Psi_2^{1,3}\Psi_3^{1,2} }+
\underset{\{i,j,k\}}{\sum} {\Psi^i_{jk} \log \Psi^i_{jk} \over    \Psi^{ij,ik}\Psi_j^{i,k}\Psi_k^{i,j} } - {\Psi^{ij}_{k} \log \Psi^{ij}_{k} \over    \Psi^{ij,ik}\Psi^{ij,jk}\Psi_k^{i,j} } \nonumber
\end{equation}
where the sum is  over all $\{i,j,k\}=\{1,2,3\}$, giving 8 terms in total.
\end{cor}
Notice  that the integral $I^3$ now has unipotent monodromy, and  that its weight has dropped by one (the dilogarithms are
replaced by logarithms, see \S\ref{sectInitialInt}).
% This is the crucial fact that explains why the transcendental weight of periods of Feynman graphs are always shifted, as we shall explain in \S.
\end{example}
\begin{rem} 
%One might  hope to compute the Landau variety $L_k$ of a general linear hypersurface $\psi=0$, and the  stratification on $L_k$ which makes $\pi_k$ into a locally trivial stratified %map.  The graph polynomial of any graph $G$
%lies in   a subvariety of $T$  given by all identities between its coefficients $\psi^I_J$.  One might then  hope to  computing the vanishing cycles, and then the monodromy of the %integral $I_k$ on $T\backslash L_k$.  But given the complexity of the geometry, and 
Instead of trying  to pursue  this approach for higher $k$, %since the identities considered in $\S. $ are highly dependent on the combinatorics of $G$, 
we will instead  \emph{approximate} the Landau varieties of graph hypersurfaces by a simpler inductive method.
The previous example might  lead one to think that the  difficulties of example \ref{exlandau3}  went away simply because the graph polynomial $\Psi_G$
is the determinant of a symmetric matrix.  This is indeed  true for $k\leq 4$, but we will see that  the components of $L^1(G,K)$  for general graphs with $|E_G|-|E_K|\geq 5$   are non-linear
in the Schwinger parameters. %  as we shall see in the following chapter. 
Beyond this point,  the structure of the Landau varieties become  highly dependent on the combinatorics of $G$. 
% has to exploit the specific combinatorics of individual graphs to generate enough identities between Dodgson polynomials  if one wants the 
%Landau varieties to remain tractable.
\end{rem}

\newpage
\section{Genealogy of singularities}\label{sectGenealogy}
We describe  a naive  inductive method   %to stratify
to compute an upper bound for the Landau variety  (discriminant) of  %the complement of
 a configuration of singular hypersurfaces. %algorithmic
%upper bound for its Landau variety.
%In the case of a graph hypersurface relative to a hypercube,   we exploit the Dodgson identities to
%yield  sufficient conditions for it to be  stratified by Tate varieties.
% which, for certain graphs, can be written in terms of the Dodgson graphs.

\subsection{Basic observations}
The  idea is  to approximate the Landau variety of a projection by considering  one-dimensional projections at a time. 
\begin{lem}\label{lemfib} Let $P_1,P_2,P_3$ be smooth connected complex analytic manifolds, and 
consider two smooth proper maps
$P_1 \overset{\pi_1}{\To} P_2 \overset{\pi_2}{\To} P_3  .$
Let $S\subset P_1$ be a  closed analytic subset. Then the Landau variety $L(S,\pi_1)$ is a closed analytic subset of $P_2$, and 
$$L(S,\pi_2\circ \pi_1) \subseteq L( L(S,\pi_1), \pi_2) \ .$$
If   $P_1 \overset{\pi'_1}{\To} P'_2 \overset{\pi'_2}{\To} P_3$  are also smooth proper maps of connected complex analytic manifolds satisfying $\pi_2\circ \pi_1 = \pi_2'\circ \pi_1'$, then it follows that 
$$L(S,\pi_2\circ \pi_1) \subseteq L( L(S,\pi_1), \pi_2)\cap L( L(S,\pi'_1), \pi'_2) \ .$$
\end{lem}
\begin{proof} There  is a stratification on $L(S,\pi_1)$ such that $\pi_1$ is a locally trivial stratified map on each stratum \cite{Morse}, and similarly for $L(L(S,\pi_1),\pi_2)$ with respect to $\pi_2$. The  lemma follows from the fact that  the composition of two locally trivial stratified maps is still a locally trivial stratified map.
%and the second part is obvious.
% If a map is locally trivial over two open sets $U_1$ and $U_2$,
%then it is locally trivial over their union $U_1\cup U_2$, which  proves the second part.
\end{proof}

By  theorem $\ref{thmsingofint}$, this lemma corresponds to the  fact that  the composition of two holomorphic functions is holomorphic, and that
a function which is locally holomorphic on two open sets is locally holomorphic on their union.

\subsection{Resultants} We recall  some properties  of resultants of polynomials in one variable (\cite{GKZ}, Chapter 12).
Let $f=\sum_{i=0}^n a_i x^i$, $g=\sum_{i=0}^m b_i x^i$ denote two polynomials with complex coefficients, where $a_nb_m\neq 0$. If we write 
 $f=a_n\prod_{i=1}^n (x-\alpha_i)$, and $g=b_m\prod_{j=1}^m (x-\beta_j)$,  then  the resultant of $f$ and $g$ is defined by: 
$$[f,g]_x=a_n^m b_m^n \prod_{i,j} (\alpha_i-\beta_j)\ .$$
It  is clear that the resultant is multiplicative, {\it i.e.}, $[f_1f_2,g]_x = [f_1,g]_x [f_2,g]_x$, and
that $[g,f]_x= (-1)^{mn}[f,g]_x.$ 
%For this reason we will often consider the resultant only up to sign. 
It is an irreducible polynomial in the coefficients of $a_i, b_j$, given by Sylvester's determinantal formula. 
 We also write
\begin{equation}\label{discdef}
D_x(f) := a_n^{-1} [f,f']_x \ , \quad  [0,f]_x := a_0 \ ,\quad [\infty, f]_x:=a_n\ , \end{equation}
for the discriminant of $f$,  and for the constant and leading terms of $f$. If $f=f_0+f_x x$ and $g=g_0+g_x x$ are both of degree one in $x$, Sylvester's formula reduces to
\begin{equation}
[f,g]_x= g_xf_0-g_0f_x \ .
\end{equation}
It follows from definition  $(\ref{discdef})$ that 
%\begin{equation}\label{discofproduct}
$D_x(fg) = D_x(f) D_x(g)\, [f,g]_x^2\ .$%

\subsection{Iterated one dimensional projections} % We first consider what happens when we fiber  one dimension at a time.
Let $P=\Pro^1 \times T$ and consider the one-dimensional projection 
$\pi:P \rightarrow T$, $\pi(x,t) =t$, where $x$ is the coordinate on $\Pro^1$.
 Let $S=S_1\cup\ldots \cup S_N$, where the $S_i$ are  (possibly singular) distinct, irreducible
hypersurfaces in $P$.  Then the Landau variety is given by:
\begin{equation}
L(S,\pi) = \bigcup_{1\leq i<j\leq N} \pi(S_i\cap S_j)\,\, \cup\,\, \bigcup_{1\leq i\leq N} \pi (c(S_i)) \ .
\end{equation}
%where  $\NS(S_i)$ denotes the non-submersive locus of $\pi|_{S_i}$, {\it i.e.}, the union of the  set of singular points of $S_i$
%with the set of smooth critical points $cS_i$.

% All strata of codimension 2 of the form $S_i \cap S_j$ are critical, by definition. Let $\NS(S_i)$ denote the non-Define   $S_{\pi}\subset T$ to be the codimension one part of..
%\begin{lem} \label{lemLand1} The Landau variety $L(X,\pi)$ is contained in $S_{\pi}$.
%\end{lem}%
%\begin{proof} It is clear that the map $X\backslash ( S\cup \pi^{-1}(S_{\pi})) \rightarrow T\backslash S_{\pi}$ is a locally trivial fibration,
% since  $S_i\backslash S_i \cap \pi^{-1}(S_\pi)$  are smooth hypersurfaces in  $X\backslash \pi^{-1}(S_{\pi})$, and
% we have removed all intersections $S_i\cap S_j$, so the corresponding
%stratification is trivial.
%\end{proof}

\begin{lem}  \label{lemLand1} Suppose that $S$ contains the hyperplanes $x=0$, $x=\infty$, and 
every other   $S_i$ is  the zero locus of  a  polynomial
$ f_i=\sum_{n\geq 0 } a_{i,n} x^n$, where $a_{i,n}:T \rightarrow \Pro^1$. % which is irreducible at the generic point of $T$. 
 Then 
$L(S,\pi)$ is given by the zeros of the resultants:
$$[0,f_i]_x \ ,\  [\infty,f_i]_x \ , \ [f_i,f_j]_x \ , \ D_x(f_i) \ .$$
%In particular, $L(S,\pi)=L^1(S,\pi)$.
\end{lem}
\begin{proof}
If $f_i$ is of degree $\geq 2$ in $x$, then  $\pi(c(S_i))$ is given by the zero locus of the discriminant $D_x(f_i)$. 
The terms $\pi(S_i\cap S_j)$ are given by the zero locus of the resultants $[f_i,f_j]_x$, and $[0,f_i]_x$, $[\infty,f_i]_x$.  The key remark  is that in the degenerate case when $f_i$ is of degree 1, the non-submersive locus $\{f_i=0\}\cap \{\partial f_i/\partial x=0\}$ is already  contained in both   $[0,f_i]_x$ and  $[\infty,f_i]_x$.
\end{proof}
%Note that the fact that we have removed $x=0,x=\infty$ takes care of the factors $a_n^mb_m^n$ in the definition of the resultant.

  %$[f_i,f'_i]_x$.
% the non-submersive locus $\NS(S_i)$ is defined by
%$f_i=\partial f_i /\partial x = 0$, and therefore its projection. This  is nothing other than the
%discriminant of $f_i$. This gives the following algorithm.
\begin{figure}[h!]
  \begin{center}
%    \leavevmode
    \epsfxsize=6.0cm \epsfbox{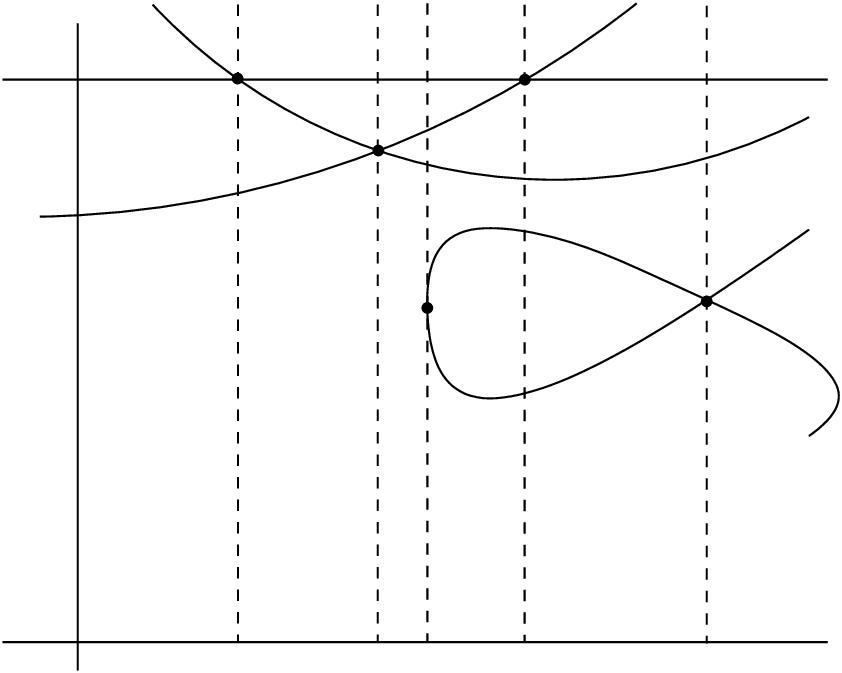}
  \label{2gen}
   \put(10,2){ $T$}
  \put(10,90){ $T\times \Pro^1$}
\put(-210,4){ $x=0$}
\put(-210,118){ $x=\infty$} 
\put(-160,138){ $f_1=0$} 
\put(-30,40){ $f_2=0$} 
\put(-60,138){ $f_3=0$} 
  \end{center}
\end{figure}

Now suppose that $P=(\Pro^1)^{N}$, with coordinates $x_1,\ldots, x_N$.  By abuse of notation, we let $\pi_i$ denote a projection 
onto the hyperplane $x_i=0$ without specifying the source and target space.  Thus if $1\leq k\leq N$, we write
\begin{equation}\label{piscommute}
\pi_{[1,\ldots, k]}= \pi_1\circ \ldots \circ \pi_k = \pi_{\sigma(1)} \circ \ldots \circ \pi_{\sigma(k)}\ ,
\end{equation}
where $\sigma$ is any permutation of $\{1,\ldots, k\}$.  If $S=\bigcup_i S_i$ is a set of irreducible hypersurfaces which contains the coordinate hypercube $\cup_{i}\{x_i=0,\infty\}$ we can approximate
$L(S,\pi_{[1,\ldots, k]})$ inductively using lemma \ref{lemfib}.  We can represent $S$  by the set of  polynomials $f_i$ which define
each non-trivial  component $S_i$ of $S$. 
%When each $S_i$ is given by the zero locus of a polynomial $f_i$ in the variables $x_1,\ldots, x_n$,
%this leads, by the previous remarks,  to the following reduction algorithm.

\begin{defn}\label{defreduction}  Let $S=\{f_1,\ldots, f_N\}$ denote  a set of irreducible polynomials in  variables $x_1,\ldots, x_N$ as above. For any $1\leq r\leq N$,  let
$$\widetilde{S}_{x_r} = \{ [0,f_i]_{x_r}\ , \ [\infty,f_i]_{x_r}\ , \   [f_i,f_j]_{x_r} \ ,\  D_{x_r}( f_i) \}\ ,$$
and let   %the \emph{simple reduction} of $S$ with respect to $x_r$ to be the set  
$S_{x_r}$  be the set of irreducible factors of  elements of $\widetilde{S}_{x_r}$. By iterating, we set
\begin{equation}\label{Sitreddef}
S_{(x_1,x_2,\ldots, x_k)}= \big(S_{(x_1,\ldots,x_{k-1})}\big)_{x_k}\ .
\end{equation}
%which we call the  \emph{simple reduction} of $S$ with respect to $x_1,\ldots, x_k$. 
\end{defn}
%By abuse of notation, we write $S$ for both the set of polynomials $S$ and their zero locus. 

% It follows from lemmas $\ref{lemfib}$ (i) and
%$\ref{lemLand1}$, that:
%\begin{equation}
%(\Pro^1)^{n} \backslash S \overset{\pi_{x_1}\,}{\To} (\Pro^1)^{n-1} \backslash S_{x_1} \overset{\pi_{x_2}\,}{\To} \ldots \overset{\pi_{x_k}\,}{\To}  (\Pro^1)^{n-k}\backslash% S_{(x_1,\ldots, x_k)}
%\end{equation}
%is a fibration. Let $\pi_{x_1,\ldots, x_k}$ denote the composed projection map $\pi_{x_k}\circ \ldots \circ \pi_{x_1}$.
\begin{cor} \label{corLand1}  It follows from lemma \ref{lemfib} that
the Landau variety  $L(S,\pi_{[1,\ldots, k]})$ is contained in the zero locus of $S_{(x_1,\ldots, x_k)}$. In particular, its irreducible components  can be expressed as factors of  iterated resultants of the defining polynomials of $S$.
\end{cor}

%It follows from the corollary that  the singular subvarieties of a graph hypersurface can be expressed as an iterated resultant, or family tree,
%of pairs of polynomials.

\subsection{Genealogy}

\begin{defn}We say that an irreducible factor of $[f_i,f_j]_x$ (resp. $D_x(f_i)$), is a \emph{descendent} of $f_i$ and $f_j$ (resp. $f_i$).
Conversely, a \emph{set of parents} of an irreducible  polynomial factor $c$  for the projection $\pi_x$ is a set $\{f_i,f_j\}$ (resp. $\{f_i\}$) such that $c$ is a descendent
of $f_i$ and $f_j$ (resp. $f_i$).
 Note that $c$ may have several sets of parents, since  different polynomials may give rise to the same descendents.
Likewise,  we define a \emph{set of grandparents} of $c$ for the projection $\pi_{[1,2]}$ to be a set of parents of a set of parents of $c$, and, more generally, a \emph{set of $k$-ancestors}
of $c$ for the projection $\pi_{[1,\ldots, k]}$
to be a set of ancestors going $k$ generations back.
\end{defn}

A set of $k$-ancestors has   at most $2^k$  elements.
However, not all  iterated resultants define components which actually occur in the Landau variety, and the upper bound in corollary $\ref{corLand1}$
is a gross over-estimate. In the generic case,  it is enough to consider  iterated resultants whose $k$-ancestors number at most $k+1$.
In this paper, it is enough    to require only the weaker  condition that every term has at most 3  grandparents. In other words, we discount all resultants of the form
$$[[f_1,f_2],[f_3,f_4]] \hbox{ with } f_1,\ldots, f_4 \hbox{ distinct}. $$

%Turns out to be a good compromise. In order to  exploit the many factorizations which occur, a good compromise is to use the iterative 1-dimensional reduction we defined earlier, %but to control artificially the number of ancestors  with some extra combinatorial data.

\begin{rem}\label{defnFubini}
In  \cite{BrCMP} we eliminated  spurious resultants by a different, but equally simple-minded method.
%One simple-minded way to improve this estimate is as follows.
%is to use the following algorithm, defined in  \cite{Br4}.
 Let $S$ be as above, and set  $S_{[x_i]}= S_{x_i}$.
%=\{f_1,\ldots, f_N\}$ denote  a set of irreducible polynomials as above, and let $x_1,\ldots, x_k$ be variables occurring in the $f_i$.
The `Fubini reduction' of $S$ with respect to $x_1,\ldots, x_k$ was defined by  the inductive  formula:
%\begin{equation} \label{Fubinired}
$$
S_{[x_1,\ldots, x_k]} = \bigcap_{i=1}^k \big( S_{[x_1,\ldots, \widehat{x_i}, \ldots, x_k]}\big)_{x_i} \ .
$$
%\end{equation}
In particular, $S_{[x_i,x_j]}= S_{(x_i,x_j)} \cap S_{(x_j,x_i)}$, and 
$S_{[x_1,\ldots, x_k]} \subseteq \bigcap_{\sigma\in \Sym_k} S_{(x_{\sigma(1)},\ldots, x_{\sigma(k)})},$
where the intersection is over all permutations $\sigma$ of $\{1,\ldots, k\}$. % but that this inclusion is strict in  general.
It follows from lemma \ref{lemfib} and  $(\ref{piscommute})$ that 
$L(X,\pi_{[1,\ldots, k]})$ is contained in the zero locus of $S_{[x_1,\ldots, x_k]}$.
%The Fubini reduction gives a finer upper bound for the Landau variety of $\pi$, since it  eliminates certain spurious resultants.
\end{rem}

\subsection{Examples and identities in  the linear case}\label{sectLandauCases}
For our study of  Feynman graphs, it suffices to consider only the  case of a two-dimensional projection of degree (1,1) hypersurfaces, generalising example 
\ref{2redexample}. 
%
%
%use the c-reduction in the special case when the hypersurfaces are linear.
% In this simplified case, it is useful to write down  the various iterated resultants which can occur for  a two-dimensional projection and prove directly that this gives an upper bound %for the Landau variety.
%
Let $P=(\Pro^1)^2\times T$, and consider a projection $\pi: P\rightarrow T$, where the coordinates on $(\Pro^1)^2$ are $x,y$. 
Let $S=X \cup B \subset P$, where  $B=\{x,y=0,\infty\} = B_1\cup B_2\cup B_3 \cup B_4$ is the coordinate square and 
 $X=\bigcup_{i=1}^k X_i$, where $X_i$ is the zero locus of a polynomial
$$f_i = f^i_0 + f^i_x x+ f^i_y y+ f^i_{xy}xy \hbox{ for } 1\leq i \leq k .$$
%
%and
%Firstly, the strata of codimension 3 are of the form:
%$$\pi(A_{i_1} \cap B_{j_1}\cap B_{j_2}) \ ,  \pi(A_{i_1} \cap A_{i_2}\cap B_{j_1})\ , \pi(A_{i_1} \cap A_{i_2}\cap A_{i_3})\ .$$
%where $i_1,i_2,i_3$ (resp. $j_1,j_2,j_3$)  are distinct indices.
%Next, we must remove the loci where $\pi|_{A_i}$ are not-submersive. Landau variety is..

%As before, let  $T=(\Pro^1)^n$, and let  each divisor $A_i$ will be  the zero locus
% of an irreducible polynomial $f_i$. Suppose that
%$$f_i = a_ix_0y_0 + b_i xy_0 + c_i x_0 y + d_i xy\ .$$

\noindent 
We consider each type of critical stratum which can occur, and relate the corresponding Landau variety to a set of iterated resultants.

\begin{enumerate}
  \item  \textit{Strata of the form} $\pi(X_{i} \cap B_{j_1}\cap B_{j_2})$. Suppose that we reduce first with respect to  %the variable
$x$, then $y$. %and then the variable $y$. 
The equations of the  4 %four
  possible Landau varieties  are:
$$[0,[0,f_i]_x]_y\ , \ [0,[\infty,f_i]_x]_y\ ,\  [\infty,[0,f_i]_x]_y\ ,\  [\infty,[\infty,f_i]_x]_y \ ,$$
corresponding to the fibres where $X_i$ passes through one of the  four corners of the  square $B$.
%It is clear that the same polynomials occur when we reduce first with respect to $y$ and then  $x$,  by obvious identities of the form:
Clearly, reducing first with respect to $y$ then $x$ gives rise to the same terms.
% by identities such as
%$[0,[0,f_i]_x]_y =[0,[0,f_i]_y]_x$, and so on.
  \item  \textit{Strata of the form} $\pi(X_{i} \cap X_{j}\cap B_{k})$. If we  reduce first with respect to $x$ and then $y$, the Landau varieties
are given by the four possible cases:
$$[0,[f_i,f_j]_x]_y  \ , [\infty,[f_i,f_j]_x]_y \ , [[f_i,0]_x,[f_j,0]_x]_y\ , [[f_i,\infty]_x,[f_j,\infty]_x]_y\ ,$$
corresponding to the four sides of the square $B$.  We can reverse the order of reduction, 
%, the same polynomials occur,
 by obvious identities such as
$[0,[f_i,f_j]_x]_y  = [[f_i,0]_y,[f_j,0]_y]_x$.

If we write $(X_i\cap X_j)\cap (X_j\cap B_k)$ as $(X_i\cap B_k ) \cap  (X_j \cap B_k)$,  we obtain the  following identities between resultants, where $\omega=0, \infty$:
\begin{equation}
[[\omega,f_i]_x,[f_i,f_j]_x]_y = [[\omega,f_i]_x,[\omega,f_j]_x]_y \times[[0,f_i]_x,[\infty,f_i]_x]_y \ . \end{equation}
%$$
%[[{\infty},f_i]_x,[f_i,f_j]_x]_y = [[{\infty},f_i]_x,[{\infty},f_j]_x]_y \times[[0,f_i]_x,[\infty,f_i]_x]_y\ . $$
 \item \textit{Strata of the form} $\pi(X_{i} \cap X_{j}\cap X_{k})$. The equation of the Landau variety of a triple intersection is
 given by a %rather
   polynomial $\{f_i,f_j,f_k\}$ of degree 6 in
the coefficients of $f_i,f_j,f_k$. Using the fact that $X_i\cap X_j \cap X_k = (X_i \cap X_j)\cap(X_j\cap X_k)$, it corresponds to a factor of the iterated resultants:
$$[[f_i,f_j]_x,[f_j,f_k]_x]_y = \{f_i,f_j,f_k\} \times [[0,f_j]_x,[\infty,f_j]_x]_y\ .$$

\item \textit{Non-submersive strata of the form} $X_i$. A single stratum $X_i=\{f_i=0\}$ is non-submersive when $f_i=\partial f_i/\partial x =\partial f_i/\partial y=0$.
This implies that $f^i_xf^i_y-f^i_0f^i_{xy}=0$, and that $X_i$ degenerates into a product of two lines. This situation corresponds to terms of the form:
$$[[0,f_i]_x,[\infty,f_i]_x]_y= [[0,f_i]_y,[\infty,f_i]_y]_x\ .$$

 \item \textit{Non-submersive  strata of the form} $X_i\cap X_j$.  The final possibility is that  two curves  $X_i, X_j$ in the fiber meet  at
 a double point. This corresponds to:
$$D_y([f_i,f_j]_x)= D_x([f_i,f_j]_y)\ .$$

\end{enumerate}
%These are the only cases which can occur. A closer inspection  shows that 
In conclusion, every possible iterated resultant can occur, except for the three cases:
\begin{equation} \label{badcases}
[[0,f_i]_x,[f_j,\infty]_x]_y \ , \ [[0 \hbox{ or } \infty ,f_i]_x,[f_j,f_k]_x]_y  \  , \ [[f_i,f_j]_x,[f_k,f_\ell]_x]_y\ ,
\end{equation}
where $i,j,k,\ell$ are distinct. These spurious singularities are precisely the resultants  which have four distinct grandparents.
Note that we must include terms of the form $[[0,f_i]_x,[f_i,\infty]_x]_y$ \emph{even though the three hyperplanes $x=0,x=\infty, f_i=0$ do not intersect}. This is because we are considering    the degenerate degree (1,1) case.

\begin{rem} The general (non-linear) case is similar, except that  $(4)$ is replaced by:

\qquad (4') \textit{Non-submersive  strata of the form} $X_i$ given by  a double discriminant:
$$D_x D_y (f_i) = D_y D_x (f_i) $$
 We need not consider resultants of the form $[D_xf, [f,g]_x]_y$ or $[D_x f, D_x g]_y$.
\end{rem}
\subsection{A  reduction algorithm}
We can refine our previous reduction algorithm by  keeping track of the number of ancestors  with some combinatorial data.

%removing iterated resultants which involve too many ancestors. The Fubini reduction will remove such terms
%in the generic case, but there may be degenerate situations in which such bad terms will survive.

%\begin{eqnarray} \label{reductiontypes}
%& i) & [0,f_i] \hbox{ and } [f_i,\infty] \nonumber \\
%& ii) & [0,f_i] \hbox{ and } [0,f_j]  \,\, (\hbox{or } [\infty,f_i] \hbox{ and } [\infty,f_j]),  \hbox{ for compatible }  f_i,f_j  \nonumber \\
%& iii) & [0,f_i] \hbox{ and } [f_i,f_j] \,\, (\hbox{or } [\infty,f_i] \hbox{ and } [f_i,f_j])    \nonumber \\
%& iv) & [f_i,f_j] \hbox{ and } [f_j,f_k]    \hbox{ for compatible }  f_i,f_k  \nonumber 
%\end{eqnarray}

%\begin{rem} 
%\label{remcommonancestor} Cases $i)-iv)$ can be reduced to the single case $iv)$ by formally adding  symbols $0, \infty$ to $S$ and declaring $0$ and $\infty $ to be compatible %with all
%elements of $S$.  Note that  any two compatible elements in $S_m$ must have a common ancestor.
%\end{rem} 

\begin{defn} \label{defcred} Let $S=\{f_1,\ldots, f_k\}$ denote a set of irreducible polynomials in $\C[x_1,\ldots, x_N]$, and let $C$ be a simple graph with $k$  vertices indexed by the  elements of $S$. Two polynomials $f_i,f_j$ are \emph{compatible} if there is  an edge connecting vertices $i$ and $j$ in $C$. Let $1\leq m \leq k$. 
Define a new set of polynomials:
$$\widetilde{S}_{x_m}= \{ D_{x_m}(f_i), [0,f_i]_{x_m},[\infty,f_i]_{x_m}, [f_i,f_j]_{x_m} \hbox{ for all compatible pairs } f_i,f_j.\}$$
Let $S_{x_m}$ denote the set of irreducible factors of $\widetilde{S}_{x_m}$. Define  a new set of compatibilities $C_{x_m}$  between all irreducible factors of resultants of the form 
$[s_1,s_2]$ and $[s_2,s_3]$, where $s_1,s_2,s_3 \in \{0,\infty, f_1,\ldots, f_k\}$, and also between the irreducible factors of a single discriminant $D_{x_m}(f_i)$.
 % if $f_i$ and $f_k$ are compatible.
%Naturally, two irreducible factors of the same resultant are also compatible.
We set $(S,C)_{x_m}:=(S_{x_m},C_{x_m})$. %with respect to $\alpha_m$ to be:
\end{defn}

Unfortunately, iterating this reduction  is not 
independent of the chosen order of variables
in the case when the resultants are degenerate.

\begin{defn} \label{defFubcred}
Let $S=\{f_1,\ldots, f_k\}$ be a set of irreducible polynomials in variables $x_1,\ldots, x_N$, and let $C$  be the complete graph on the 
polynomials $f_i$. Let $i_1,\ldots, i_k$ be an ordered subset of indices $1,\ldots ,N$. 
Define $(S,C)_{[i_1,\ldots, i_k]}$ inductively as follows:
\begin{equation} 
(S,C)_{[i_1,\ldots, i_k]}= \bigcap_{1\leq j\leq k} (S_{[i_1,\ldots, \widehat{i_j},\ldots, i_k]}, C_{[i_1,\ldots, \widehat{i_j},\ldots, i_k]})_{i_j} \ ,
\end{equation}
 where  $(S,C)_{[i_j]}=(S,C)_{i_j}$ and the intersection of pairs  $(S_1,C_1) \cap (S_2,C_2)$  is defined by $(S',C')$, where $S'=S_1\cap S_2$, and $C'$ is the set of compatibilities between elements of $S'$ that are common to both $C_1$ and $C_2$.
\end{defn}

 \noindent  In practice, the commutativity properties for linear  resultants  mean that we only ever
 compute $(S_1,C_1)\cap (S_2,C_2)$, where the underlying sets  $S_1=S_2$ are the same.

\subsection{Linearly reducible graph hypersurfaces} We apply the  above to the case of graph hypersurfaces relative to the coordinate hypercube. Therefore, let  $G$ be a (Feynman)  graph, and let $S_0(G)=\{\Psi_G(\alpha_e)\}$ and $C_0(G)=\emptyset$.  For simplicity of notation, we omit the coordinate hyperplanes $\alpha_i=0,\infty$ from the set  $S_0(G)$.

\begin{defn} Given an ordering $e_1,\ldots, e_N $ on the set of edges of $G$,
we set
\begin{equation}
(S_{[e_1,\ldots, e_k]}(G), C_{[e_1,\ldots, e_k]}(G)) = (S_0(G), C_0(G))_{[\alpha_{e_1},\ldots, \alpha_{e_k}]} \ .
\end{equation}
for $1\leq k\leq N$,  according to definition \ref{defFubcred}.
%We call these the \emph{reduction} of $G$ with respect to the ordering $e_1,\ldots, e_N$. 
\end{defn}

Thus to any graph, and any ordering on its edges,  we associate a cascade of polynomials together with compatibility relations between them.
\begin{defn}  A graph $G$ is \emph{linearly reducible} if there is  an ordering on its edges $e_1,\ldots, e_N$ such that every term in $S_{[e_1,\ldots, e_k]}(G)$ is of degree at most one in  $\alpha_{e_{k+1}}$. %, the Schwinger parameter corresponding to the edge $e_{k+1}$. % in the reduction.
\end{defn}
If $G$ is linearly reducible, then the reduction of $G$ never requires computing any discriminants $D_x(f)$.
The main result of this section is the following theorem.

\begin{thm} Let $G$ be linearly reducible for some ordering $e_1,\ldots, e_N$ of its edges. For each $1\leq k\leq N$, the Landau variety $L(G,\{e_1,\ldots, e_k\})$ is contained in the union of the  zero locus of the polynomials
$S_{[e_1,\ldots, e_k]}(G)$ and the hyperplanes $\alpha_i=0,\infty$.
\end{thm}

\begin{proof} %Strictly speaking, the Landau varieties of a graph $G$ were  defined in terms of the blow-up $\BP_G$. But since the exceptional divisors lie over coordinate %hyperplanes $\alpha_i=0,\infty$, this  makes no difference. %The proof then follows by applying corollary \ref{corLand1} by induction. 
 The computations in \S\ref{sectLandauCases}  show that only resultants with at most  three  grandparents need be considered. Since this holds for all possible orderings of the edges $e_1,\ldots, e_i$  with respect to which $G$ is linearly reducible,  for all $1\leq i\leq k$, it follows by lemma \ref{lemfib} that the Landau variety is contained in the output of the linear reduction of $G$.
\end{proof}

% Explicitly, the reduction $S_r$ of a set $S=\{f_1,\ldots, f_N\}$ with compatibilities $C$ with respect to $\alpha_r$
 %is given by the irreducible factors of terms:
%$$ \{[0,f_i]_{\alpha_r}, [\infty,f_i]_{\alpha_r}, [f_i,f_j]_{\alpha_r} \hbox{ for all compatible } f_i,f_j\}$$ 
%and the set of compatibilities $C_r$ are between factors of resultants of the form:
%\begin{eqnarray} \label{reductiontypes}
%& i) & [0,f_i]_{\alpha_r} \hbox{ and } [f_i,\infty]_{\alpha_r}  \\
%& ii) & [0,f_i]_{\alpha_r} \hbox{ and } [0,f_j]_{\alpha_r}  \,\, (\hbox{or } [\infty,f_i]_{\alpha_r} \hbox{ and } [\infty,f_j]_{\alpha_r}),  \hbox{ for compatible }  f_i,f_j  \nonumber \\
%& iii) & [0,f_i]_{\alpha_r} \hbox{ and } [f_i,f_j]_{\alpha_r} \,\, (\hbox{or } [\infty,f_i]_{\alpha_r} \hbox{ and } [f_i,f_j]_{\alpha_r})    \nonumber \\
%& iv) & [f_i,f_j]_{\alpha_r} \hbox{ and } [f_j,f_k]_{\alpha_r}    \hbox{ for compatible }  f_i,f_k  \nonumber 
%\end{eqnarray}

\subsection{Properties of linear reducibility}
It is clear that if $\gamma\minor G$ is a graph minor, and $e_1,\ldots, e_k$ are edges of $\gamma$, then: 
\begin{equation} \label{minormonotonereduction}
(S_{[e_1,\ldots, e_k]}(\gamma), C_{[e_1,\ldots, e_k]}(\gamma)) \subseteq (S_{[e_1,\ldots, e_k]}(G), C_{[e_1,\ldots, e_k]}(G)) \ .
\end{equation}
This follows immediately from the contraction-deletion relations.  The reduction of a graph is in this sense  minor-monotone.
In particular, if $G$ is linearly reducible, then $\gamma\minor  G$ is linearly reducible for all minors $\gamma$ of $G$.
%
%
%It is also clear that if $\gamma=G'$ is the  maximal core subgraph of $G$, then we have equality in $(\ref{minormonotonereduction})$.
%
%
%
\begin{lem} \label{lemsimp} $G$ is linearly reducible if and only if its simplification is.
%$\gamma \leq G$ is a simplification of $G$ and $e_1,\ldots, e_k$ are 
%edges of  $\gamma$, then 
%we have equality in $(\ref{minormonotonereduction})$.
\end{lem}
%\begin{proof} Consider the case when $G$ has a two-valent vertex. The lemma follows from the following identity, valid for any polynomials $f_i, f_j$:
%\begin{equation}
%\big[f_i(\alpha_1+\alpha_2), f_j(\alpha_1+\alpha_2)\big]_{\alpha_1} =\big[f_i(\alpha_1),f_j(\alpha_1)\big]_{\alpha_1} 
%\end{equation}
%The case when $G$ has multiple edges  is similar. % or follows from this by duality.
%\end{proof}
%
This follows  from  lemma \ref{lemsimplif}. %$ that $G$ is linearly reducible if and only if its simplification $G'$ is linearly  reducible.
 In particular, all series-parallel 
graphs obtained by applying  operations $S$ and $P$ to the trivial graph are linearly reducible.
%
%SIMPLIFICATION:  I want: an analytic version, stratification version, and this: a genealogical version.
%Two different graphs have the same compatibilities and elements at same time ... so same periods. kindof.
%\\
%\begin{ex} Suppose that $G$ contains a double edge. Then $\Psi_{12}=0$ and $\Psi^1_2 = \Psi^2_1$ since
%the operation of  deleting  one edge  and contracting the other leads to the same minor. ...
%NB Any order? Need to check that $e,f,g,...,1,2 = e,f,g,...2,1$ also?
%\end{ex}
%Reduction preserved under simplification.
%
The effect of the star-triangle operations on the linear reducibility of graphs is  subtle.

\begin{rem}
Suppose that $G_{\triangle},G_Y$ are two Feynman graphs related by a star-triangle relation with edges $e_1,e_2,e_3$. Then one can show that
$$(S(G_{\triangle}),C(G_{\triangle}))_{[e_1,e_2, e_3]} = (S(G_Y), C(G_Y))_{[e_1,e_2, e_3]}\ .$$ It is not true in general that 
$(S(G_{\triangle}),C(G_{\triangle}))_{[e_1,\ldots, e_m]} = (S(G_Y), C(G_Y))_{[e_1,\ldots, e_m]}$ for  all $m\geq 5$, because the definition of  reduction involves intersecting over all possible orderings, not just those 
commencing with the edges $e_1,e_2,e_3$.
Indeed,
 there are counter-examples  ({\it e.g.}, $K_{3,4}$ considered in \S\ref{sectK34} below) to show that linear reducibility is not preserved by the star-triangle relations in general.
\end{rem}

\subsection{Degeneration}
%
%
%The following lemma states that the Landau varieties of $S$ are stable after intersecting with a hyperplane given by a  fiber of $\pi_{n-1}$.
%
Let $P=(\Pro^1)^N\times T$, and let  $\pi:P  \times T\rightarrow T$ be the projection.  Let $x_1,\ldots, x_N$ denote the coordinates in the fiber $(\Pro^1)^N$. Let $S\subset P$ be a union of irreducible hypersurfaces $S_i$ as above, which contains the  hypercube $B=\cup_i \{x_i=0,\infty\}$. Suppose that  $S$ is linearly reducible, and let $V\subset T$ denote the zero locus of the outcome  of the linear reduction  algorithm applied to $S$. We know that  $L(S,\pi)\subset V$.

\begin{lem}\label{lemdegen} 
Let $R\rightarrow T$ be a closed subvariety of $T$, and write $P_R= P\times_T R$,  $S_R = S\big|_R$ and $\pi_R: P_R \rightarrow R$.  If $S$ is linearly reducible with respect to $\pi$, then so is 
$S_R$ with respect to $\pi_R$, and $L(S_R,\pi_R) \subset V \times_T R$.
%Let $x \in \Pro^1$, and let 
%$S_x = S \cap \pi_{n-1}^{-1}(x)$.
%Then $S_x \subset (\Pro^1)^{n-1} \cong \pi_{n-1}^{-1}(x)$ is linearly reducible, and its Landau varieties are  
%$L(S_x,\pi_i) = L(S,\pi_i) \cap \pi_{n-1}^{-1}(x)$.
\end{lem}
\begin{proof} Clearly, linearity is preserved since the restriction to $R$ of a polynomial which is of degree at most one in a reduction variable is still of degree at most one.
The linear reduction on $R$ retains triple resultants of polynomials of the form $[[g_1,g_2],[g_2,g_3]]$. By induction these can always be written
as the restriction to $R$ of  a triple $[[f_1,f_2],[f_2,f_3]]$, where $f_i$ occur in the linear reduction of $S$. 
%
%
%
% By linearity, any four distinct non-trivial polynomials $a_1,\ldots, a_4$ which are of degree one in a reduction variable $x_i$ must be the restriction to $R$
%of four distinct non-trivial polynomials $f_1,\ldots, f_4$.
\end{proof}

In  \S\ref{sectMatrixRed}, we will apply the linear reduction to graph hypersurfaces. The previous lemma will be used implicitly, since for certain graphs it may happen that some polynomials defining components in the Landau varieties will  vanish identically.

\newpage

\section{Graphs of matrix-type} \label{sectMatrixRed}
We study the linear reduction of a graph $G$, and show that  the first obstruction to $G$ being matrix type is having a non-trivial 5-invariant. We then  prove that graphs of vertex width at most 3 are of matrix type.
\subsection{Resultants and Dodgson polynomials} In many cases, we  can translate the  reductions of \S\ref{sectLandauCases} into identities between Dodgson polynomials.

\begin{notat} \label{graphnot}
Let $(G,O)$ be an ordered graph, and suppose that we have reduced $\Psi_G$  with respect to  $K_i=\{e_1,\ldots, e_i\}$. Then many of the  typical terms in the reduction will be of the form $\Psi^{I,J}_{K_i}$ where $|I|=|J|$, and $I, J \subset K_i$. Since the $K_i$ is implicit in the choice of reduced variables, we will frequently omit the subscripts, and denote $\Psi^{I,J}_{K_i}$ simply by the pair $(I,J)$ ( $=(J,I)$).
\end{notat}

Firstly, there are trivial identities from the contraction-deletion formulae:
\begin{equation}\label{resid0}
[\Psi_K^{I,J},0]_x = \Psi_{K x}^{I,J} \quad \hbox{and}\quad [\Psi_K^{I,J},\infty]_x = \Psi_{K}^{Ix,Jx} \ .
\end{equation}
These can be depicted graphically using notation $\ref{graphnot}$ as:
\vspace{-0.05in}
\begin{figure}[h!]
  \begin{center}
    \leavevmode
  \label{figure1}
\put(-120,0){$x\rfloor$}
\put(-64,0){$(I,J)$}
\put(-50,10){\line(1,1){10}}
\put(-50,10){\line(-1,1){10}}
\put(-82,24){$(I,J)$}
\put(-39,24){$0$}
\put(32,0){$(Ix,Jx)$}
\put(50,10){\line(1,1){10}}
\put(50,10){\line(-1,1){10}}
\put(18,24){$(I,J)$}
\put(60,24){$\infty$}
 \end{center}
\end{figure}
\vspace{-0.05in}

\noindent
The reduction variable (here, $x$) is indicated on the left.
The reason why  linear reduction works is  because of the following identities.
\begin{lem} Let $I,J$ be two subsets of $[n]$ such that $|I|=|J|$ and let $a,b,x\notin I\cup J$. Then identity $(\ref{FirstDodgsonId})$ implies that:
\begin{equation} \label{resid1}
[\Psi_K^{I,J}, \Psi_K^{Ia,Jb}]_x   =  \Psi_K^{Ix,Jb}\,\Psi_K^{Ia,Jx} \ .
\end{equation}
Let $I,J$ be two subsets of $[n]$ such that $|J|=|I|+1$, and let $a,b,x\notin I\cup J$. Then identity $(\ref{SecondDodgsonId})$ implies that:
\begin{equation} \label{resid2}
[\Psi_K^{Ia,J}, \Psi_K^{Ib,J}]_x   =  \Psi_K^{Ix,J}\,\Psi_K^{Iab,Jx}  \ .
 \end{equation}
\end{lem}
These two identities can be depicted graphically as follows:
\vspace{-0.05in}
\begin{figure}[h!]
  \begin{center}
    \leavevmode
  \label{figure1}
\put(-120,0){$x\rfloor$}
\put(-86,0){$(Ix,Jb)(Ia,Jx)$}
\put(-50,10){\line(1,1){10}}
\put(-50,10){\line(-1,1){10}}
\put(-80,24){$(I,J)$}
\put(-48,24){$(Ia,Jb)$}
\put(18,0){$(Ix,J)(Iab,Jx)$}
\put(50,10){\line(1,1){10}}
\put(50,10){\line(-1,1){10}}
\put(18,24){$(Ia,J)$}
\put(55,24){$(Ib,J)$}
 \end{center}
\end{figure}
\vspace{-0.05in}
\begin{defn} Given two pairs of subsets $(P,Q)$ and $(R,S)$, define the distance between them by the formula:
$$||(P,Q),(R,S)|| = \min\{|P\circ R|+|Q\circ S|, |P\circ S|+|Q\circ R|\}\ ,  $$
where for two sets $A$ and $B$,  $A\circ B= (A\cup B) \backslash (A\cap B)$ denotes their symmetric difference.
This 
takes values in the set of non-negative even integers.
\end{defn}
\begin{rem} \label{remres} There are two observations which will be important in the sequel:
\begin{enumerate}
\item The distance between a parent and its child in $(\ref{resid0})$-$(\ref{resid2})$ is at most 2.
\item After reducing with respect to  $x$, all offspring $(A,B)$ in identities $(\ref{resid1})$ and $(\ref{resid2})$  have the property that  exactly one of the two sets $A,B$ contains $x$.
\end{enumerate}
\end{rem}
Using these identities, we can perform the first few linear reductions for any  graph $G$. Writing $\Psi=\Psi_G$, and reducing with respect to edges numbered $1,2,3,4$ in order,
we have % $S=\{\Psi\}$,  $S_1 = \{\Psi^1, \Psi_1\}$
$S=\{\emptyset\}$, $S_{[1]}=\{\emptyset, (1,1)\}$,  and 
$S_{[1,2]}= \{\emptyset, (1,1),(2,2),(12,12),(1,2)\}$, 
which corresponds to the fact that $L_2= \{\Psi_{12}\Psi^{1,1}_2\Psi^{2,2}_1\Psi^{12,12}\Psi^{1,2}=0\}$,
  as computed in  example $\ref{2redexample}$,  which  
has the following genealogy:  %be represented by the following ancestral diagram:
\\
\begin{figure}[h!]
\vspace{-0.1in}
  \begin{center}
%    \leavevmode
    \epsfxsize=6.0cm \epsfbox{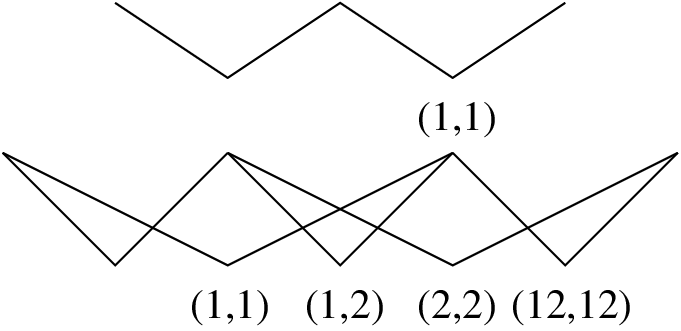}
  \label{2gen}
 \put(-250,4){ $2\rfloor$} 
 %\put(-220,4){ $0$} 
 %\put(-10,4){ $\infty$} 
 \put(-122,51){ $\emptyset$}
 \put(-151,4){ $\emptyset$} 
 \put(-180,51){ $0$} 
 \put(-250,51){ $1\rfloor$} 
 \put(-10,51){ $\infty$} 
 \put(-153,87){ $0$}
 \put(-95,87){ $\emptyset$}
 \put(-35,87){ $\infty$}
  \end{center}
\end{figure}
\vspace{-0.1in}

\noindent
The compatibilities between the elements of the sets $S_{[1]}$ and $S_{[1,2]}$ can be represented by the following graphs $C_{[1]}$ and  $C_{[1,2]}$:
 \begin{figure}[h!]
  \begin{center}
%    \leavevmode
 \vspace{-0.1in}
    \epsfxsize=9.0cm \epsfbox{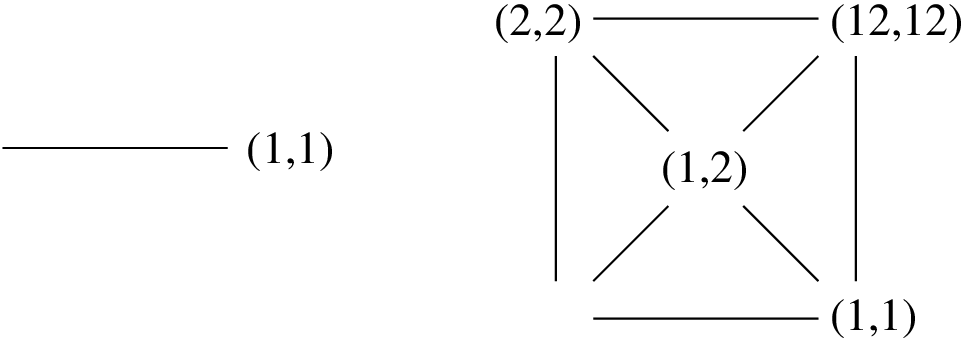}
  \label{compats}
 \put(-270,48){\Large $\emptyset$}
 \put(-119,4){ \Large$\emptyset$} % \caption{Compatibilities after one reduction (left) and two reductions (right).}
  \end{center}
\end{figure}

\noindent
At the third stage, the identities   $(\ref{resid0})$-$(\ref{resid2})$  imply that  $S_{[1,2,3]}$ is the union of 
$$\{\emptyset, (1,1),(2,2),(3,3),(12,12),(13,13),(23,23),(123,123)\}$$
with the six Dodgson terms $\{(1,2),(1,3),(2,3),(12,13),(12,23),(13,23)\}$.
%\begin{figure}[h!]
 % \begin{center}
%    \leavevmode
 %   \epsfxsize=8.0cm \epsfbox{3cube.eps}
%  \label{3cube}
% \put(-225,39){\Large $\emptyset$}
% \put(-119,4){ \Large$\emptyset$} % 
%\caption{The compatibility graph $C_{[1,2,3]}$ is a representation  of the Landau variety of the projection of a general graph hypersurface relative to a hypercube (example $%%\ref{3redexampleDODGSON}$).
%  The polynomials $(1,3)$,$(2,3)$ and $(13,23)$ on the back faces of the cube have not been drawn.
%The six Dodgson polynomials $\{(1,2),(1,3),(2,3),(12,13),(12,23),(13,23)\}$ are all mutually compatible, and form a complete graph $K_6$ (not drawn). }
%  \end{center}
%\end{figure}
%
%
%
We will show (proposition $\ref{propcompat5desc}$) that    two pairs $(A,B)$, $(C,D)$ are compatible if and only if the distance between them is 2.
The compatibility graph can be represented by a cube, with one of the eight terms above at each corner, and a Dodgson term inscribed in the middle of each face (each face is isomorphic to   a $C_{[1,2]}$). The six Dodgson terms
are mutually compatible and form a complete graph $K_6$.

A new phenomenon occurs  at the fourth stage.  The set $S_{[1,2,3,4]}$ consists of sixteen graph polynomials  of minors  
$(A,A)$ such that $A\subset \{1,2,3,4\}$ (corresponding to the vertices of a 4-cube), 
a further  24 Dodgson terms  $(A,B)$ where $A,B\subset \{1,2,3,4\}$ such that  $|A|=|B|$ and $|A\circ B|=2$, 
and three further terms:
$$(12,34) \ ,(13,24) \ ,(14,23) \ ,$$
which are mutually compatible. There is no determinantal identity to compute the resultant of any two of these terms at the fifth stage of reduction, 
and in general one obtains a term (the `five-invariant' defined below) which does not factorize.% For a particular graph however, the existence or non-existence of cycles will entail 
%further identities which will enable us to circumvent this problem.

\begin{defn} We say that a graph $G$ is of \emph{matrix type} if there exists an ordering $O$ of the set of edges of $G$ such that the  only  elements which occur  in the reduction of
$G$ with respect to $O$ are Dodgson polynomials $\Psi^{I,J}_K$.
\end{defn}

If $G$ is of matrix type, then it is  linearly reducible.

%First obstruction to being of matrix type is the five-invariant.
\subsection{The five-invariant}\label{sectFiveinvariant} This was  first observed implicitly in \cite{B-E-K}, equation (8.13). 
\begin{defn}Let $G$ be a graph, and let $e_1,\ldots, e_5$ be five distinct edges of $G$. We define the \emph{five-invariant}
${}^5\Psi_G(e_1,\ldots, e_5)$ to be
$\pm [\Psi_G^{e_1e_2,e_3e_4}, \Psi_G^{e_1e_3,e_2e_4}]_{e_5}$.
\end{defn}
\noindent The five-invariant is the first quadratic term which may occur in a reduction of $\Psi_G$.
\begin{lem} \label{5invwelldef} The five-invariant is well-defined up to a sign, i.e., it does not depend on the choice of the ordering of the five edges in the definition above. 
\end{lem}
\begin{proof}
The proof follows from the fact that if $f,g$ are  two (generic) polynomials which are of degree at most $1$ in two variables, $x$ and $y$, then
$$D_y([f,g]_x) = D_x([f,g]_y)\ .$$
Since this is a polynomial identity in the coefficients of $f,g$, it remains valid in all cases.
Now take $f=\Psi^{1,2}_{G,3}$ and $g=\Psi_G^{13,23}$, and $x=\alpha_4, y=\alpha_5$. By   $(\ref{resid1})$,
$$D_y([f,g]_x) = D_y (\Psi_G^{14,23} \Psi_G^{13,24}) = \big(\big[\Psi_G^{14,23},\Psi_G^{13,24}\big]_5\big)^2 \ ,$$
$$D_x([f,g]_y) = D_y (\Psi_G^{15,23} \Psi_G^{13,25}) = \big(\big[\Psi_G^{15,23},\Psi_G^{13,25}\big]_4\big)^2\ ,$$
so interchanging any pairs of  indices in  $[\Psi_G^{14,23},\Psi_G^{13,24}]_5$  multiples it by $\pm 1$. Since the symmetric group on 5 letters is generated by 
transpositions, we are done.
\end{proof}

%First we  consider the case when 5-invariants are  locally trivial in a strong sense.
%It is easy to write down sufficient conditions for a five-invariant to factorize.
For any four edges $i,j,k,l$ of $G$, we have the identity 
$ \Psi^{ij,kl}-\Psi^{ik,jl}+\Psi^{il,jk}=0,$ (having chosen a representative for  $M_G$ to fix the signs).  This implies that:
\begin{eqnarray} \label{parallel5eq}
\Psi_m^{ij,kl} - \Psi_m^{ik,jl} + \Psi_m^{il, jk} & = & 0 \ , \\
\Psi^{ijm,klm} - \Psi^{ikm,jlm} + \Psi^{ilm, jkm} & = & 0 \ . \nonumber
\end{eqnarray}
\begin{defn}\label{defn5split} We say that the five-invariant ${}^5\Psi(ijklm)$ \emph{splits} if, for some ordering of the indices,   one of the 30 terms in  $(\ref{parallel5eq})$ vanishes, {\it i.e}.,
\begin{equation}\label{eqn5split}
\Psi_m^{ij,kl} =0 \quad \hbox{ or } \quad  \Psi^{ijm,klm}=0 \ . \end{equation}
%$$\Psi_m^{ij,kl}  \Psi_m^{ik,jl}  \Psi_m^{il, jk}  \Psi^{ijm,klm} \Psi^{ikm,jlm}  \Psi^{ilm, jkm}  =  0 \ .$$
%either $\Psi^{ij,kl}=0$ (i.e., $\Psi_m^{ij,kl}=0$ and $\Psi^{ijm,klm}=0$)  or $\Psi^{ij,kl}_m=0$ and $\Psi^{ikm,jlm}=0$.
\end{defn}
In this case, the 5-invariant ${}^5\Psi(i,j,k,l,m)$ either vanishes altogether, or after permuting the indices if necessary, can be written in  the form
$${}^5\Psi(i,j,k,l,m) = \Psi^{ij,kl}_m \Psi^{ikm,jlm}\ .$$

By corollary $ \ref{corvanishingcond}$, equation $(\ref{eqn5split})$ is a minor monotone property which is equivalent to the non-existence of certain cycles in 
$G\backslash m$ or $G\q m$.

\subsection{Linear reduction} The following result is the key to the main theorem.
%It turns out that the first obstruction to the linear reducibility of a graph $G$ is the non-triviality of 5-invariants of minors of $G$. % at each stage.

\begin{prop}\label{propcompat5desc}
In the generic case, the polynomials occurring  in the linear reduction of $\Psi_G$ are either Dodgson polynomials $\Psi^{I,J}_K$ or 
 descendents of five-invariants of minors of  $G$.
Two   Dodgson polynomials $\Psi_K^{A,B}$ and $\Psi_K^{C,D}$ which are not descendents of any five-invariants  can only be compatible if either:
\begin{eqnarray}
& i). & \qquad  ||(A,B),(C,D)||=2\ ,  \nonumber \\
\hbox{or } & ii). & \qquad  \{(A,B),(C,D)\} = \{(Mij,Mkl), (Mik,Mjl)\} \ ,\nonumber
\end{eqnarray}
where, {\it e.g.}, $Mij$ denotes $M\cup \{i,j\}$.
Thus  any further compatibilities between Dodgson polynomials  only occur  if they are descendents of five-invariants. 
\end{prop}

\begin{proof}
By induction, suppose  that the statement is true after $k$ reductions, and let $x$ be the reduction index.
Suppose that $\Psi_K^{A,B}$ and $\Psi_K^{C,D}$ are compatible. If we are in case $i)$, then  we must have 
either $\{(A,B),(C,D)\} = \{ (I,J), (Ia,Ja)\}$, where $|I|=|J|$ and $a\notin I \cup J$, or 
$\{(A,B),(C,D)\} = \{ (Ia,J), (Ib,J)\}$, where $|J|=|I|+1$ and $a,b\notin I \cup J$. By applying either  $(\ref{resid1})$ or  $(\ref{resid2})$, we see that the
resultant   $[\Psi_K^{A,B},\Psi_K^{C,D}]_x$ factorizes as a product of Dodgson polynomials. If we are in case $ii)$, then the resultant  $[\Psi_K^{A,B},\Psi_K^{C,D}]_x$
is by definition a five-invariant $\Psi_{G'} (i,j,k,l,x)$ where $G'$ is  a minor of $G$. We must next check the compatibilities for any two  Dodgson polynomials $\Psi_{Kx}^{P,Q}, \Psi_{Kx}^{R,S}$ in the  new generation.

 Firstly, by definition of linear reduction, two such  polynomials can only be compatible if they share a common ancestor $\Psi_K^{U,V}$.  
Furthermore, by remark \ref{remres} (1), the distance between parent and child is at most 2, so we have:
 $$||(P,Q),(R,S)||\leq ||(P,Q),(U,V)||+||(R,S),(U,V)||\leq 2+2 = 4\ .$$ 
Since we are in the linear case, we need only consider the resultants of   $\S\ref{sectLandauCases}$. For reductions of type $[[0,f],[f,\infty]]$, $[[0,f],[0,g]]$, or $[[0,f],[f,g]]$ (and those obtained by interchanging $0$ with $\infty$) the distance between  offspring is exactly 2: {\it i.e.}, 
 $||(P,Q),(R,S)||=2$ and we are in case $i).$ 
Otherwise, the interesting case is when  we  have  a reduction of type $[[f,g],[f,h]]$ and  $||(P,Q),(R,S)||=4$, and we must show that we obtain case $ii)$.  We  first claim that  $P\cup Q = R\cup S$. If not, then without loss of generality there exists $y\in P$ such that $y\notin R \cup S$. 
By changing the order of reduction, we can also assume that $y=x$, the reduction variable. But by remark \ref{remres} (2), the index $y$ must occur in either $R$ or $S$, %since of type $[[f,g],[g,h]]$,
 which gives a contradiction. Therefore $P\cup Q =R\cup S$, and we can assume that $\{(P,Q),(R,S)\}= \{(Aij,Bkl),(Aik,Bjl)\}$. By passing to the subgraph $G\backslash (A\cap B)$, we can further assume that $A\cap B=\emptyset$.
If $A=B=\emptyset$, then we are in case $ii)$. Therefore suppose in the opposite case that  $|A|=|B|\geq 1$, 
and let $a\in A$ and $b\in B$.
 We must show that  $(Aij,Bkl)$ and $(Aik,Bjl)$ cannot be compatible.  Otherwise  $(Aij,Bkl)$ and $(Aik, Bjl)$ would have to  possess a common ancestor $(U,V)$, which is of distance at most 2 from each. It follows from this that, after interchanging $U$ and $V$ if necessary, we must have $Ai\subseteq U$
and $Bl\subseteq V$. In fact, there are only three possibilities: $(U,V)$ is one of $(Ai,Bl)$, $(Aij,Bjl)$ or $(Aik,Bkl)$. But, if the reduction variable $x$ is $a$ or $b$, then this clearly cannot be the case, since we must have $x\notin U, V$ at the previous generation. It follows that  $(Aij,Bkl)$ and $(Aik,Bjl)$ are not compatible.
\end{proof} 

The proposition states that the first non-trivial obstruction to being of matrix type is  the five-invariant, and furthermore, that there  can  occur no  higher obstructions
of the form $[\Psi^{123,456},\Psi^{134,256}]_7$ and so on (unless these  Dodgson polynomials  themselves happen to occur as descendents of five-invariants, {\it e.g.}, lemma \ref{lem5invcompat}).

\begin{defn} \label{defnbasiccompat} We call the compatibilities defined in $i), ii)$ of  proposition \ref{propcompat5desc} the \emph{basic compatibilities} between Dodgson polynomials.
We call \emph{extra compatibilities} the new ones induced by the possible  descendents of  a split 5-invariant. % and above the set of basic compatibilities.
\end{defn}

%\begin{rem} The argument at the end of the  proof requires the assumption $|A|=|B|\geq 1$ and does not contradict the existence  of the five-invariant.
%Indeed, 
%$$(13,24)(14,23)= [[0,(1,2)]_3,[\infty,(1,2)]_3]_4= [[0,(1,2)]_4,[\infty,(1,2)]_4]_3\ ,$$
%$$(13,24)(14,23)= [[0,(3,4)]_1,[\infty,(3,4)]_1]_2= [[0,(3,4)]_2,[\infty,(3,4)]_2]_1\ ,$$
%so $(13,24)$ and $(14,23)$ are compatible with respect to all orderings, and  the five-invariant does arise. Were this not the case, then by the previous proposition
%all Feynman graphs would be linearly reducible and would be of MZV-type.
%\end{rem}

\begin{cor} \label{corminorinduction} Let $G$ be a graph, and let $e_1,\ldots, e_n$ denote an ordering on the set of  edges in $G$. Suppose that at the 5th stage of reduction
there are only basic compatibilities (i.e., the 5-invariant ${}^5\Psi_G(e_1,\ldots, e_5)$ splits and induces no extra compatibilities), and suppose that  for all $i=1,\ldots, 5$ the minors
$$G \backslash e_i \hbox{ and } G\q e_i $$
with the induced ordering of edges are linearly reducible. Then $G$ is linearly reducible with respect to the ordering $e_1,\ldots, e_n$. 
\end{cor}
%Note that the induced edge-ordering on the minors $G \backslash e_i$ and $G\q e_i $ must be the same. We will obtain a stronger result in section ... by removing this restriction.

%If the five-invariants do not split, we may obtain higher invariants (see ..).

The previous corollary will enable us to  do inductions.
The entire difficulty of the problem is that one has to go all the way down to the fifth level to see the first  non-trivial combinatorial  phenomena.

\subsection{Triangles and 3-valent vertices} %We saw earlier that the reduction of a graph is the same as its simplification.  
We study the effect of the presence of a triangle
or 3-valent vertex on the five-invariants of $G$.
\begin{lem}  \label{lemtrisplit} Let $a,b,c,i,j$ be any 5 distinct edges of $G$. 
Suppose that $\{a,b,c\}$ forms a triangle. Then $\Psi_{abc}=0$ and 
\begin{equation} \label{triangleid}  \Psi^{ab,ij}_c=\Psi^{ac,ij}_b=\Psi^{bc,ij}_a= 0 \end{equation}
$$ \Psi^{ai,bj}_c= \pm \Psi^{aj,bi}_c = \ldots = \pm \Psi^{bi,cj}_a  \qquad  (= \Psi_{G\backslash a\q bc}^{i,j})$$
$$\Psi^{abc,aij}=\pm \Psi^{abc,bij}=\pm \Psi^{abc,cij}$$
and the five invariant ${}^5\Psi(abcij)$ is the product of any  element in the second row with one in the third.
If $\{a,b,c\}$ forms a 3-valent vertex, then $\Psi^{abc}=0$ and 
$$\Psi^{abc,aij}=\Psi^{abc,bij}=\Psi^{abc,cij}=0  $$
$$ \Psi^{aci,bcj}= \pm \Psi^{acj,bci} = \ldots = \pm \Psi^{abi,acj}  \qquad  (= \Psi_{G\backslash ab\q c}^{i,j})$$
$$\Psi^{ab,ij}_c=\pm \Psi^{ac,ij}_b=\pm \Psi^{bc,ij}_a$$
and the five invariant ${}^5\Psi(abcij)$ is the product of any  element in the second row with one in the third.

%\begin{eqnarray} \label{triangleid}
%\Psi_{abc} & = & 0 \quad \Longrightarrow \quad  \Psi^{ab,ij}_c=\Psi^{ac,ij}_b=\Psi^{bc,ij}_a=0 \ , \\
%\Psi^{abc}  & =& 0 \quad \Longrightarrow \quad  \Psi^{abc,aij}=\Psi^{abc,bij}=\Psi^{abc,cij}=0 \ . \nonumber  
%\end{eqnarray}
%The former case occurs if  $\{a,b,c\}$ forms a triangle, and the latter when $a,b,c$  meet at a 3-valent vertex. In either case, the five invariant
%${}^5\Psi(abcij)$ splits.

If $a,b,c,i,j$ contains a 2-valent vertex or a 2-loop, then ${}^5\Psi(abcij)$ vanishes.
\end{lem}
\begin{proof} 
Suppose $a,b,c$ forms a triangle, and hence  $\Psi_{abc}$ vanishes. By 
proposition $\ref{PropGenPsitreeformula}$,  $\Psi^{ab,ij}_c$ is a sum over all common monomials in $\Psi^{ab,ab}_{cij}$  and $\Psi^{ij,ij}_{abc}$.  Since the latter vanishes, 
so must   $\Psi^{ab,ij}_c$.  By symmetry, this gives the first row of  $(\ref{triangleid})$.
Now by the Pl\"ucker identity $\Psi_c^{ab,ij}-\Psi_c^{ai,bj}+\Psi_c^{aj,bi}=0$, we deduce that
$\Psi_c^{ai,bj}=\Psi_c^{aj,bi}$. We showed in example \ref{ExTriangle}  that $\Psi^{a,b}_c = \Psi^{a}_{bc}$, so we have:
$$[\Psi^{ai,bi}_c, \Psi^{a,b}_{ic}]_j = [\Psi^{ai,ai}_{bc}, \Psi^{a}_{ibc}]_j $$
The left-hand side is $\Psi^{ai,bj}_c \Psi^{aj,bi}_c =(\Psi^{ai,bj}_c)^2$ and the right-hand side is $(\Psi^{ai,aj}_{bc})^2$. Since $ G\backslash a\q bc$ is symmetric in $a,b,c$, we obtain the second row of  $(\ref{triangleid})$. We can therefore write the five-invariant ${}^5\Psi(abcij)=\pm (\Psi^{ab,ij}_c\Psi^{aci,bcj}- \Psi^{abc,cij}\Psi^{ai,bj}_c)$ in all possible  ways
 $ {}^5\Psi(abcij) = \pm \Psi_c^{ai,bj}\Psi^{abc,cij} $
obtained by permuting the indices $\{a,b,c\}$ and $\{i,j\}$. It follows that the third row of $(\ref{triangleid})$ must hold too.
The case where $a,b,c$ forms a 3-valent vertex is similar.

For the  last statement suppose, for example, that $\Psi_{ab}=0$. Then $\Psi^{ab,ij}=0$ by the same argument, and hence ${}^5\Psi(abcij)=0$.
% since $\Psi^{ij,ij}_{ab}$ has no spanning trees.
%First observe, that for any two edges $a,b$ of $G$, we have:
%\begin{equation} \label{2id}
%\Psi_{ab}  = 0 \quad \hbox{ or } \quad \Psi^{ab}  = 0 \quad \Longrightarrow \quad  \Psi^{ab,ij}=0 \ .
%\end{equation}
%Suppose that $\Psi_{ab}=0$, i.e., $\{a,b\}$ forms a loop, by corollary $\ref{corVanishingsubgraphs}.$
%The case $\Psi^{ab}=0$ is similar. 
%
%Then $G$ is completely symmetric in $a,b$.  It follows from the equation 
%$$\Psi^{ab,ij} + \Psi^{ai,bj} + \Psi^{aj,bi}=0$$
%and the symmetry of the last two terms, that we must have $\Psi^{ab,ij}=0$.
%
% Then $\Psi^a_b=\Psi^b_a$, and 
%Applying $(\ref{2id})$ to the minor $G\backslash c$ or $G\q c$, we deduce that:
%
\end{proof}

We need to proceed one step further in the  reduction.

\begin{lem}\label{lem5invcompat} At the 5th stage of reduction, the 5-invariant $\Psi(ijklm)$ is only  compatible with Dodgson polynomials of the form $\Psi_m^{ij,kl}$
or $\Psi^{ijm,klm}$. We have
$$[\Psi(ijklm),\Psi^{ij,kl}_m]_x = \Psi_m(ijklx) \Psi^{ijx,klm}\Psi^{ijm,klx}\ .$$
$$[\Psi(ijklm),\Psi^{ijm,klm}]_x = \Psi^m(ijklx) \Psi^{ijx,klm}\Psi^{ijm,klx}\ ,$$
where  we write $\Psi_m(ijklx)=\Psi_{G\q m}(ijklx)$ and $\Psi^m(ijklx)=\Psi_{G\backslash m}(ijklx)$\ . 
\end{lem}
\begin{proof} For two polynomials to be compatible in the reduction requires them to have a common parent. All parents of five invariants at the fourth generation are  of the form $(ij,kl)$, which can only have descendents of the form $(ij,kl)$ or $(ijm,klm)$, or the five invariant $\Psi_G(ijklm)$.  This proves the first statement.
Recall that for linear polynomials $f,g$ in the variables $x$ and $m$, we have
the identity  $[[0,f]_m,[f,g]_m]_x= [[0,f]_m,[\infty,f]_m]_x \times  [[0,f]_m,[0,g]_m]_x$.  Applying this in the case $f=\Psi^{ij,kl}$ and $g=\Psi^{ik,jl}$ gives:
$$[\Psi_m^{ij,kl}, \Psi(ijklm)]_x = [\Psi_m^{ij,kl},\Psi^{ijm,klm}]_x \times [\Psi_m^{ij,kl},\Psi_m^{ik,jl}]_x $$
$$= \Psi^{ijm,klx}\Psi^{ijx,klm} \times \Psi_m(ijklx)$$
%which proves the first equation. 
The second equation is similar, on replacing $0$ in the previous formula with $\infty$.
\end{proof}

\begin{lem} Let $a,b,c,i,j$ be five distinct edges in $G$. If $\{a,b,c\}$ forms a triangle in $G$, then at the $5^\mathrm{th}$ stage of its linear reduction, we only have terms of the form $(A,B)$, where $|A|=|B|\subset \{a,b,c,i,j\}$. The  only possible %
%compatibilities are the basic compatibilities (definition  $\ref{defnbasiccompat}$), plus possible  
  extra compatibilities are between $(abc,aij)=(abc,bij)=(abc,cij)$  and terms of the form $(pq,rs)$ or  $(pqr,rst)$, where $\{p,q,r,s,t\}=\{a,b,c,i,j\}$.

If  $\{a,b,c\}$ is a 3-valent vertex in $G$   the same holds, except that the extra compatibilities are at most between 
$(ab,ij) =(ac,ij)=(bc,ij)$  and $(pq,rs)$ or  $(pqr,rst)$.
\end{lem}
\begin{proof}
Suppose that $\{a,b,c\}$ forms a triangle in $G$. The case when $a,b,c$ forms a 3-valent vertex is similar (and dual to it).  %follow by duality.
%
% It follows from 
%$(\ref{triangleid})$ and  identities of the form $\Psi_c^{ab,ij}-\Psi_c^{ai,bj}+\Psi_c^{aj,bi}=0$, that  
%$$\Psi_a^{bi,cj}=\pm \Psi_a^{bj,ci}\ , \ \Psi_b^{ai,cj}=\pm \Psi_b^{aj,ci}\  ,   \Psi_c^{ai,bj}=\pm \Psi_c^{aj,bi} .$$
%Thus we can write the five-invariant $\{abcij\}$ as:
%$$ \{ai,bj\}\{abc,cij\} \ , \{ai,cj\}\{abc,bij\} \ , \hbox{ or } \{bi,cj\}\{abc,aij\} \ .$$
By lemma \ref{lemtrisplit},   we can write  
\begin{equation}\label{5inv6ways}
 {}^5\Psi(abcij) = \pm \Psi_c^{ai,bj}\Psi^{abc,cij} \end{equation}
 in all possible ways by  permuting the indices $\{a,b,c\}$ and $\{i,j\}$.  Now consider the reduction of $\Psi$ with respect to the five edges $\{a,b,c,i,j\}$. Since it factorizes,  the five-invariant ${}^5\Psi(abcij)$ will be  replaced with the factors occurring in the right-hand side of $(\ref{5inv6ways})$. This will introduce  extra compatibility conditions between these
factors owing to the fact (lemma $\ref{lem5invcompat}$) that the five-invariant is compatible with all terms of the form $\Psi^{pq,rs}_t$ and $\Psi^{pqt,rst}$, where
$\{p,q,r,s,t\}=\{a,b,c,i,j\}$. 
These  can be of  five different kinds (since terms of type $\Psi^{ab,ij}_c$ vanish):
%$$i) \quad \Psi^{ai,bj}_c \  ,\qquad ii)\quad   \Psi_j^{ab,ci} \  ,$$
%$$ \Psi^{abi,abj} ,\Psi^{abc,aij},  \Psi^{abi,cij}$$
%\begin{eqnarray}
%&i)   &  \Psi^{ai,bj}_c \ , \nonumber  \\%
%&ii)   &  \Psi_i^{ac,bj} \ , \nonumber \\
%&iii)   &  \Psi^{abi,abj} \ , \nonumber \\
%&iv)   &  \Psi^{abc,aij} \ , \nonumber \\
%&v)   &   \Psi^{aci,bij} \ , \nonumber 
%\end{eqnarray}
$$i) \quad \Psi^{ai,bj}_c \  ,\qquad\qquad  ii)\quad   \Psi_i^{ac,bj} \  ,$$
$$iii)\quad  \Psi^{aci,bcj}\  ,\qquad iv)  \quad   \Psi^{aci,bij}  \ , \qquad v) \quad  \Psi^{abc,cij}$$
where $\{a,b,c\}$ and $\{i,j\}$ are to be permuted in all possible ways.
% Consider, for example, the  term of type $ii)$ given by $\Psi_j^{ab,ci}$. We have a compatibility between $\Psi_j^{ab,ci}$ and ${}^5\Psi(abcij)$. Since  the latter is $\Psi_a ^{ci,bj}$
Each representative of $i)-iv)$  listed above is already   compatible with the left-hand factor $\Psi_c^{ai,bj}$ of $(\ref{5inv6ways})$ by the basic compatibilities of definition $\ref{defnbasiccompat}$. Since we can always write
${}^5\Psi(abcij)$ in many ways $(\ref{5inv6ways})$, it follows that any polynomial obtained from $i)-iv)$ by permuting $\{a,b,c\}$ and $\{i,j\}$ will always be compatible
with an appropriate choice of left-hand factor of $(\ref{5inv6ways})$. Therefore, the only new compatibilities that must be taken into account are the ones arising from  the right-hand factor of $(\ref{5inv6ways})$, and $v)$,   which are both of the form $\Psi^{abc,cij}$.
\end{proof}

We will only use this lemma in the following context.

\begin{cor} \label{cortri3vertex} Suppose that $e_1,\ldots, e_5$ are five distinct edges in $G$, three of which form a triangle, and three of which form a  3-valent vertex. Then 
at the $5^\mathrm{th}$ stage of the linear reduction of $G$, we have only Dodgson polynomials of the form $(A,B)$, where $|A|=|B|\subset \{e_1,\ldots, e_5\}$. The  compatibilities are precisely the basic compatibilities of definition  $\ref{defnbasiccompat}$,  and no others.
\end{cor} 

\begin{proof}  Put the first and second halves of the previous lemma together to deduce that any extra compatibilities that might arise actually reduce to basic compatibilities 
after  rewriting terms using the identifications  in lemma \ref{lemtrisplit}.
%
% The extra compatibilities in the first half of the previous lemma (when $G$ contains a triangle) are disjoint from those in the second half (when $G$ contains a 3-vertex).
%Therefore there are no extra compatibilities in this situation.
\end{proof}

\subsection{Graphs of vertex-width at most 3}
The most basic family of graphs that are linearly  reducible are given by the following theorem.
\begin{thm}  \label{thmvw3impliesmt} Any graph  $G$  satisfying $vw(G)\leq 3$  is of matrix type.
\end{thm}

\begin{proof}
 Let  $O=(e_1,\ldots, e_N)$  be an ordering  on $E_G$ such that $vw_O(G) \leq 3$, and consider the reduction of $\Psi=\Psi_G$ with respect to  $e_1,\ldots, e_N$. 
%Let $n\geq 2$, and suppose by induction that at the $(n-1)^{\mathrm{th}}$ stage, the reduction of $\Psi$  consists only of Dodgson polynomials  and basic compatibilities.
Let $G_1, G_2 $ denote the subgraphs of $G$ spanned by the  edges $\{e_1,\ldots, e_n\}$ and $\{e_{n+1},\ldots, e_N\}$, respectively. By assumption, $G_1, G_2$ meet in at most 3 vertices which we denote by $v_1,v_2,v_3$ (the case where there are fewer than 3 vertices is trivial and left to the reader).
Let $n\geq 2$ and suppose by induction that   the reduction at the $(n-1)^{\mathrm{th}}$ stage consists of Dodgson polynomials of the form 
 $\Psi^{A,B}_{G_1}$,  where  $|A|=|B|$ and  $A,B\subset G_1,$
and basic compatibilities given by cases $i)$, and $ii)$ of proposition $ \ref{propcompat5desc}$.  The compatibilities in case $i)$
%, are of the form $\{P,Q\}$ and $\{R,S\}$ where $||\{P,Q\},\{R,S\}||=2$. In this case, 
give rise to descendents which are products of Dodgson polynomials, and % given by identities $(\ref{resid1})$ and $(\ref{resid2})$, and
 one verifies the induction step  exactly as in the proof of proposition $ \ref{propcompat5desc}$. The compatibilities in case $ii)$ are between polynomials of the form
 $(Mij,Mkl)$ and $(Mik,Mjl)$. Let $m=e_n\in E_{G_1}$, and  consider the minor   $G'=G \backslash M \q N,$ where $N=G_1\backslash M \cup \{i,j,k,l,m\}$.  
 % The two polynomials above are just
 %the Dodgson polynomials $\{ij,kl\}$ and $\{ik,jl\}$ of $G'$, and we wish to
 It suffices to  show that  the five-invariant ${}^5\Psi_{G'}(ijklm)$ splits and induces no extra compatibilities. The minor $G'$
 inherits  a decomposition $G_1' \cup G_2$, where $G'_1=\{i,j,k,l,m\}$ has exactly 5 edges, and $G'_1 $ and $G_2$ meet in at most 3 vertices (again we only consider the non-trivial case when there are exactly three: $v_1,v_2,v_3$). There are  a limited number of possibilities. Suppose that $G'_1$ has no   vertex besides $v_1,v_2,v_3$. Since it has 5 edges, it must contain  two double edges or one triple edge. One of these situations is pictured below (left).
 \begin{figure}[h!]
  \begin{center}
%    \leavevmode
    \epsfxsize=12.0cm \epsfbox{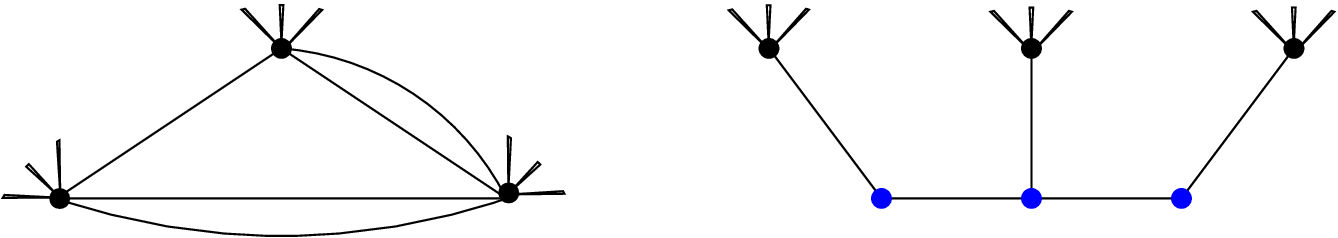}
  \label{figure3}
\put(-140,46){$v_1$} 
\put(-73,46){$v_2$} 
\put(-6,46){$v_3$} 
\put(-112,52){$G_2$} 
\put(-100,15){$G'_1$}
\put(-240,52){$G_2$}
\put(-275,15){$G'_1$}
% \caption{}
  \end{center}
\end{figure}
 
 \noindent
  It follows from lemma \ref{lemtrisplit} that  ${}^5\Psi_{G'}(ijklm)$ vanishes. If $G'_1$ contains three or more vertices different from $v_1,v_2,v_3$, then  one of them  must have degree $2$ or less (e.g. above right), and again we conclude that  ${}^5\Psi_{G'}(ijklm)=0$ by lemma \ref{lemtrisplit}.

\begin{figure}[h!]
  \begin{center}
%    \leavevmode
    \epsfxsize=12.0cm \epsfbox{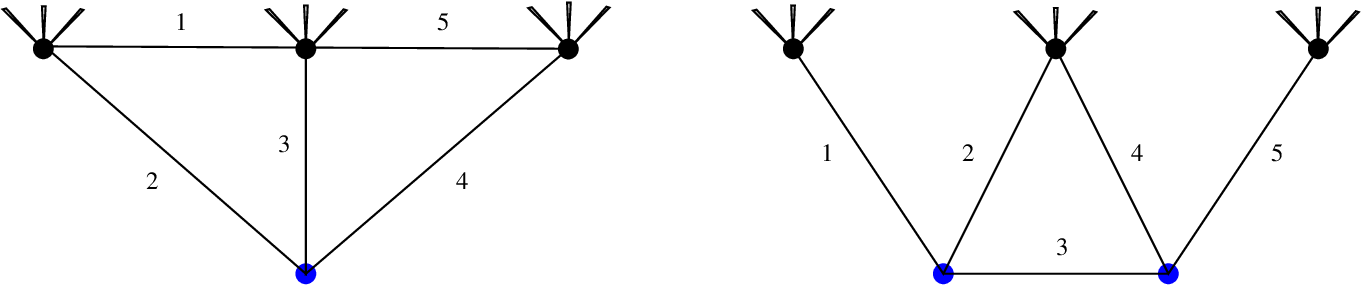}
  \label{figure3}
\put(-137,52){$v_1$} 
\put(-74,52){$v_2$} 
\put(-10,52){$v_3$} 
\put(-322,52){$v_1$} 
\put(-262,52){$v_2$} 
\put(-198,52){$v_3$} 
\put(-122,72){$G_2$} 
\put(-140,05){$G'_1$}
\put(-250,72){$G_2$}
\put(-250,05){$G'_1$}
% \caption{}
  \end{center}
\end{figure}

\noindent
For $G_1'$ to be simple it must have one or two vertices disjoint from $v_1,v_2,v_3$ and there are exactly two cases (above).
In either case,  $\{i,j,k,l,m\}$ contains both a triangle and a 3-valent vertex. We conclude by corollary $\ref{cortri3vertex}$ that 
${}^5\Psi_{G'}(ijklm)$ splits and induces no extra compatibilities.  This completes the induction step.
\end{proof}

\begin{rem}The proof consists  in analysing the local 5-minors in a graph of vertex width 3, and checking the triviality of the 5-invariant. 
It is likely that a similar argument will extend to the vertex width 4 case if one excludes the appropriate local minors.
Also, instead of appealing to corollary $\ref{cortri3vertex}$, it suffices  to compute the possible compatibilities for the two situations depicted above.  %figure above.
 One can show \cite{WD} that there is a universal formula for the graph polynomial of  a 3-vertex join, in the same spirit as corollary \ref{universal3joinformula}.  In other words,  the heart of the previous theorem can be reduced
 to computing the linear reduction of a few  small graphs, the most non-trivial of which is the wheel with four spokes $W_4$.
 In conclusion: \emph{the fact that $W_4$ is of matrix type implies that all graphs of vertex width 3 are of matrix type}.
 %
 % Note that  they are dual to each other in the sense that $ijk$ forms a triangle (3-vertex) in one if and only if it forms a 3-vertex (resp. triangle) in the other. 
% One therefore only has to compute a single Feynman graph.
%
\end{rem}
The previous theorem is closely related to the star-triangle relations, as the diagram below illustrates. The first row shows  a  sequence of operations which splits a triangle down the middle, the second splits a 3-valent vertex in two.
% and this operation  preserves the matrix reduction. 
%The 5-invariant of the diagram on the left (top row) is clearly trivial, since it has only 3 edges.
%Since the star-triangle relations are valid in this situation, this proves that the 5-invariant in the diagram on the right (top row) is also trivial.

\begin{figure}[h!]
  \begin{center}
%    \leavevmode
    \epsfxsize=11.0cm \epsfbox{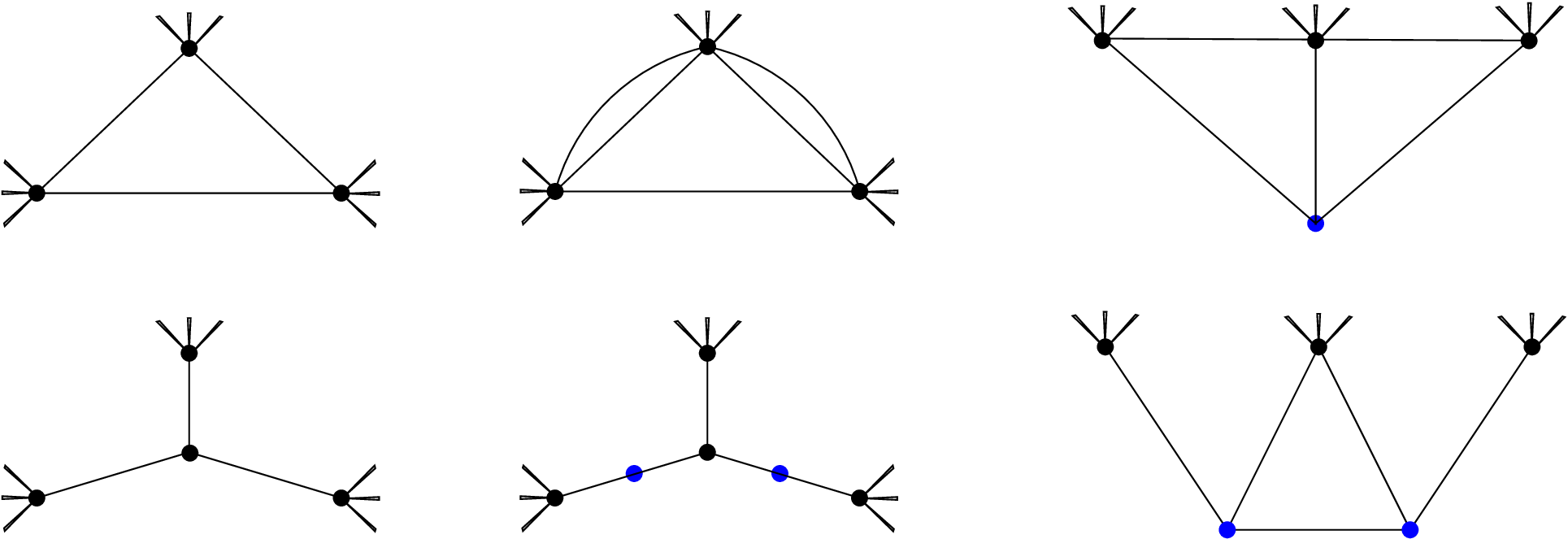}
  \label{figure3}
\put(-122,90){$ST$} 
\put(-122,82){$\leftrightarrow$} 
\put(-122,28){$ST$} 
\put(-122,20){$\leftrightarrow$}
\put(-230,90){$P$} 
\put(-230,82){$\leftrightarrow$}
\put(-230,28){$S$} 
\put(-230,20){$\leftrightarrow$}
 %\caption{The two situations in the proof of the vertex width $\leq 3$ theorem can be obtained by star-triangle operations. Note that the operation of splitting
% a triangle (top) and splitting a vertex (bottom) are dual to each other, and preserve primitive divergence. }
  \end{center}
\end{figure}
The operation of splitting a triangle  preserves primitive divergence.
\noindent 
%We will see below, however, that the star-triangle relations are not valid in all situations. Is it true for planar graphs (C-de-Verdieres?).
%\end{rem}

\begin{thm} \label{startriangleinduction}  Let $G$ be a graph with an ordering $e_1,\ldots, e_N$  on its edges. Suppose that $e_1,\ldots, e_5$ forms  a split triangle or a split vertex as indicated above (far right).
Let $G_{\Delta}$ and $G_{Y}$ denote the graph with $N-2$ edges obtained by replacing $e_1,\ldots, e_5$ with a triangle and a star (far left). Then if both $G_{\Delta}$ and $G_Y$
are linearly reducible with respect to the induced orderings, then so is $G$.
\end{thm}
\begin{proof} 
Apply corollaries \ref{corminorinduction} and \ref{cortri3vertex} in turn. It suffices to consider successive minors of the form $G\backslash e_i$ and $G\q e_i$, for $1\leq i\leq 5$.
In all cases one eventually obtains a minor of either $G_{\Delta}$ or $G_Y$.
\end{proof}

%begin{rem}
%The proof in fact gives a stronger result. We can admit the case where the vertex width is 4 provided that locally we have a triangle and a 3-valent vertex.
%\subsection{Positivity}

\begin{example} Let $W_3$ be the wheel with 3 spokes. Choose a triangle in $W_3$ and split it. This creates two new triangles. Choose either of these two triangles, and split them. Continuing in this way generates an infinite family of planar graphs which contains the wheels and zig-zags. They are clearly of vertex width 3, and hence  by theorem  \ref{thmvw3impliesmt}, of matrix type.
\end{example}

\newpage
\section{Moduli spaces and linearly  reducible hypersurfaces}\label{sectcohomperiods}
For any set of linearly reducible hypersurfaces, we construct  explicit maps to the moduli spaces $\Mod_{0,n}$ of genus 0 curves with $n$ marked points. This will enable us to compute the periods 
in the following section using the function theory of $\Mod_{0,n}$.

\subsection{Recap}
We briefly recall  the situation so far.  We begin with a set of irreducible hypersurfaces $S$ contained in $(\Pro^1)^N$  (typically, the zero locus of a graph polynomial), along with the coordinate hypercube $B$.  In this section, we shall work on the open complement of $S$, so it is convenient to   replace $\Pro^1$ with   $\G_m=\Pro^1\backslash\{0,\infty\}$.
Consider the maps $\pi_{[i]}:\G_m^{N} \rightarrow \G_m^{N-i}$ given by projecting out the first $i$ coordinates, and write $\pi_{[i]}=\pi_i\circ \ldots \circ \pi_1$, where $\pi_i: \G_m^{N-i+1} \rightarrow \G_m^{N-i}$ denotes each successive projection.  
In \S \ref{SectStrat} we defined the Landau varieties $L_i = L(S,\pi_{[i]})$ which, in particular,  have the property that: 
$$\pi_{[i]}:\G_m^N\backslash (S\cup \pi_{[i]}^{-1}(L_i) ) \To \G_m^{N-i}\backslash L_i $$
is a locally trivial fibration. 
We write $L_0= S$, and define:
$$\widetilde{L_i} = L_i \cup \pi^{-1}_{i+1} (L_{i+1}) \cup \ldots \cup (\pi_{N-1}\circ\ldots \circ \pi_{i+1})^{-1} (L_{N-1})\ .  $$
The situation is summarized by the following diagram:
\begin{equation}
\xymatrix{ \quad\G_m^N\backslash \widetilde{L_0}  \ar[d]^{ \pi_1}   \ar@/_19pt/ 
 [dd]_{\pi_{[2]}}  \ar@/_50pt/ 
 [dddd]_{ \pi_{[k]}}  \\
       \quad \G_m^{N-1}\backslash \widetilde{L_1}          \ar[d]^{ \pi_2}                  \\ 
        \quad  \G_m^{N-2}\backslash \widetilde{L_2}   \ar[d]^{ \pi_3}     \\
        \vdots   \ar[d]^{ \pi_k}   \\
          \quad \G_m^{N-k}\backslash \widetilde{L_k} 
         } \label{tower}
\end{equation}
The maps $\pi_{[k]}$ are  trivial fibrations, the $\pi_k$ in general are not. 
We say $S$ is  linearly reducible if each component of $\widetilde{L}_i$ is of degree at most one in the  fiber of $\pi_{i+1}$.

\subsection{Maps to the  moduli spaces  $\Mod_{0,n}$} \label{sectmapstomod}
Consider a set of distinct irreducible hypersurfaces $S$ contained in $\G_m^n$, and let $\pi$ denote the coordinate projection $\G_m^{n} \rightarrow \G_m^{n-1}$.  The set $S$ is partitioned into a set of vertical components $S_v$ (of  degree 0 in the fiber of $\pi$), and  horizontal components $S_h$ (of degree $\geq 1$ in the fiber):
$$S = S_v \cup S_h\ .$$
Suppose that every component of $S_h$ is linear in the fiber of $\pi$, and let $M=|S_h|$ denote the number of horizontal components in $S$.  Then there is a natural map:%\footnote{fiber products are over $\Spec \Z$ unless specified otherwise}
%$$\rho: \G_m^n \backslash S \To \Mod_{0,M+3} \times \G_m^{n-1} \backslash S_v $$
$$\rho: \G_m^n \backslash S \To \overline{\Mod}_{0,M+3}  $$
This follows from the univeral property of the moduli spaces $\overline{\Mod}_{0,M+3}$. The fiber of the map $\pi$ is isomorphic to $\G_m$, and the horizontal components of $S_h$   cut out $M$  points  on $\G_m$, and hence $M+2$  points on $\Pro^1$.  
This maps to  the universal curve $\overline{\Mod}_{0,M+3}\rightarrow \overline{\Mod}_{0,M+2}$ with $M+2$ marked points. More precisely, we have:
\begin{lem} \label{lemmaptomod} There is a commutative diagram mapping to the open moduli spaces:
$$
\xymatrix{ \G_m^n\backslash (S \cup \pi^{-1}L(S,\pi)) \ar[r]^{\qquad \rho} \ar[d]^{\pi} & \Mod_{0,M+3}  \ar[d]^{f} \\
\G_m^{n-1}\backslash L(S,\pi) \ar[r]^{\qquad \overline{\rho}} & \Mod_{0,M+2}}
$$
where $f$ is the  map which forgets  one of the marked points.
\end{lem}
\begin{proof} We write down all the maps explicitly.  Let $\alpha_1,\ldots, \alpha_n$ denote  coordinates on $\G_m^n$, and
let  $A$ denote the coordinate ring of $\G_m^{n-1}\backslash S_v$.
Let the set of horizontal components of $S$ be the zeros of irreducible polynomials  $f_i= a_i \alpha_1+ b_i$, where $a_i,b_i \in \Z[\alpha_2,\ldots, \alpha_n]$ and  $a_ib_i\neq 0$, for $1\leq i\leq M$. Recall from lemma \ref{lemLand1}  that the Landau variety $L(S,\pi)$ is given by the zeros of the resultants 
$$[f_i,f_j]_{\alpha_1} = a_ib_j-a_jb_i\ , \ [0,f_i]_{\alpha_1}=b_i \ ,  [\infty, f_i]_{\alpha_1}=a_i \ .$$ Define  $A^+=A[a_i^{-1}, b_i^{-1}, (a_ib_j-a_jb_i)^{-1}]$.
The lemma follows on taking the Spec of the following commutative diagram:
%Let $g_i$, for $1\leq i\leq m$, be irreducible polynomials defining  the vertical components of $S$.  
$$
\xymatrix{ A^+[\alpha_1^{\pm 1}, f_i^{-1}]   &  \Z[t_1^{\pm 1},\ldots, t_M^{\pm 1}, (1-t_i)^{-1}_{1\leq i \leq M}, (t_i-t_j)^{-1}_{1\leq i<j\leq M}]\ar[l]_{\rho^*\qquad\qquad\quad} \\
A^+    \ar[u]_{\pi^*}&     \Z[t_1^{\pm 1},\ldots, t_{M-1}^{\pm 1}, (1-t_i)^{-1}_{1\leq i\leq M-1}, (t_i-t_j)^{-1}_{1\leq i<j\leq M-1}]  \ar[l]_{\overline{\rho}^*\qquad\qquad\qquad \qquad}
 \ar[u]_{f^*}}
$$
where  $f^*$ is the inclusion, and the horizontal maps $\rho^*$, $\overline{\rho}^*$ are given by 
$$t_M \mapsto - {a_M \over b_M} \alpha_1\ , \quad \hbox{ and }   \quad  t_i \mapsto {a_M b_i \over b_M a_i} \quad  \hbox{ for } \quad  i<M\ . $$
The map $\rho^*$ is uniquely determined up to a choice of ordering on the hypersurfaces in $S_h\cup \{0,\infty\}$ (resp. marked points on $\Mod_{0,M+3}$). 
\end{proof}
We wish to  apply this lemma to the tower of maps $(\ref{tower})$. Let $R_i$ denote the coordinate ring of $\G_m^{N-i} \backslash \widetilde{L}_i$.
We obtain a nested sequence of rings:
$$R_0 \supseteq R_1 \supseteq \ldots \supseteq R_N $$

%$$\Spec R_i = \G_m^{n_i} \backslash \widetilde{L}_i$$

\begin{defn} Let $\widetilde{L}_i = V_i \cup H_i$ denote the decomposition of $\widetilde{L}_i$ into horizontal and vertical components with respect to the projection $\pi_i$. We write
\begin{equation}\label{Nihorizdef} M_i= |H_i|
\end{equation}
for the number of irreducible horizontal components. This is of course equal to the number of horizontal components of $L_i$.
\end{defn}

\begin{rem}If $(G,O)$ is an ordered graph, and $S=X_G,$ %is the zero locus of the graph hypersurface, 
then the numbers $M_0,\ldots, M_{e_G}$  are interesting invariants of  $(G,O)$.
% The numbers $N_0,\ldots, N_{e_G}$  are interesting invariants of an ordered graph, and in the generic case will be given by  
 % the  number of components in the Landau variety $L_i$, which  will 
 In the generic case, when $G$ is of matrix type, these numbers  initially  coincide with the number of Dodgson polynomials  
  $M_0=1, M_1=2, M_2=5, M_3=14, M_4=43,\ldots $, given by  %for which the general formula is 
  $$M_k=\sum_{i=0}^{\lfloor k/2 \rfloor} \binom{2i-1}{i-1} \binom{k}{2i} 2^{k-2i} = 2^{k-1}+{1\over 2}\binom{2k}{k} \ . $$
For any given graph, however, the numbers $M_k$ are typically smaller because many Dodgson polynomials  vanish due to lemma $\ref{lemTechVanish}$,
and so the $M_i$  tail off for large $i$.
% For example, for the wheel with 3 spokes, we have
% $M_0,\ldots, M_5=1,2,4,3,1$.
\end{rem}
Now suppose that $S$ is linearly reducible, and write the    components $H_k$ of $\widetilde{L}_k$ as zeros of  linear terms 
$f_1,\ldots, f_{M_k}$ where $f_i= a_i \alpha_k +b_i,$  and $a_i,b_i\in R_{k+1}$. We have
\begin{equation}
R_k= R_{k+1}\big[ \alpha_k , {1\over \alpha_k},   {1\over f_i} :f_i \in H_k\big] \ .
\end{equation}
Since   the Landau variety of $\widetilde{L}_k$ with respect to $\pi_k$ contains, but is not  in general equal to $\widetilde{L}_{k+1}$, we cannot apply the previous lemma directly to each step in the tower $(\ref{tower})$. But if we  set $L^+_{k+1} =  L(\widetilde{L}_k, \pi_{k+1})$, lemma \ref{lemmaptomod}  gives:
$$
\xymatrix{ \G_m^{N-k}\backslash \widetilde{L}_k \ar[d]^{\pi_{k+1}}  &   \ar@{_{(}->}[l] \G_m^{N-k}\backslash \widetilde{L}_k \backslash \pi_{k+1}^{-1} L^+_{k+1}  \ar[r]^{\qquad \rho} \ar[d]^{\pi_{k+1}} & \Mod_{0,M_k+3} \ar[d] \\
\G_m^{N-k-1}\backslash \widetilde{L}_{k+1}
 &\ar@{_{(}->}[l] \G_m^{N-k-1}\backslash L^+_{k+1} \ar[r]^{\overline{\rho}} & \Mod_{0,M_k+2}}
$$
We  use the notation   $R^+_{k+1} = R_{k+1}[(a_ib_j-a_jb_i)^{-1}, 1\leq i<j\leq M_k ]$ (see \cite{BrCMP}). Then the  left-hand square of the above diagram is simply  the Spec
of the diagram:
$$
\xymatrix{ R_k \ar[r]& R_k\otimes_{R_{k+1}} R_{k+1}^+ \\
         R_{k+1} \ar[u]\ar[r] & R^+_{k+1} \ar[u]}
$$

\begin{lem} \label{lembetagammadef}There exists a  pair of spaces $\Mod_{0,M_k+j}^{\dag}$, with $j=2,3$ which satisfy    
%$\Mod_{0,N}\subset \Mod^{\dag}_{0,N}  \subset \overline{\Mod}_{0,N}$ which satisfies:
$\Mod_{0,M_k+j} \subseteq \Mod_{0,M_k+j}^{\dag} \subseteq \G_m^{M_k+j-3}$,
such that $\rho$, $\overline{\rho}$  extend to  maps:
$$
\xymatrix{\Spec R_k= \G_m^{N-k}\backslash \widetilde{L}_k \ar[d]^{\pi_{k+1}} \ar[r]^{}  & \Mod^{\dag}_{0,M_k+3} \ar[d] \\
 \Spec R_{k+1}=\G_m^{N-k-1}\backslash \widetilde{L}_{k+1}\ar[r]^{\qquad \qquad \beta_k}
 &  \Mod^{\dag}_{0,M_k+2} }
$$
and induce an isomorphism:
$$\gamma_k : \Spec R_{k} \overset{\sim}{\To} \Mod^{\dag}_{0,M_k+3} \times_{\Mod^{\dag}_{0,M_k+2}} \Spec R_{k+1}\ .$$ 
 \end{lem}

\begin{proof} The map $\overline{\rho}$ defined in  the proof of lemma  \ref{lemmaptomod}  gives a   map   from $\Spec R_{k+1}^+ $ to $\Mod_{0,M_k+2}=\Spec \Z[t_1^{\pm 1} ,\ldots, t_{M_k-1}^{\pm 1} ,(1-t_i)^{-1}, (t_i-t_j)^{-1}]$  such that:
$$ t_i \mapsto {a_{M_k} b_i \over b_{M_k} a_i} \quad  \hbox{ for } \quad  1\leq i\leq M_k-1\ . $$
Define $\Mod^{\dag}_{0,M_k+2}$ to be the Spec of the ring $Q$ obtained from $\Z [t_1^{\pm 1} ,\ldots, t_{M_k-1}^{\pm 1}]$ by  inverting terms $(1-t_i)$ (respectively $(t_i-t_j)$) if and only if 
$[f_i,f_{M_k}]=a_{M_k}b_i-a_ib_{M_k}$ (respectively $[f_i,f_j]=a_ib_j-a_jb_i$) is invertible in $R_{k+1}$.  Now  define 
$$\Mod^{\dag}_{0,M_k+3}=\Spec Q\big[t_{M_k}^{\pm 1}, {1\over 1- t_{M_k}}, {1 \over t_{M_k}-t_i} \quad 1\leq i\leq M_k-1\big]\ ,$$
along with  the natural map from $\Mod^{\dag}_{0,M_k+3}\rightarrow \Mod^{\dag}_{0,M_k+2}$. This is no longer a fibration in general, reflecting the fact that $\Spec R_k \rightarrow \Spec R_{k+1}$ is also not a fibration.
%$$\Q[t^{\pm 1}_1,\ldots, t^{\pm 1}_N, {1\over 1-t_i} \hbox{ if } i\in I, {1 \over t_i-t_j} \hbox { if }  ]$$ 
\end{proof}

The space $\Mod_{0,M+2}^{\dag}$ is the complement of only those hyperplanes in $\G_m^{M-1}$ which correspond to resultants $[f_i,f_j]$ which 
survive in the Landau variety $\widetilde{L}_{k+1}$.
The superscript $\dag$ reflects the discrepancy between $\widetilde{L}_{k+1}$ and $L(\widetilde{L}_k, \pi_{k+1})$, or, in other words, the failure of $\pi_{k+1}$ from being a  fibration.

\begin{thm}\label{thmmaptomdag} Let  $S$ be a set of  linearly reducible hypersurfaces in $\G_m^N$, and let $L_i=L(S,\pi_{[i]}) \subset \G_m^{N-i}$ denote the corresponding Landau varieties. Let $\widetilde{L}_i$ be defined as above, and let  $R_i$ denote the coordinate ring of $\G_m^{N-i} \backslash \widetilde{L}_i$. Let us write
$$m_i = M_i+2\geq 3\ ,$$
where $M_i$ is the number of horizontal components of $L_i$, for $1\leq i\leq N$. Then for each $1\leq i \leq N$, there exists a pair of affine schemes
$ \Mod^{\dag}_{0,m_i}  ,  \Mod^{\dag}_{0,m_i+1}  $ satisfying
\begin{equation}\label{Mplusdefn}
\begin{array}{ccc}
 \Mod_{0,m_i+1} \subseteq & \Mod^{\dag}_{0,m_i+1}  &\subseteq\G_m^{m_i-2}   \\
 \downarrow & \downarrow & \downarrow \\
   \Mod_{0,m_i} \subseteq & \Mod^{\dag}_{0,m_i}  &\subseteq\G_m^{m_i-3}   
\end{array}
\end{equation}
and hence defined over $\Z$.  There are  connecting morphisms $\phi_i$  such that:
$\phi_N:\Spec R_N \To \Mod^{\dag}_{0,m_{N-1}},$ 
and for all $1\leq i\leq N-1$,
\begin{equation}
\phi_i : \Mod^{\dag}_{0,m_i+1}\underset{\Mod^{\dag}_{0,m_i}}{\times} \ldots \underset{\Mod^{\dag}_{0,m_{N-2}}}{\times} \Mod^{\dag}_{0,m_{N-1}+1} \underset{\Mod^{\dag}_{0,m_{N-1}}}{\times} 
\Spec R_N \To   \Mod^{\dag}_{0,m_{i-1}} \nonumber 
\end{equation}
where the maps in the fiber product are given on the left  of each $\times$ by the natural maps  $\Mod^{\dag}_{0,m_{i}+1} \rightarrow \Mod^{\dag}_{0,m_i}$ $(\ref{Mplusdefn})$, and on the right 
by the maps $\phi_{j}$ for $i<j\leq N$. Then for each $1\leq i \leq N$, there is an isomorphism:
\begin{equation} \label{eqnRiisom}
\psi_i:\Spec R_i \cong  \Mod^{\dag}_{0,m_i+1}\underset{\Mod^{\dag}_{0,m_i}}{\times} \ldots \underset{\Mod^{\dag}_{0,m_{N-1}}}{\times}\Spec R_N
\end{equation}
such that the projections between the $\Spec R_i$ are induced by the natural maps   $\Mod^{\dag}_{0,m_{i+1}} \underset{\Mod^{\dag}_{0,m_i}}{\times} X\To \Mod^{\dag}_{0,m_i} \underset{\Mod^{\dag}_{0,m_i}}{\times} X\overset{\sim}{\To}X$, i.e: 
\begin{equation}\label{Ricompat}
\xymatrix{ \Spec R_i \ar[d]^{\pi_i} \ar[r]^{\sim \qquad\qquad\qquad}  &     \Mod^{\dag}_{0,m_i+1}\underset{\Mod^{\dag}_{0,m_i}}{\times} \ldots \underset{\Mod^{\dag}_{0,m_{N-1}}}{\times}\Spec R_N \ar[d] \\
\Spec R_{i+1} \ar[r]^{\sim \qquad \qquad \qquad}
 &       \Mod^{\dag}_{0,m_{i+1}+1}\underset{\Mod^{\dag}_{0,m_{i+1}}}{\times} \ldots \underset{\Mod^{\dag}_{0,m_{N-1}}}{\times}\Spec R_N       }
\end{equation}
In other words, the schemes $R_i$ can be entirely constructed out of $\Spec R_N\subseteq \G_m$, the modified moduli spaces $\Mod_{0,m_i}^{\dag}$, and the connecting maps $\phi_i$.
\end{thm}
%\begin{equation}
%\phi_i : \Md_{0,m_i+3}\underset{\Md_{0,m_i+2}}{\times} \ldots \underset{\Md_{0,m_{N-1}+2}}{\times} \Md_{0,m_N+3}\To   \Md_{0,m_{i-1}+2}
%\end{equation}

\begin{proof} In the previous lemma we defined morphisms  $\beta_k: \Spec R_{k+1} \rightarrow \Mod^{\dag}_{0,m_k}$, and isomorphisms $\gamma_k: \Spec R_k \overset{\sim}{\To} \Mod^{\dag}_{0,m_k+1}  \times_{\Mod^{\dag}_{0,m_k}} \Spec R_{k+1}$. Formally set $\gamma_N=\psi_{N+1}=1$, and define, by decreasing induction for $2\leq i\leq N$:
\begin{eqnarray}\label{phipsifrombetagamma}
\phi_i& = & \beta_{i-1} \circ \psi_{i+1}^{-1} \\
\psi_i & = & (1 \times \psi_{i+1}) \circ \gamma_i \ . \nonumber  
\end{eqnarray}
%\vspace{-0.2 in}
%
% Equation $(\ref{eqnRiisom})$ is trivial for $i=N$. For simplicity of notation, we only prove the case $i=N-1$.  By lemma.. there is a commutative diagram
%$$
%\xymatrix{ \Spec R_{N-1} \ar[d]^{\pi_N} \ar[r]^{\qquad\qquad\qquad}  &     \Mod^{\dag}_{0,m_{N-1}+1} \ar[d] \\
%\Spec R_{N} \ar[r]^{ \qquad \qquad \qquad}
% &       \Mod^{\dag}_{0,m_{N-1}}      }
%$$
%Call the horizontal map on the bottom $\phi_N$. The diagram gives a morphism
%$$\Spec R_{N-1} \To \Mod^{\dag}_{0,m_{N-1}+1} \underset{ \Mod^{\dag}_{0,m_{N-1}} }{ \times} \Spec R_N\ .$$
%That this is an isomorphism follows from the definition of $\Mod^{\dag}_{0,m_{N-1}+1}$. The result follows in a similar way by induction. 
\end{proof}

\subsection{The case when $G$ is of matrix type} \label{sectMonGmatrixtype} 
Let $(G,\{1,\ldots, e_G\})$ be of matrix type. 
%We can write down the previous constructions in the case when   $G$ is   a graph of matrix type  with respect to the ordering $1,\ldots, e_G$ on its set of edges. 
Let $Q_0=\{(\{\emptyset,\emptyset\},\emptyset)\}$,
and for each $1\leq i \leq e_G$,   denote by $Q_i$ the set of all  indices 
$ (\{I,J\},K)$ 
 where $I,J,K$ are subsets of $\{1,\ldots, i\}$ such that $ I\cup J\cup K = \{1,\ldots, i\}$,  $ |I|=|J|$, and such that the 
   Dodgson polynomial $\Psi^{I,J}_K$ does not vanish.
   Since $G$ is  of matrix type,  its  Landau varieties are contained in the zero locus of the Dodgson polynomials. Therefore, working over $\Q$, we have:
   $$R_i= \Q [\alpha^{\pm 1}_{i+1},\ldots, \alpha^{\pm 1}_{e_G} , { (\Psi^{I,J}_K)^{-1}} \ , \hbox{ for } (I,J,K) \in Q_i \cup Q_{i+1}\cup \ldots \cup Q_{e_G} ] \ , $$
for $0\leq i \leq e_G-1$, see also (\cite{BrCMP},  \S 7). 
%
%
%Then we have: 
%$$\mathrm{Spec } \, R_0 =  (\Pro^1)^{N} \backslash L_0  \cup  \pi_1^{-1}(L_1) \cup \ldots \cup \pi_{e_G}^{-1}(L_{e_G})\ .$$
%is the complement of the graph hypersurface, and all preimages of lower Landau varieties.
Next, we decompose $Q_i$ into its set of horizontal and vertical components:
$Q_i = Q^h_i \cup Q^v_i$
where $Q^h_{i-1} $ is the set of indices $(\{I,J\},K)$ such that $\Psi^{I,J}_K = \Psi^{iI,iJ}_K\alpha_{i} +\Psi^{I,J}_{iK}$ satisfies
$\Psi^{iI,iJ}_K \Psi^{I,J}_{iK}\neq 0$ and $Q^v_{i-1}$ is its complement in $Q_{i-1}$. 
The numbers $M_i = |Q^h_i|$ are  the number of horizontal Dodgson polynomials at the $i^{\mathrm{th}}$ level. 
%
%  $N_0=1, N_1=2, N_3=4, N_4=3, N_5=1$
 
Next, we define the spaces $\Mod^{\dag}_{0,M_i+3}\subseteq \G_m^{M_i+3}$ in terms of  simplicial coordinates we denote by $t^{I,J}_K$, 
for every  $(\{I,J\},K) \in Q^h_i$, i.e., first set:
\begin{equation} \label{tijkcoords}
\Mod_{0,M_i+3} =\mathrm{Spec\,}\, \Q\big[(t^{I,J}_K)^{\pm}, { 1 \over 1-t^{I,J}_K }, {1\over t^{I,J}_K - t^{L,M}_N  } \big]\ .\end{equation}
Now, fixing some  element $(A,B,C)\in Q^h_{i-1}$, we can write down  a map  $\rho_i:\Spec R_i \rightarrow \Mod_{0,M_i+3}$  as in lemma $\ref{lemmaptomod}$
%$$ \rho_i: (\Pro^1)^{N-i} \backslash L_i \To \Mod_{0,N_i+3} \times (\Gm)^{N-i-1}$$
 defined on the affine coordinate  rings by 
%$$(\rho_i)_* : \Q\big[(t^{I,J}_K)^{\pm}, (1-t^{I,J}_K )^{-1}, (t^{I,J}_K - t^{L,M}_N)^{-1} \big]\To F_i $$
\begin{equation} %\label{rhoDodmap}
(\rho_i)_*: t^{I,J}_K  \mapsto {\Psi^{I,J}_{iK} \Psi^{iA,iB}_C \over \Psi^{iI,iJ}_{K} \Psi^{A,B}_{iC}} , \quad 
(\rho_i)_*:t^{A,B}_C \mapsto -\alpha_i {\Psi^{iA,iB}_C \over \Psi^{A,B}_{iC}} \nonumber
\end{equation}
The coordinate on the universal curve, i.e., the fiber of the  forgetful map $\Mod_{0,M_i+3}\rightarrow \Mod_{0,M_i+2}$,
is $t^{A,B}_C$. 
The modified space  $\Mod_{0,M_i+2}^+$ is defined in a similar manner to ($ \ref{tijkcoords})$ except that  we denote the coordinates by $s^{I,J}_K$ for $(\{I,J\},K) \neq (\{A,B\},C)$, and the only  terms
besides  $s^{I,J}_K$  that are inverted are  terms of  the form 
$(1-s^{I,J}_K)$ if $(I,J),(A,B)$ satisfy a basic compatibility (definition \ref{defnbasiccompat}), and 
$(s^{I,J}_K-s^{L,M}_N)$
 if $(I,J), (L,M)$ do.  %\begin{eqnarray}
%(1-s^{I,J}_K)^{-1} & & \hbox{ if }\quad ||(I,J),(A,B)||=2 \ , \nonumber  \\
%(s^{I,J}_K-s^{L,M}_N)^{-1} & & \hbox{ if }\quad ||(I,J),(L,M)||=2 \ .  \nonumber 
%\end{eqnarray}
The map $\beta_i: \Spec R_{i+1}\rightarrow \Mod_{0,m_{i}}^{\dag}$ of lemma \ref{lembetagammadef} is given by:
\begin{equation} \nonumber %\label{betamap} 
(\beta_i)_*: s^{I,J}_K  \mapsto {\Psi^{iI,iJ}_K \Psi^{iA,iB}_C \over \Psi^{I,J}_{iK} \Psi^{A,B}_{iC}}\end{equation} 
To verify that this map is  well-defined uses the identities $(\ref{FirstDodgsonId})$ and $(\ref{SecondDodgsonId})$.
The map $\rho_i$ defined above also gives the isomorphism 
$$\gamma_i:\Spec R_i\overset{\sim}{ \To} \Mod_{0,m_i+1}^{\dag}\underset{\Mod_{0,m_i}^{\dag}}{\times} \Spec R_{i+1}\ , $$
where the fiber product is defined using $\beta_i$. Thus on the affine rings we have
\begin{eqnarray} \nonumber %\label{gammamap}
(\gamma_i)_*: R_{i+1}[(s^{A,B}_C)^{\pm 1},(1- s^{A,B}_C)^{-1}, (s^{A,B}_C-s^{I,J}_K)^{-1}] & \overset{\sim}{\To} & R_i \nonumber \\
%(\alpha_i)_*: s^{I,J}_K  \mapsto {\Psi^{iI,iJ}_K \Psi^{iA,iB}_C \over \Psi^{I,J}_{iK} \Psi^{A,B}_{iC}} ,   
(\gamma_i)_*:s^{A,B}_C & \mapsto & -\alpha_i {\Psi^{iA,iB}_C \over \Psi^{A,B}_{iC}} \nonumber
\end{eqnarray}
The maps $\phi_i$ and $\psi_i$ are deduced from  $(\ref{phipsifrombetagamma})$. %
%The maps $\beta_i$ and $\gamma_i$ above suffice to describe everything, since the maps $\phi_i, \psi_i$ of theorem $\ref{thmmaptomdag} $ are determined inductively by: $%\psi_N=\phi_N=1$ and for $2 \leq i \leq N$ 
%$$\psi_i = (1\times \psi_{i+1} ) \circ \gamma_i$, and $\phi_i = \beta_i \circ \psi_{i+1}^{-1}
%
%\begin{eqnarray}
%\psi_{i-1} & = & (1\times \psi_{i} ) \circ \gamma_{i-1} \ , \nonumber\\
%\phi_{i-1} & = & \beta_{i-1} \circ \psi_{i}^{-1}\ . \nonumber
%\end{eqnarray}
Concretely, the maps $(\phi_i)_*$ send  each $s^{I,J}_K$ to a product of cross-ratios involving other $s^{P,Q}_R$'s.
 
 \begin{rem}
 The upshot of  this   is  to write down   changes of variables which  essentially turn  the  graph polynomial into products of cross ratios.  
 \end{rem} % for graphs of matrix type.  
% For example, the
%map $\gamma_{i+1}^{-1} \beta_i$ sends $s^{I,J}_K$ to 

%Finally, the  connecting map:
%$$\phi_i : \Mod_{0,N_{i+1}+3} \times \Spec F_{i-1} \To \Mod^+_{0,N_i+2}\ , $$
 %is given on the affine coordinate rings by:
%\begin{equation} \label{phistardef}
%\phi^*( s^{I,J}_K ) = { \Psi^{aI,aJ}_{bK} \Psi^{aA,aB}_{bC} \over \Psi^{I,J}_{abK} \Psi^{bA,bB}_{aC}} {(t^{aI,aJ}_K-t^{A,B}_C ) (t^{aA,aB}_C-t^{A,B}_C )
%\over (t^{I,J}_K- t^{A,B}_C)(1-t^{A,B}_C) } 
%\end{equation}
%\begin{rem}
%If $G$ is matrix reducible, then   the  number of horizontal components in the Landau variety $L_i$   will  in general coincide with the number of Dodgson polynomials  
%for small $i$. Thus, 
 % $N_0=1, N_1=2, N_2=5, N_3=14, N_4=43,\ldots $, and generically 
 % $$N_k=\sum_{i=0 }^{\lfloor k/2\rfloor} \binom{2i-1}{ i-1} \binom{k}{2i} 2^{k-2i}   =   2^{k-1}+{1\over 2}\binom{2k}{k}\ . $$
 % \end{rem}

\newpage 
\section{Calculation of the Periods}\label{sectCalcPeriods}
 %The technical heart of the period calculation for Feynman integrals uses the bar construction.  
 %See \cite{BrCMP} for  a concrete example in the case of the wheel with 3 spokes.
\subsection{The bar construction of $\Mod_{0,n}$} We briefly recall some properties  of the $\Q$-algebra of iterated integrals on the moduli spaces $\Mod_{0,n}$ 
(see also \cite{BrENS}, \S 3).

\begin{defn} \label{Vdefn}
Let $V(\Mod_{0,n}) = H^0(\overline{B}(\Omega^*_{\overline{\Mod}_{0,n}}(\log \overline{\Mod}_{0,n}\backslash \Mod_{0,n})))$ denote the zeroth cohomology  of Chen's reduced bar complex on the global logarithmic forms on $\Mod_{0,n}$, for  $n\geq 3$. It is a graded  commutative Hopf algebra defined over $\Q$.
\end{defn}

 One can also view $V(\Mod_{0,n})$ as  the de Rham realization of the motivic fundamental group %$\pi_1^{\mathrm{mot}} (\Mod_{0,n})$
 of $\Mod_{0,n}$.   Suppose we are given a tangential basepoint $t$ on $\Mod_{0,n}$ which is defined over $\Z$. Then we obtain an isomorphism (\cite{BrENS}, \S6.7):
 \begin{equation}\label{rhot}
 \rho_t:V(\Mod_{0,n}) \overset{\sim}{\To} L(\Mod_{0,n})\ , 
 \end{equation}
 where $L(\Mod_{0,n})\otimes_{\Q}\C$ is the graded Hopf algebra of homotopy-invariant iterated integrals on $\Mod_{0,n}(\C)$. Its $\Q$-structure $L(\Mod_{0,n})$ is given by  the choice of basepoint $t$.  %the $\Q$-structure  on $L(\Mod_{0,n})\otimes_{\Q}\C$.
The elements of $L(\Mod_{0,n})$ can be expressed as multiple polylogarithms in $n-3$ variables, which are multivalued  functions on $\Mod_{0,n}(\C)$ with unipotent monodromy.

\begin{defn}Let $\MZV=\Q[\zeta(n_1,\ldots,n_r): n_1,\ldots, n_r\in \N, n_r\geq 2]$ denote the ring of multiple zeta values. It is filtered by the weight $n_1+\ldots+n_r$.
\end{defn}

Choosing a different tangential base point  over $\Z$ changes the $\Q$-structure on $L(\Mod_{0,n})\otimes_{\Q}\C$, but preserves the $\Q$-structure on the filtered algebra
$L(\Mod_{0,n}) \otimes_{\Q} \MZV$ (the $\Q$-structure is modified by a product of Drinfel'd assocatiors, which have coefficients in $\MZV$). 
By abuse of notation, we will sometimes consider elements of $V(\Mod_{0,n})$ as elements of $L(\Mod_{0,n})\otimes_{\Q}\MZV$, for some unspecified isomorphism $\rho_t$.

%Since $\rho_t$ is an isomorphism,  we will henceforth consider elements of $V(\Mod_{0,n})$ as functions 
We also require a relative version of the bar construction. Let 
$$\Mod_{0,n+1} \To \Mod_{0,n}$$
denote the map which forgets a marked point. 
It comes with  $n$ sections $\sigma_1,\ldots, \sigma_{n}$, and  is a  fibration with fibers isomorphic to $\Pro^1\backslash \{\sigma_1,\ldots, \sigma_n\}$.
 We define
 $$V_{\Mod_{0,n}}(\Mod_{0,n+1}) = H^0(\overline{B}(\Omega^*_{\Pro^1/\Mod_{0,n}} (\log (\sigma_1\cup\ldots \cup \sigma_n))))\ ,$$
 to be the   bar construction of the fiber relative to the base, which  is again graded by the weight.  Let $L_{\Mod_{0,n}}(\Mod_{0,n+1})$ be its realization in terms of hyperlogarithms (\cite{BrENS} \S5, \cite{BrCMP}, \S5.1), which again depends on the choice of a tangential basepoint.
 %We denote by $L_{\Mod_{0,n}(\Mod_{0,n+1})$ the corresponding algebra of iterated integrals.
 %\footnote{One can construct from it, the universal pro-unipotent variation of mixed Hodge-Tate structures over $\Mod_{0,n}$.}
 The algebraic structure of $V_{\Mod_{0,n}}(\Mod_{0,n+1})$  is  that of a free shuffle algebra on $n-1$ generators.
 %
  %
 % simple: it is isomorphic to the universal envelopping algebra of the free Lie algebra on generators $e_{\sigma_1},\ldots, e_{\sigma_{n}}$ subject to the single relation $%  \sum_{i=1}^n e_{\sigma_i}=0$.  
We will only require the following fact, proved in \cite{BrENS}:

\begin{thm}\label{thmModint} Let $\omega \in \Omega^1_{\Mod_{0,n+1}/\Mod_{0,n}} \otimes_{\Q} L_{\Mod_{0,n}}(\Mod_{0,n+1})$ of weight $k$. For any two distinct sections $\sigma_i,\sigma_{j}:\Mod_{0,n}\rightarrow \Mod_{0,n+1}$ the integral in the fiber satisfies
$$\int_{\sigma_i}^{\sigma_{j}} \omega \in \MZV\otimes_{\Q}  L(\Mod_{0,n})\ ,$$
and is of total weight at most $k+1$. For the integral to make sense, we must assume that in each fiber, the domain of integration is a continuous path from $\sigma_i$ to $\sigma_j$ along which $\omega$ is single-valued.
\end{thm}
 It is proved in (\cite{BrENS}, \S3.5) that  there is an isomorphism of algebras (which does not respect the coproducts) $V(\Mod_{0,n+1}) \cong V_{\Mod_{0,n}}(\Mod_{0,n+1})\otimes_{\Q} V(\Mod_{0,n})$. The algebraic structure of $V(\Mod_{0,n})$ is a product of shuffle algebras and is therefore well suited to algorithmic calculations of Feynman integrals. See \cite{BrCMP} for  a concrete example.

\subsection{Linearly reducible spaces} 
Let $S\subset (\Pro^1)^N$ be a linearly reducible set of hypersurfaces over $\Q$. Recall from  \S\ref{sectmapstomod} that this gives rise to a nested sequence of rings
$R_0\supset R_1\supset \ldots \supset R_{N-1} \supset R_N\supset \Q$,   %$\Spec R_k \rightarrow \Spec R_{k+1}$ is a linear fibration:
along with two families of maps:
\begin{equation}\label{LRbeta} \beta_i:\Spec R_{i+1} \rightarrow \Mod^{\dag}_{0,m_i}
\end{equation}
and
\begin{equation}\label{LRgamma}
\gamma_i:\Spec R_i\overset{\sim}{ \To} \Mod_{0,m_i+1}^{\dag}\underset{\Mod_{0,m_i}^{\dag}}{\times} \Spec R_{i+1}\ , 
\end{equation} 
where $\Mod^{\dag}_{0,m_i}\supseteq \Mod_{0,m_i}$ is isomorphic to an affine complement of hyperplanes, and $\Spec R_N$ 
is an open subscheme of $\Pro^1$.  
%In  general, $\Spec R_i\cong $ is isomorphic to the complement of the Landau variety $L_i$ of $S$ in $(\Pro^1)^{n-i}$.
 We will use the maps $\beta_i$ and $\gamma_i$ to construct a space of functions $V(R_i)$ on $\Spec R_i$ in which to compute  periods.
 
\begin{defn} 
Let
 $V(\Mod^{\dag}_{0,m_i})\subseteq V(\Mod_{0,m_i})$  be the zeroth cohomology of the reduced bar construction on $\Mod^{\dag}_{0,m_i}\supseteq \Mod_{0,m_i}$.
 \end{defn}
 By an isomorphism 
 $\rho_t$, elements of $V(\Mod^{\dag}_{0,m_i})$  correspond to  multivalued functions in $L(\Mod_{0,m_i})$ which are unramified along boundary components of $\Mod^{\dag}_{0,m_i}\backslash
 \Mod_{0,m_i}$.

\begin{prop} 
Let $E$  be a complement of hypersurfaces in $\A^n\times \A^1$ which are defined over $\Q$, let $\pi:\A^n \times \A^1\rightarrow \A^n$
denote the induced projection, and let  $B=\pi(E) \subset \A^n$. Assume that the components of $E$ are of degree $\leq 1$ in the fiber of $\pi$, and   let 
$B'$ be the largest open subvariety of $B$ such that 
$$\pi:E' \To B'$$
is a  fibration, where $E'=\pi^{-1}(B')$. Denote the fiber over the generic point of $B'$ by $F$.
 Thus $B'$ contains the Landau variety $L(E,\pi)$ as defined in lemma \ref{lemLand1}.

For any  $X=E,B,E',B'$, let us write
$b(X) = H^0 (\overline{B}( \Omega^* (X))$ for  the zeroth cohomology group of the reduced bar construction on the de Rham complex of $X$ with coefficients in $\Q$, and 
$b_{B'}(F)$ for the relative version  $H^0 (\overline{B}( \Omega^* (F/B'))$.%It is isomorphic to th

(i). There is an isomorphism of algebras $b(E') \cong b(B') \otimes_{\Q} b_{B'}(F)$

(ii). The image of the  map $b(E) \hookrightarrow b(E')$  is contained in $b(B) \otimes_{\Q} b_{B'}(F)$.
\end{prop}

\begin{proof} Write $B=\Spec R$, $E=\Spec R[ (a_i\alpha+b_i)^{-1}]$, where $a_i,b_i \in R$. We have $B'=\Spec R^+$, where $R^+ = R[(a_ib_j-a_jb_i)^{-1})]$, and 
$E'=B' \times_{B} E$.
The isomorphism $(i)$ was proved in \cite{BrENS}, corollary 3.24, and is
 % and depends on the choice of a (tangential) basepoint on $E'$.
dual to  theorem 3.1 in \cite{FR}.
% and  Chen's $\pi_1$-de Rham theorem, which states that $b(X)$ is dual to  the prounipotent completion of the fundamental group of $X(\C)$.  
To prove $(ii)$, note that the natural  map $H^1(E;\Q) \rightarrow H^1(F;\Q)$ is split by lifting  each logarithmic form on the fiber to $d\log(a_i\alpha+b_i)$, and this gives a map $H^1(E;\Q)\rightarrow  H^1(B;\Q)$. In general,
 it follows from the Eilenberg-Moore spectral sequence that the  associated graded of $b(X)$ (with respect to the length filtration) is a certain subspace of the tensor algebra $T(H^1(X;\Q))$  on $H^1(X;\Q)$.  
 %Since $F$ is linear and one-dimensional, 
%we have $b_{B'}(F) \cong T(H^1(F))$.
Thus 
$$\gr^* b(E) \subset \gr^* b(E') \cong  b(B') \otimes_{\Q} b_{B'}(F)$$
But the map $b(E) \rightarrow b(E') \rightarrow b(B')$ lands in $b(B)$, since it is induced by  $H^1(E;\Q)\rightarrow  H^1(B;\Q)$, and    $\gr^* b(B)= \gr^* b(B') \cap T(H^1(B))$, in $T(H^1(B'))$.
This implies  $(ii)$.
\end{proof} 
By $(\ref{LRgamma})$, this motivates the following inductive definition of  the spaces $V(R_i)$. % (with $E=\Spec R_i$, $B=\Spec R_{i+1}$) 
% inductively as follows. First, since $\Spec R_N\cong \Pro^1\backslash...$, we merely set $V(R_N) = ..$.

\begin{defn} For every $1\leq i \leq N$ define $V(R_i)$ inductively by the formula:
\begin{equation}\label{VRidef}
V(R_i) =  V_{\Mod_{0,m_i}}(\Mod_{0,m_i+1})  \otimes_{\Q} V(R_{i+1}) \ .\end{equation}
\end{defn}

Thus  $V(R_i)$ is defined to be a tensor product of  shuffle algebras on $m_{i}-1$ elements   $ V_{\Mod_{0,m_i}}(\Mod_{0,m_i+1}) $. The elements of $V(R_i)$ can be thought of  as    multivalued functions on an open subset of $(\Spec R_i)(\C)$. % Most of  the complexity of the period calculation is contained in the following lemma.

\begin{lem}\label{lembetatransfer}  The  map
$\beta_i: \Spec R_{i+1} \rightarrow \Mod^{\dag}_{0,m_i}$ induces  a map:
$$(\beta_i)_*:  V(\Mod^{\dag}_{0,m_i}) \To V(R_{i+1})$$
 \end{lem}
\begin{proof}
By functoriality, $\beta_i$ induces a map:
$$(\beta_i)_*: H^0(\overline{B}(\Omega^*( \Mod^{\dag}_{0,m_i})) ) \To H^0(\overline{B}(\Omega^* (\Spec R_{i+1}) ))\ .$$
It suffices to show that for all $j$,
$$H^0(\overline{B}(\Omega^* (\Spec R_{j}) )) \subset V(R_j)\ ,$$
i.e., our recursive definition of $V(R_j)$ does indeed contain the iterated integrals on $\Spec R_j$.
To see this, consider the  open subscheme  $U_{j+1}\subset \Spec R_{j+1}$  given by
$U_{j+1} =\beta_j^{-1}  (\Mod_{0,m_j})$. Then the open subscheme $W_j\subseteq \Spec R_j$ given by
$$W_j=\gamma_j^{-1}(\Mod_{0,m_{j+1}}   \underset{\Mod_{0,m_j}}{\times} U_{j+1})\ , $$
fibers linearly over $U_{j+1}$. 
The lemma follows from the previous proposition on setting $E=\Spec R_j$, $B=\Spec R_{j+1}$, $B'=U_{j+1}$, $E'=W_j$.
%we see that   
% $V(R_j)\otimes\C$  contains   
%the image of $ H^0(\overline{B}(\Omega^* (\Spec R_{j}) ))$. % the zeroth cohomology group of the reduced bar construction on $\Spec R_i$.
%In this case, since everything is defined over $\Q$, the map in the proposition respects the $\Q$-structures too.
\end{proof}

Therefore  lemma   \ref{lembetatransfer} defines a  transfer map
\begin{equation} \label{idtimesbeta}
 id\otimes(\beta_i)_* :V(R_{i+1})\otimes_{\Q} V(\Mod^{\dag}_{0,m_i}) \To  V(R_{i+1})
 \end{equation}
This is the only data we will require to compute the periods in \S \ref{subsectCalcFeynInt}.

\subsection{Tangential base points and rational structures}
In order to compute periods, the maps induced by $\beta_i,\gamma_i$ on the iterated integrals should be defined over   $\Q$.
%\footnote{Since we require a verifiable criterion in terms of the  polynomials $\Psi^{I,J}_K$, it is better to define a  tangential base point on  $\Spec R_i$ and demand that $\beta_i$ %and $\gamma_i$ preserve tangential basepoints.}
%We take an analytic point of view, bearing  explicit computations in mind.  
%
Therefore  we must first define a  $\Q$-structure on the algebra of iterated integrals on $\Spec R_i(\C)$, which follows from a choice of 
tangential base-point on $\Spec R_i$.  For this, it is natural to use 
the  Schwinger coordinates $\alpha_i,\ldots,\alpha_N$, as follows.
 Recall that there are linear functions $f_k=a_k \alpha_i+b_k$ of $\alpha_i$ such that: 
 $$R_i = R_{i+1}[\alpha^{\pm 1}_i, f_k^{-1},\  1\leq k\leq M_k] \ .$$
  It follows from \S3.5 in \cite{BrENS}  that $R_i$ inherits a tangential base-point by induction. Set
$$L(R_i ) = L(R_{i+1}) \otimes_{\Q} \gamma_i^{*}\big( L_{\Mod_{0,m_i}}(\Mod_{0,m_i+1})\big)\ ,$$
by induction. 
%where $L_{\Mod_{0,m_i}}(\Mod_{0,m_i+1})$ is the $\Q$-structure on the $\C$-algebra  of hyperlogarithms  on the fiber of $\Mod_{0,m_{i+1}}\rightarrow \Mod_{0,m_i}$,  %equipped with the tangential basepoint  corresponding  to  the image of $\alpha_i=0$  with direction  $(\gamma_i)_*({\partial \over \partial \alpha_i})$.
%viewed as multivalued functions on an open subset of  $(\Spec R_i)(\C)$.\
The tangential basepoint means that the  elements of $L(R_i)$  are regularized  by letting $\alpha_i \rightarrow 0$, $\ldots$, $\alpha_N\rightarrow 0
$ in that order. 
 Concretely, the functions  in $L(R_i)$ are $\Q$-linear combinations  of products of elements in $L(R_{i+1})$, with functions
$$f(\alpha_i,\ldots, \alpha_N)= \sum_k \log^k(\alpha_i) f_k(\alpha_i,\ldots, \alpha_N)\ , $$
where $f_k\in \gamma_i^{*}\big( L_{\Mod_{0,m_i}}(\Mod_{0,m_i+1})\big)$ is identically zero on $\alpha_i=0$. 
 This defines a $\Q$-structure on a space of  iterated integrals on $\Spec R_i$, {\it i.e.}  the  left-hand side of the  diagram of lemma \ref{lembetagammadef}. 
  The spaces of iterated integrals on the  moduli spaces  $\Mod^{\dag}_{0,m_i+1}$ on the right-hand side already have   a canonical $\Q$-structure after tensoring with $\MZV$. More precisely, if
   $L(\Mod^{\dag}_{0,m_i+1})  \subseteq L(\Mod_{0,m_i})$  is the $\Q$-subalgebra of $L(\Mod_{0,m_i})$  (defined earlier, for some tangential basepoint over $\Z$) given by  the homotopy-invariant iterated integrals on $\Mod^{\dag}_{0,m_i}$, then  $L(\Mod^{\dag}_{0,m_i+1})\otimes_{\Q}\MZV$ is well-defined.

\subsubsection{Compactification of $\Mod^{\dag}_{0,m_i+1}$} Recall that $\Mod^{\dag}_{0,m_i+1} \subset (\G_m)^{m_i-2}$ is the complement of a configuration of hyperplanes of the form $t_i=t_j$, $t_i=1$, where $t_i$ are coordinates on each component $\G_m$. There exists a minimal smooth compactification $\overline{\Mod}^{\dag}_{0,m_i}$ such that 
$\Mod^{\dag}_{0,m_i+1}\subset \overline{\Mod}^{\dag}_{0,m_i+1}$ is the complement of a smooth normal crossing divisor. It is defined over $\Z$ and  there is a  map
$\overline{\Mod}_{0,m_i+1} \rightarrow  \overline{\Mod}^{\dag}_{0,m_i+1}$ which blows down certain divisors of $\overline{\Mod}_{0,m_{i+1}}$.
For example, if $\Mod^{\dag}_{0,m_i+1} = (\G_m)^{m_i-2}$, then its minimal compactification is  $(\Pro^1)^{m_i-2}$. There is a stratification on  $\overline{\Mod}^{\dag}_{0,m_i+1}$ obtained by intersecting components of boundary divisors, and its deepest stratum is of dimension 0, i.e., a collection of  points. By standard techniques, one can write down  normal coordinates in the neighbourhood of each such point as quotients of terms $t_i$, $t_i-1$ and $t_i-t_j$ corresponding to the hyperplanes in  $\A^{m_i-2}\backslash \Mod^{\dag}_{0,m_i+1}$.

\subsubsection{Ramification}
In order to determine if the rational  structures on the  algebras of iterated integrals defined previously are  compatible, consider  any  open path 
$$\gamma: (0,\varepsilon]   \To (\Spec R_{i})(\C) $$
whose image lies in the sector $0\ll \alpha_{i} \ll \alpha_{i+1} \ll \ldots \ll \alpha_N\ll 1\ . $ By composing with the map
$\rho : \Spec R_{i} \rightarrow \Mod^{\dag}_{0,m_i+1}  \hookrightarrow \overline{\Mod}^{\dag}_{0,m_i+1}$ we obtain a path  in $(\overline{\Mod}^{\dag}_{0,m_i+1})(\C)$. 
By compactness,   its limit $ \lim_{t\rightarrow 0} (\rho ( \gamma (t)) )$ defines a point $p \in \overline{\Mod}^{\dag}_{0,m_i+1}$.

\begin{defn} \label{unramdef} The map $\rho$ is unramified if
$\rho^*: W^1 L(\Mod^{\dag}_{0,m_i+1}) \rightarrow  W^1 L(R_i)\otimes \C$ is defined over $\Q$,
and
 $p\in \overline{\Mod}^{\dag}_{0,m_i+1}$  lies in a deepest possible boundary stratum.
\end{defn}

\begin{prop}\label{propunram}  If $\rho$ is unramified, then   $\beta_i^* : L(\Mod^{\dag}_{0,m_i})\otimes_{\Q}\MZV \rightarrow L(R_{i+1})\otimes_{\Q}\MZV$ and
 $(\gamma_i^{-1})^*: L_{\Mod_{0,m_i+1}}(\Mod_{0,m_i+1}) \otimes_{\Q} L(R_{i+1}) \rightarrow L(R_i)\otimes_{\Q} \MZV$
are defined over $\Q$.
\end{prop}

\begin{proof}  Every function in $ L(\Mod^{\dag}_{0,m_i})$ lifts to a multivalued function on 
$\overline{\Mod}^{\dag}_{0,m_i}$ with logarithmic singularities along the boundary divisors. By choosing 
appropriate normal coordinates $z_1,\ldots, z_\ell$ which vanish at the point $p$, and taking this as tangential base-point, we see that 
an algebra basis for  $ L(\Mod^{\dag}_{0,m_i})\otimes_{\Q}\MZV$ is spanned by 
$$\log (z_1) \ , \  \ldots\ , \  \log (z_\ell)\ , \   f(z_1,\ldots, z_{\ell})\ ,$$ 
where $f(z_1,\ldots, z_{\ell})\in L(\Mod^{\dag}_{0,m_i})\otimes_{\Q}\MZV$ vanishes at $p$. By assumption, 
$$\beta_i^* (\log z_k)  \in L(R_{i+1})$$
is defined over $\Q$ for $1\leq k\leq \ell $. It suffices to check that $\beta_i^* f(z_1,\ldots, z_\ell)$ is in $L(R_{i+1}).$ But this follows from the fact that 
$p= \lim_{t\rightarrow 0} (\beta_i \circ \gamma (t) )$
and the definition of the $\Q$-structure on $L(R_{i+1})\otimes_{\Q}\C$.  The case $\gamma_i^*$ is similar.
\end{proof}

\begin{rem} One can show that  definition \ref{unramdef} is equivalent to the property that 
$\rho^*: W^1 L(\Mod^{\dag}_{0,m_i+1}) \rightarrow  W^1 L(R_i)\otimes_{\Q} \C$ be defined over $\Q$. 
\end{rem}

Definition \ref{unramdef} can be made explicit. 
Let $L_i$ be the Landau variety of $S$  with horizontal components $f_k= a_k \alpha_i + b_k$, $a_kb_k\neq 0$. %Define
Recall that the map 
$\rho_i: \Spec R_i= (\Pro^1)^{N-i}\backslash \widetilde{L_i} \To \Mod_{0,m_i+1}$ is given in simplicial coordinates by:
$$\rho_i^* (t_{M_i}) = - \alpha_i   {a_{M_i} \over b_{M_i}}  \qquad \hbox{ and } \qquad \rho_i^* (t_k)= {a_{M_i} b_k \over  b_{M_i} a_k}$$
A $\Q$-basis of $W^1 V(\Mod^{\dag}_{0,m_i}) $ is given by  the logs of $t_i,1-t_i$,  and $t_i-t_j$ for certain pairs $i,j$, and therefore definition \ref{unramdef}  requires that their images in $W^1 L(R_{i+1})$:
\begin{equation} \label{logquots} \log \big({a_i \over b_j} \big) \qquad \hbox{ and } \qquad \log \big({a_ib_j-a_jb_i \over a_i b_j} \big)\end{equation}
be defined over $\Q$, and that certain cross-ratios in the same quantities $t_i, 1-t_i, t_i-t_j$ tend to $\{0,1,\infty\}$ as $\alpha_i\rightarrow 0,\ldots, \alpha_N\rightarrow 0$ in that order. By linear reducibility, the arguments of the logarithms $(\ref{logquots})$  factorize into terms which are linear  in $\alpha_{i}$:
$$g= g_0\, \alpha_{i}^{s_0}\prod_{k=1}^m (\rho_k\alpha_{i} +1)^{s_k} \in R_{i+1}$$
where $s_i \in \Z$, and $g_0\in R_{i+2}$ is the leading non-zero term in the Laurent expansion of $g$ at $\alpha_{i}=0$. Thus $\log g\in V(R_{i+1})$ is defined over $\Q$
 if and only if $\log g_0 \in V(R_{i+2})$ is defined over $\Q$. Proceeding by induction, we see that it is enough that
 \begin{equation}\label{ramlimit}
 \lim'_{\alpha_N\rightarrow 0}   \ldots \lim'_{\alpha_{i+2}\rightarrow 0} \lim'_{\alpha_{i}\rightarrow 0}  g = 1 
 \end{equation}
 where  $\lim'_{x\rightarrow 0} h(x)$ denotes the first non-zero term in the Laurent expansion of a rational function $h$ at $x=0$, and where $g$ ranges over  the set
 %\footnote{In \cite{BrCMP} } 
\begin{equation} \label{ramset}
\{[f_k,0],  [f_k,\infty] \hbox{ for all } k\} \cup \{  [f_k,f_l]  \hbox{ for all compatible } f_k, f_l\} \ . 
\end{equation}

\begin{defn} Let $S$ be a linearly reducible set of hypersurfaces.  We say that $S$ is \emph{unramified} if  $(\ref{ramlimit})$  holds for all  $i=1,\ldots, N$
\end{defn}

This implies in particular that $\Spec R_N  \supset  \Pro^1\backslash \{0,1,\infty\}$ is the complement of at most 3 marked points.
If $G$ is of matrix type, then there exists an ordering on the edges of $G$ such that the components of the Landau varieties $L_i$ are zeros of Dodgson polynomials
$\Psi^{I,J}_K$. % and the above criterion simplifies. 
We say that $G$ is \emph{positive} if, for some ordering on its set of edges, every such Dodgson polynomial has only positive coefficients.
% there exists an ordering on the edges of $G$ such that 
%and every auxiliary polynomial $\Psi^{I,J}_K$ which defines a component of a Landau 
%variety of $G$ has only positive coefficients, 

\begin{lem} If $G$ is of matrix type and positive  then $\Psi_G=0$ is unramified.
\end{lem}
\begin{proof} Since $G$ is of matrix type, the terms occuring in  $(\ref{ramset})$ are products of  polynomials $(\Psi^{I,J}_K)^{\pm1}$. By the positivity assumption,
each such $\Psi^{I,J}_K$ is a sum of monomials with coefficient $+1$ and so the limit in  $(\ref{ramlimit})$ is necessarily 1.
\end{proof}  

%\begin{lem} Let $G'$ be obtained from  $G$  by splitting  a triangle or  a 3-valent vertex. Then if $G$ is positive and of matrix type with respect to ... + usual conditions.

%\end{lem}
%\begin{proof} Any auxiliary polynomail $\Psi^{i,j}_k$ where $i,j$ are adjacent  edges is positive.
%\end{proof} 
%It follows that any graph obtained by succesively splitting triangles is unramified.

\begin{thm}  \label{thmvw3positive} If $G$ has vertex width $\leq 3$, then $G$ is of matrix type and positive.

\end{thm}

\begin{proof} We have already shown that vertex width 3 implies matrix type. It suffices to show that $\Psi^{I,J}_{G,K}$ has positive coefficients if $I\cup J\cup K = \{1,\ldots, i\}$,
where $1,\ldots, N$ is the ordering on the edges of $G$.  
%By passing to a minor, we can assume that $I\cap J=K=\emptyset$. 
By theorem 23 of \cite{WD}, there is a universal formula for 3-vertex joins, in terms of the graph obtained by adding a 3-valent vertex to the subgraph spanned by $\{1,\ldots, i\}$ and connecting it  to the three distinguished vertices (see $\S4.6$ of \cite{WD}). Then $\Psi^{I,J}_{G,K}$ is a polynomial in the variables $x,y,z$ corresponding to these three edges, which are themselves spanning forest polynomials of $G$, and therefore have positive coefficients (proposition 38 of \cite{WD}). By proposition \ref{PropGenPsitreeformula} there are very few possibilities, and all are checked to be positive.  
\end{proof}

\begin{rem} One can also consider the  case when the ramification is contained in a set of roots of unity (\cite{BrCMP}), by working throughout with the moduli spaces 
$\Mod_{0,n}^{[\mu]}$, which are finite covers of $\Mod_{0,n}$,  and their ring of  periods  $\MZV^{[\mu]}$, which are values of multiple polylogarithms at these roots of unity.  
\end{rem}

\subsection{Calculation of Feynman integrals}\label{subsectCalcFeynInt} Let $S$ be  linearly reducible and defined over $\Q$, and let 
 $f_0$ be  a multivalued function on $\G_m^N\backslash S$.
 % Suppose that $f_0$ is single-valued 
 %on the interior of the hypercube $[0,\infty]^N$ an
 Consider an integral
 $$I= \int_{[0,\infty]^N}  f_0\, d\alpha_1\ldots d\alpha_N \ .$$
 For this to make sense, we must assume that $f_0$ is single-valued along the domain of integration and 
 that the integral converges.  If $f_0$ has unipotent monodromy and is defined over $\Q$, then it corresponds to an element in $R_0\otimes_{\Q}V(R_0)$,
 and  $I$ can be computed by integrating one variable at a time via the diagram:
  \begin{equation}
\xymatrix{ (\Pro^1)^N\backslash S  \ar[d]^{ \pi_1}    &  \ar[l]\Spec R_0  \ar[d]\\  
(\Pro^1)^{N-1} \backslash L_1 \ar[d] & \ar[l]\Spec R_1 \ar[d] \\
%(\Pro^1)^{N-2} \backslash L_2  & \ar[l]\Spec R_2\\
\vdots & \vdots 
         } 
\end{equation}
This can be done by working in  the rings $V(R_i)$, which correspond to  algebras of unipotent functions on  $(\Spec R_i)(\C)$.
More precisely, we define 
\begin{equation}\label{fidef} f_{i+1} = \int_{0}^{\infty} f_i\, d\alpha_{i+1}\ ,
\end{equation}
and $I=f_{N}$. By theorem \ref{thmsingofint} and remark \ref{remmultivalued}, we know that  the partial  integrals $f_i$  
are multivalued functions on $(\Pro^1)^{N-i} \backslash L_i$, i.e., 
\begin{equation}\label{singinLi} \mathrm{Sing } (f_i ) \subseteq L_i\ .\end{equation}
%and hence on $\Spec R_i$. The key remark is that the the ramification locus of $f_i$ is contained in $L_i$.

\begin{example} The Feynman case is similar.  Let $G$ be a connected graph,  and consider a convergent affine Feynman integral of the form: 
$$I=\int_{[0,\infty]^{N}} {P(\alpha_i, \log \alpha_i, \log \Psi_G)\over \Psi_G^k} \,\delta(\alpha_N=1) d\alpha_1\ldots d\alpha_{N} $$
where $P$ is a polynomial with coefficients in $\Q$, and  $S=X_G$. The integrand defines an element $f_0 \in R_0 \otimes_{\Q}V(R_0)$, and $I=f_{N-1}\big|_{\alpha_N=1}$, where the $f_i$ are given by $(\ref{fidef})$.
\end{example}

\begin{thm}\label{mainthm} Let $S$ be a linearly reducible set of hypersurfaces which is unramified.  Then $I\in \MZV$. 
\end{thm}
\begin{proof} We write
$B(R_i) = R_i \otimes_{\Q} L(R_i) \otimes_{\Q} \MZV$,  equipped with the weight filtration.
The elements of $B(R_i)$ can either be seen as elements of the bar construction or as iterated integrals on $\Spec R_i$, and we will use both points of view interchangeably.
Likewise, let $B_{\Mod_{0,m_i}}(\Mod^{\dag}_{0,m_i+1}) = \Or(\Mod^{\dag}_{0,m_i+1})\otimes_{\Q} L_{\Mod_{0,m_i}}(\Mod_{0,m_i+1}) \otimes_{\Q} \MZV$.

Recall that there is a commutative diagram:
\begin{equation}
\xymatrix{ \Spec R_i  \ar[d]^{ \pi_{i+1}}   \ar[r]^{\sim\qquad \qquad }_{\gamma_i\qquad \qquad} &   \Mod^{\dag}_{0,m_i+1} \textstyle{\underset{\Mod^{\dag}_{0,m_i}}{\times} } \Spec R_{i+1} \ar[d]\\  
\Spec R_{i+1} \ar[r]^{\sim\qquad \qquad }  & \Mod^{\dag}_{0,m_i}  \underset{\Mod^{\dag}_{0,m_i}}{\times}  \Spec R_{i+1} \\
         } 
\end{equation}
We can therefore compute the integrals by travelling around the right-hand side of this diagram and working on the moduli spaces $\Mod_{0,m_i+1}$.
More precisely, we have:
\begin{equation}
\xymatrix{ B(R_i)      & \ar[l]_{\sim\qquad \qquad\qquad}  B_{\Mod_{0,m_i}}(\Mod^{\dag}_{0,m_i+1}) \otimes_{\Or(\Mod^{\dag}_{0,m_i})}   B(R_{i+1}) \\  
B(R_{i+1}) \ar[u]^{ \pi_{i+1}^*} &\ar[l]   B(\Mod^{\dag}_{0,m_i})  \otimes_{\Or(\Mod^{\dag}_{0,m_i})}   B(R_{i+1}) \ar[u]\\
         } 
\end{equation}
The horizontal isomorphism along the top follows from the definition of $L(R_i)$ as a tensor product, and the horizontal map along the bottom 
is given by $(\ref{idtimesbeta})$.  Suppose  by induction that 
 $f_i \in B(R_i)$.  Since $\alpha_{i+1}$ corresponds to the coordinate in the fiber of  $\Mod_{0,m_i+1} \rightarrow \Mod_{0,m_i}$, the 1-form $f_i \, d\alpha_{i+1}$  corresponds via $(\gamma^{-1}_i)^*$ to an element
 $$ \sum_{k} \eta_k \otimes g_k \in  \Omega^1 B_{\Mod_{0,m_i}}(\Mod_{0,m_i+1})  \otimes_{\Or(\Mod^{\dag}_{0,m_i})}   B(R_{i+1})    $$
 where  $\Omega^1B_{\Mod_{0,m_i}}(\Mod_{0,m_i+1}) = \Omega_{\Mod_{0,m_i+1}/\Mod_{0,m_i}}^1\otimes_{\Q} L_{\Mod_{0,m_i}}(\Mod_{0,m_i+1}) \otimes_{\Q} \MZV$.
%
 %
% $\Omega^1 B_{\Mod_{0,m_i}}(\Mod_{0,m_i+1}) =  \Omega_{\Mod_{0,m_i}}^1(\Mod_{0,m_i+1})\otimes_{\Or(\Mod_{0,m_i+1})} B_{\Mod_{0,m_i}}(\Mod_{0,m_i+1}) .$
  We have  
%   By theorem .. we have
$$\int_{\gamma_i(0,\infty)} (\gamma^{-1}_i)^* f_i\, d\alpha_{i+1}  = \sum_k \Big(\int_{\gamma_i(0)}^{\gamma_i(\infty)} \eta_k\Big) g_k$$
which by theorem \ref{thmModint}    lies in 
$$  B(\Mod_{0,m_i})  \otimes_{\Or(\Mod^{\dag}_{0,m_i})} B(R_{i+1})$$
and  defines a multivalued function on  the  subset $U_i= \Mod_{0,m_i}\times_{\Mod^{\dag}_{0,m_i}} \Spec R_{i+1}$ of $\Spec R_{i+1}$. By theorem \ref{thmsingofFeynInt}  or $(\ref{singinLi})$, its ramification locus is contained in $L_{i+1}$, and therefore has trivial monodromy around components of $\Spec R_{i+1} \backslash U_i$. In other words, it 
 must actually lie in the subspace:
 $$ B(\Mod^{\dag}_{0,m_i})  \otimes_{\Or(\Mod^{\dag}_{0,m_i})}  B(R_{i+1})$$
This maps via $id \otimes (\beta_i)^*$ to $B(R_{i+1})$ by proposition \ref{propunram}. Thus we conclude that
$$f_{i+1} = \int_0^{\infty} f_i\, d\alpha_{i+1} \in B(R_{i+1})$$
and has weight at most one greater than $f_i$.
At the penultimate stage we deduce that 
$f_{N-1} \in B(R_{N})$, and hence
$I = f_{N} \in \MZV$.
\end{proof}

%Pointers. Remark on $\Z$ structure in next section. Pole structure and polytopes?

%NB next section notation uses bar construction on dodgsons.

\begin{cor} Let $G$ be a positive graph of matrix type. Then the periods of Feynman integrals associated to $G$ are $\Q$-linear combinations of multiple zeta values.
\end{cor}
\begin{rem} In the case of graphs which are ramified at roots of unity, the same result holds with $\MZV$ replaced with $\MZV^{[\mu]}$ simply by working throughout with $\Mod_{0,n}^{[\mu]}$. In practice, for many Feynman graphs at low loop orders, the ramification occurs at the final step, so one can in fact carry out most of the  computations in $\Mod_{0,n}$. 
 \end{rem}

%\subsection{Integration over polytopes and pole structure}

\newpage

\section{Leading terms and denominators}
We study the  residue $I_G$ of a primitive divergent graph $G$, its most interesting period.
We show that there is a simple iterative way to compute the denominators (or polar singularities) of  its partial integrals, and deduce an upper bound
for its  transcendental weight. We also sketch an integrality result for its coefficients.  

\subsection{Higher graph invariants}
Let  $G$ be a connected graph with edges  $e_1,\ldots, e_N$.

\begin{defn} Recall the definition of the 5-invariant ${}^5\Psi_G(e_1,\ldots, e_5)$ from \S\ref{sectFiveinvariant}. If $D_{x}$ denotes the discriminant with respect to $x$,  for $n\geq 6$ define inductively: $${}^n\Psi_G(e_1,\ldots, e_n) = D_{\alpha_{e_n}} ({}^{n-1}\Psi_G(e_1,\ldots, e_{n-1}))\ .$$
%where $D_{e_n}$ denotes the discriminant with respect to $\alpha_{e_n}$.
\end{defn}
In the generic case  the notation is consistent: i.e., ${}^n\Psi_G(e_1,\ldots, e_n)$ does not depend on the ordering of $e_1,\ldots, e_n$. 
For a general graph,  ${}^n\Psi_G(e_1,\ldots, e_n)$ will be an irreducible polynomial (the `worst' contribution to the Landau variety), % of degree $2(n-4)$.
 but in many cases it will factorize due to Dodgson-type identities.
If, for example, $G$ is linearly reducible, 
then ${}^5 \Psi_G(e_1,\ldots, e_5)$ 
is of the form  $(a \alpha_6+b)(c\alpha_6+d)$,
so the six invariant ${}^6\Psi_G(e_1,\ldots, e_6)=(ad-bc)^2$ is a square 
and  the higher  invariants ${}^n\Psi_G(e_1,\ldots, e_n)$ will  vanish for all $n\geq 7$.
In the linearly reducible (or nearly linearly reducible) case  it is  more interesting, therefore, to take irreducible factors of each ${}^n\Psi_G(e_1,\ldots, e_n)$, and repeatedly take discriminants of these terms. 
This will give  a sequence of polynomials which, as we show below, computes the denominators of the partial Feynman integrals, and yields  information about the residue.
  
\subsection{Unipotent functions and primitives}
Let $n\geq 4$ and consider, as in the previous section, the universal curve of genus $0$ with $n$ marked points, given by the fiber of the forgetful map:
$\Mod_{0,n+1}\rightarrow \Mod_{0,n}$.
We identify the fiber  with $\Pro^1\backslash \Sigma$, where $\Sigma=\{\sigma_1,\ldots, \sigma_n\}$, and $\sigma_i$ are sections of the map above.
Let $x$ be the coordinate on $\Pro^1\backslash \Sigma$, and  assume that $\sigma_1,\sigma_2,\sigma_3$ correspond to $x=0,1,\infty$ respectively. 
Since it is not known whether the ring of multiple zeta values is graded by the weight, we are obliged to  consider $\gr_{\bullet}^W \MZV$, which is 
 the associated graded (for the weight filtration) of the ring of multiple zeta values over  $\Q$. Recall that $L(\Mod_{0,n})$ is a $\Q$-structure on the graded algebra of iterated integrals on $\Mod_{0,n}$
 corresponding to a (canonical) tangential basepoint.
 We  need a more precise version of theorem \ref{thmModint}. % a  simple property of primitives of  unipotent functions on these moduli spaces.
\begin{thm}  \label{propmoduliprim} Let $F(x) \in L(\Mod_{0,n+1})$ of pure weight $m$. Then for  $\sigma \in \Sigma\backslash \{0,1,\infty\}$, 
\begin{equation} \label{intprim1}
\int_{0}^{\infty} {F(x) \over x-\sigma } dx \in L(\Mod_{0,n}) \otimes_{\Q} \gr^W_\bullet \MZV\ ,
\end{equation}
where the right-hand side is pure of  total weight $m+1$. Likewise, for $i\neq j$, 
\begin{equation} \label{intprim3}
\int_{0}^{\infty} {F(x) \over (x-\sigma_i)(x-\sigma_j) } dx \in {1\over \sigma_i-\sigma_j}  L(\Mod_{0,n}) \otimes_{\Q} \gr^W_{\bullet} \MZV\ ,
\end{equation}
is also pure of total weight $m+1$.   By contrast, 
\begin{equation} \label{intprim2}
\int_{0}^{\infty} {F(x) \over (x-\sigma)^2 } dx \in \sum_{\tau \in \Sigma\backslash \infty, \tau\neq \sigma} {1 \over \sigma-\tau} \, L(\Mod_{0,n}) \otimes_{\Q} \gr^W_{\bullet} \MZV\  , 
\end{equation}
where the right-hand side  can be of any weight up to and including $m$.
%$\sigma_{i_0}$ can be infinite.  
\end{thm}
\begin{proof} Decompose $L(\Mod_{0,n+1})$ as a tensor product of algebras of hyperlogarithms ($(6.7)$ in \cite{BrENS}). The function $F(x)$ can be written as  a $\Q$-linear combination of hyperlogarithms in the variable $x$ with singularities in $\Sigma$ and it follows from the definitions that there exists a 
 primitive $G(x)$ of  $(x-\sigma)^{-1} F(x) dx$ which is  in $L(\Mod_{0,n+1})$, and is pure of weight $m+1$.
   It follows from theorem 6.25 in \cite{BrENS}   that the regularized limit
$\Reg_{x=1} G(x) - \Reg_{x=0} G(x)$ is in $L(\Mod_{0,n})\otimes_{\Q} \gr^W_{\bullet} \MZV$, and is pure of total weight $m+1$. This proves $(\ref{intprim1})$. The proof of 
 $(\ref{intprim3})$ follows  by partial fractions, and $(\ref{intprim2})$ follows  by integration by parts  (exact formulae are given in \cite{WD}, \S4).
\end{proof}

Note that in the case $(\ref{intprim2})$ the denominators on the right-hand side can be any descendents (in the sense of \S\ref{sectGenealogy})  of the  denominator $(x-\sigma)^2$ of the integrand. We apply the previous proposition to  partial Feynman integrals.
\begin{cor} \label{lemprim} Let $F$ be an iterated integral  of weight $n$ on the universal curve $\Pro^1\backslash \Sigma$ with coordinate $x$, 
and single-valued along some path from $0$ to $\infty$.  In particular, it has at most logarithmic singularities along $\Sigma$, so the integrals below are well-defined and convergent.
%Let $U$ be an open subset of $\Mod_{0,n}$, and let $\alpha,\beta,\gamma,\delta :U \rightarrow X$ such that $-\alpha^{-1}\beta \in \Sigma$ and  $\alpha\delta-\beta\gamma\neq 0$, 
Then
 there exists a function $\widetilde{F}$ on $\Mod_{0,n}(\C)$ which is unipotent of weight $n+1$ such that: 
\begin{equation}\label{intFprim1}
\int_0^{\infty} {F(x) \over (\alpha x+ \beta) (\gamma x+\delta) } dx = {\widetilde{F} \over \alpha\delta-\beta\gamma} \quad \hbox{ if } \quad \alpha\delta-\beta\gamma\neq 0 \ .
\end{equation}
In other words, the denominator after one integration  is just the square root of the discriminant  of the previous denominator $\sqrt{D_{x}(\alpha x+\beta)(\gamma x+\delta) }$. Similarly,
\begin{equation} \label{intFprim2}
\int_0^{\infty} {F(x) \over (\alpha x+ \beta)^2 } dx \quad \hbox{ and } \quad   \int_0^{\infty} {F(x) } dx \  , \end{equation}
are linear combinations of unipotent functions of weight at most $n$. 
\end{cor}
\begin{proof}
By a change of variables, we reduce to the previous proposition.   Equation $(\ref{intFprim1})$ is precisely  $(\ref{intprim3})$,  and the first equation of 
$(\ref{intFprim2})$ is exactly $(\ref{intprim2})$.  The case $\int_0^\infty F(x) dx$ follows by integration by parts and is left to the reader.
\end{proof}

\begin{defn} In    either case of $(\ref{intFprim2})$ we say that the integral has a \emph{weight drop}.
\end{defn}

\subsection{Initial integrations} \label{sectInitialInt}
Using the above, we can compute denominators in partial Feynman integrals inductively.
Let $G$ be any primitive-divergent graph with $N$ edges, so that the following integral converges: 
$$I_G=\int_{\alpha_N=1} {\prod_{i=1}^{N-1} d\alpha_i \over \Psi^2}\ .$$
 A priori the transcendental  weight is bounded above by $N-1$, but 
the denominator is a perfect square, so we immediately have a weight drop. The only descendents of $\Psi$ ($=\Psi_G$) are $\Psi^1$ and $\Psi_1$.
After a single integration with respect to $\alpha_1$ we have:
$$I_1=\int_{\alpha_N=1} {\prod_{i=2}^{N-1} d\alpha_i  \over \Psi_1\Psi^1}\ ,$$
After a second integration with respect to $\alpha_2$, we had (example \ref{exlandau2}):  %the second stage, we have
$$I_2=\int_{\alpha_N=1} { \log(\Psi^1_2) +\log(\Psi^2_1)-\log(\Psi^{12})-\log(\Psi_{12})  \over (\Psi^{1,2})^2} \prod_{i=3}^{N-1} d\alpha_i$$
The denominator is a perfect square once again, so we get a second weight drop. Recall from example \ref{3redexampleDODGSON}  that we had:
$$I_3 =\int_{\alpha_N=1}\Big( {\Psi^{123} \log \Psi^{123} \over    \Psi^{12,13}\Psi^{12,23}\Psi^{13,23} }-{\Psi_{123} \log \Psi_{123} \over    \Psi^{2,3}_1\Psi_2^{1,3}\Psi_3^{1,2} }
+\underset{\{i,j,k\}}{\sum} {\Psi^i_{jk} \log \Psi^i_{jk} \over    \Psi^{ij,ik}\Psi_j^{i,k}\Psi_k^{i,j} } - {\Psi^{ij}_{k} \log \Psi^{ij}_{k} \over    \Psi^{ij,ik}\Psi^{ij,jk}\Psi_k^{i,j} }\Big) \prod_{i=4}^{N-1} d\alpha_i  \ , $$
where the sum is over all $\{i,j,k\}=\{1,2,3\}$.
Since the weight can increase by at most one at each integration, the transcendental weight is therefore at most $N-3$, which is the generic case.
The denominators are the Dodgson polynomials $\Psi^{i,j}_k$ and $\Psi^{ij,ik}$ for $\{i,j,k\}=\{1,2,3\}$, which are precisely the descendents of $\Psi^{1,2}$.
\begin{lem} The integrand at the fourth stage is a sum of three terms: %with denominators
$$I_4=\int_{\alpha_N=1} \Big({A \over \Psi^{12,34}\Psi^{13,24}}+{B \over \Psi^{14,23}\Psi^{13,24}}+{C \over \Psi^{12,34}\Psi^{13,24}}\Big)  \prod_{i=5}^{N-1} d\alpha_i $$
where $A,B,C$ are unipotent of weight 2.
\end{lem}
\begin{proof} The terms $A,B,C$ can  be computed explicitly, but  it is instructive to see why the denominators are what they are.  Let us illustrate by  decomposing only  the coefficient of $\log \Psi_{123} $ of $I_3$ into partial fractions with respect to $\alpha_4$. It is:
$${\Psi_{123} \over \Psi^{2,3}_1\Psi^{1,3}_2\Psi^{1,2}_3 } =  {[\Psi_{123},\Psi^{1,2}_3]_4 \over [\Psi^{2,3}_1,\Psi^{1,2}_3]_4[\Psi^{1,3}_2,\Psi^{1,2}_3]_4}\Big( {\Psi^{14,24}_3\over
\Psi^{14,24}_3\alpha_4 + \Psi^{1,2}_{34} } \Big) + 2 \hbox{ other terms}$$
$$= { 1 \over \Psi^{13,24}\Psi^{14,23}}\Big( {\Psi^{14,24}_3\over
\Psi^{14,24}_3\alpha_4 + \Psi^{1,2}_{34} } \Big)+  2 \hbox{ other terms} $$
The second line follows from the first by applying identities $(\ref{resid1})$ and $(\ref{resid2})$ and cancelling terms. The result   follows from theorem  \ref{propmoduliprim} after integrating out $\alpha_4$.
\end{proof}
 Note that one can show that any unipotent function of weight $\leq 3$ 
% It is not hard to show that any unipotent function of weight 2
 can be expressed  only in terms of logarithms, dilogarithms and trilogarithms. 
 % One can prove, although it is more subtle, that every unipotent function of weight 3 can be expressed in terms of the trilogarithm,
%and products of dilogarithms and logarithms.
 \begin{cor}
 It follows that at the fifth stage of integration, we have
$$I_5 = \int_{\alpha_N=1}  {F \over {}^5\Psi(1,2,3,4,5)}   \prod_{i=6}^{N-1} d\alpha_i  $$
where $F$ is a unipotent function of weight 3. 
\end{cor}
Thus after five integrations there is a single denominator, which is  the  5-invariant.
Thereafter we can compute the denominators by applying  corollary \ref{lemprim}.

\subsection{Denominator reduction}
We obtain the following reduction algorithm.

\begin{prop}  \label{propdenomred} Suppose that  $P_{n}$  is the denominator at the $n^{th}$ stage. It is a polynomial of degree at most 2 in $\alpha_{n+1}$. 
Suppose that  $P_n$ factorizes into a product of linear factors\footnote{This is necessarily the case if $G$ is linearly reducible} in $\alpha_{n+1}$, i.e., the discriminant $D_{\alpha_{n+1}}(P_{n})$ is a perfect square.  If it is identically zero, 
then there is a drop in the transcendental weight. If it is non-zero then  
the denominator at the $(n+1)^{th}$ stage is 
$$P_{n+1} = \sqrt{D_{\alpha_{n+1}}(P_{n})}\ . $$
\end{prop}

Since we know that the denominator $P_5$ at the fifth stage is just the 5-invariant ${}^5\Psi(1,2,3,4,5)$, this gives an algorithm to compute the denominators at each stage, up to the point where a weight drop occurs (if it occurs at all).
We call this  the \emph{denominator reduction}. This is much easier to compute than the full reduction.

\begin{cor} Suppose that $G$ is  linearly reducible for  an ordering $e_1,\ldots, e_N$ on its edges. % Compute the denominators $P_n$ iteratively for $n\geq 5$ according
%to
The terms $P_n$ occuring in the denominator reduction (proposition $\ref{propdenomred}$) are contained in the linear reduction of $G$. In particular, 
if  for some $n\geq 5$, $P_n$ is irreducible and of degree two in every variable $\alpha_{n+1},\ldots, \alpha_N$, then $G$ is not linearly reducible
with respect to that ordering.
\end{cor}

%\begin{proof}  The numerator is unipotent, so there exists a monodromy operator reducing it to 1. The resulting function has a pole along $P_n$, so it follows 
%that $P_n$ must be contained in the Landau variety.
%\end{proof}

This gives a simple way to show if a graph \emph{fails} to be linearly reducible.  In this case, we say that a denominator of degree  exactly 2 in every (unintegrated) variable is \emph{totally quadratic}.  It is these kinds of considerations  which led to the discovery of the first  non-Tate counter examples (\S\ref{sectcritgraph}).

\begin{defn}  A graph $G$ is \emph{denominator reducible} if there exists an ordering $e_1,\ldots, e_N$ on its edges such that, for all $5\leq n\leq N$,
$$P_n\in \overline{\Q}[\alpha_{n+1},\ldots, \alpha_N] \ .$$ 
In other words, $D_{\alpha_n}(P_{n-1})$ is a perfect square in $\overline{\Q}[\alpha_{n+1},\ldots, \alpha_N]$. 
\end{defn}
Note that  denominator reducibility  is not minor monotone. To rectify this, we can say that a graph $G$ is \emph{strongly denominator reducible}  if it, and all its minors,
are denominator reducible  for the induced  edge ordering.

%In other words,  linear reducibility implies denominator reducibility.

%\noindent The previous corollary implies \begin{example} Two-vertex reducible.
%\end{example}

\begin{example} The smallest primitive divergent graph in $\phi^4$ that exhibits a weight-drop is the 5-loop non-planar graph depicted below.
\begin{figure}[h!]
  \begin{center}
%    \leavevmode
    \epsfxsize=3.5cm \epsfbox{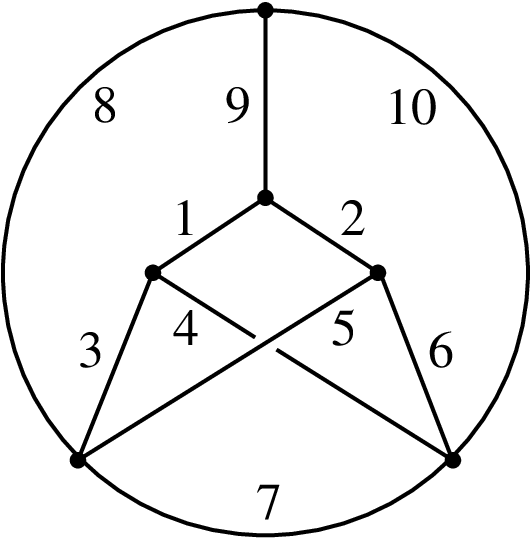}
  \label{5lnp}
  \end{center}
\end{figure}
\noindent After five integrations, the denominator is the five-invariant:
$${}^5\Psi(1,2,3,4,5)=\alpha_{7}(\alpha_{8}\alpha_{9}+\alpha_{8}\alpha_{10}+\alpha_{9}\alpha_{10})(\alpha_{8}\alpha_{6}+\alpha_{10}\alpha_{6}+\alpha_{7}\alpha_{10}+\alpha_6\alpha_{7})\ ,$$
which is split, since the edges $1,3,4$ form a 3-valent vertex. Since this is linear in $\alpha_6$, the denominator at the sixth stage is just the coefficient of $\alpha_6$ in this polynomial.
%$(\alpha_{8}\alpha_{9}+\alpha_{8}\alpha_{10}+\alpha_{9}\alpha_{10})(\alpha_{7}+\alpha_{8}+\alpha_{10}),$ and therefore 
Likewise, at the seventh stage, the denominator is:
$$P_7=(\alpha_{8}\alpha_{9}+\alpha_{8}\alpha_{10}+\alpha_{9}\alpha_{10})(\alpha_{8}+\alpha_{10})$$
and after one more integration, the denominator $P_8$ is the resultant of both factors with respect to $\alpha_8$, which is simply
$\alpha_{10}^2$, a perfect square. Hence we get a weight drop at the 8th stage, and the expected transcendental weight is 10-4=6.
%Now consider the same example and integrate with respect to the edges in the order 4,5,6,7,10,3,8. We have
%${}^5\Psi(4,5,6,7,10) = \alpha_2\alpha_3\alpha_8(\alpha_1\alpha_8-\alpha_3\alpha_9)$, and thus at the 6th stage, the denominator
%is simply  $\alpha_1\alpha_2\alpha_8^2$, which is a perfect square in $\alpha_8$. Thus we can also see the  weight drop at the 7th stage in this new ordering.
\end{example}

\subsection{Split triangles and 3-valent vertices}
The denominator reduction behaves well with respect to split triangles and 3-valent vertices, as expected.
\begin{thm} Let $G$ be a graph containing a triangle (resp. 3-valent vertex), and let $\widehat{G}$ be the graph obtained by subdividing that triangle (resp. vertex).
% as in figure 3. 
Denote  the five (resp. three) distinguished edges in $\widehat{G}$ (resp. $G$)  by $e_1,\ldots, e_5$  (resp. $e_1,e_2,e_3$), and let $e_a,e_b$ be any other two edges in $G$. Then
\begin{equation} P_{e_1,\ldots,e_5,e_a,e_b}(\widehat{G}) = P_{e_1,e_2,e_3,e_a,e_b} (G)\ ,
\end{equation}
and  $\widehat{G}$   has a weight drop if and only if   $G$ does.
\end{thm}
The proof is similar to lemma \ref{lemtrisplit}. A more general version of splitting triangles, and a combinatorial interpretation of the $P_n$'s  will be considered in \cite{WD}. Splitting triangles or vertices should be the first in a family of operations
on graphs which relate the higher invariants ${}^n\Psi_G(e_1,\ldots, e_n)$ to  ${}^m\Psi_{\gamma}(e_1,\ldots, e_m)$, where $\gamma$ is a minor of $G$ and $m< n$.

\subsection{Purity and weight drops} Putting all the previous elements together:

\begin{thm}  \label{thmpurity} Suppose that $G$ is primitive divergent, linearly reducible, has no weight drop,  and is  unramified.  Then the   Feynman integral $I_G$ is of  `pure' transcendental weight $e_G-3$ in $\gr^W_{\bullet} \MZV $.
\end{thm}

\begin{proof} Repeat the proof of theorem  \ref{mainthm}, using    theorem  $\ref{propmoduliprim}$.
\end{proof}
In the case when there is a weight-drop, then we may expect the mixing of weights, by $(\ref{intprim2})$. This was somewhat unexpected, but, indeed,   numerical examples of non-pure  residues $I_G$ have since been found by Schnetz.

\subsection{Towards integrality of the leading term}
The denominator reduction should also give us a handle on the integrality of the Feynman residue. First of all, one can define a $\Z$-structure on $V(\Mod_{0,n})$, and hence on $L(\Mod_{0,n})$. The crucial reason for this is that the Drinfel'd associator has integral coefficients in the multiple zeta values. Then, in theorem \ref{propmoduliprim}, all statements should also hold over $\Z$, although we have not checked that all the proofs in \cite{BrENS} go through unchanged. The reason for this  is  due to the crucial fact that the power in the denominator of $I_G$, and all subsequent integrals, is $2$. This motivates the following definition.

\begin{defn}\label{defnfinaldenom}
 Let $G$ be denominator-reducible with respect to the ordering $e_1,\ldots, e_N$, and  has no weight drop.
  The \emph{final denominator} $P_N$ is a non-zero element $d_G \in \Z$, defined up to sign. It does not  depend on this ordering. 
\end{defn}
Up to checking all the details of the above, we expect  in theorem $\ref{thmpurity}$  that the coefficients of the multiple zeta values for   $I_G$ can be taken to  lie in  ${1\over d_G} \Z$.

\begin{rem} \label{remintegrality}  Almost all primitive divergent graphs up to 7 loops have final denominator $1$.
Oliver Schnetz has kindly verified that in fact all known residues in $\phi^4$ theory  evaluate to integer linear combinations of multiple zeta values (personal communication).
It also should be possible, by refining the argument above, to give a bound for the heights  of the coefficients which occur  by counting the number of terms which can occur in the numerators.
\end{rem}

\newpage

\section{Relative cohomology of linearly reducible hypersurfaces} \label{sectrelativecohom}
%We use the previous results to prove that the cohomology of the complement of a linearly reducible set of hypersurfaces is mixed Tate.
%We recall some terminology from \cite{FR}.
%Consider a fibration 
%$ F\To E \To B$, where $F$ is a smooth open curve.

Let $X$ be a linearly reducible set of hypersurfaces in $(\Pro^1)^N$  defined over $\Q$ and let $B\subset (\Pro^1)^N$ be the coordinate hypercube. Let $L_i=L(S,\pi_i)$ denote the Landau variety of $S=X\cup B$ with respect to 
$\pi_i:(\Pro^1)^N\rightarrow (\Pro^1)^{N-i}$. 
 In particular, the map
$$\pi_i: (\Pro^1)^N\backslash S \backslash \pi_i^{-1} L_i \To (\Pro^1)^{N-i}\backslash L_i $$
has topologically  constant  fibers we denote by $F_i$.

\begin{prop} \label{propunipmonodromy}  $\pi_1((\Pro^1)^{N-i}\backslash L_i)$  acts trivially on $H^*(F_i)$.
\end{prop}
\begin{proof}
It is proved in  \cite{FR} that if $p:E\rightarrow B$ is a fibration with one-dimensional fibers $F$, then $\pi_1(B)$ acts trivially on $H^*(F)$ whenever $p$ admits a section and 
the map $H_1(F)\rightarrow H_1(E)$ is injective. Both conditions hold for  the linear fibrations considered in $\S\ref{sectmapstomod}$.  Therefore for any linear projection $\pi:(\Pro^1)^{k} \backslash T \rightarrow (\Pro^1)^{k-1}\backslash L(T)$ (recall this means that the components of $T$ are of degree at most 1 in the fiber), the higher direct images $R^q \pi_*$ map constant sheaves to constant sheaves.
%(This also follows  by comparison with the case of the universal linear fibration $\Mod_{0,n+1}\rightarrow \Mod_{0,n}$).

Assume by induction that  the statement is true for  $i= k$. Then we can compare
$$\pi_{k+1}:(\Pro^1)^N \backslash S \backslash \pi_{k+1}^{-1}L_{k+1}\To (\Pro^1)^{N-k-1} \backslash L_{k+1}$$
with the diagram
\begin{eqnarray} (\Pro^1)^N \backslash S \backslash \pi_{k}^{-1}L_{k}\overset{\pi_k}{\To} &(\Pro^1)^{N-k} \backslash L_{k}& \nonumber \\
& \cup & \nonumber \\
&
 (\Pro^1)^{N-k}\backslash L_k \backslash \pi^{-1}L(\pi,L_k) & \overset{\pi}{\To} (\Pro^1)^{N-k-1} \backslash L(\pi,L_{k})   
 \nonumber
 \end{eqnarray}
where $\pi: (\Pro^1)^{N-k} \rightarrow (\Pro^1)^{N-k-1}$ is the projection such that $\pi_{k+1}= \pi\circ \pi_{k}$.
By induction hypothesis, $R^q(\pi_{k})_*$ maps constant sheaves to constant sheaves, and likewise $R^s \pi_*$, since it is a linear fibration. The same is true for the restriction maps. It follows from a Grothendieck spectral sequence that the compositum  also maps constant sheaves to constant sheaves.
For all $n\in \N$, $R^n (\pi_{k+1})_*$  maps constant sheaves to locally constant sheaves on $(\Pro^1)^{N-k-1}\backslash L_{k+1}$,
which by the above have trivial monodromy on  the open subset $(\Pro^1)^{N-k-1} \backslash L(\pi,L_{k}) $. 
Therefore  $R^n (\pi_{k+1})_*$ also maps constant sheaves to constant sheaves. This proves the proposition.
\end{proof}

Now suppose that   $P\rightarrow (\Pro^1)^N$ is a blow up of intersections of components of $B$. Let $X'$ be the strict transform of $X$ and  $B'$ the total transform of $B$,
where $B'$ is normal crossing,  as in the case  (\S\ref{sectblowups}) when $X$ is the graph hypersurface of a connected graph. 
We are interested in the mixed Hodge structure:
$$M_S = H^N(P\backslash X', B' \backslash B' \cap X')\ .$$
Let $j:P\backslash X' \backslash B' \rightarrow P\backslash X'$ be the inclusion. Then $M_S=H^N(P\backslash X'; j_{!}\Q)$. 
%This is equal to  the sheaf cohomology $H^N(P, \mathcal{F})$ where CHECK!
%$$\mathcal{F} = i_{B'*} i_{S'!} \Q\ .$$
We wish to apply the Leray spectral sequence to the  composed map 
$$\pi: P\backslash X' \To (\Pro^1)^N \overset{\pi_i}{\To} (\Pro^1)^{N-i}\ .$$
Since  $P$ blows up only over intersections of components of $B$ (whose vertical components are critical), this is still a locally trivial fibration over the complement of $L_i$.
Thus $R^n \pi_* j_!\Q$ is a local system over $(\Pro^1)^{N-i}\backslash L_i$, for $n\in \N$.  Let $P_i, X'_i, B'_i$ denote the fibers of $P,X', B'$ over the generic point 
of $(\Pro^1)^{N-i}$. The stalks of $R^n \pi_* j_!\Q$ are isomorphic to $H^n(P_{i} \backslash X'_{i}, B'_{i} \backslash B'_{i} \cap X'_{i})$.
\begin{cor}\label{corunipmonodromy}  $\pi_1((\Pro^1)^{N-i}\backslash L_i)$        acts unipotently (with respect to the weight filtration) on the restriction of
 $R^n \pi_* j_!\Q$ to $(\Pro^1)^{N-i} \backslash L_i$. 
\end{cor}
\begin{proof}
Let $\{W_j\subset P_i\}_{j\in J}$ be the set of horizontal irreducible components of $B'_i\backslash X' \cap B'_i$ and write $W_A=\cap_{j\in A} W_j$ for $A\subseteq J$, and $W_{\emptyset}=P_i$. The same argument  as in proposition \ref{propunipmonodromy}  shows that 
$\pi_1((\Pro^1)^{N-i}\backslash L_i)$   acts trivially on $H^*(W_A)$ for any $A\subseteq J$. 
The mixed Hodge structure on $H^n(P_{i} \backslash X'_{i}, B'_{i} \backslash B'_{i} \cap X'_{i})$     is computed  from the spectral sequence with $E_1^{pq} = \bigoplus_{|A|=q}H^p(W_A)$.  Since $\pi_1((\Pro^1)^{N-i}\backslash L_i)$  acts trivially on all $E_1^{pq}$  terms, it also acts trivially on 
$\gr_{\bullet}^W H^n(P_{i} \backslash X'_{i}, B'_{i} \backslash B'_{i} \cap X'_{i})$,  and therefore acts unipotently on $H^n(P_{i} \backslash X'_{i}, B'_{i} \backslash B'_{i} \cap X'_{i})$. 
\end{proof}

\subsection{Mixed Tate cohomology}

\begin{thm} \label{thmMixedTate} If $S$ is linearly reducible,  then  $M_S$ is mixed Tate.
\end{thm}
\begin{proof} The proof is by induction on the dimension $N$. If $N=1$, then $S\cup B\subset \Pro^1$  is a finite set of points, and therefore $M_S$ is a sum of Kummer motives.
In the case $N>1$, apply the Leray spectral sequence to the map
$$\pi:P\backslash X' \To (\Pro^1)^N \overset{\pi_{N-1}}{\To} \Pro^1\ .$$
which gives $E^{pq}_2=H^p(\Pro^1; R^q \pi_* j_! \Q) \Rightarrow H^{p+q}(P\backslash X'; j_! \Q)= M_S.$ The map $\pi$ is a locally trivial fibration on the complement of the Landau variety $ \Pro^1\backslash L_{N-1}$. By proposition  $\ref{propunipmonodromy}$, the local system $R^q \pi_* j_!\Q$ on $(\Pro^1)^{N-i} \backslash L_i$ has unipotent monodromy.
 If $\xi$ is the generic point of $\Pro^1$, then $S_\xi = S \cap  \pi_{N-1}^{-1}(\xi)$ is  also linearly reducible, has dimension $N-1$, and   by induction hypothesis  $M_{S_\xi}$ is mixed Tate.
Therefore $R^q\pi_* j_!\Q$ is a unipotent variation of mixed Hodge-Tate structures over $\Pro^1\backslash L_{N-1}$, i.e., 
$\gr^W_{\bullet} R^q\pi_* j_!\Q$ is a  constant Hodge-Tate  variation  on $\Pro^1\backslash L_{N-1}$.

 Now let $x\in L_{N-1}$. By  lemma $\ref{lemdegen}$, $S_x= S\cap \pi_{N-1}^{-1}(x)$ is linearly reducible of dimension $N-1$.
 The stalk of $R^q\pi_* j_!\Q$  at the point $x$ will not necessarily coincide with  $M_{S_x}$, because of possible exceptional divisors in the fiber. But these 
 are also projective spaces, and their intersection with $S$ are linearly reducible also. 
 Thus a  similar inductive argument proves that  $R^q\pi_* j_!\Q$ is mixed Tate over $L_{N-1}$ as well.

 We have shown that $\gr^W_{\bullet} R^q\pi_* j_! \Q$ is mixed Tate and constant on $\Pro^1\backslash L_{N-1}$, and mixed Tate over points in $L_{N-1}$.
 By \S14.4 of \cite{P-S},  there is a mixed Hodge structure on  $H^p(\Pro^1\backslash L_{N-1}, R^q \pi_* j!\Q)$, which is mixed Tate.
 We can  apply the Leray spectral sequence in the category of mixed Hodge structures to conclude that $M_S$ is mixed Tate, since its $E_2$ terms are.
  \end{proof}

\begin{rem}  The previous proof essentially constructs a mixed Tate stratification of  the graph hypersurface complement for linearly reducible graphs.
The largest open stratum is $U=(\Pro^1)^N \backslash X\backslash \pi^{-1} L_{N-1}$. To see that $U$ is mixed Tate, it follows from 
proposition \ref{propunipmonodromy} that $H^*(U) \cong H^*(\Pro^1\backslash L_{N-1})\otimes H^*(F_{N-1})$, and  we know that  $H^*(F_{N-1})$ is mixed Tate by induction hypothesis.
All other strata in $P$, constructed by blowing-up linear spaces  in \S\ref{sectblowups}, are complements of graph hypersurfaces of sub and quotient graphs, which are also linearly reducible,
and hence mixed Tate by induction. So it follows that $P\backslash X'$ has a Tate stratification, and one can deduce that $M_G$ is mixed Tate by a standard argument.
\end{rem}
 
In other situations, it may be the case that only the subquotient of $M_G$ which carries the period is mixed Tate. A more sophisticated discussion of this case would require introducing 
a suitable category of  equivalence classes of framed mixed Tate sheaves. We hope to return to this question in the future.

\subsection{Denominator reduction revisited} \label{sectDRrevisited} We give a cohomological interpretation of the denominator reduction. There are two cases: the non-weight drop  and weight-drop cases.
In both cases, consider a projection $\pi:(\Pro^1)^n\times \Pro^1 \rightarrow  (\Pro^1)^n,$ and let $x$ denote the coordinate in  the fiber. Let $\Omega_n = dx_1\ldots dx_n$, where $x_i$ are coordinates on the base, and let $\Omega_{n+1}=\Omega_n \wedge dx$. 
\subsubsection{General case} Let $R\subset (\Pro^1)^n\times \Pro^1$, where $R=R_1\cup R_2$ has exactly two irreducible components which are of degree one in the fiber, 
  and let $R_x\subset (\Pro^1)^n$ be the discriminant. The inclusion of the open subset
  $$U = (\Pro^1)^{n+1}\backslash (R \cup \pi^{-1} R_x) \overset{i}{\To} (\Pro^1)^{n+1}\backslash R$$
  gives rise to a map
\begin{eqnarray}
H^{n+1} ((\Pro^1)^{n+1}\backslash R)  &\overset{i_*}{\To}  &H^{n+1}(U) \cong H^{n}((\Pro^1)^n\backslash R_x) \otimes H^1(\G_m)  \nonumber \\
\Big[ {\Omega_{n+1} \over (f^1x+f_1)(g^1x+g_1) } \Big] & \mapsto & \Big[{\Omega_n \over f^1g_1-f_1g^1}\Big] \otimes \Big[ {f^1 dx \over f^1x+f_1} - {g^1 dx \over g^1 x+g_1}\Big] \nonumber 
\end{eqnarray}
Since $H^1(\G_m)\cong \Q(-1)$, we can simply write
\begin{equation}\label{istarone}  i_* \Big[ {\Omega_{n+1} \over (f^1x+f_1) (g^1x+ g_1) }\Big] = \Big[ {\Omega_n \over f^1g_1-f_1g^1}\Big] (-1)
\end{equation}

\subsubsection{Weight-drop case}  Let $R\subset (\Pro^1)^n\times \Pro^1$, where $R$ is irreducible  of degree one  in the fiber, and let $B_0: x=0$, $B_{\infty} :x=\infty$, and $R_0=R\cap B_0$, $R_{\infty} = R\cap B_{\infty}$. Write $B^{o} = (B_0\cup B_{\infty}) \backslash (R_0\cup R_{\infty})$, and consider the inclusion
  $$V = (\Pro^1)^{n+1} \backslash (R\cup \pi^{-1} (R_0\cup R_{\infty}))\overset{i}{\To} (\Pro^1)^{n+1}\backslash R$$
  which gives rise to a map on relative cohomology
\begin{eqnarray}
H^{n+1} ((\Pro^1)^{n+1} \backslash R,B^o)  &\overset{i_*}{\To}  &H^{n+1}(V,  B^o) \cong H^{n}((\Pro^1)^n\backslash (R_0\cup R_{\infty})) \otimes H^1(\A^1,\{0,\infty\})  \nonumber \\
\Big[ {\Omega_{n+1} \over (f^1x+f_1)^2 } \Big] & \mapsto & \Big[{\Omega_n \over f^1f_1}\Big] \otimes \Big[ {dy \over (y+1)^2}\Big] \nonumber 
\end{eqnarray}
after changing fiber coordinate  $y= {f^1\over f_1}x$. Since $H^1(\A^1,\{0,\infty\})\cong \Q(0)$, we write:
\begin{equation}\label{istartwo}  i_* \Big[ {\Omega_{n+1} \over (f^1x+f_1)^2 }\Big] = \Big[ {\Omega_n \over f^1f_1}\Big] (0)
\end{equation}
Note that the form ${dy\over (y+1)^2}$ is exact in absolute cohomology.
\\

Now apply this argument to a (primitive-divergent) Feynman integrand 
$$\Big[{\Omega_N\over \Psi^2}\Big] \in H^N((\Pro^1)^N\backslash X_G, B_2 \backslash (B_2\cap X_G))$$
where $B_2$ is the part of the hypercube given only by $\alpha_1,\alpha_3=0,\infty$. 
Applying $(\ref{istarone}), (\ref{istartwo})$ in turn gives a cohomological version of the denominator reduction algorithm. Heuristically, and dropping all $i_*$'s from the notation, we
 get: 
\begin{eqnarray} \label{omegasequence} 
 && \Big[{\Omega_N \over \Psi^2}\Big] \mapsto \Big[{\Omega_{N-1} \over \Psi^1\Psi_1}\Big](0) \mapsto
\Big[{\Omega_{N-2} \over (\Psi^{1,2})^2}\Big](-1)  \mapsto  \Big[{\Omega_{N-3} \over \Psi^{13,23}\Psi^{1,2}_3}\Big](-1)  \mapsto \\
  && \qquad  \mapsto \Big[{\Omega_{N-4} \over \Psi^{13,24}\Psi^{14,23}}\Big](-2) \mapsto   \Big[{\Omega_{N-5} \over {}^5\Psi(1,2,3,4,5)}\Big](-3) \mapsto  \nonumber \\
&  &  \qquad\qquad  \qquad     \ldots \mapsto \Big[{\Omega_{N-k} \over P_k}\Big](2-k) \mapsto \ldots  \nonumber
\end{eqnarray}
where $P_k$ is the $k^{th}$ term in the denominator reduction. Unfortunately, the maps $i_*$ may lose some information, and this calculus only yields information about the Hodge type of the  Feynman integrand after restricting it to some open subset.

More precisely, let $P_k$ be the $k^{th}$ denominator in  the  sequence $(\ref{omegasequence})$, for $k\geq 1$, {\it e.g.},  $P_1=\Psi^1\Psi_1$.  As usual,
let $\pi_i :(\Pro^1)^N \rightarrow (\Pro^1)^{N-i}$ project out the first $i$ Schwinger parameters, and let 
$$Y_G = X_G \cup \pi^{-1}V(P_1) \cup \ldots \cup \pi^{-1}_k V(P_k)\ .$$
The above argument can be used to  compute the position  of  
$$ \Big[{\Omega_N \over \Psi^2}\Big]\in H^N((\Pro^1)^N\backslash Y_G, B_2 \backslash (B_2\cap Y_G))$$
in the Hodge and weight filtrations (see \S\ref{sectCY}).  For example, one can take the denominator reduction of the 8-loop graph computed in  \cite{BrSch}, \S6, 
and  by performing all the reductions and changes of variables therein, one arrives at a Tate twist of  the canonical form of a singular K3 surface, which is  of type $(2,0)$. It is these kind of 
considerations which motivate conjecture \ref{conjTD}.

\begin{rem} This argument, as it stands, is not quite sufficient to deduce general results about the framing in the graph motive $M_G$, but we expect it to be valid in most cases.
 In the special  case of a  linearly reducible non-weight drop graph $G$ with $N$ edges, the argument of  \S\ref{sectrelativecohom} should suffice to construct a framing
$$\Q(3-N) \To \gr^W_{N-6} M_G$$
since  contributions of  lower-dimensional strata only  affect the lower weight-graded pieces  of the graph motive. For the non-Tate counterexamples  of  \S\ref{sectCY}, one needs only
show that the lower-dimensional strata are Tate to deduce that the framing is non-Tate. Again, we expect this to be true for these examples.

Yet  another way  around this problem,  suggested by S. Bloch, is if the denominators $D_k$ all have positive coefficients, for then one can pass to the smaller open subset 
$(\Pro^1)^N\backslash Y_G \hookrightarrow (\Pro^1)^N\backslash X_G$ without affecting the period. But one is still left with the problem that  $(\Pro^1)^N\backslash Y_G$
is only an  open in the full blown-up motive.
%(but one must still control the contributions of the blown up  strata).  
The positivity assumption  is in fact possible for certain non-trivial graphs such as $K_{3,4}$.

In short, the question of framings merits a more detailed analysis.
\end{rem}

\begin{rem} A third interpretation of the denominator reduction is in terms of the point-counts of the graph hypersurface $X_G$ over finite fields $\mathbb{F}_q$. In \cite{BrSch}, we show 
that  for any connected graph $G$ satisfying $2h_G\leq N_G$ and $N_G \geq 5$, 
$$X_G (\mathbb{F}_q) \equiv  \pm \, q^2 \,V(P_k)(\mathbb{F}_q) \mod q^3 $$
where $q$ is any prime power, and $P_k$  for $k\geq 5$, is any term in the denominator reduction of $G$, with respect to any ordering.
\end{rem}

%\vspace{0.1in}
%\newpage
\section{Some critical graphs}\label{sectcritgraph}
We discuss some examples  of graphs in increasing order of complexity.
% some primitive divergent counter examples coming from  $\phi^4$ theory.

\subsection{Some critical minors for non-split 5-invariants}
Let $G$ be a connected  simple graph, and $i,j,k,l,m\in E_G$. Recall that  the splitting of the five-invariant ${}^5\Psi_G(ijklm)$
is a minor-monotone property, and this occurs if $\{i,j,k,l,m\}$ contains a triangle or 3-valent vertex by lemma \ref{lemtrisplit}.

% can only be of degree $2$  in variables $\alpha_n$ which correspond to edges which do not form a 3-valent vertex or
%a triangle with any two  edges of $\{i,j,k,l,m\}$. 
\begin{lem} \label{lem5invlinearterms} Let $i,j,k,l,m,n$ be any 6 edges in a graph $G$. If $i,j,n$ forms a triangle, then $\Psi(ijklm)$ is  divisible by $\alpha_n$. % (i.e., the constant term in  $\alpha_n$ vanishes)
 If $i,j,n$ forms a 3-valent vertex, then
$\Psi(ijklm)$ is of degree at most 1 in $\alpha_n$. % (i.e., the coefficient of $\alpha_n^2$ vanishes).
\end{lem}
\begin{proof}
Suppose that $i,j,n$ forms a triangle. Then $i,j$ forms a loop in the quotient graph $G\q n$.  By lemma \ref{lemtrisplit}, we have
${}^5\Psi_G(ijklm)\big|_{\alpha_n=0}={}^5\Psi_{G\q n}(ijklm)=0.$
It follows that $\alpha_n$ divides ${}^5\Psi_G(ijklm)$. The other case is similar.
\end{proof}
Similar results hold for the higher invariants  too. The smallest  simple graphs with an irreducible 5-invariant are $K_5$ (fewest vertices) and $K_{3,3}$ (fewest edges):
 \begin{figure}[h!]
  \begin{center}
%    \leavevmode
    \epsfxsize=9.0cm \epsfbox{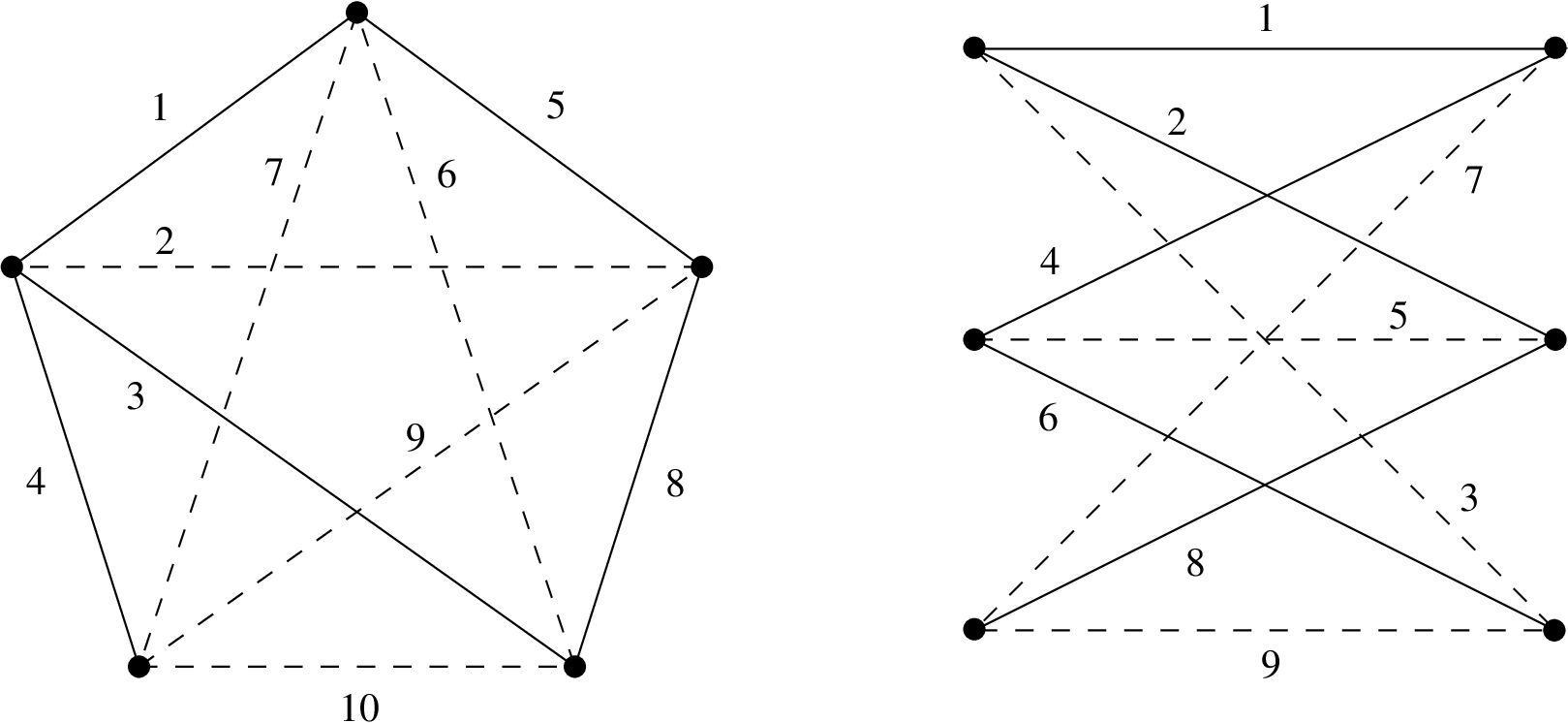}
  \label{figure3}
   \put(-270,10){$K_5$}
  \put(-130,10){$K_{3,3}$}
% \caption{}
  \end{center}
\end{figure} 

\noindent
%We will show that if a subset of five edges of any graph contains a triangle or a 3-valent vertex, then the corresponding five-invariant necessarily splits. 
With these numberings, one can check that:
\begin{eqnarray} \label{nonsplit5}
{}^5\Psi_{K_5}(1,3,4,5,8)&= & \alpha_2\alpha_6\alpha_7\alpha_{10}(\alpha_6\alpha_9^2+\alpha_9\alpha_2\alpha_6+\alpha_2\alpha_9\alpha_{10}+\alpha_2\alpha_7\alpha_9+\alpha_2\alpha_7\alpha_{10}) \nonumber \\
{}^5\Psi_{K_{3,3}}(1,2,4,6,8)&=& \alpha_5\alpha_9^2+\alpha_3\alpha_5\alpha_9+\alpha_5\alpha_7\alpha_9+\alpha_3\alpha_5\alpha_7 -\alpha_3\alpha_7\alpha_9 \nonumber
 \end{eqnarray}
%
%
%I%n the case of $K_5$, this only leaves two possibilities, up to symmetries of $K_5$.
% The first is indicated above (solid edges). With this numbering,  one can check that:
%$${}^5\Psi_{K_5}(1,3,4,5,8)= \alpha_2\alpha_6\alpha_7\alpha_{10}(\alpha_6\alpha_9^2+\alpha_9\alpha_2\alpha_6+\alpha_2\alpha_9\alpha_{10}+\alpha_2\alpha_7\alpha_9+%\alpha_2\alpha_7\alpha_{10})\ , $$
%which is irreducible.
The only other sets of 5 edges (up to symmetries) which do not contain a triangle or 3-valent vertex   are  $1,4,5,8,10$  (for $K_5$), and $1,2,4,5,9$ (for $K_{3,3}$), and these
give rise to 5-invariants which do not split, but factorize into   terms of degree at most $(1,\dots,1)$ nonetheless. %
% is the outer pentagon, 1,5,8,10,4. In this case, one verifies that the 5-invariant does not split, but  factorizes into linear terms nonetheless:
%\begin{eqnarray}
%{}^5\Psi_{K_5}(1,5,8,10,4) &= & \alpha_2\alpha_3\alpha_6\alpha_7\alpha_9(\alpha_2\alpha_6+\alpha_2\alpha_7+\alpha_6\alpha_9+\alpha_3\alpha_9+\alpha_3\alpha_7)  %\nonumber \\
%{}^5\Psi_{K_{3,3}}(1,2,4,5,9)& = &\alpha_3\alpha_6\alpha_7+\alpha_3\alpha_6\alpha_8-\alpha_3\alpha_7\alpha_8-\alpha_6\alpha_7\alpha_8. \nonumber
%\end{eqnarray}
%It follows from this counter-example that a totally  linear graph  is not necessarily of  matrix type.
%Now consider the graph $K_{3,3}$, with the numbering shown.  Again, there are only two configurations of 5 edges which do not span a 3-valent vertex modulo symmetries of %$K_{3,3}$. The first (indicated above) gives:
%$${}^5\Psi_{K_{3,3}}(1,2,4,6,8)= \alpha_5\alpha_9^2+\alpha_3\alpha_5\alpha_9+\alpha_5\alpha_7\alpha_9+\alpha_3\alpha_5\alpha_7 -\alpha_3\alpha_7\alpha_9\ ,$$
%which is irreducible. The second  gives rise  a non-split 5-invariant which  is linear all the same:
Since $\{K_5,K_{3,3}\}$ are
 the critical minors for the set of planar graphs (Wagner's theorem) we obtain:
\begin{cor} Every non-planar graph contains an irreducible 5-invariant.
\end{cor}
Unfortunately, there also exist planar graphs with irreducible 5-invariants.
Consider the   graph $G$ with seven vertices, and edges $e_1,\ldots, e_{12}$ given by: 
$$ \{1, 2\}, \{2, 3\}, \{3, 4\}, \{4, 5\}, \{5, 6\}, \{6, 7\}, \{7, 1\}, \{1, 3\}, \{1, 4\}, \{2, 6\}, \{3, 5\}, \{4, 7\}\ ,$$
where $\{i,j\}$ denotes an edge connecting vertex $i$ and vertex $j$.
It is planar  primitive divergent  with six loops. One can check that the five-invariant ${}^5\Psi_G(e_1,e_2,e_3,e_4,e_{6})$ is irreducible, and quadratic in the variables $e_{7} ,e_9,e_{12}$. All other planar primitive divergent graphs at 6 loops or less have five-invariants which either split or factorize into terms which are  linear in every variable.
%Note that if the five-invariant of a simplification of graph splits, it does not imply that the five-invariant of the original graph splits.
%  do not behave well with respect to simplification.
%Only need LOCAL 5-invariants.
%For linear reducibility, however, we only require the 5-invariants to be 

\subsection{Linear reducibility in $\phi^4$ up to 6 loops}
It turns out that all the above examples are in fact in linearly reducible, despite having a non-trivial 5-invariant.

\begin{lem} The vertex-width 4  graphs $K_5$ and  $K_{3,3}$   are linearly reducible.
\end{lem}

The proof is by calculation, made plausible by the fact that the non-split 5-invariants  above are of degree 1 in almost all the Schwinger parameters.
The single case $K_{3,3}$ is enough  to prove that all primitive divergent graphs in $\phi^4$ theory up to six loops (with the sole
exception of $K_{3,4}$ considered below) are linearly reducible, since by applying  %theorem \ref{startriangleinduction} 
corollaries \ref{cortri3vertex} and \ref{corminorinduction} 
we either reduce to 
the trivial graph or  $K_{3,3}$.  
%
%
%all graphs up to 4 loops have vertex width 3. The first graph of vertex width 4 is the non-planar graph a 5 loop level. By corollary \ref{corminorinduction}, we
%can easily reduce it to a minor, which is either of the form $K_{3,3}$ or $K_5$ minus an edge. In fact, all graphs at the 6 loop level which are planar reduce to a point by minor induction, and those with crossing number 1 likewise reduce to the two cases $K_{3,3}$ and $K_5$.
\begin{cor} All primitive-divergent graphs in $\phi^4$ theory  up to 6 loops, bar $K_{3,4}$, are linearly reducible (this is the main result of \cite{BrCMP}).
\end{cor}
%This result explains the a most analytic calculations of
\subsection{The case $K_{3,4}$}\label{sectK34}
Consider the bipartite graph $K_{3,4}$ with the given  edge labelling. It has vertex width 4.
\begin{figure}[h!]
  \begin{center}
%    \leavevmode
    \epsfxsize=3.8cm \epsfbox{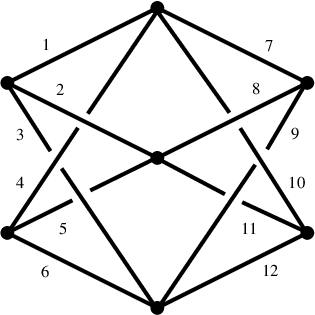}
  \label{figure3}
% \caption{The two situations in the proof of the vertex width $\leq 3$ theorem can be obtained by star-triangle operations. Note that the operation of splitting
 %a triangle (top) and splitting a vertex (bottom) are dual to each other, and preserve primitive divergence. }
  \end{center}
\end{figure}
Since the edges $1,2,3$ form  a three-valent vertex, the first five-invariant ${}^5\Psi(1,2,3,4,5)$ splits. Likewise, because $4,5,6$ form a second 3-valent vertex disjoint from the first, all
five-invariants ${}^5\Psi(i,j,k,l,m)$ where $\{i,j,k,l,m\}\subset \{1,\ldots, 6\}$ must split too. The first problem must occur at the 7th reduction, and indeed, ${}^5\Psi(1,2,4,5,7)$ is irreducible and quadratic. One checks that it is linear in $\alpha_{10}$, so we can pass to $K_{3,4}\q 10$.
Direct computation gives:
\begin{equation}\label{K34bad1}
% D_3(D_6(\Psi_{G\q 10}(1,2,4,5,7))=\alpha_{11}\alpha_{12}(\alpha_8^2\alpha_{12}+\alpha_9^2\alpha_{11}+\alpha_{8}\alpha_{11}\alpha_{12}+\alpha_{9}\alpha_{11}\alpha_{12})\ . 
 {}^7\Psi_{K_{3,4}\q 10}(1,\ldots,7)=\alpha_{11}\alpha_{12}(\alpha_8^2\alpha_{12}+\alpha_9^2\alpha_{11}+\alpha_{8}\alpha_{11}\alpha_{12}+\alpha_{9}\alpha_{11}\alpha_{12})\ . \
 \end{equation}
By the symmetry of $K_{3,4}$ which interchanges $(7,8,9) \leftrightarrow(10,11,12)$, 
we also have:
\begin{equation}\label{K34bad2}
  %D_3(D_6(\Psi_{G\q 7}(1,2,4,5,10))=\alpha_{8}\alpha_{12}^2+\alpha_{9}\alpha_{11}^2+\alpha_{8}\alpha_{9}\alpha_{11}+\alpha_{8}\alpha_{9}\alpha_{12}\ .
  {}^7\Psi_{K_{3,4}\q 7}(1,\ldots, 6, 10)=\alpha_{8}\alpha_{9}(\alpha_{8}\alpha_{12}^2+\alpha_{9}\alpha_{11}^2+\alpha_{8}\alpha_{9}\alpha_{11}+\alpha_{8}\alpha_{9}\alpha_{12})\ .\end{equation}
Thus, after  the $8^{\mathrm{th}}$ stage of reduction, we find that $S_{[1,\ldots, 7,10]}(K_{3,4})$ contains  two distinct quadratic terms (in fact,  one can show that all other terms are Dodgson polynomials), and there remain no more variables with respect to which every term is linear. Thus $K_{3,4}$ is the first example of a primitive-divergent graph in $\phi^4$ which is not linearly reducible.  

Its motive is surely mixed Tate, since (after setting, say $\alpha_{12}=1$)  at the 9th stage we obtain a finite cover of  $\Pro^1\times \Pro^1\backslash L_{9}$, and this surface can continue to be fibered in curves of genus $0$.  Thus  we expect the methods of \S \ref{sectrelativecohom} to generalize to  this, and  similar families (but not all) graphs of vertex width 4.

%We can continue generating increasingly exotic identities between graph polynomials. For example, the identity $(\Psi^{2,3})^2 = \Psi^{2,2}_3\Psi^{3,3}_2 - \Psi^{23,23}\Psi_{23}$
%implies that $2\Psi^{12,13}\Psi^{2,3}_1 = [\Psi^{1,1}_2,\Psi^{2,2}_1]_3 +  [\Psi^{12,12},\Psi_{12}]_3$.
%\end{rem}

\subsection{Splitting}\label{sectSplitting} There is a completely different way to deal with the non-linear reducibility of $K_{3,4}$, which should be more practical for computations.
The idea is that, since $(\ref{K34bad1})$ and $(\ref{K34bad2})$ are not compatible, one should be able to split $T=S_{[1,\ldots, 7,10]}(K_{3,4})$ into two sets $T_1\cup T_2$
such that $T_1$ contains  only Dodgson polynomials and  $(\ref{K34bad1})$, and $T_2$ contains only Dodgson polynomials and  $(\ref{K34bad2})$.
Then $T_1$, $T_2$ are individually linearly reducible (with respect to different orderings).
Analytically, this corresponds to the following, in the case of the residue $I_G$. We  easily check that $K_{3,4}$ is denominator reducible, and that
its denominator at the 8th stage is $P_8 = (\alpha_8 \alpha_{12} - \alpha_9 \alpha_{11})^2$ (in particular, $K_{3,4}$ has weight drop and the expected  weight is $12-4=8$. The period is identified  numerically with $\zeta(5,3) - {29\over 12} \zeta(8)$ \cite{BK}). The idea is to split the integral %\footnote{Ideally, framed mixed Hodge structure}  
 at the 8th stage into a sum of two pieces
$$I_G = \int {f_1\over P_8}+ \int {f_2 \over P_8}\ ,$$
where $f_i$ only contains one of the quadratic terms above. Then, after regularizing if necessary, each piece of the integral is linearly reducible and can be treated with a different order of integration.
We have not carried this out explicitly, but we expect that the idea translates naturally to the context of framed mixed Hodge structures and  generalizes to certain families of graphs with non-trivial local 5-invariants.

%\begin{example} Consider the graph $K_{3,4}$ with the numbering of its edges as shown in figure. Integrate in the order of the edges $1,2,3,4,5,6,7$. At the seventh stage,
%we obtain the denominator:
%$$(\alpha_{11}\alpha_{10}+\alpha_{12}\alpha_{11}+\alpha_{10}\alpha_{12})(\alpha_{10}\alpha_{8}^2+\alpha_{12}\alpha_{8}^2+2\alpha_{10}\alpha_{9}\alpha_{8}+%\alpha_{11}\alpha_{10}\alpha_{8}+$$
%$$\alpha_{12}\alpha_{11}\alpha_{8}+\alpha_{10}\alpha_{12}\alpha_{8}+\alpha_{9}\alpha_{12}\alpha_{11}+\alpha_{11}\alpha_{9}\alpha_{10}+\alpha_{11}\alpha_{9}^2+%\alpha_{10}\alpha_{9}^2+\alpha_{10}\alpha_{9}\alpha_{12})$$
%Taking the resultant of both factors in this polynomial with respect to $\alpha_{11}$ yields the denominator at the 8th stage, which is miraculously a perfect square:
%$$(\alpha_8\alpha_{12}+\alpha_9\alpha_{10}+\alpha_8\alpha_{10})^2\ .$$
%Thus, once again, we have a weight-drop, and the expected transcendental weight is therefore $12-4=8$. The period is 
%$\zeta(3,5)$.
%\end{example}

Thus the first serious obstruction to being non-Tate is likely to come from the  failure of denominator reducibility, which first occurs  in $\phi^4$ at 7 loops.

\subsection{Totally quadratic denominators}
By direct computation,  every primitive-divergent graph in $\phi^4$ at 7 loops is  denominator-reducible, except for a single example. The graph (called $Q48$ in \cite{BK}) with  vertices $1,\ldots, 8$ and edges 
$ \{1,2\}$, $\{2,3\}$, $\{3,4\}$, $\{4,5\}$, $\{5,6\}$, $\{6,7\}$, $\{7,8\}$, $\{8,1\}$, $\{1,3\}$, $\{1,5\}$, $\{2,6\}$, $\{3,7\}$, $\{4,6\}$, $\{4,8\} $,
leads to an irreducible denominator which is of degree two in every variable, for all possible orderings on the set of edges. For example, 
reducing in the order $1,2,3,4,9,13,5,10,11,7$ leads to the  denominator
$$P_{10}=\alpha_6^2\alpha_8\alpha_{14}+\alpha_8^2\alpha_6\alpha_{12}+3\,\alpha_6\alpha_8\alpha_{12}\alpha_{14}-\alpha_{12}^2\alpha_{14}^2\ .$$
This graph is one of two examples at 7 loops whose period remains unindentified by numerical methods, at the time of writing.
By inspection of $P_{10}$ we suspect that one can show, as for $K_{3,4}$, that this graph defines a mixed Tate motive. 
Indeed, by changing variables, one easily shows that $P_{10}$ defines a Tate variety.
This takes us to the limit of all presently known analytic and experimental numerical data.

%Discussion of 8 loops and change integration cycle.

\subsection{Non-trivial higher invariants and non-Tate examples}
It is easy to construct graphs with non-trivial higher invariants ${}^n\Psi(e_1,\ldots, e_n)$, for any $n$,  by considering complete graphs $K_n$ or bipartite graphs $K_{a,b}$, which can in turn be promoted
to $\phi^4$ primitive divergent graphs using  the methods from $\S2$.
%which suggests that they  do not split in general. 
The first  interesting  6-invariants in primitive $\phi^4$ theory occur at the 8 loop level. Indeed, the graphs  obtained by deleting a vertex in the last 5 graphs $P_{8,37},\ldots, P_{8,41}$  of the census  \cite{SchnetzCensus} are not denominator  reducible. Since writing the first versions of this paper, we studied in detail the example $G^3_{8,37}$ 
obtained by deleting  a vertex in  $P_{8,37}$  in  \cite{BrSch}. It has  vertices 1, \ldots, 9 and edges
$$12, 13, 14, 25, 27, 34, 58, 78, 89, 59, 49, 47, 35, 36, 67, 69$$
where $ij$ denotes an edge connecting vertex $i$ and $j$.
Its denominator reduction  is carried out explicitly in  \cite{BrSch}
 and  leads to an irreducible polynomial which is totally quadratic and not of mixed Tate type. Indeed,  we prove   that  its point-counting function is given by a singular K3 surface with complex multiplication by $\Q(\sqrt{-7})$.
 % and thus by the coefficients of the weight 3 modular form $(\eta(q) \eta(q^7))^3$, where $\eta$ denotes the Dedekind $\eta$ function.
  This counter-example has vertex width 4.
%
%\begin{cor} The statement: `$vw(G) \leq 3$ implies $\Mot_G$  mixed Tate' is optimal.
%\end{cor}
%
In \cite{BrSch}, we also exhibit a planar example at 9 loops (but higher vertex-width) which has  the same point-counting function. 
Thus,   the fact that non-Tate examples occur at exactly the point where the methods of this paper fail suggests that linear reducibility is a complete explanation for 
mixed Tate motives in $\phi^4$ theory.
We expect to see the first  non-MZV Feynman amplitudes to start to appear at  8 loops. % integrals.

\subsection{Reduction to Calabi-Yau varieties} \label{sectCY} Let $G$ be a primitively divergent, ordered graph.
The denominator reduction provides a sequence of polynomials 
$$P_5= {}^5\Psi_G(1,2,3,4,5)\ ,\  P_6\ , \ \ldots\ , \ P_k$$
where each $P_i$ has the `Calabi-Yau' property of having degree equal to the number of remaining  Schwinger variables, namely  $2h_G-i$.
Suppose that $G$ does not have weight drop, and is not denominator reducible. Then the denominator reduction terminates with a polynomial $P_k$
which is totally quadratic, and typically in far fewer variables than $\Psi_G$.  One  then has a hope of  counting the points of $P_k$ over $\mathbb{F}_q$ modulo $q$, which is well-defined, since it is independent of $k$ and of the chosen ordering on the $G$  (this is the $c_2(G)$ invariant of \cite{BrSch}). In $\phi^4$ up to 7 loops we find that all graphs are denominator reducible or potentially Tate, but at 8 loops we also  have point counts determined by the coefficients of the following 
modular forms:

\vspace{0.1in}

\begin{center}
\begin{tabular}{|c|c|c|c|}
\hline
4-regular completed graph & Modular Form & Weight & Level \\
\hline
 $  P_{8,39}$ &   $\eta(q)^2 \eta(q^2) \eta(q^4) \eta(q^8)^2$ & 3 & 8 \\ 
 $P_{8,37}$    &  $(\eta(q) \eta(q^7))^3$  & 3 & 7   \\
 $ P_{8,38}$  &   $(\eta(q) \eta(q^5))^4$   &  4  & 5 \\
$ P_{8,41}$ &   $(\eta(q) \eta(q^3))^6$  &  6  &   3 \\
 \hline
\end{tabular}\end{center}
\vspace{0.1in}

The point counts are obtained from the completed graphs on the left  (as listed in \cite{SchnetzCensus}) by deleting a vertex. Conjecturally (\cite{SchnetzFq}, \cite{BrSch}), the choice
of deleted vertex does not affect the $c_2$-invariant. The  entry for $P_{8,37}$ has been proved \cite{BrSch}, but the rest were found experimentally with O. Schnetz by counting points, and are expected to  yield counterexamples to Kontsevich's conjecture. The modular forms should correspond to the framing of the graph motives by  the kind of argument sketched in \S\ref{sectDRrevisited}.  We are therefore led to make the following definitions.

\begin{defn} Let $[\omega_G] $ denote the relative de Rham class of the Feynman integrand in $\Mot_G \otimes_{\Q}{\C} = H^{N-1}(\BP_G \backslash X_G', B'\backslash (B'\cap X_G');\C)$ (see $(\ref{GraphMotive})$).  The mixed Hodge structure $\Mot_G \otimes_{\Q}\C$ has  Hodge and weight filtrations  $F_{\bullet}, W^{\bullet}$.
Let 
$$w(G) = \min  \{k: \ [\omega_G] \in W_k \Mot_G \}  $$
$$\mathfrak{h}(G) =  \max \{k: \  [\omega_G] \in F^k (\Mot_G \otimes_{\Q}\C)\} $$
We define the (motivic) \emph{weight drop} to be the quantity
 $$wd(G) = 2 N_G - w(G) -6 \geq 0\ ,$$
where $N_G$ denotes the number of edges of $G$,  and the \emph{Tate defect} to be 
  $$td(G) =  2 \mathfrak{h}(G) - w(G)\ .$$ 
\end{defn}
The results on the transcendental weights \cite{WD} suggest that $wd(G)$ can be arbitrarily high, and, in particular, increases for every double triangle contained in a graph. If a Feynman diagram is mixed Tate, then $td(G)=0$. By  \S\ref{sectDRrevisited}, we expect the non-Tate counter-examples above  to satisfy $td(G)>0$.
\begin{conj} \label{conjTD} There exists a sequence of  primitive-divergent  graphs $G_i$ in $\phi^4$ such that $td(G_i) \rightarrow \infty$.  Furthermore, the $G_i$ can be taken to be planar.
\end{conj}
\noindent 

Since  our  criteria for graphs  to be mixed Tate are combinatorially   restrictive, we expect that the proportion of   primitive-divergent  $\phi^4$ graphs $G$ satisfying $wd(G)=td(G)=0$ at loop order $n$  tends to zero  as a function of $n$.  In particular,   there should be  asymptotically few multiple zeta value  graphs of maximal weight. 
%This should have implications for the radius of convergence of  the perturbative series of primitive graphs.

\subsection{Towards a classification} 
The following picture summarizes  some of the  minor-monotone classes  of graphs we have considered in this paper.
\begin{figure}[h!]
  \begin{center}
    \leavevmode
    \epsfxsize=10.0cm \epsfbox{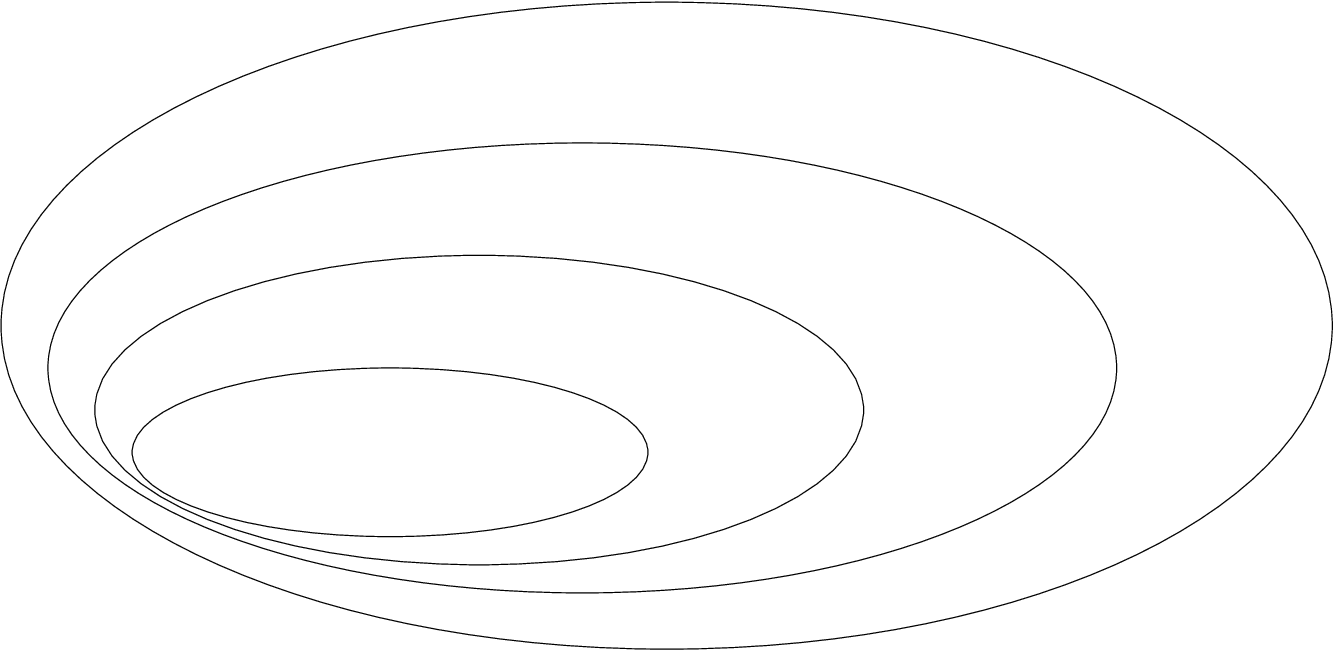}
  \label{2gen}
 \put(-235,40){$vw(G)\leq 3$}
 \put(-230,67){ $G$ of matrix type}
 \put(-220,90){ $G$ linearly reducible}
  \put(-100,60) { $K_{3,3},K_5$ } 
 \put(-210,115){ $G$ strongly denominator  red.}
 \put(-45,60){ $K_{3,4}$ } 
 \put(0,60){ $G^3_{8,37}$ } 
  \end{center}
%  \caption{Just for info: the world of accessible graphs.  Each class is minor monotone. More onion rings to be added }
\end{figure}

In conclusion,  the hypersurfaces of graphs in the  inner onion rings fiber in curves of genus 0 and therefore give mixed Tate motives and multiple zetas.
 This gives a geometric and combinatorial explanation for the numerical  Feynman amplitude computations in the physics literature.
 Beyond  the outer onion ring, we find non-denominator reducible graphs  which give non-Tate counterexamples to Kontsevich's conjecture, and we expect them to have non-MZV  amplitudes.

%However, the  existence of many local irreducible 6-invariants will pose a serious obstruction to the present method, % A general polynomial of degree 4 will fiber in curves
%as we cannot  expect the complement of arbitrary graph hypersurfaces to fiber in curves in genus 0.  We therefore expect 
%non mixed Tate motives,  and the first non-MZV Feynman amplitudes to appear on the fringe of this  %above
% diagram. % of the outer onion rings.

\subsection{Some open questions}
%Direct approach using motivic fundamental groups? Since we know that Bar is tensor product, and Betti too?

%To Add: What is minimal counter-example for Grassmannian?

\begin{enumerate}
\item What is the precise relationship between the combinatorics of  $G$ and the existence of identities between the polynomials $\Psi^{I,J}_{G,K}$ or the higher  invariants
${}^n\Psi_{G\backslash I\q K} (e_1,\ldots, e_n)$?  There is an associated reconstruction problem: to what extent does the vanishing of these polynomials determine the graph?  
\item One can show  by conformal transformations that two graphs $G_1$, $G_2$ have the same residue if adding an apex to the four 3-valent vertices in $G_1$  and $G_2$ gives isomorphic  graphs \cite{SchnetzCensus}. Interpret this using graph polynomials.
%Study the relationship between combinatorics and identities between Dodgson polynomials. Reconstruction problems. Completion
\item Extend our main results to  classes of graphs of vertex width 4  by allowing quadratic terms or by splitting the motives/linear reduction (cf \S\ref{sectK34},\ref{sectSplitting}). %The case vw=4
%\item Splitting with log reg
\item By splitting triangles, one can generate  infinite families of graphs $G$ in $\phi^4$  of vertex width 3. Not counting isomorphisms, these number at least $2^n$. The number of MZVs of weight $n$ grows  approximately like $(4/3)^n$, so there must exist many identities between  the residues $I_G$. Generate such identities
by finding relations between their partial Feynman integrals.
\item Can one compute the motivic coproduct \cite{Go} %(as framed mixed Hodge-Tate structures \cite{Go}) 
for graphs of matrix type?  Certain families of graphs should be highly constrained in the coradical filtration, which would explain the prevalence of certain linear combinations of MZVs in the Feynman integral calculations known to date \cite{BK,SchnetzCensus}.
\item Use the fibrations $(\ref{tower})$  to relate the polynomial point counts of $X_G$ over finite fields to the combinatorics of $G$, for $G$ of matrix type.
%\item Hopf algebra of Dodgsons/Motivic coproduct
%\item motivic galois group.
\item For which (ordered) graphs is it true that its Landau varieties are contained in the zero locus of the Dodgson polynomials $\Psi^{I,J}_K$ and the higher
graph invariants ${}^n\Psi_{\gamma}(e_1,\ldots, e_n)$ as $\gamma$ ranges over the minors of $G$?
\item Sum the total contribution in the perturbative expansion of a family of graphs obtained by, say, splitting triangles.  
\item What proportion of graphs in $\phi^4$ theory have a weight drop or positive Tate defect? What is the physical significance of the `maximal Hodge or weight' part of the perturbative expansion in a quantum field theory? What are the implications for the radius of convergence of the perturbative series generated by primitive graphs?
\item   The zig-zag graphs are known to evaluate to rational multiples of $\zeta(n)$ \cite{BK}, but on the other hand should be integral linear combinations of MZVs.
Therefore, let $\MZV\subset \R$ denote the $\Z$-module spanned by all multiple zeta values $\zeta(n_1,\ldots, n_r)$. For what prime powers $q$ is $ \zeta(n) \in q \MZV$?
\end{enumerate}


\begin{thebibliography}{99}





\bibitem{BB} {\bf P. Belkale, P. Brosnan}: {\it Matroids, motives, and a conjecture of Kontsevich},
Duke Math. J. 116, No.1, 147-188 (2003).

%\bibitem{BB2}{\bf P. Belkale, P.  Brosnan}: {\it  Periods and Igusa local zeta
%functions}, Int. Math. Res. Not. 2003, No.49, 2655-2670 (2003).


\bibitem{BK}{\bf D. Broadhurst, D.  Kreimer}: {\it  Knots and numbers in $\phi^4$ theory to 7 loops
and beyond}, Int. J. Mod. Phys. C 6, 519 (1995).


%\bibitem{B-W}{\bf I. Bierenbaum, S. Weinzierl}: {\it The massless two-loop two-point function},
 %Eur. Phys. J. C32, No.1, 67-78 (2003).

\bibitem{BlochL} {\bf S. Bloch}: {\it Letters to  F. Brown}, dated 28 June 2006 and 28 Sep 2009.

\bibitem{BlochJ} {\bf S. Bloch}: {\it Motives associated to graphs}, Japan J. Math. 2 (2007), 165--196

\bibitem{B-E-K} {\bf S. Bloch, H. Esnault, D. Kreimer}: {\it On motives associated to graph polynomials},
Comm. Math. Phys. 267 (2006), no. 1, 181-225.


%\bibitem{Broad} {\bf D. Broadhurst}: {\it Exploiting the 1,440-fold Symmetry of the
%Master Two-Loop Diagram}, Z. Phys. C32,
%249-253 (1986).

\bibitem{BrSch} {\bf F. C. S. Brown, O. Schnetz}: {\it A K3 in $\phi^4$},  	\url{arXiv:1006.4064v2}

\bibitem{WD} {\bf  F. C. S. Brown, K. Yeats}: {\it Weights of Feynman graphs}, 	\url{arXiv:0910.5429v1},  to appear in Communications in Math. Physics.


%\bibitem{Br1} {\bf  F. C. S. Brown}: {\it P\'eriodes des espaces des modules
%$\overline{\mathfrak{M}}_{0,n}$ et multiz\^etas}, C.R. Acad. Sci.
%Paris, Ser. I {\bf 342} (2006),  949-954.






%\bibitem{Br2} {\bf  F. C. S. Brown}: {\it Single-valued multiple polylogarithms in one variable.},
% %(French. Abridged English version),
% C. R., Math., Acad. Sci. Paris
%338, No.7, 527-532 (2004).


\bibitem{BrENS} {\bf  F. C. S. Brown}: {\it Multiple zeta values and periods of moduli spaces
$\mathfrak{M}_{0,n}$},  Ann. Scient. \'Ec. Norm. Sup., $4^{e}$ s\'erie, t. 42, p. 373-491, (2009),  \url{math.AG/0606419}





\bibitem{BrCMP} {\bf F. C. S. Brown}: {\it The massless higher-loop two-point function},
	Comm. in Math. Physics  287, Number 3, 925-958 (2009)

%\bibitem{BGL} {\bf  F. C. S. Brown, H. Gangl, A. Levin}: {\it Polygons}




\bibitem{Ch1}  {\bf  K.\ T.\ Chen}: {\it Iterated path integrals}, Bull.
Amer. Math. Soc. {\bf 83}, (1977), 831-879.

\bibitem{Doryn} {\bf D. Doryn} :{\it Cohomology of graph hypersurfaces associated to certain Feynman graphs}, (2008)
 \url{arXiv:0811.0402}.

\bibitem{DG} {\bf P. Deligne, A. Goncharov}:{\it  Groupes fondamentaux motiviques de Tate mixte}, Ann. Sci. ƒcole Norm. Sup. (4) 38 (2005), no. 1, 1--56. 


\bibitem{FR}{\bf M. Falk, R. Randell}, {\it The lower central series of a fiber-type arrangement}, Invent. Math. 82 (1985), no. 1, 77--88. 


\bibitem{GKZ} {\bf I. Gelfand, M.  Kapranov, A. Zelevinsky}: {\it Discriminants, resultants, and multidimensional determinants}, BirkhŠuser Boston, MA, (1994) %. x+523 pp. ISBN: 0-8176-3660-9 


%\bibitem{HZ}{\bf R. Hain, S. Zucker}: {\it Unipotent variations of mixed Hodge structure}, Invent. Math. 88 (1987), 83--124


\bibitem{Go}  {\bf A.\ B.\ Goncharov}: {\it Multiple polylogarithms and
mixed Tate motives}, preprint (2001),
\url{arXiv:math.AG/0103059v4}.



\bibitem{Morse} {\bf M. Goresky, R. MacPherson}: {\it  Stratified Morse theory}, 
Ergebnisse der Mathematik und ihrer Grenzgebiete (3), 14. Springer-Verlag, Berlin, (1988) %. xiv+272 pp. ISBN: 3-540-17300-5 


\bibitem{P-S}{\bf C. Peters, J. Steenbrink}: {\it Mixed Hodge Structures}, 
Ergebnisse der Mathematik und ihrer Grenzgebiete (3),  52. Springer-Verlag, Berlin, (2008)

\bibitem{Riv}{\bf R Gurau, V. Rivasseau}, {\it   Parametric Representation of Noncommutative Field Theory},
Comm. in Math.  Physics, vol. 272, no. 3, pp. 811-835 (2007)


\bibitem{IZ}
{\bf C. Itzykson, J.B. Zuber}: {\it Quantum field theory},
International Series in Pure and Applied Physics, McGraw-Hill, New York, (1980) %, xxii, 705 p.

%\bibitem{Ko-Za} {\bf M. Kontsevich, D. Zagier}: {\it Periods}, inMathematics unlimited - 2001 and beyond, Ed. Engquist and Schmidt, pp. 771-808, Springer, 2001.

\bibitem{Pham} {\bf  F. Pham}: {\it Int\'egrales Singuli\`eres},  EDP Sciences,  CNRS \'Editions, Paris, (2005) %. x+225 pp. ISBN: 2-86883-799-9; 2-271-06186-5 

%\bibitem{Re}  {\bf C.\ Reutenauer}: {\it Free Lie Algebras}, London Math.
%Soc. Mono. 7, Clarendon Press, Ox. Sci. Publ., (1993).


\bibitem{SchnetzCensus} {\bf O. Schnetz} {\it Quantum periods: A census of $\phi^4$ transcendentals}, \url{arxiv:0801.2856} (2008)

\bibitem{SchnetzFq} {\bf O. Schnetz} {\it Quantum field theory over $F_q$}, 
\url{arXiv:0909.0905v1} (2009).

\bibitem{Simon}{\bf B. Simon} {\it Orthogonal Polynomials on the Unit Circle, Part 1},
Colloquium Publications, vol. 54 (2004).

\bibitem{Sm} {\bf V. A. Smirnov}: {\it  Evaluating Feynman
integrals}, Springer Tracts in Modern Physics 211. Berlin: Springer.
ix, 247 p. (2004).  %[ISBN 3-540-23933-2]



\bibitem{Sta} {\bf  R. P. Stanley}: {\it Spanning trees and a conjecture of
Kontsevich}, Ann. Comb. 2, No.4, 351-363 (1998).

\bibitem{Stem} {\bf J. Stembridge}: {\it Counting points on varieties over finite fields related to a
conjecture
    of Kontsevich},  Ann. Combin. 2 (1998) 365--385.

%\bibitem{Teiss} {\bf B. Teissier}:{\it The hunting of invariants in the geometry of discriminants}, Real and complex singularities,
% Proc. Ninth Nordic Summer School, Oslo, (1976)  565--678. %Sijthoff and Noordhoff, Alphen aan den Rijn, 1977. 

%\bibitem{KY} {\bf  K. Yeats}: {\it List of primitive graphs in $\phi^4_4$ up to 7
%loops}, personal communication.

%\bibitem{Br1} {\bf  F. C. S. Brown}: {\it P\'eriodes des espaces des modules $\overline{\Mod}_{0,n}$ et valeurs z\^eta multiples}, to appear
%in C. R. Acad. Sci.  Paris, (2006).
%\bibitem{Br2} {\bf  F. C. S. Brown}: {\it Multiple zeta values and periods of moduli spaces
%$\overline{\Mod}_{0,n}$}

\end{thebibliography}
\end{document}